\documentclass[11pt,oneside]{nyjm}
\usepackage{amsmath, amssymb, amsthm, amscd, amsfonts, graphicx, fancyhdr, color, mathrsfs,epstopdf, tikz,caption, tikz-cd, newtxmath}

\usepackage{tikz-cd}
\usetikzlibrary{matrix,arrows,decorations.pathmorphing}
\usetikzlibrary{positioning,automata}
\input epsf.tex
\usepackage{latexsym}
\usepackage{multicol}
\usepackage{subfig}
\usepackage{multirow}
\usepackage{verbatim}
\usepackage[margin=0.8in]{geometry}
\usepackage[T1]{fontenc}
\usepackage[utf8]{inputenc}
\usepackage{amsmath,amsfonts,amssymb,amsthm,amscd,mathrsfs,dsfont}
\usepackage[french,english]{babel}
\usepackage{fancyhdr}
\usepackage{color}
\usepackage{graphicx}
\usepackage[autostyle,english=british]{csquotes}

\newtheorem{theorem}{Theorem}[section]
\newtheorem{lemma}[theorem]{Lemma}
\newtheorem{proposition}[theorem]{Proposition}
\newtheorem{prop}[theorem]{Proposition}
\theoremstyle{definition}
\newtheorem{definition}[theorem]{Definition}

\newtheorem{remark}[theorem]{Remark}
\newtheorem{corollary}[theorem]{Corollary}
\newtheorem{example}[theorem]{Example}
\newtheorem{ex}[theorem]{Example}

\newtheorem{convention}[theorem]{Convention}

\newcommand{\re}{\color{red}} 
\newcommand{\mat}[4]{\begin{bmatrix} #1  & #2 \\ #3 & #4 \end{bmatrix}} 
\newcommand{\matp}[4]{\begin{pmatrix} #1  & #2 \\ #3 & #4 \end{pmatrix}} 
\newcommand{\sm}[4]{\left[\begin{smallmatrix} #1 & #2 \\ #3 & #4 \end{smallmatrix}\right]} 

\newcommand{\bm}{Bouw-M\"oller } 
\newcommand{\mnbms}{\M mn }  
\newcommand{\G}[2]{\ensuremath{\graphg_{#1,#2}}} 
\newcommand{\graphg}{\ensuremath{\mathcal H}} 
\newcommand{\D}[3]{\ensuremath{\mathcal D^{#1}_{#2,#3}}}
\newcommand{\T}[3]{\ensuremath{\mathcal T^{#1}_{#2,#3}}}
\newcommand{\GD}[3]{\ensuremath{\mathcal G}^{#1}_{#2,#3}}
\newcommand{\UD}[2]{\ensuremath{\mathcal U_{#1,#2}}}
\newcommand{\M}[2]{\ensuremath{\textrm S_{#1,#2}}}
\newcommand{\Sec}[3]{\ensuremath{\Sigma^{#1}_{#3}}}
\newcommand{\Tree}[2]{\ensuremath{{\mathcal{T}}_{#1,#2}}}
\newcommand{\Subsec}[3]{\ensuremath{\Sigma^{0}_{#2,#3}(#1)}}
\newcommand{\Subsubsec}[4]{\ensuremath{\Sigma^{0}_{#3,#4}(#1,#2)}}
\newcommand{\Der}[2]{\ensuremath{{\bf D}_{#1}^{#2}}}
\newcommand{\Norm}[2]{\ensuremath{{\bf N}}_{#1}^{#2}}
\newcommand{\F}[2]{\ensuremath{\mathcal{F}_{#1}^{#2}}} 
\newcommand{\FF}[2]{\ensuremath{{\mathcal{F}}_{#1,#2}}}
\newcommand{\FFF}[4]{\ensuremath{{\mathcal{F}}_{#1,#2}(#3,#4)}}
\newcommand{\Lalphabet}{\ensuremath{\mathscr{L}}} 
\newcommand{\LL}[2]{\ensuremath{\mathscr{L}_{#1,#2}}}
\newcommand{\Ar}[2]{\ensuremath{\mathscr{A}_{#1,#2}}}
\renewcommand{\AA}[2]{\ensuremath{\mathscr{A}_{#1,#2}}}
\newcommand{\NA}[2]{3#1#2-2#1-4#2+2}
\newcommand{\I}{I}
\newcommand{\Tr}[3]{\ensuremath{{\bf T}_{#1}^{#2,#3}}} 
\newcommand{\Sub}[4]{\ensuremath{\sigma_{#1,#2}^{#3,#4}}} 
\newcommand{\PSub}[3]{\ensuremath{\sigma_{#1}^{#2,#3}}} 

\newcommand{\gen}[3]{\ensuremath{\mathfrak{g}_{#1}^{#2,#3}}} 
\newcommand{\AD}[2]{\ensuremath{\Psi_{#1}^{#2}}}
\newcommand{\shear}[2]{\ensuremath{s_{#1,#2}}}
\newcommand{\flip}{f}
\newcommand{\diag}[2]{\ensuremath{d_{#1}^{#2}}}
\newcommand{\derAD}[2]{\ensuremath{\gamma_{#1}^{#2}}}
\newcommand{\refl}[3]{\ensuremath{\phi^{#1}_{#3}}}
\newcommand{\perm}[3]{\ensuremath{\pi^{#1}_{#3}}}

\newcommand{\fl}[1]{\ensuremath{f_{#1}}} 

\newcommand{\R}{\mathbb{R}} 
\newcommand{\ZZ}{\mathbb{Z}} 
\newcommand{\RP}{\mathbb{R}\mathbb{P}^1}
\newcommand{\octcdot}{}

\newcommand{\Sing}{\mathcal{S}}
\newcommand{\Aff}{Af\!f}

\newcommand{\be}{\begin{equation}}
\newcommand{\ee}{\end{equation}}

\newcommand{\bes}{\begin{equation*}}
\newcommand{\ees}{\end{equation*}}

\newcommand{\Disk}{\mathbb{D}}

{\theoremstyle{remark}
}

\newcommand{\N}{\mathbb{N}}

\definecolor{red}{RGB}{255,0,0} 
\definecolor{green}{RGB}{0,120,0}
\definecolor{purple}{RGB}{138,43,226}
\newcommand{\gr}{\color{green}}  
\newcommand{\rd}{\color{red}}  

\title[Cutting sequences on  Bouw-M\"oller surfaces]{Cutting sequences on Bouw-M\"oller surfaces:\\ an $\mathcal{S}$-adic characterization \vspace{3mm}\\ Suites de coupage sur les surfaces de Bouw-M\"oller:\\ une caract{\'e}risation $\mathcal{S}$-adique}
\author{Diana Davis, Irene Pasquinelli \\ and Corinna Ulcigrai}

\begin{document}

\begin{abstract}
We consider a symbolic coding for geodesics on the  family of \emph{Veech  surfaces} (translation surfaces rich with affine symmetries) recently discovered by \emph{Bouw} and \emph{M\"oller}. These surfaces, as noticed by \emph{Hooper}, can be realized by cutting and pasting a collection of \emph{semi-regular polygons}. We characterize the set of symbolic sequences (\emph{cutting sequences}) that arise by coding linear trajectories by the sequence of polygon sides crossed. We provide a full characterization for the closure of the set of cutting sequences, in the spirit of the classical characterization of Sturmian sequences and the recent characterization of Smillie-Ulcigrai of cutting sequences of linear trajectories on regular polygons. The characterization is in terms of a system of finitely many substitutions (also known as an \emph{$\mathcal{S}$-adic presentation}), governed by a one-dimensional continued fraction-like map. As in the Sturmian and regular polygon case, the characterization is based on \emph{renormalization} and the definition of a suitable combinatorial \emph{derivation} operator. One of the novelties is that derivation is done in two steps, without  directly using Veech group elements, but by exploiting an  affine diffeomorphism that maps a \bm surface to the \emph{dual} \bm surface in the same Teichm\"uller disk. As a technical tool,  we crucially exploit the presentation of \bm surfaces via \emph{Hooper diagrams}.
\end{abstract}

\begin{otherlanguage}{french}
\begin{abstract}

On considère un codage symbolique des géodésiques sur une famille de surfaces de Veech (surfaces de translation riches en symétries affines) récemment découverte par Bouw et Möller. Ces surfaces, comme l'a remarqué Hooper, peuvent être réalisées en coupant et collant une collection de polygones semi-réguliers. Dans cet article, on caractérise l'ensemble des suites symboliques (``suites de coupage'') qui correspondent au codage de trajectoires linéaires, à l'aide de la suite des côtés des polygones croisés. On donne une caractérisation complète de l'adhérence de l'ensemble des suites de coupage, dans l'esprit de la caractérisation classique des suites sturmiennes et de la récente caractérisation par Smillie-Ulcigrai des suites de coupage des trajectoires linéaires dans les polygones réguliers. La caractérisation est donnée en termes d'un système fini de substitutions (connu aussi sous le nom de présentation $\mathcal{S}$-adique), réglé par une transformation unidimensionnelle qui ressemble à l'algorithme de fraction continue. Comme dans le cas sturmien et dans celui des polygones réguliers, la caractérisation est basée sur la renormalisation et sur la définition d'un opérateur combinatoire de dérivation approprié. Une des nouveautés est que la dérivation se fait en deux étapes, sans utiliser directement les éléments du groupe de Veech, mais en utilisant un difféomorphisme affine qui envoie une surface de Bouw-Möller vers sa surface ``duale'', qui est dans le même disque de Teichmüller. Un outil technique utilisé est la présentation des surfaces de Bouw-Möller par les diagrammes de Hooper.

\phantom{i}\hspace{-1.3em}{\scriptsize A}{\tiny BSTRACT.} We consider a symbolic coding for geodesics on the  family of \emph{Veech  surfaces} (translation surfaces rich with affine symmetries) recently discovered by \emph{Bouw} and \emph{M\"oller}. These surfaces, as noticed by \emph{Hooper}, can be realized by cutting and pasting a collection of \emph{semi-regular polygons}. We characterize the set of symbolic sequences (\emph{cutting sequences}) that arise by coding linear trajectories by the sequence of polygon sides crossed. We provide a full characterization for the closure of the set of cutting sequences, in the spirit of the classical characterization of Sturmian sequences and the recent characterization of Smillie-Ulcigrai of cutting sequences of linear trajectories on regular polygons. The characterization is in terms of a system of finitely many substitutions (also known as an \emph{$\mathcal{S}$-adic presentation}), governed by a one-dimensional continued fraction-like map. As in the Sturmian and regular polygon case, the characterization is based on \emph{renormalization} and the definition of a suitable combinatorial \emph{derivation} operator. One of the novelties is that derivation is done in two steps, without  directly using Veech group elements, but by exploiting an  affine diffeomorphism that maps a \bm surface to the \emph{dual} \bm surface in the same Teichm\"uller disk. As a technical tool,  we crucially exploit the presentation of \bm surfaces via \emph{Hooper diagrams}.
\end{abstract}
\end{otherlanguage}

\maketitle


\paragraph*{\textbf{Mots cl\'e}}
Suites de coupage,
surfaces de translation, 
surfaces de Bouw-Möller, 
renormalisation pour une surface de Veech, 
systèmes S-adiques, substitutions, 
suites de complexité linéaire.
\vspace{5mm}

\paragraph*{\textbf{Short form of the title}}
Cutting sequences on Bouw-Möller surfaces
(Suites de coupage sur les surfaces de Bouw-Möller)

\vspace{5mm}

\paragraph*{\textbf{Math subject classification}}
\emph{Primary}: 37B10 Symbolic dynamics, 37E35 Flows on surfaces
\emph{Secondary}:
11J70 Continued fractions and generaliztions,
37D40 Dynamical systems of geometric origin and hyperbolicity (geodesic and horocycle flows, etc).

\pagestyle{myheadings}


\tableofcontents

\section{Introduction} \label{intro}
{
In this paper we give  a complete characterization  of a class of symbolic sequences that generalizes the famous class of  \emph{Sturmian sequences}, that arises geometrically by coding bi-infinite linear trajectories on \bm surfaces. A gentle introduction for the non-familiar reader is given below, but first we will give a short version of our main result. The {\it \bm} family of translation surfaces is a family of Veech surfaces (see \S \ref{background} for definitions) indexed by two parameters, $(m,n)$, so that the \bm surface $\M mn$  is obtained by identifying parallel sides  of $m$ semi-regular polygons with symmetry of order $n$ (the definition is given in $\S$~\ref{bmdefsection}). Let $\mathscr{A}_{m,n}$ be an alphabet that labels (pairs of identified) sides of these polygons, and let $w \in \mathscr{A}_{m,n}^\mathbb{Z}$ be a sequence that codes a bi-infinite linear trajectory on  $\M mn$ (called a \emph{cutting sequence}). Our main result  characterizes the closure of the set of such sequences in  $\mathscr{A}_{m,n}^\mathbb{Z}$ in terms of a finite family of substitutions as follows (see \S~\ref{sec:substitutionscharacterization} for the definition of a substitution). 

\begin{theorem}\label{firstversion}
For any \bm surface  $\M mn$, there exist $(m-1)(n-1)$ substitutions $\sigma_i=\sigma_i^{m,n}$, $i=1, \dots, (m-1)(n-1)$ on an alphabet $\mathscr{A}'_{m,n}$, and operators ${\bf {T}}_{i}= \Tr i m n$, $i=0,\dots, 2n-1$ from sequences in  ${\mathscr{A}'_{m,n}}^\mathbb{Z}$ to sequences in $\mathscr{A}_{m,n}^\mathbb{Z}$, such that the following characterization holds:

A sequence $w \in \mathscr{A}_{m,n}^\mathbb{Z}$ is in the closure of the set of cutting sequences of bi-infinite linear trajectories on the \bm surface $\M mn$ if and only if there exists a sequence $(s_k)_{k \in \mathbb{N}}$ of indices $1\leq s_k \leq (m-1)(n-1)$, $0\leq s_0 \leq 2n-1$, and a sequence of  letters $a_k \in {\mathscr{A}'}_{m,n}$, such that $w$ can be written as\footnote{The limit is taken along a sequence of \emph{finite} words, for which convergence in $\mathscr{A}_{m,n}^\mathbb{Z}$ to the infinite word $w$ means that the finite words, as $k$ grows, share larger and larger central blocks of letters.}
\begin{equation}\label{limit}
w=  \lim_{k \to \infty} {\bf {T}}_{s_0}\circ \sigma_{s_1} \circ \sigma_{s_2} \circ \dots \circ \sigma_{s_{k}}( a_k).
\end{equation}
\end{theorem}
The above expression is called an  \emph{$\mathcal{S}-$adic expansion} of the word $w$, and this type of characterization, which is well known in the world of word combinatorics, is known as an \emph{$\mathcal{S}-$adic characterization} (see for example \cite{BD} or \cite{fogg}). We emphasize that the notion of $\mathcal{S}-$adic expansions  is used to describe words with \emph{low complexity} since, under some general assumptions, a word with an  $\mathcal{S}-$adic  presentation has \emph{zero entropy} (see Theorem 4.3 in \cite{BD} for a precise statement).  The substitutions $\sigma_i^{m,n}$ and the operators $\Tr i m n$ are explicitly constructed in the paper\footnote{A technical detail is that this $\mathcal{S}-$adic  presentation is possible only at the level of \emph{transitions}, namely  pairs of consecutive letters in $w$; indeed the alphabet $\mathscr{A}'_{m,n}$ on which the substitutions are defined is an alphabet labeling {transitions},  and the operators $\Tr i m n$ (defined in \S \ref{sec:substitutionscharacterization}) simply transform a sequence of transitions into a sequence of letters in $\mathscr{A}_{m,n}$.}.  

Furthermore, in this paper we show that the sequence  $(s_k)_{k \in \mathbb{N}}$ that appears in the $\mathcal{S}-$adic expansion of $w$ is the itinerary of a certain Farey-type map $\FF mn$ on the set of directions. In particular, it is completely determined by knowing the direction  of the trajectory coded by the cutting sequence $w$. This can be used in two  ways: on one hand, given a direction $\theta \in S^1$, one can hence algorithmically produce, by iterating our substitutions, all (finite length blocks of) cutting sequences of trajectories in direction $\theta$. On the other hand, given a sequence $w$ that is the cutting sequence of a trajectory in an unknown direction $\theta$, one can recover the sequence $(s_k)_{k \in \mathbb{N}}$ from $w$ (see the informal discussion after Theorem~\ref{thm:main_characterization}  in this introduction and Section \S~\ref{sec:sectors_sequences} for more details) and hence use the map $\FF mn$ to recover (uniquely, if $w$ is non-periodic) the direction $\theta$ (see Proposition \ref{directionsthm}). 
}

\medskip 

In order to introduce the problem of charactarization of cutting sequences and motivate the reader, we start this introduction by recalling in \S \ref{sec:Sturmian} the geometric construction of \emph{Sturmian sequences} in terms of coding linear trajectories in a square,  and then  both  their characterization using derivation, as described by Series,  and their  $\mathcal{S}-$adic presentation by a system of substitutions. We then recall in \S \ref{sec:polygons} how this type of description was recently generalized by several authors to the sequences coding linear trajectories in regular polygons. Finally, in \S \ref{sec:ourresults} we explain why \bm sequences are the next natural example to consider to extend these symbolic characterizations, and state a simple case of our main result. 

\subsection{Sturmian squences} \label{sec:Sturmian}
\emph{Sturmian sequences} are an important class of sequences in two symbols that often appear in mathematics, computer science and real life. They were considered by Christoffel \cite{C:obs} and Smith \cite{S:note} in the 1870's, by Morse and Hedlund \cite{MH:sym} in 1940 and by many authors since then (see \cite{fogg} for a contemporary account and \cite{Alg} for a historical survey). Sturmian sequences are interesting because of their geometric origin, and are also of interest because they give the simplest non-periodic infinite sequences (see \cite{CH:seq}), having the lowest possible complexity.\footnote{For each $n$ let $P(n)$ be the number of possible strings of length $n$. For Sturmian sequences, $P(n)=n+1$.} They admit the following geometric interpretation:

 Consider an \emph{irrational line}, i.e. a line in the plane in a direction $\theta$ such that $\tan \theta$ is irrational, in a \emph{square grid} (Figure \ref{square1}). As we move along the line, let us record with a $0$ each time we hit a horizontal side and with a $1$ each time we hit a vertical side. We get in this way a bi-infinite sequence of $0$s and $1$s which, up to choosing an origin arbitrarily, we can think of as an element in $\{0,1\}^{\mathbb{Z}}$. The sequences obtained in this way as the line vary among all possible irrational lines are exactly all \emph{Sturmian sequences}.  (For further reading, see the beautiful expository paper by Series  \cite{Series}, and also the introduction of \cite{SU2}.)

Equivalently, by looking at a fundamental domain of the periodic grid, we can  consider a square with opposite sides identified by translations. 
We define a \emph{linear trajectory} in direction $\theta$ to be a path that starts in the interior of the square and moves with constant velocity vector making an angle $\theta$ with the horizontal, until it hits the boundary, at which time it re-enters the square at the corresponding point on the opposite side and continues traveling with the same velocity.  
For an example of a trajectory see Figure \ref{square1}. 
We will restrict ourselves to trajectories that do not hit vertices of the square.  As in Figure \ref{square1}, let us  label by $0$ and $1$ respectively its horizontal and vertical sides.\footnote{Since squares (or, more generally, parallelograms) tile the plane by translation, the cutting sequence of a trajectory in a  square (parallelogram) is the same than the cutting sequence of a straight line in $\mathbb{R}^2$ with respect to a square (or affine) grid.}  The \emph{cutting sequence} $c(\tau)$ associated to the linear trajectory $\tau$ is the bi-infinite word  in the symbols (edge labels, here $0$ and $1$) of the alphabet $\Lalphabet$, which is obtained by reading off the labels of the pairs of identified sides crossed by the trajectory $\tau$ as time increases. 
 
\smallskip
Let us explain now how to characterize Sturmian sequences. One can assume without loss of generality (see \cite{SU2} for details) that $0\leq \theta \leq \pi/2$. If $0\leq \theta \leq \pi/4$, as in Figure \ref{square1}, the cutting sequence does not contain the subword $00$,  and if $\pi/4 \leq \theta \leq \pi/2$, it does not contain the subword $11$. 
Let us say that a word $w \in \{0,1\}^{\mathbb{Z}}$ is \emph{admissible} if either it does not contain any subword $00$, so that $0$s separate blocks of $1$s (in which case we say it it admissible of type $1$) or it does not contain any subword $11$ and $1$s separate blocks of $0$s (in which case we say it it admissible of type $0$).

\begin{figure}[!h] 
\centering
\includegraphics[width=300pt]{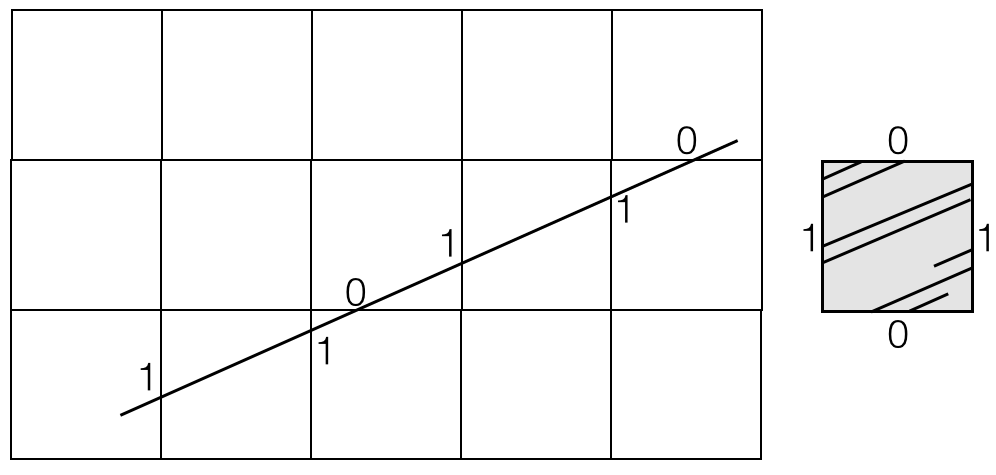}
\begin{quote}\caption{A trajectory with $\theta < \pi/4$ and irrational slope on the square torus \label{square1}} \end{quote}
\end{figure}

Given an admissible word $w$, denote by $w'$ the \emph{derived sequence}\footnote{In this section, we are using the terminology from Series \cite{Series}. } obtained by erasing one $1$ (respectively one $0$) from each block of consecutive $1$'s (respectively $0$'s) if  $w$ is admissible of type $1$ (respectively $0$).  


\begin{ex}\label{sigmaex} A $w$ and its derived sequence $w'$:
\begin{eqnarray}
w &=& \dots 011101111011101110 11110\dots  \nonumber \\ 
w'& = & \dots 0 11\phantom{1}0 111\phantom{1}011\phantom{1}011\phantom{1}0111\phantom{1}0 \dots \nonumber%
\end{eqnarray}
\end{ex}
 
We say (following Series \cite{Series}) that a word is \emph{infinitely derivable} if it is admissible and each of its derived sequences is admissible. It turns out that \emph{cutting sequences of linear trajectories  on the square are infinitely derivable} (see Series \cite{Series} or also the introduction of \cite{SU2}). 
Moreover, the converse 
is \emph{almost} true; the exceptions, i.e. words in $\{0, 1\} ^{\mathbb{Z}}$ which are infinitely derivable and are not cutting sequences such as $\overline{w} = \dots 111101111 \dots $,
 can be explicitly described. 
The space of words has a natural topology that makes it a compact space (we refer e.g.~to  \cite{LM:sym}).  The word $\overline{w}$ is not a cutting sequence, but it has the property that any finite subword can be realized by a finite trajectory. This is equivalent to saying that it is in the \emph{closure} of the space of cutting sequences. In fact, \emph{the closure of the space of cutting sequences is precisely the set of infinitely derivable sequences}.

\smallskip
An alternative related characterization of Sturmian sequences can also be given in terms of substitutions. The definition of substitution is recalled in \S \ref{sec:substitutionscharacterization} (see Definition \ref{def:substitution}). 
Let $\sigma_0$ be   the substitution given by $\sigma_0(0)=0$ and $\sigma_0(1)=10$ and let  $\sigma_1$ be   the substitution given by $\sigma_1(0)=01$ and $\sigma_1(1)=1$. Then, words in Sturmian sequences can be obtained by starting from a symbol (0 or 1) and applying all possible combinations of the substitutions $\sigma_0$ and $\sigma_1$. More precisely, given a Sturmian word $w$ corresponding to a cutting sequence in a direction $0<\theta<\pi/4$, there exists a sequence $(a_i)_{i \in \mathbb{N}} $ with integer entries $a_i \in \mathbb{N}$ such that 

\begin{equation}\label{Sturmian:substitutions_char}
w \in \bigcap_{k \in \mathbb{N}} \sigma_0^{a_0}\sigma_1^{a_1} \sigma_0^{a_2}\sigma_1^{a_3}  \cdots  \sigma_0^{a_{2k}} \sigma_1^{a_{2k+1}}\{0,1\}^{\mathbb{Z}}.
\end{equation}
If $\pi/4<\theta<\pi/2$, the same type of formula holds, but starting with $\sigma_1$ instead of $\sigma_0$. 
Furthermore, $w $ is in the closure of the set of cutting sequence in $\{ 0,1\}^\mathbb{Z}$ if and only if there exists  $(a_i)_{i \in \mathbb{N}} $ with integer entries $a_i \in \mathbb{N}$ such that \eqref{Sturmian:substitutions_char} holds, thus this gives an alternative characterization via substitutions and more precisely  $\mathcal{S}$-\emph{adic} expansions. 


We refer to \cite{BD} for a nice exposition on \emph{$\mathcal{S}$-\emph{adic} systems}, which are a generalization of \emph{substitutive systems} (see also \cite{ferenczi} and \cite{Chap12PF}). While in a substitutive system  one considers sequences obtained as a fixed point of a given substitution and the closure of its shifts,  the sequences studied  in an  $ \mathcal{S}$-\emph{adic system}, 
are obtained by applying products of permutations from a finite set, for example from the set $\mathcal{S} = \{ \sigma_0, \sigma_1\}$ in  \eqref{Sturmian:substitutions_char}. Equivalently, we can write  \eqref{Sturmian:substitutions_char} in the form of a  limit, which is known as $\mathcal{S}$-adic expansion (see \eqref{limit}  
or more in general \cite{BD}). The term $\mathcal{S}-$adic was introduced by Ferenczi in \cite{ferenczi}, and is meant to remind of Vershik \emph{adic} systems \cite{vershik} (which have the same inverse limit structure) where $\mathcal{S}$ stands for \emph{substitution}.

The sequence of substitutions in an $\mathcal{S}$-adic system 
is often 
 governed by a dynamical system, which in the Sturmian case is a one-dimensional map, i.e. the Farey (or Gauss) map (see  Arnoux's chapter \cite{fogg} and also the discussion in \S 12.1 in \cite{Chap12PF}).  

Indeed, the sequence $(a_i)_{i \in \mathbb{N}}$ in \eqref{Sturmian:substitutions_char}  is exactly the sequence of \emph{continued fraction entries} of the slope of the coded trajectory and hence can be obtained as \emph{symbolic coding of the Farey} (or \emph{Gauss}) \emph{map} (see for example the introduction of \cite{SU}, or \cite{fogg}). There is also a classical and beautiful connection with the geodesic flow on the modular surface (see for example the papers \cite{Series}, \cite{Se:mod}, \cite{Se:sym} by Series).  For more on Sturmian sequences, we  also refer the reader to the excellent survey paper  \cite{fogg} by Arnoux.

\subsection{Regular polygons} \label{sec:polygons}
A natural geometric generalization of the above Sturmian characterization  is the question of \emph{characterizing  cutting sequences of linear trajectories in  regular polygons} (and on the associated  surfaces).

Let ${O_n}$ be a regular $n$-gon.  When $n$ is even, edges come in pairs of opposite parallel sides, so  we can identify opposite parallel sides by translations. When $n$ is odd, no sides are parallel, but we can take two copies of $O_n$ and glue parallel sides in the two copies (this construction can also be done for $n$ even). \emph{Linear trajectories} in a regular polygon are defined as for the square. We will restrict our attention to \emph{bi-infinite trajectories} that never hit the vertices of the polygons. If one labels pairs of identified edges with \emph{edge labels} in the alphabet $\Lalphabet_n =\{0,1, \dots , n-1\}$, for example from the alphabet $\Lalphabet_4:= \{ 0,1,2,3 \}$ when $n=8$ (see Figure \ref{intro-oct}),  one can associate as above to each bi-infinite linear trajectory $\tau$ its \emph{cutting sequence}  $c(\tau)$, which is a sequence in $\Lalphabet_n^{\mathbb{Z}}$.  For example, a trajectory that contains the segment  in Figure \ref{intro-oct} will contain the word $10123$.

\begin{figure}[!h] 
\centering
\includegraphics[width=100pt]{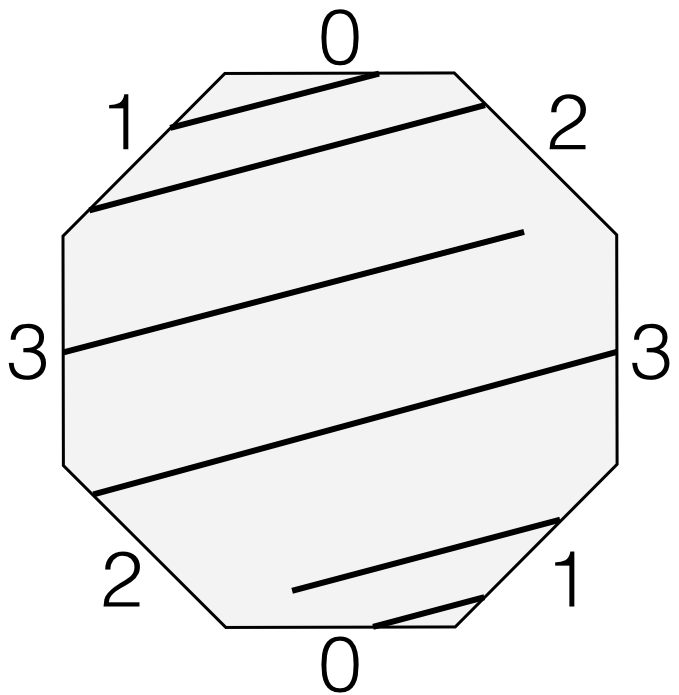} \hspace{0.5in}
\begin{tikzpicture}[node distance=2em]
\node (0) {0};
\node (1) [right=  of 0] {1};
\node (2) [right= of 1] {2};
\node (3) [right= of 2] {3};
\path[thick,-to] 
                 (0) edge [bend left=30] (1)
                 (1) edge [bend left=30] (0)
                 (1) edge [bend left=30] (2)
                 (2) edge [bend left=30] (1)
                 (2) edge [bend left=30] (3)
                 (3) edge [bend left=30] (2)
                 (3) edge [loop right=60] (3);
\end{tikzpicture}
\begin{quote}\caption{A trajectory on the regular octagon surface, and the corresponding transition diagram for $\theta \in [0,\pi/8)$ \label{intro-oct}} \end{quote}
\end{figure}
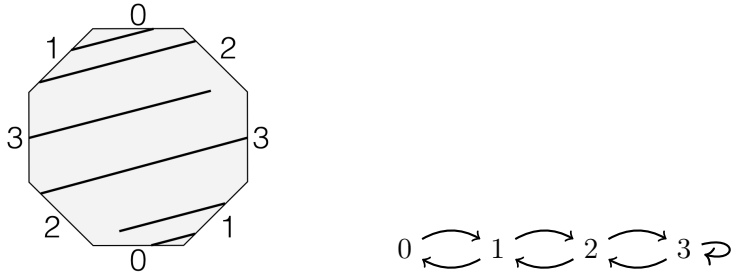

In the case of the square, identifying opposite sides by translations yields a torus or surface of genus $1$. When $n \geq 4$, one obtains in this way a surface of higher genus.  We call all the surfaces thus obtained (taking one or two copies of a regular polygon) \emph{regular polygonal surfaces}.
 Regular polygonal surfaces inherit from the plane an Euclidean metric (apart from finitely many points coming from vertices), with respect to  which linear trajectories are \emph{geodesics}.  

\smallskip
The full characterization of cutting sequences for the octagon, and more in general for regular polygon surfaces coming from the $2n$-gons, was recently obtained by Smillie and the third author in the paper \cite{SU2}; see also \cite{SU}. Shortly after the first author, Fuchs and Tabachnikov  described  in \cite{DFT} the set of periodic cutting sequences in the regular pentagon, the first author showed in \cite{davis} that the techniques in Smillie and Ulcigrai's work \cite{SU2} can be applied also to regular polygon surfaces with $n$ odd. We now recall the \emph{characterization of cutting sequences for the regular octagon surface} in \cite{SU2}, since it provides a model for our main result. 

One can first describe the set of pairs of consecutive edge labels, called \emph{transitions}, that can occur in a cutting sequence. By symmetry, one can consider only cutting sequences of trajectories in a direction $\theta \in [0,\pi)$ and up to permutations of the labels, one can further assume that $\theta \in [0,\pi/8)$. One can check that the transitions that are possible in this sector of directions are only the ones recorded in the graph in Figure \ref{intro-oct}. Graphs of the same form with permuted edge labels describe transitions in the other sectors of the form  $[\pi i /8 ,\pi (i+1)/8)$ for $i=1,\dots, 7$. We say that a sequence $w \in \Lalphabet_4^\mathbb{Z}$ is \emph{admissible} or more precisely \emph{admissible in sector $i $} if it contains only the transitions  allowed for the sector $[\pi i /8 ,\pi (i+1)/8)$. 

One can then define a \emph{derivation rule}, which turns out to be different than Series' rule for Sturmian sequences, but is particularly elegant. We say that an edge label is \emph{sandwiched} if it is preceded and followed by the same edge label. 
The \emph{derived sequence} of an admissible sequence is then obtained by \emph{keeping only sandwiched edge labels}. 

\begin{ex}\label{generationex} In the following sequence $w$ sandwiched edge labels are  written in bold fonts:
$$  w= \cdots  \mathbf{2}\, 1\, \mathbf{3}\, 122 \,\mathbf{1}\, 2213\, \mathbf{0}\,312213\,\mathbf{0}\,3122\,\mathbf{1}\,221\,\mathbf{3}\,122\,\mathbf{1}\,221\,\mathbf{3} \cdots $$ 
Thus, the derived sequence $w'$ of $w$ will contain  the string
$$w'=\cdots 231001313 \cdots .$$
\end{ex}

One can then prove that cutting sequences of linear trajectories on the regular octagon surface are \emph{infinitely derivable}. Contrary to the Sturmian case, though, this condition is only necessary and fails to be sufficient to characterize the closure of the space of cutting sequences. In \cite{SU2} an additional condition,  \emph{infinitely coherent} (that we do not want to recall here), is defined in order to characterize the closure. It is also shown on the other hand that one can give an $\mathcal{S}$-adic presentation of the closure of the octagon cutting sequences. In \cite{SU2} the language of substitutions was not used, but it is shown that one can define some combinatorial operators called \emph{generations} (which are essentially substitutions on pairs of labels)  and that each sequence in the closure can be obtained by a sequence of generations. One can rewrite this result in terms of substitutions; this is done for the example in the case of the regular hexagon  in \cite{report},  thus obtaining a characterization that generalizes \eqref{Sturmian:substitutions_char} and provides an $\mathcal{S}-$adic presentation, which for a regular $2n$-gon surface  consists of $2n-1$ substitutions.  The $1$-dimensional map that governs the substitution choice is a generalization of the Farey map (called the \emph{octagon Farey map} for $2n=8$ in \cite{SU2}). A symbolic coding of this generalized Farey map applied to the direction of a trajectory coincides with the sequence of sectors in which derived sequences of the trajectory's cutting sequence are admissible. 

\smallskip 
Both in the Sturmian case and for regular polygon surfaces the proofs of the characterizations are based on \emph{renormalization} in the following sense.  
Veech was the first to notice in the seminal paper \cite{Veech} that the square surface and the regular polygon surfaces share some special property that might make their analysis easier. He realized that all these surfaces are rich with \emph{affine symmetries} (or more precisely, of affine diffeomorphisms) and are examples of what are nowadays called \emph{Veech surfaces}  or \emph{lattice surfaces}, see \S \ref{sec:Veech} for definitions. It turns out that these affine symmetries can be used to \emph{renormalize} trajectories and hence produce a characterization of cutting sequences.  In the case of the square torus, they key idea behind a geometric proof of the above mentioned results on Sturmian sequences is the following: by applying an affine map of the plane, a linear trajectory is mapped to a linear trajectory whose cutting sequence is the derived sequence of the original trajectory. From this observation, one can easily show that cutting sequences are infinitely derivable.   In the case of the regular octagon, Arnoux and Hubert have used affine symmetries in \cite{AH:fra} to renormalize directions and define a continued fraction-like map for the octagon, but could not use their renormalization to describe cutting sequences and left this as an open question in  \cite{AH:fra}. 
An important point in Smillie and Ulcigrai's work \cite{SU2, SU} is  to also use  non-orientation-preserving affine diffeomorphisms, since this makes the continued fraction simpler and allows to use an element which acts as a \emph{flip and shear}, which accounts for the particularly simple sandwiched derivation rule. 


\subsection{Our results 
on \bm surfaces}\label{sec:ourresults}
In addition to  the  regular polygon surfaces, there are other known examples (see \S \ref{sec:Veech_history}) of surfaces which,  being rich with affine symmetries,  are   \emph{lattice} (or \emph{Veech}) surfaces (the definition is given in \S \ref{sec:Veech}). A full classification of Veech surfaces  is an ongoing big open question in Teichm\"uller dynamics (see again \S \ref{sec:Veech_history} for some references). 
Two new infinite families of Veech surfaces were discovered almost two decades after regular polygonal surfaces, respectively one by
Irene Bouw and Martin M\"oller \cite{BM} and the other by Kariane Calta \cite{calta} and Curt McMullen \cite{Mc} independently. 

The family found by Irene Bouw and Martin M\"oller was initially described algebraically (see \S \ref{sec:Veech_history}); later, Pat Hooper presented the construction of what we here call \emph{\bm surfaces} as created by identifying opposite parallel edges of a collection of \emph{semi-regular} polygons (see \S \ref{sec:Veech_history} for more detail). We give a precise description in \S \ref{bmdefsection}.  An example is the surface in Figure \ref{intro-bm43}, obtained from two semi-regular hexagons and two equilateral triangles by gluing by parallel translation the sides with the same edge labels. Surfaces in the \bm family  are parametrized by two indices $m, n$, so that the  $\M mn$ \bm surface is glued from 
$m$ polygons, the first and last of which are regular $n$-gons, and the rest of which are semi-regular $2n$-gons. The surface in the example is hence known as $\M 43$.  


\begin{figure}[!h] 
\centering
\includegraphics[width=.8\textwidth]{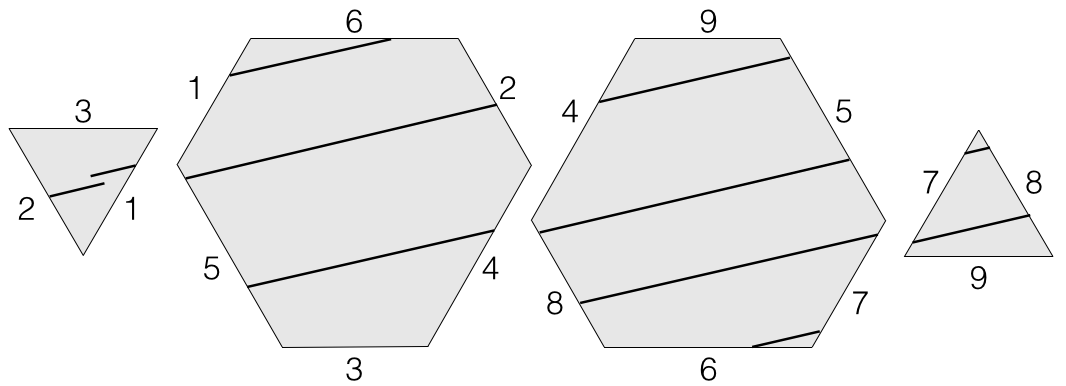}
\begin{quote}\caption{Part of a trajectory on the \bm surface $\M 43$ \label{intro-bm43}} \end{quote}
\end{figure}

\bm surfaces can be thought in some sense as the next simplest classes of (primitive) Veech surfaces after regular polygon surfaces, and the good \emph{next} candidate to generalize the question of  characterizing  cutting sequences. Indeed, the Veech group, i.e. the group generated by the linear parts of the affine symmetries (see \S\ref{sec:Veech} for the definition) of  both regular polygon surfaces and \bm surfaces are \emph{triangle groups}. More precisely, regular $n$-gon surfaces have $(n,\infty, \infty)$-triangle groups as Veech groups, while the Veech groups of  \bm surfaces  are $(m,n, \infty)$-triangle groups for $m$ and $n$ not both even (when $m$ and $n$ are both even, the Veech group has index $2$ inside the $(m,n, \infty)$-triangle group) \cite{Hooper}.  In \cite{D14}, Davis studied cutting sequences on \bm surfaces and analyzed the effect of a \emph{flip and shear} (as in Smillie-Ulcigrai's work \cite{SU2}) in order to define a derivation operator and renormalize trajectories. Unfortunately, with this approach it does not seem possible to cover all angles, apart from the surfaces with $m=2$ or $m=3$ in which all polygons are regular. Part of the reason behind this difficulty is that the Veech group contains two rotational elements, one of order $m$ and one of order $n$, but they do not act simultaneously on the same polygonal presentation.

In this paper, we give a \emph{complete characterization of the cutting sequences on \bm surfaces}, in particular providing an \emph{$\mathcal{S}$-adic presentation} for them. The key idea behind our approach is the following. It turns out that the $\M mn$ and the  $\M nm$ \bm surfaces are intertwined in the sense that they can be mapped to each other by an affine diffeomorphism.\footnote{In other words, they belong to the same Teichm\"uller disk.} While the $\M mn$ surface has a rotational symmetry of order $n$, the  $\M nm$ surface has a rotational symmetry of order $m$. We will call $\M mn$ and the  $\M nm$ \emph{dual \bm surfaces}.  Instead of normalizing using an affine automorphism as in the regular polygon case, we  renormalize trajectories and define associated derivation operators on cutting sequences in two steps, exploiting the affine diffeomorphism between  the $\M mn$ and the  $\M nm$ \bm surfaces. In particular, we map cutting sequences on the $\M mn$ surface to cutting sequences on the  $\M nm$ \bm surface. This allows us the freedom in between to apply the $n$ rotational symmetry and the $m$ rotational symmetry respectively, and this allows us to renormalize all cutting sequences. 

Note that since we frequently use the relationship between the surfaces {\rd \M mn} and {\gr \M nm}, we use the colors red and green to distinguish them throughout the paper, as here and as in Figure \ref{intro34auxtd} below.

We now give an outline of the statement of our main result, with an example in the special case of the $\M 43$ surface. The general results for $\M mn$ surfaces are stated precisely at the end of our paper, in \S \ref{howtolabel}. Let us label pairs of identified edges of  the $\M mn$ surface   with labels in the alphabet   $\LL mn=\{1,2,\dots, (m-1)n\}$. The surface $\M 43$ is for example labeled by $\LL 43=\{1,2,\dots, 9\}$ as in Figure \ref{intro-bm43}. The way to place edge labels for $\M mn$ is described in \S\ref{sec:labelingHooper} and is chosen in a special way that simplifies the later description. 
By applying a symmetry of the surface and exchanging edge labels by permutations accordingly, we can assume without loss of generality that the direction of trajectories we study belongs to the sector $[0, \pi/n]$.

As in the case of the regular octagon, we can first  describe the set of \emph{transitions} (i.e. pairs of consecutive edge labels) that can occur in a cutting sequence. 
For trajectories on  $\M 43$  whose direction belongs to sector $[0, \pi/3]$, the possible transitions are shown in the graph in Figure \ref{intro34auxtd}. The structure of \emph{transition diagrams} $\T i m n$  for trajectories on  $\M mn$  whose direction belong to sector $[\pi i/n, \pi (i+1)/n]$ are described in \S \ref{sec:other_sectors}. 
We say that a sequence $w \in \LL mn^\mathbb{Z}$ is \emph{admissible} (or more precisely \emph{admissible in sector $i $}) 
if it contains only the transitions represented by arrows in the diagram $\T i m n$. 


\begin{figure}[!h] 
\centering
\includegraphics[width=350pt]{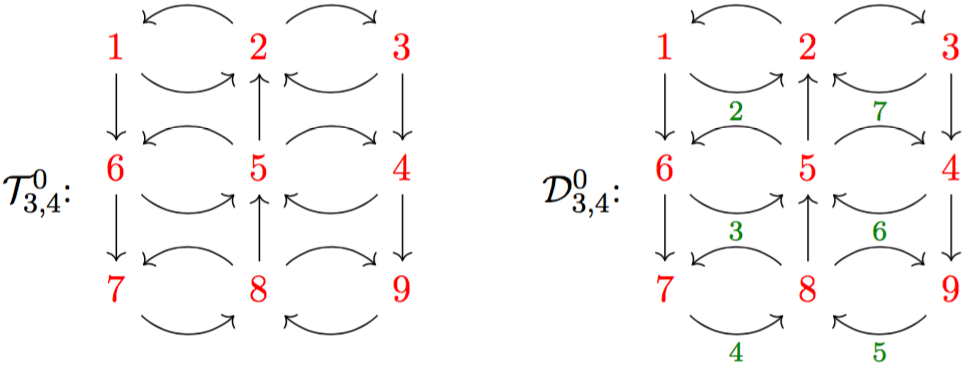}
\begin{quote}\caption{The transition diagram $\T 0 34$ for $\M 34$ and its derivation diagram $\D 0 34$, used to define $\Der 34$ \label{intro34auxtd}} \end{quote}
\end{figure}

We define a  \emph{derivation operator} $\Der mn$, which maps admissible sequences in $\LL mn^\mathbb{Z}$ to (admissible) sequences in $\LL nm^\mathbb{Z}$. The derivation rule for sequences admissible in sector $0$ is described by a labeled diagram as follows. We define \emph{derivation diagrams}  $\D 0 m n$  for the basic sector $[0,\pi/n]$ in which some of the arrows are labeled by edge labels of the dual surface $\M nm$. The derivation diagram for $\M43$ is shown in Figure \ref{intro34auxtd}.  The derived sequence $w'= \Der mn w $
 of a sequence $w$ admissible in diagram $0$ is obtained by reading off only the arrow labels of a bi-infinite path which goes through the vertices of $\D 0 mn$ described by $w$. 

\begin{example}\label{ex:der34}
Consider the trajectory on  $\M 43$ in Figure \ref{intro-bm43}. Its cutting sequence $w$ contains the word $\cdots {\rd 1678785452} \cdots$. This word corresponds to a path on the derivation diagram  $\D 0 43$ in Figure \ref{intro34auxtd}, which goes through the edge label vertices. By reading off the labels of the arrows crossed by this path, we find that $w'= \Der 4 3 w$ contains the word $\cdots {\gr 434761} \cdots$.  
\end{example}

This type of derivation rule is not as concise as for example the \emph{keep the sandwiched labels} rule for regular polygons, but we remark that the general shape of the labeled diagram that gives the derivation rule is quite simple, consisting of an $(m-1)\times n$ rectangular diagram with vertex labels and arrows labels snaking around as explained in detail in \S \ref{sec:labeled_def}. 

We say that a sequence  $w \in \LL mn^\mathbb{Z}$  is \emph{derivable} if it is admissible and its derived sequence  $\Der mn w \in \LL nm^\mathbb{Z}$ is admissible (in one of the diagrams of the dual surface $\M nm$). The derivation operator is defined in such a way that it admits the following geometric interpretation: if $w$ is a cutting sequence of a linear trajectory on $\M mn$, the derived sequence  $\Der mn  w$ is the cutting sequence of a linear trajectory on the dual surface $\M nm$ (see  \S \ref{sec:der_ex} for this geometric interpretation). In the special case $m=4,n=3$ this result was proved by the second author in \cite{report} (see also the Acknowledgments), where the derivation diagram in Figure \ref{intro34auxtd} was first computed.
 

In order to get a derivation from sequences $\LL mn^\mathbb{Z}$ back to itself, we compose this derivation operator with its dual operator $\Der nm$: we first \emph{normalize} the derived sequence, i.e. apply a permutation to  the labels to reduce to a sequence admissible in $\T 0 m n$. The choice of the permutations used to map sequences admissible in $\T i m n$ to sequences admissible in $\T 0 m n$ is explained in \S \ref{sec:normalization}. We can then apply $\Der nm$. This composition maps  
cutting sequences of trajectories on $\M mn$ first to cutting sequences of trajectories on $\M nm$ and then back to cutting sequences on $\M mn$.

We say that a sequence in $\LL mn^\mathbb{Z}$ is \emph{infinitely derivable} if by alternatively applying normalization and the two dual derivation operators $\Der mn$ and $\Der nm$ one always obtains sequences that are admissible (see formally Definition \ref{def:infinitely_derivable} in \S\ref{sec:infinitely_derivable}). With this definition, we then have our first result:

\begin{prop}\label{thm:main_infinite_diff}
Cutting sequences of linear trajectories on \bm surfaces are infinitely derivable.
\end{prop}
As in the case of regular polygon surfaces, this is only a necessary and not a sufficient condition to characterize the closure of cutting sequences. We then define in \S \ref{sec:generation}  \emph{generation} combinatorial operators that invert derivation (with the additional knowledge of starting and arrival admissibility diagram) as in the work by Smillie-Ulcigrai \cite{SU,SU2}. Using these operators, one can obtain a characterization, which we then also convert in \S \ref{sec:substitutionscharacterization} into a statement using substitutions. More precisely, we explain how to explicitly construct, for every \bm surface \M mn, $(m-1)(n-1)$ substitutions $\sigma_i$ for $1\leq i \leq (m-1)(n-1)$ on an alphabet of cardinality  $N=N_{m,n}:= 3mn-2m-4n+2$  and an operator $\Tr i mn$ that maps admissible sequences in the alphabet of cardinality $N$ (see details in \S \ref{sec:generation_characterization})  to admissible sequences on $\T i mn$ such that:
\begin{theorem}\label{thm:main_characterization}
A sequence $w$ is in the closure of the set of cutting sequences on the \bm surface \M mn if and only if there exists a sequence $(s_i)_{i\in \mathbb{N}}$ with $s_i \in \{1, \dots, (m-1)(n-1)\}$ and $0\leq s_0 \leq 2n -1$ such that
\begin{equation}\label{intersection}
w \in \bigcap_{k \in \mathbb{N}} \Tr {s_0} mn \sigma_{s_1} \sigma_{s_2} \cdots  \sigma_{s_{k}}  \{1, \dots, N \}^{\mathbb{Z}}.
\end{equation}
Furthermore, when $w$ is a non periodic cutting sequence the sequence $(s_i)_{i\in \mathbb{N}}$ can be uniquely recovered from the knowledge of $w$.
\end{theorem}
{We remark that \eqref{intersection} gives  the $\mathscr{S}-$adic presentation as a limit of finite sequences, which is equivalent to the result given in \eqref{limit} in the statement of Theorem \ref{firstversion} at the  beginning of the paper.  }   
Theorem \ref{thm:main_characterization}, which is proved as Theorem \ref{thm:substitutionscharacterization} in  \S \ref{sec:generation_characterization},\footnote{We remark that  Theorem \ref{thm:substitutionscharacterization}  the notation used is slightly different than the statement above, in particular the substitutions are labeled by two indices $i,j$ and similarly the entries $s_i$ are pairs of indices which code the two simpler Farey maps, see \S \ref{sec:generation_characterization} for details.} and the relation with itineraries mentioned above, which is proved by  Proposition \ref{prop:itineraries_vs_sectors}, provide the desired $\mathcal{S}-$adic characterization of \bm cutting sequences (recall the discussion on  $\mathcal{S}-$adic systems in the paragraph following equation \eqref{Sturmian:substitutions_char} previously in this introduction) and is indeed the main result of our work.

We remark also that Theorem \ref{thm:main_characterization} provides an algorithmic way to test (in infinitely many steps) if a sequence belongs to the closure of cutting sequences. The sequence  $(s_i)_{i\in \mathbb{N}}$ can be recovered algorithmicaly   when $w$ is a cutting sequence and hence infinitely derivable and is the sequence of indices of diagrams in which the successive derivatives of $w$ are admissible
(see Definition \ref{def:seq_sectors} in Section \ref{sec:sectors_sequences}).

Furthermore, the sequence  $(s_i)_{i\in \mathbb{N}}$ is governed by a $1$-dimensional dynamical system as follows. There exists a piecewise expanding map $\FF mn$, which we call the \emph{\bm Farey map}, which has \mbox{$(n-1)(m-1)$}
branches, such that if $w$ is the cutting sequence of a trajectory in direction $\theta$, the sequence $(s_i)_{i\in \mathbb{N}}$ is given by the symbolic coding of the orbit of $\theta$ under $\FF mn$. More precisely, it is the itinerary of $((\FF mn )^k(\theta))_{k \in \mathbb{N}}$  with respect to the natural partition of the domain of $\FF mn$ into monotonicity intervals. This is explained in \S \ref{farey}, where the map $\FF mn$ is defined as composition of two simpler maps, describing the projective action on directions of the affine diffeomorphisms from $\M mn$ to $\M nm$ and from $\M nm$ to $\M mn$ respectively. 

The \bm Farey map can be used to define a generalization of  the continued fraction expansion (see \S \ref{sec:direction_recognition}) which can be then in turn used to recover the direction of a 
 trajectory corresponding to a given cutting sequence. More precisely, the itinerary of visited sectors for the \bm Farey map described above gives us the indices for the \emph{\bm additive continued fraction expansion} of the direction $\theta$ (Proposition \ref{directionsthm}). 



\smallskip
{
Finally, let us conclude by commenting that, even though our characterization of cutting sequences is still only specific to the \bm   translation surfaces, we believe that this new family of Veech surfaces contains a substantial new layer of complexity and that our methods (briefly described in the next section) actually provide insight on how to potentially characterize cutting sequences on any Veech surface (see the very final subsection of the Appendix \ref{teich} for more insight in this direction). In particular, a key novelty that distinguishes our methods from the characterization of Sturmian sequences, or of  cutting sequences on regular polygons (e.g. from \cite{SU}) is that, though as in the other works we crucially exploit renormalization via affine diffeomorphisms, we do not directly describe  the action of (suitable generators of) the group of affine diffeomorphisms of $\M mn$, but we first describe the action on cutting sequences of intermediate and simpler maps, i.e. diffeomorphisms from the \bm surface $\M mn$ to an affine image of the \emph{dual} \bm surface $\M mn$. An interpretation of these elementary moves acting on the Teichueller disk of $\M mn$ is given in  Appendix \ref{teich}.}

\subsection{Structure and outline of the paper}
Let us now comment on some of the tools and ideas used in the proofs and describe the structure of the rest of the paper. As a general theme throughout the paper, we will first describe properties and results on an explicit example,  then give general results and proofs for the general case of $\M mn$. The example we work out in detail is the characterization of cutting sequences on the \bm surface $\M 43$ which already appeared in this introduction, exploiting also  its dual \bm surface $\M 34$. 
This is the first case that could not be fully dealt with by D. Davis in \cite{D14}.\footnote{On the other hand derivation on $\M 34$ can be fully described using Davis' flip and shear because whenever $m=3$, all the polygons are regular.}

In the next section, \S \ref{background}, we include some background material, in particular the definition of translation surface (\S \ref{sec:trans_surf}), affine diffeomorphisms (\S \ref{subsec:affine}) Veech group and Veech (or lattice) surfaces (\S \ref{sec:Veech}) and  a brief list of known classes of Veech surfaces (\S \ref{sec:Veech_history}). 
In \S \ref{bmdefsection} we then give the formal definition of \bm surfaces, describing the number and type of semi-regular polygons to form \M mn and giving formulas for their side lengths. We also describe their Veech group (see \S \ref{veechofbm}).

The main tool used in our proofs is the presentation of \bm surfaces through \emph{Hooper diagrams}, introduced by P. Hooper in his paper \cite{Hooper} and originally called \emph{grid graphs} by him.
{ These are decorated diagrams that encode combinatorial information on how to build \bm surfaces via the Thurston-Veech construction.}  
The surface $\M mn$ can be decomposed into \emph{cylinders} in the horizontal direction, and in the direction of angle $\pi/n$. The Hooper diagram encodes how these transversal cylinder decompositions intersect each other. In \S \ref{hooperdiagrams} we first explain how to construct a Hooper diagram starting from a \bm surface, while in \S \ref{hoopertobm} we formally define Hooper diagrams and then explain how to construct a \bm surface from a Hooper diagram.  

As we already mentioned in the introduction, the definition of the combinatorial derivation operator is 
motivated by the action on cutting sequences of  affine diffeomorphism (a \emph{flip and shear}) between $\M mn$ and its dual \bm surface $\M nm$. This affine diffeomorphism is described in \S \ref{sec:affine}. A particularly convenient presentation is given in what we call the \emph{orthogonal presentation}: this is an affine copy of $\M mn$, so that the two directions of cylinder decomposition forming an angle of $\pi/n$ are sheared to become orthogonal. In this presentation, both $\M mn$ and $\M nm$ can be seen simultaneously as diagonals of rectangles on the surface (that we call \emph{basic rectangles}, see Figure \ref{hexort1}). 

In \S \ref{stairsandhats} a useful tool for later proofs is introduced: we describe a local configuration in the Hooper diagram, that we call a \emph{hat} (see Figure \ref{hat} to understand choice of this name)
 and show that it translates into a \emph{stair} configuration of basic rectangles in the orthogonal presentation mentioned before.  Proofs of both the shape and labeling of transition diagrams and of derivation rules exploit the local structure of Hooper diagrams by switching between hat and stairs configurations. 

Section \S \ref{transitiondiagrams} is devoted to transition diagrams: we first explain our way of labeling edges of \bm surfaces. This labeling, as mentioned before, works especially well with Hooper diagrams. The structure of transition diagrams is then described in \S \ref{sec:labeled_def} (see Theorem \ref{tdtheorem}) and proved in the later sections using hats and stairs. In the same sections we prove also that derivation diagrams describe intersections with sides of the affine image of the dual \bm surface, which is a key step for derivation.

In Section \S \ref{derivation} we describe the \emph{derivation process} obtained in two steps, 
 by first deriving cutting sequences on \M mn to obtain cutting sequences on the dual surface \M nm (see \S \ref{sec:der_general}) and  then, after \emph{normalizing} them (see \S \ref{sec:nor}), deriving them another time but this time applying the \emph{dual} derivation operator.
This two-step process of derivation and then normalization is called \emph{renormalization}.
In \S \ref{farey} we define a one-dimensional map, called the \emph{Bouw-M\"oller Farey map}, that describes the effect of renormalization on the direction of a trajectory.

In \S \ref{sec:characterization} we \emph{invert} derivation through generation operators. This allows to prove the characterization in \S \ref{sec:generation}, where first the characterization of \bm cutting sequences through generation is proved in \S\ref{sec:generation_characterization}, then the version using substitutions is obtained in \S\ref{sec:substitutionscharacterization}, see Theorem \ref{thm:substitutionscharacterization}.



\subsection{Acknowledgements}
The initial idea of passing from $\M mn $ to $\M nm$ to define derivation in \bm surfaces came from conversations between the third author and John Smillie, whom we thank also for explaining to us Hooper diagrams. 
We also thank Samuel Leli\`evre, 
Pat Hooper, Rich Schwartz and Ronen Mukamel for useful discussions and Alex Wright and Curt McMullen for their comments on the first version of this paper. 

A special case  of the derivation operator defined in this paper (which provided the starting point for our work) was worked out by the second author for her Master's thesis \cite{report} during her research project under  the supervision of the third author. We thank Ecole Polytechnique and in particular Charle Favre for organizing and supporting this summer research project and the University of Bristol for  hosting her as a visiting student.   

The collaboration that led to the present paper was made possible by the support of ERC  grant ChaParDyn, which provided funds for a research visit  of the three authors at the University of Bristol, and by the hospitality during the ICERM's workshop \emph{Geometric Structures in Low-Dimensional Dynamics} in November 2013, and the conference \emph{Geometry and Dynamics in the Teichm\"uller space} at CIRM in July 2015, which  provided excellent conditions for continued collaboration.

I. Pasquinelli is currently supported by an EPSRC Grant.
C. Ulcigrai is currently supported by ERC Grant ChaParDyn.  


\section{Background} \label{background}
In this section 
 we present some general background on the theory of translation surfaces, in particular giving the definition of translation surfaces (\S \ref{sec:trans_surf}), of  affine deformations and of  Veech groups  (\S \ref{sec:Veech}) and we briefly list known examples of Veech surfaces (\S \ref{sec:Veech_history}).

\subsection{Translation surfaces and linear trajectories}\label{sec:trans_surf}

The surface  $T$  obtained by identifying opposite parallel sides of the square, and the surface $\mathcal O$ obtained by identifying opposite parallel sides of the regular octagon,   are examples of  translation surfaces. The surface $T$ has genus $1$, and the surface $\mathcal O$ has genus $2$. Whenever we refer to a translation surface $S$, we will have in mind a particular collection of polygons in $\mathbb{R}^2$ with identifications. We define translation surfaces as follows:

\begin{definition}\label{translationsurface}
A \emph{translation surface} is a collection of polygons $P_j$ in $\R^2$, with parallel edges of the same length identified, so that
\begin{itemize}
\item edges are identified by maps that are restrictions of translations,
\item every edge is identified to some other edge, and 
\item when two edges are identified, the outward-pointing normals point in opposite directions. 
\end{itemize}
If $\sim$ denotes the equivalence relation coming from identification of edges, then we define the surface $S=\bigcup P_j/\sim$. 
\end{definition}


Let $\Sing$ be the set of points corresponding to vertices of polygons, which we call \emph{singular points}.\footnote{Standard usage says that such a point is singular only if the angle around it is greater than $2\pi$, but since all of our vertices satisfy this, we call all such points singular points.}

We will consider geodesics on translation surfaces, which are straight lines: any non-singular point has a neighborhood that is locally isomorphic to the plane, so geodesics locally look like line segments, whose union is a straight line. We call geodesics \emph{linear trajectories}. We consider trajectories that do not hit singular points, which we call   \emph{bi-infinite} trajectories. 

A trajectory that begins and ends at a singular point is a \emph{saddle connection}. Every periodic trajectory  is parallel to a saddle connection, and is contained in a maximal family of parallel periodic trajectories of the same period. This family fills out a \emph{cylinder} bounded by saddle connections.

 A \emph{cylinder decomposition} is a partition of the surface into parallel cylinders. The surfaces that we consider, \bm surfaces, have many cylinder decompositions (see Figure \ref{octcyl}). For a given cylinder, we can calculate the \emph{modulus} of the cylinder, which is ratio of the width (parallel to the cylinder direction) to the height (perpendicular to the cylinder direction). For the cylinder directions we use on \bm surfaces, all of the cylinders have the same modulus (see Theorem \ref{modthm}, proven in \cite{D14}).

\subsection{Affine deformations and affine diffeomorphisms}\label{subsec:affine}

Given  $\nu \in GL(2,\mathbb{R})$, we denote by $\nu  P \subset \mathbb{R}^2$ the image of $P \subset \mathbb{R}^2$ under the linear map $\nu$. Note that parallel sides in $P$ are mapped to parallel sides in $\nu P$. If $S $ is obtained by gluing the polygons $P_1, \ldots, P_n$, we define a new translation surface that we will denote by $\nu \cdot S$,  by gluing the corresponding sides of $\nu \octcdot P_1, \dots, \nu \octcdot P_n $. 
The map from the surface $S$ to the surface $\nu \cdot S$, which is given by the restriction of the linear map $\nu$ to the polygons $P_1 , \dots ,  P_n$, will be called the \emph{affine deformation given by $\nu$}.



Let $S$ and $S'$ be translation surfaces.  Consider a homeomorphism $\Psi$ from $S$ to $S'$ that takes $\Sing$ to $\Sing'$ and is a diffeomorphism outside of $\Sing$. We can identify the derivative $D\Psi_p$ with an element of $GL(2,\mathbb{R})$. We say that $\Psi$ is an \emph{affine diffeomorphism} if the derivative $D\Psi_p$ does not depend on $p$. In this case we write $D\Psi$ for $D\Psi_p$.
The affine deformation $\Phi_\nu$ from $S$ to  $\nu \cdot S$ described above is an example of an affine diffeomorphism. In this case $D\Phi_\nu=\nu$.

We say that $S$ and $S'$ are \emph{affinely equivalent} if there is an affine diffeomorphism $\Psi$ between them.  
We say that $S$ and $S'$ are \emph{translation equivalent} if they are affinely equivalent with $D\Psi=Id$.  If $S$ is given by identifying sides of polygons $P_j$ and $S'$ is given by identifying sides of polygons $P'_k$  then a translation equivalence $\Upsilon$ from $S$ to $S'$ can be given by a ``\emph{cutting and pasting}'' map. That is to say we can subdivide the polygons $P_j$ into smaller polygons and define a map $\Upsilon$ so that  the restriction of $\Upsilon$ to each of these smaller polygons is a translation and the image of $\Upsilon$ is the collection of polygons $P'_k$. 


An affine diffeomorphism from $S$ to itself is an \emph{affine automorphism}. The collection of affine diffeomorphisms is a group which we denote by $\Aff(S)$. 
If $S$ is given as a collection of polygons with identifications then we can realize an affine automorphism of $S$ with derivative $\nu$ as a composition of a map $\Psi_\nu:S\to\nu\cdot S$ with a translation equivalence, or cutting and pasting map, $\Upsilon:\nu\cdot S \to S$. 

\subsection{The Veech group and Veech surfaces}\label{sec:Veech}
The Veech homomorphism is the homomorphism $\Psi\mapsto D\Psi$  from $\Aff(S)$ to $GL(2,\R)$. The image of this homomorphism lies in the subgroup of matrices with determinant $\pm1$ which we write as $SL_{\pm}(2,\mathbb{R})$. We call \emph{Veech group} and we denote by $V(S)$ the image of $\Aff(S)$ under the Veech homomorphism.
It is common to restrict to orientation-preserving affine diffeomorphisms in defining the Veech group, but 
since we will make essential use of orientation-reversing affine automorphisms, we will use the term \emph{Veech group} for the larger group $V(S)$. Note that the term \emph{Veech group} is used by some authors to refer to the image of the group of orientation-preserving affine automorphisms in the projective group $PSL(2,\R)$. 


A translation surface $S$ is called a \emph{Veech surface} if 
$V(S)$ is a lattice in $SL_{\pm}(2, \mathbb{R})$. The torus $T^2=\mathbb{R}^2 / \mathbb{Z}^2$ is an example of a Veech surface whose  Veech group is $GL(2, \mathbb{Z})$. Veech proved  more generally that all  translation surfaces obtained from regular polygons are Veech surfaces.
Veech surfaces satisfy the \emph{Veech dichotomy} (see \cite{Veech}, \cite{Vorobets}) which says that if we consider a direction $\theta$ then one of the following two possibilities holds: either there is a saddle connection in direction $\theta$ and the surface decomposes as a finite union of cylinders each of which is a union of a family of closed geodesics in direction $\theta$, or each trajectory in direction $\theta$ is dense and uniformly distributed.

We will use the word \emph{shear} to denote an affine automorphism of a surface whose derivative is $\sm 1s01$ for some real number $s$. If a translation surface admits a shear, we can decompose it into cylinders of commensurable moduli, so a power of the shear acts as a Dehn twist in each cylinder.


\subsection{Known examples of Veech surfaces}\label{sec:Veech_history}

Several families of Veech surfaces are known. A brief history of known Veech surfaces is as follows.

\begin{itemize}
\item The simplest example of a Veech surface is the square, with pairs of parallel sides identified to create the square torus.
\item Covers of the square torus, called \emph{square-tiled surfaces}, are created by identifying opposite parallel edges of a collection of congruent squares. Eugene Gutkin and Chris Judge \cite{GJ} showed that square-tiled surfaces are equivalent to those surfaces whose Veech group is \emph{arithmetic}, i.e. commensurable with $SL(2,\ZZ)$. Subsequently, Pascal Hubert and Samuel Leli\`evre showed that in genus $2$, all translation surfaces in $H(2)$ that are tiled by a prime number $n > 3$ of squares fall into exactly two Teichm\"uller discs.
\item William Veech was the first to define in \cite{Veech} Veech groups and lattice surfaces, and to prove that all regular polygon surfaces are Veech surfaces and satisfy the Veech dichotomy described above.

\item Clayton Ward  discovered a new family of Veech surfaces  about $10$ years after regular polygonal surfaces \cite{Ward}.  These surfaces are created by identifying opposite parallel edges of three polygons: two regular $n$ gons and a regular $2n$-gon (see Figure \ref{m34} for the case when $n=4$.)    Ward's surfaces turn out to be a special case of \bm surfaces, those made from exactly $3$ polygons, the $\M3n$ family.

\item Veech surfaces are related to billiards on triangles; we will not describe the correspondence here. Rick Kenyon and John Smille \cite{KS} showed that, other than the triangles corresponding to the examples above, only three other triangles correspond to Veech surfaces. Two of these were already known to Yaroslav Vorobets \cite{Vorobets}.

\item  Kariane Calta \cite{calta} and Curt McMullen \cite{Mc}  discovered independently infinitely many new Veech surfaces in genus $2$, each of which can be presented as an L-shaped polygon with certain integer measurements in a quadratic vector field. 

\item {In \cite{McP}, McMullen discovered a new infinite family of primitive Teichm\"uller curves in $\mathcal M_g$ for $g=2,3,4$, as well as new closed $SL(2,\R)$-invariant loci in the space of holomorphic 1-forms $\Omega \mathcal M_g$, $g \leq 5$ and new Teichm\"uller curves generated by strictly quadratic differentials.  
All these new examples are built using a construction that generalizes  to higher genus the work he did for genus $2$ using Jacobians with real multiplication (see \cite{McJ}).
}
\item  Irene Bouw and Martin  M\"oller discovered a new family of Veech curves (i.e. quotients of $SL(2, \mathbb{R})$ by a lattice Veech group) with triangular Veech groups in \cite{BM}.  Then Pat Hooper  in \cite{Hooper} showed that special points on these Veech curves  can be obtained by gluing  semi-regular polygons; see the definition given in the next section. In this paper, we will call \bm surfaces this family of Veech surfaces obtained by gluing semi-regular polygons (as it has been done often in previous literature).  We remark that Hooper showed that the Teichm\"uller curves associated to his semi-regular polygon surfaces were the same as Bouw and M\"oller's Veech curves in many cases, with a few exceptions. Later, Alex Wright \cite{Wright} showed this equality in all the remaining cases. We remark also that while Ward's surfaces are always glued from exactly $3$ polygons (they correspond as mentioned above to the $\M3n$ \bm family), \bm surfaces 
can be obtained by gluing  any number $m\geq 2$ of (semi-regular) polygons.

\item Hooper found that the $(\pi/12, \pi/3, 7\pi/12)$ triangle has the lattice property, and computed its Veech group. The surface unfolded from this triangle is one of the Veech surfaces of genus $4$ discovered by McMullen in \cite{McJ}), described above, and Hooper showed that it comes from a billiard in this triangle (see \cite{Hooper1}).

\end{itemize}
Providing a full classification of Veech surfaces is a big open question in Teichm\"uller dynamics, since Veech surfaces correspond indeed to closed $SL(2, \mathbb{R})$-orbits and hence are the smallest orbit closures of the $SL(2, \mathbb{R})$ action on the moduli space of Abelian differentials. Several very recent results are in the direction of proving that there exists only finitely many Veech surfaces 
 in several strata of translation surfaces, see for example  \cite{AN,ANW,BHM,LN, LNW, MaW,NW}. 

\subsection{Bouw-M\"oller surfaces: semi-regular polygonal presentation} \label{bmdefsection}

We will now describe the polygonal presentation of the \bm surfaces, given by Pat Hooper \cite{Hooper}. We create the surface $\M m n$ by identifying opposite parallel edges of a collection of $m$ \emph{semi-regular} polygons that each have $2n$ edges. 


A \emph{semi-regular polygon} is an equiangular polygon with an even number of edges. Its edges alternate between two different lengths. The two lengths may be equal (in which case it is a regular $2n$-gon), or one of the lengths may be $0$ (in which case it is a regular $n$-gon).

\begin{example} \label{m43ex}
The \bm surface \M 4 3 ($m=4$, $n=3$) is made of $4$ polygons, each of which have $2n=6$ edges (Figure \ref{m43}). From left to right, we call these polygons $P(0), P(1), P(2), P(3)$. Polygon $P(0)$ has edge lengths $0$ and $\sin\pi/4=1/\sqrt{2}$, polygon $P(1)$ has edge lengths $1/\sqrt{2}$ and $\sin(\pi/2)=1$, polygon $P(2)$ has edge lengths $1$ and $1/\sqrt{2}$, and polygon $P(3)$ has edge lengths $1/\sqrt{2}$ and $0$. 

\begin{figure}[!h] 
\centering
\includegraphics[width=350pt]{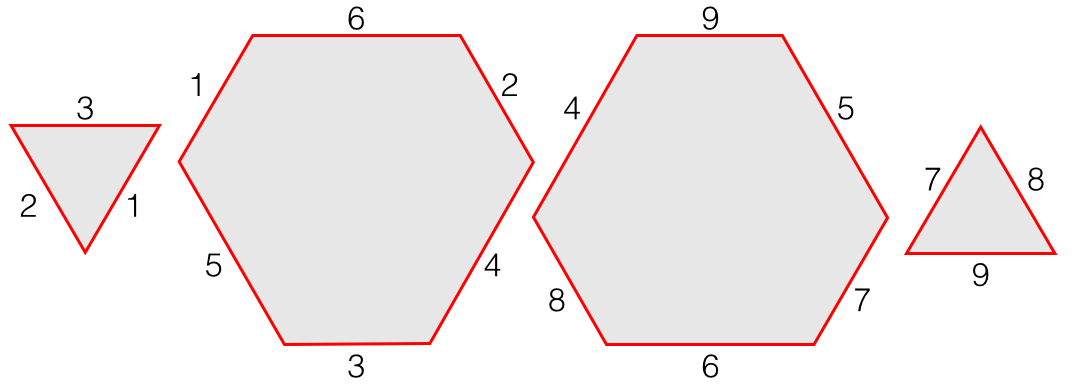}
\begin{quote}\caption{The \bm surface $\M 43$ with \mbox{$m=4$}, \mbox{$n=3$} is made from two equilateral triangles and two semi-regular hexagons. Edges with the same label are identified. \label{m43}} \end{quote}
\end{figure}

\end{example}

 
Definition \ref{semiregular} gives an explicit definition of an equiangular $2n$-gon whose edge lengths alternate between $a$ and $b$:

\begin{definition} \label{semiregular}
Let $P_n (a,b)$ be the polygon whose edge vectors are given by:
\[
\mathbf{v}_i = 
\begin{cases}
 a \ [\cos \frac{i\pi}{n},\sin  \frac{i\pi}{n}] & \text{if }i\text{ is even} \\
 b\ [\cos \frac{i\pi}{n}, \sin  \frac{i\pi}{n}] & \text{if }i\text{ is odd}
\end{cases}
\]
for $i=0,\ldots,2n-1$.  The edges whose edge vectors are $\mathbf{v_i}$ for $i$ even are called \emph{even edges}. The remaining edges are called \emph{odd edges}. We restrict to the case where at least one of $a$ or $b$ is nonzero. If $a$ or $b$ is zero, $P_n (a,b)$ degenerates to a regular $n$-gon.
\end{definition}

In creating polygons for a \bm surface, we carefully choose the edge lengths so that the resulting surface will be a Veech surface (see \S \ref{sec:Veech}). 

\begin{definition}\label{pk} Given integers $m$ and $n$ with at least one of $m$ and $n$ nonzero, we define the polygons $P(0), \ldots, P(m-1)$ as follows. 
\[
P(k) = 
\begin{cases}
 P_n \left(\sin \frac{(k+1)\pi}{m}, \sin \frac{k\pi}{m}\right) & \text{if }m\text{ is odd} \\
P_n \left(\sin \frac{k\pi}{m}, \sin \frac{(k+1)\pi}{m}\right) & \text{if }m \text{ is even and } k \text{ is even} \\
P_n \left(\sin \frac{(k+1)\pi}{m}, \sin \frac{k\pi}{m}\right) & \text{if }m \text{ is even and } k \text{ is odd.} \\
 \end{cases}
\]
\end{definition}

An example of computing these edge lengths was given in Example \ref{m43ex}.

\begin{remark}
$P(0)$ and $P(m-1)$ are always regular $n$-gons, because $\sin\frac{0\pi}m=0$ and $\sin\frac{(m-1+1)\pi}m=0$. If $m$ is odd, the central $2n$-gon is regular, because $\sin (k\pi /m) = \sin ((k+1)\pi/m)$ for $k=(m-1)/2$. Figure \ref{m34} shows both of these in $\M 34$.
\end{remark}

\begin{figure}[!h] 
\centering
\includegraphics[width=290pt]{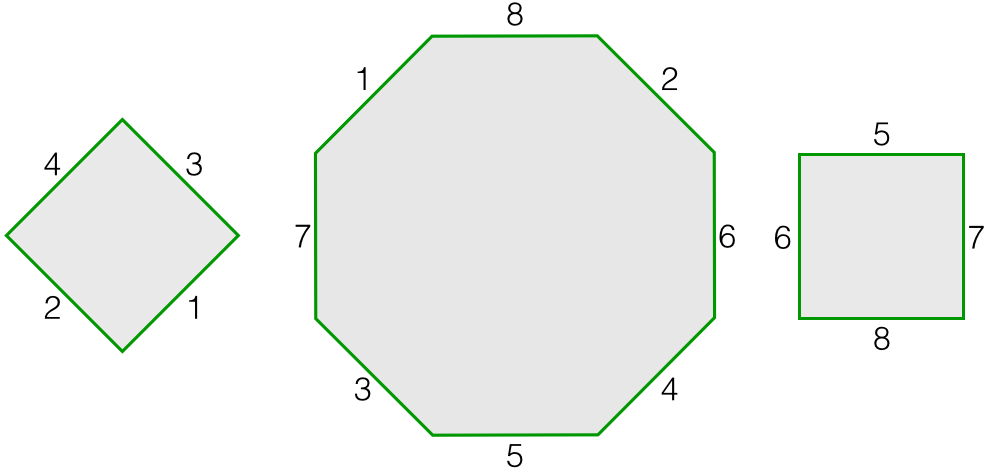}
\begin{quote}\caption{The \bm surface $\M 34$ with \mbox{$m=3$}, \mbox{$n=4$} is made from two squares and a regular octagon. Edges with the same label are identified. \label{m34}} \end{quote}
\end{figure}

Finally, we create a \bm surface by identifying opposite parallel edges of $m$ semi-regular polygons $P(0),\ldots,P(m-1)$. For each polygon in the surface, $n$ of its edges (either the even-numbered edges or the odd-numbered edges) are glued to the opposite parallel edges of the polygon on its left, and the remaining $n$ edges are glued to the opposite parallel edges of the polygon on its right. The only exceptions are the polygons on each end, which only have $n$ edges, and these edges are glued to the opposite parallel edges of the adjacent polygon. These edge identifications are shown in Figures \ref{m43} and \ref{m34}.

We now give the edge identifications explicitly:

\begin{definition}\label{bmdefinition}
The \bm surface $\M mn$ is made by identifying the edges of the $m$ semi-regular polygons $P(0),\ldots,P(m-1)$ from Definition \ref{pk}. 
We form a surface by identifying the edges of the polygon in pairs. For $k$ odd, we identify the even edges of $P(k)$ with the opposite edge of $P(k+1)$, and identify the odd edges of $P(k)$ with the opposite edge of $P(k-1)$. The cases in Definition \ref{pk} of $P(k)$ are chosen so that this gluing makes sense. 
\end{definition}

\begin{theorem}[\cite{D14}, Lemma 6.6]\label{modthm} Every cylinder of the \bm surface $\M mn$ in direction $k\pi/n$ has the same modulus. The modulus of each such cylinder  is \mbox{$2\cot\pi/n +2 \frac{\cos \pi/m}{\sin \pi/n}$}.
\end{theorem}

We will use this fact extensively, because it means that one element of the Veech group of $\M mn$ is a \emph{shear}, a parabolic element whose derivative is $\sm 1 s 01$ for some real number $s$. For $\M mn$, \mbox{$s=2\cot\pi/n +2 \frac{\cos \pi/m}{\sin \pi/n}$} as above.

\begin{theorem} [Hooper]\label{thm:Hooper}
$\M mn$ and $\M nm$ are affinely equivalent.
\end{theorem}
This means that $\M mn$ can be transformed by an affine map (a shear plus a dilation) and then cut and reassembled into $\M nm$. 

\begin{example} In Figure \ref{shear-equiv} it is for example shown how the surface in Figure \ref{m43} can be cut and reassembled into a sheared version of the surface in Figure \ref{m34}.
\begin{figure}[!h] 
\centering
\includegraphics[width=400pt]{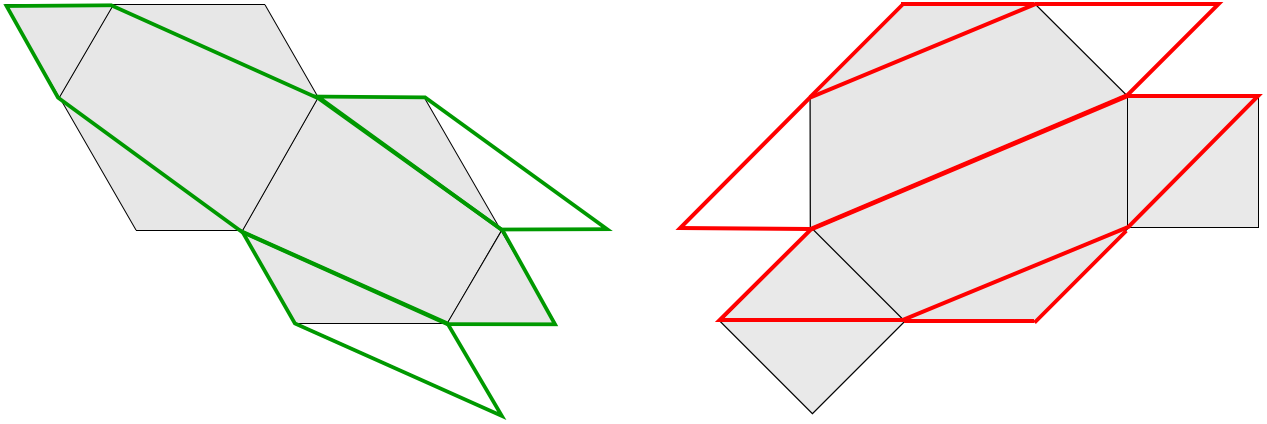}
\begin{quote}\caption{The bold outline on the left shows how the left surface can be cut and reassembled into a sheared version of the right surface, and vice-versa. \label{shear-equiv}} \end{quote}
\end{figure} 
\end{example}
We will use this affine equivalence extensively, since as already mentioned in the introduction our derivation and characterization of cutting sequences exploit the relation between cutting sequences on  $\M mn$ and  $\M nm$. The affine diffeomorphism between $\M mn$ and $\M nm$ that we use for derivation (which also includes a flip, since this allows a simpler description of cutting sequences) is described in \S \ref{sec:affine}.


\subsection{The Veech group of \bm surfaces} \label{veechofbm}
The Veech group of $\M mn$, as well as the Veech group of the $\M nm$, is isomorphic to the $(m,n,\infty)$ triangle group. The only exception to this is when $m$ and $n$ are both even, in which case the Veech group of $\M mn$ has index $2$ in the $(m,n,\infty)$ triangle group see \cite{Hooper}. Thus, the Veech group contains two elliptic elements of order $2m$ and $2n$ respectively. One can take as generators one of this two elements and  a shear (or a ``flip and shear'') automorphism from the $(m,n)$ surface to itself. In the $(n,m)$ polygon presentation of $\M mn$ the elliptic element of order $2m$ is a  rotation of order $\pi/m$ (while in the $(m,n)$ polygon decomposition the elliptic element is a rotation of order $\pi/n$). Thus, the elliptic element of order $2n$ acting on $\M mn$ can be obtained conjugating the rotation of $\pi/n$ on the dual surface $\M nm$ by the affine diffeomorphism between $\M mn$ and $\M nm$ given by Theorem \ref{thm:Hooper}. In section Section \ref{teich} we describe the action of these Veech group elements on a tessellation of the hyperbolic plane by $(m,n,\infty)$  triangles shown in Figure \ref{hypdisk}. 



\section{\bm surfaces via Hooper diagrams} \label{hooperdiagrams}
In \S\ref{bmdefsection} we recalled the construction of \bm surfaces by gluing a collection of semi-regular polygons. In his paper \cite{Hooper}, as well as this  polygonal presentation, Hooper gave a description of these surfaces by constructing a decorated diagram, that we will call the \emph{Hooper diagram $\G{m}{n}$} for the  \bm surface $\M mn$. { These diagrams are related to what is known as the   Thurston-Veech construction (sometimes also called the \emph{Bouillabaisse} construction) first described by Thurston \cite{Thurston} and independently by Veech  \cite{Veech}.  Indeed, Hooper diagrams describe the intersection pattern of two transversal cylinder decompositions of the surface, and hence provide data that can be used for the Thurston-Veech construction of pseudo-Anosov diffeormorphisms of a surface as product of two distinct multi-twists. We remark that similar diagrams and the related ribbon graphs already appeared in the literature before the work of Hooper \cite{Hooper} (for example in the works of Leininger \cite{Leininger} and McMullen  \cite{McP}). We choose to refer to them as \emph{Hooper diagrams} since it is the decorated version in \cite{Hooper} that we use, and Hooper was the first to find this  description of \bm surfaces (he also brought the use of this type of diagrams further, by exploiting them to study infinite translation surfaces in \cite{Hooper2} and \cite{Hooper3}).}

In this section we will explain how to construct the Hooper diagram given the polygonal presentation of the surface and vice versa, following the example of $\M{3}{4}$ throughout.

\subsection{From $\M{3}{4}$ to a Hooper diagram: an example}

Let us consider the \bm surface $\M{3}{4}$. 
To construct its Hooper diagram, we need to consider the two cylinder decompositions given in Figure \ref{octcyl}. 

\begin{figure}[!h]
\centering
\includegraphics[width=0.7\textwidth]{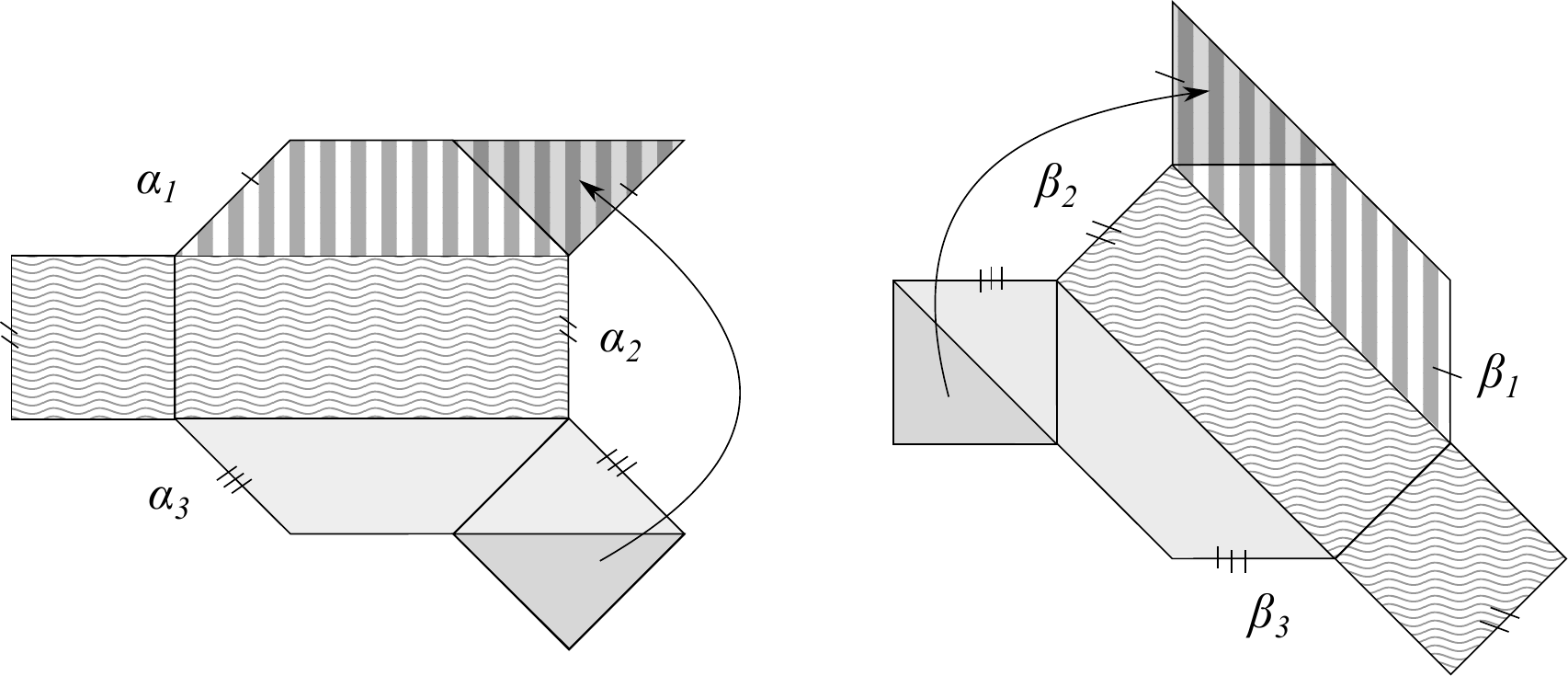} \hspace{0.08\textwidth}
\includegraphics[width=0.2\textwidth]{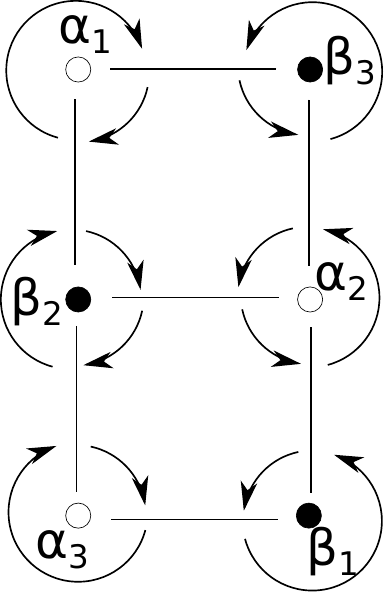}
\begin{quote}\caption{The cylinder decomposition for $\M{3}{4}$ and its Hooper diagram. \label{octcyl}} \end{quote}
\end{figure}

We have a \emph{horizontal} cylinder, and a cylinder in direction $\frac{3 \pi}{4}$ that we will call \emph{vertical}.
The horizontal cylinders will be called $\alpha_1$, $\alpha_2$ and $\alpha_3$ as in Figure \ref{octcyl}, while the vertical ones will be $\beta_1$, $\beta_2$ and $\beta_3$.
Notice that both decompositions give the same number of cylinders -- three cylinders in this case -- and this is true for all \bm surfaces, that the cylinder decompositions in each direction $k\pi/n$ yield the same number of cylinders. 

Let us now construct the corresponding graph $\G{3}{4}$ as the \emph{cylinder intersection graph} for our cylinder decompositions.
In general, it will be a bipartite graph with vertices $\mathcal V= \mathcal A \cup \mathcal B$, represented in Figure \ref{octcyl}  with black and white vertices, respectively. 
The black vertices are in  one-to-one correspondence with the vertical cylinders, while the white vertices are in  one-to-one correspondence with the horizontal cylinders. 
To describe the set of edges, we impose that there is an edge between two vertices $v_i$ and $v_j$ if the two corresponding cylinders intersect. 
It is clear that we will never have edges between vertices of the same type, because two parallel cylinders  never intersect. 
An edge will hence correspond to a parallelogram that is the intersection between  cylinders in two different decompositions.

In our case, the graph $\G{3}{4}$ will have six vertices: three white ones, corresponding to the cylinders $\alpha_i$, and three black ones, corresponding to the cylinders $\beta_i$, for $i=1,2,3$. 
Considering the intersections, as we can see in Figure \ref{octcyl}, the central cylinder $\alpha_2$ of the horizontal decomposition will cross all three cylinders of the vertical decomposition.
The other two will cross only two of them, $\beta_1$ and $\beta_2$ in the case of $\alpha_3$; $\beta_3$ and $\beta_1$ in the case of $\alpha_3$.

Finally, we need to record how the various pieces of a cylinder, seen as the various edges around a vertex, glue together. 
To do that we first establish a positive direction, gluing on the right for the orthogonal decomposition and gluing upwards for the vertical one. 
We then record this on the graph by adding some circular arrows around the vertices, giving an ordering for the edges issuing from that vertex. 
We can easily see that such arrows will have the same direction (clockwise or counter-clockwise) in each column, and alternating direction when considering the vertices on the same row. 
We start the diagram in a way such that we will have arrows turning clockwise in odd columns and arrows turning counter-clockwise in the even columns.
All we just said leads us to construct a graph as in Figure \ref{octcyl}.


We notice that the dimension of the graph is of three rows and two columns and this will be true in general: the graph $\G{m}{n}$ for $\M mn$  will have $n-1$ rows and $m-1$ columns.

\subsection{From  $\M{m}{n}$  to Hooper diagrams: the general case }\label{hoogen}

We will now explain how to extend this construction to a general \bm surface and see what type of graph we obtain. 

In general, our surface $\M{m}{n}$ will have two cylinder decompositions in two different directions that we will call \emph{horizontal} and \emph{vertical}. 
We define $\mathcal A=\{\alpha_i\}_{i \in \Lambda}$ and $\mathcal B= \{\beta_i\}_{i \in \Lambda}$ to be the set of horizontal and vertical cylinders, respectively.

The vertices of the cylinder intersection graph is the set of cylinders in the horizontal and vertical directions, $\mathcal A \cup \mathcal B$. 
The set of edges will be determined by the same rule as before: there is an edge between $\alpha_i$ and $\beta_j$ for every intersection between the two cylinders. 
Therefore, each edge represents a parallelogram, which we call a \emph{rectangle} because it has horizontal and vertical (by our definition of ``vertical'' explained above) sides. 
Let $\mathcal E$ be the collection of edges (or rectangles). 
Define the maps $\alpha \colon \mathcal E \to \mathcal A$ and $\beta \colon \mathcal E \to \mathcal B$ to be the maps that send the edge between $\alpha_i$ and $\beta_j$ to the nodes $\alpha_i$ and $\beta_j$, respectively.

The generalization of the black and white vertices is the concept of a $2$-colored graph:

\begin{definition}
A \emph{2-colored graph} is a graph equipped with a coloring function $C$ from the set of nodes $\mathcal V$ to $\{0,1\}$, with the property that for any two adjacent nodes, $x, y \in \mathcal V$, we have $C(x) \neq C(y)$.
\end{definition}

The graph we constructed is a 2-colored graph. 
To see that, simply define $C(x)=0$ if $x\in \alpha (\mathcal E)=\mathcal A$ and $C(x)=1$ if $x \in \beta (\mathcal E)=\mathcal B$.
Conversely, the maps $\alpha, \beta \colon \mathcal E \to \mathcal V$ as well as the decomposition $\mathcal V= \mathcal A \cup \mathcal B$ are determined by the coloring function.

As we said, we also need to record in our graph the way the rectangles forming the cylinders are glued to each other. 
To do that we define $\mathfrak e \colon \mathcal E \to \mathcal E$ be the permutation that sends a rectangle to the rectangle on its right, and let $\mathfrak n \colon \mathcal E \to \mathcal E$ be the permutation that sends a rectangle to the rectangle above it. (Here $\mathfrak e$ stands for ``east'' and $\mathfrak n$ stands for ``north.'')
Clearly, we will always have that $\mathfrak e(e)$  lies in the same cylinder as the rectangle $e$, hence $\alpha \circ \mathfrak e= \alpha$ and $\beta \circ \mathfrak n = \beta$.
Moreover, an orbit under $\mathfrak e$ is a horizontal cylinder and an orbit under $\mathfrak n$ is a vertical one.

\begin{corollary}
By construction, $\G mn$ is always a grid of $(n-1)\times(m-1)$ vertices.
\end{corollary}

\subsection{Definition of Hooper diagrams and augmented diagrams}
 In \S \ref{hoogen}, we showed how from a surface we can construct a Hooper diagram, which is a 2-colored graph equipped with two edge permutations. In  \S\ref{hoopertobm}, we will show how to construct a \bm surface from a Hooper diagram.  We first give the formal definition of Hooper diagrams and define their \emph{augmented} version, which provides an useful tool to unify the treatment to include degenerate cases (coming from the boundary of the diagrams).

The data of a 2-colored graph $\graphg$, and the edge permutations $\mathfrak e$ and $\mathfrak n$, determine the combinatorics of our surface as a union of rectangles, as we will explain explicitly in this section. 
We will also give the width of each cylinder, to determine the geometry of the surface as well. 

We will first describe in general the Hooper diagram for $\M mn$. Here we use Hooper's notation and conventions from \cite{Hooper}.

\begin{definition}[Hooper diagram]\label{hooperdiagram}

Let $\Lambda=\{(i,j) \in \mathbb Z^2 \mid 1 \leq i \leq m-1$ and $1 \leq j \leq n-1\}$. 
Let $\mathcal A_{m,n}$ and $\mathcal B_{m,n}$ be two sets indexed by $\Lambda$, as follows:
\[
\mathcal A_{m,n}=\{\alpha_{i,j}, (i,j) \in \Lambda \mid i+j \text{ is even}\} \text{ and } \mathcal B_{m,n}=\{\beta_{i,j}, (i,j) \in \Lambda \mid i+j \text{ is odd}\}.
\]
Here $\mathcal A_{m,n}$ are the white vertices and $\mathcal B_{m,n}$ are the black vertices.

Let $\G{m}{n}$ be the graph with nodes $\mathcal A_{m,n} \cup \mathcal B_{m,n}$ formed by adding edges according to the usual notion of adjacency in $\mathbb Z^2$. 
In other words, we join an edge between $\alpha_{i,j}$ and $\beta_{i',j'}$ if and only if $(i-i')^2+(j-j')^2=1$, for all $(i,j), (i',j') \in \Lambda$ for which $\alpha_{i,j}$ and $\beta_{i',j'}$ exist. 
We define the counter-clockwise ordering of indices adjacent to $(i,j)$ to be the cyclic ordering
\[
(i+1,j) \to (i,j+1) \to (i-1,j) \to (i, j-1) \to(i+1,j).
\]
The clockwise order will clearly be the inverse order. 
We define then the map $\mathfrak e \colon \mathcal E \to \mathcal E$ to be the cyclic ordering of the edges with $\alpha_{i,j}$ as an endpoint. 
We order edges with endpoints $\alpha_{i,j}$ counter-clockwise when $i$ is odd and clockwise if $i$ is even. 
Similarly, $\mathfrak n \colon \mathcal E \to \mathcal E$ is determined by a cyclic ordering with $\beta_{i,j}$ as an endpoint. 
The opposite rule about the ordering of the cycle will be applied for $\beta_{i,j}$: we order the edges with endpoint $\beta_{i,j}$ clockwise when $j$ is odd and counter-clockwise when $j$ is even.

$\G{m}{n}$ is called the \emph{Hooper diagram} for $\M mn$.
\end{definition}

\smallskip
We now define the \emph{augmented Hooper diagram}, which will make it easier to construct the surface associated to a Hooper diagram. 
{
As we explained, the Hooper diagram is related to the polygonal representation of the surface in that the edges of the diagram represent cylinder intersections. 
When dividing the polygonal representation of the surface into parallelograms that are cylinder intersections, each polygon edge is either a diagonal of the parallelogram, or one of the sides of the parallelogram. For uniformity of treatment, it is convenient to think of these latter sides as \emph{diagonals} of \emph{degenerate} parallelograms  where one of the side lengths is zero. 
Introducing the augmented diagrams corresponds to introducing degenerate cylinder intersections (i.e. degenerate parallelograms) with one of the side lengths equal to zero,  
so that we can describe each side of the polygons as a diagonal of one of the (possibly degenerate) parallelograms corresponding to edges of the augmented Hooper diagram. }

  The \emph{augmented graph} $\G mn '$ is obtained by adding degenerate nodes and degenerate edges to the graph $\G{m}{n}$. 
If we consider the nodes of $\G{m}{n}$ in bijection with the coordinates $(i,j) \in \mathbb Z^2$, for $0<i<m$ and $0<j<n$, the nodes of $\G mn '$ will be in bijection with the coordinates $(i,j) \in \mathbb Z^2$, for $0 \leq i \leq m$ and $0 \leq j \leq n$.
The nodes we added are the \emph{degenerate nodes}.
On the new set of nodes we add a \emph{degenerate edge} if the nodes are at distance $1$ in the plane and they are not yet connected by an edge. 
Our graph $\G mn '$ is again bipartite and we extend coherently the naming conventions we described for $\G{m}{n}$.
We can see the augmented graph for $\M{3}{4}$ in Figure \ref{augmented}.

\begin{figure}[!h]
\centering
\includegraphics[width=0.5\textwidth]{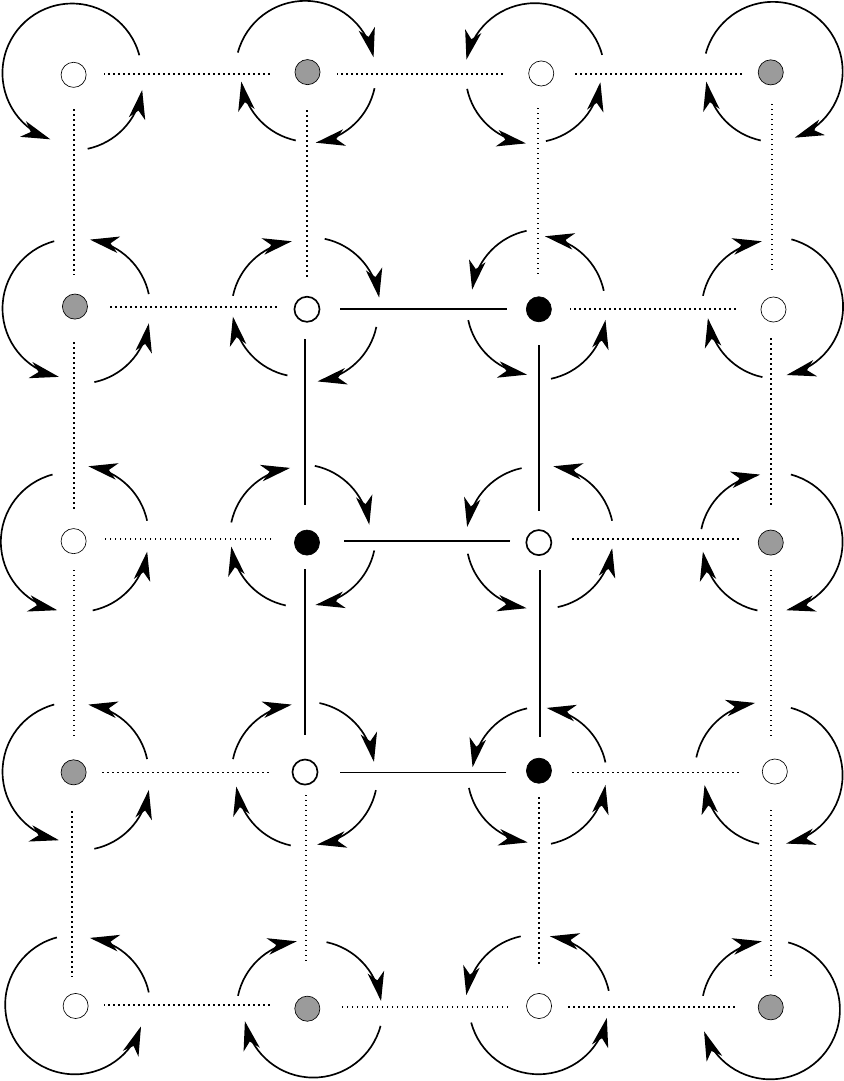}
\begin{quote}\caption{The augmented Hooper diagram for $\M{3}{4}$. \label{augmented}} \end{quote}
\end{figure}

Let $\mathcal E'$ denote the set of all edges of $\G mn '$, both original edges and degenerate ones. 
We say a degenerate edge $e \in \mathcal E'$ is \emph{ $\mathcal A$-degenerate}, \emph{$\mathcal B$-degenerate} or \emph{completely degenerate} if $\partial e$ contains a degenerate $\mathcal A$-node, a degenerate $\mathcal B$-node or both, respectively.
We also extend the edge permutations to $\mathfrak e', \mathfrak n' \colon \mathcal E' \to \mathcal E'$ following the same convention as before.

\subsection{From Hooper diagrams to \bm surfaces: combinatorics}\label{hoopertobm}

In Section \ref{hoogen}, we showed how from a surface we can construct a Hooper diagram. 
In this and the next  sections, we will show how to construct a \bm surface from a Hooper diagram $\G mn$ and describe it explicitly on the example of $\M{3}{4}$ we considered before. 
The data of a 2-colored graph $\graphg$, and the edge permutations $\mathfrak e$ and $\mathfrak n$, determine the combinatorics of our surface as a union of rectangles, as we will explain explicitly in this section.  We will also need the width of each cylinder, to determine the geometry of the surface as well. This is explained in the next section \S\ref{moduli}.

From the $(m,n)$ Hooper diagram $\G mn$ we can in fact recover the structure of two surfaces: $\M{m}{n}$ and $\M{n}{m}$ (which are affinely equivalent, see Theorem \ref{thm:Hooper}). In this section we will show how to construct $\M{m}{n}$, while in \S \ref{sec:affine} we will comment on how to recover also the dual surface. More precisely, we will often consider an \emph{intermediate} picture, that we will call the \emph{orthogonal presentation}, which contains both a sheared copy of $\M{m}{n}$ and a sheared copy of the dual $\M{n}{m}$ and allows us to easily see the relation between the two (see \S \ref{sec:affine}). 


To recover the combinatorics of the surface from its Hooper diagram, we need to decompose it into smaller pieces.
We will see that each piece corresponds to one polygon in the presentation in semi-regular polygons that was explained in the previous section.
The choice of which surface we obtain depends on our choice to decompose the graph into horizontal or vertical pieces. 
The vertical decomposition of the graph $\G{m}{n}$ will give us the combinatorics of the surface $\M{m}{n}$, while the horizontal decomposition produces $\M{n}{m}$, see \S\ref{sec:affine}. This is coherent with the operation of rotating the diagram to invert the role of $m$ and $n$, see Remark \ref{rk:rotating} for details. 



We now explain how to construct the surface starting from its graph, using the example of $\M{3}{4}$.
Let us decompose the augmented graph vertically, as shown in Figure \ref{dec}.  
We will consider as a piece a column of horizontal edges with the boundary vertices and all the edges between two of these vertices, no matter if they are degenerate or not.
In our case the decomposition will be as in the following figure, where, as before, the degenerate edges that have been added are represented with dotted lines.

\begin{figure}[!h]
\centering
\includegraphics[width=0.8\textwidth]{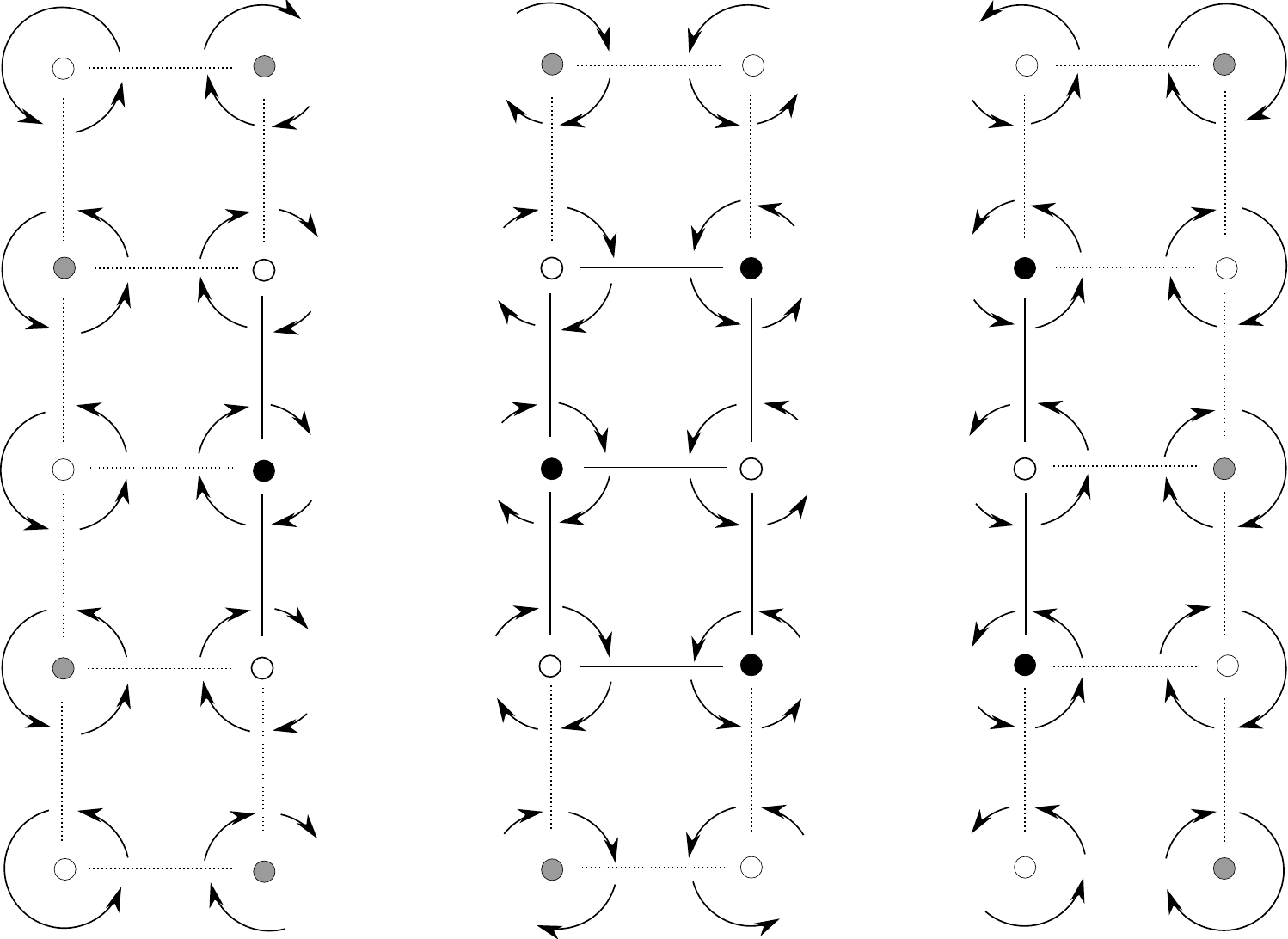}
\begin{quote}\caption{The three pieces of the vertical decomposition of $\G{3}{4}$. \label{dec}} \end{quote}
\end{figure}

Each edge will now represent a basic rectangle in our decomposition of the surface in polygons. 
We will still need the data of the width and height of the rectangle, which we will treat later. 
In Figure \ref{piece1} we label each edge and its corresponding basic rectangle with a letter, so that it is easy to see how to pass from one to the other. 

The \emph{degenerate edges} will correspond to degenerate rectangles, which means rectangles with zero width, or zero height, or both. 
The $\mathcal A$-degenerate edges correspond to rectangles with zero height (horizontal edges), the $\mathcal B$-degenerate edges correspond to rectangles with zero width (vertical edges), and the completely degenerate ones correspond to rectangles with zero width and zero height (points). 

Each rectangle coming from a vertical edge will contain a \emph{positive diagonal}, which means a diagonal with positive slope, going from the bottom left corner to the upper right corner. 
In the case of degenerate rectangles we will just identify the diagonal with the whole rectangle, so with a horizontal edge, a vertical edge or a point for $\mathcal A$-degenerate, $\mathcal B$-degenerate and completely degenerate edges respectively.
In the non-degenerate rectangles, this means that since each piece is repeated twice, in two pieces of our decomposition, each time we will include in our polygon one of the two triangles formed by the diagonal inside the rectangle. 

The permutation arrows between edges show us how the basic rectangles are glued. 
We will glue the rectangles according to the ``north'' and ``east'' conventions: following $\mathfrak e$-permutation arrows around white vertices corresponds to gluing on the right, and following 
 $\mathfrak n$-permutation arrows around black vertices corresponds to gluing  above. 
Moreover, such arrows will sometimes represent gluing in the interior of the same polygon, and other times they will represent gluing between a polygon and the following one. 
This will depend on whether the permutation arrows are internal to the piece we are considering or if they are between edges in different pieces of the decomposition. 

This is evident already in the first piece of our diagram. 
As in Figure \ref{piece1}, we can see that  the edges that contain both a black and a white degenerate vertex collapse to a point, as for the basic rectangles $a, e, f, h, j, l$. 
The edges containing only black degenerate vertices  collapse to a vertical edge, as for $b, d, g, m$. 
The edge $c$, containing a degenerate white vertex, will be a horizontal edge. 

\begin{figure}[!h]
\centering
\includegraphics[width=1\textwidth]{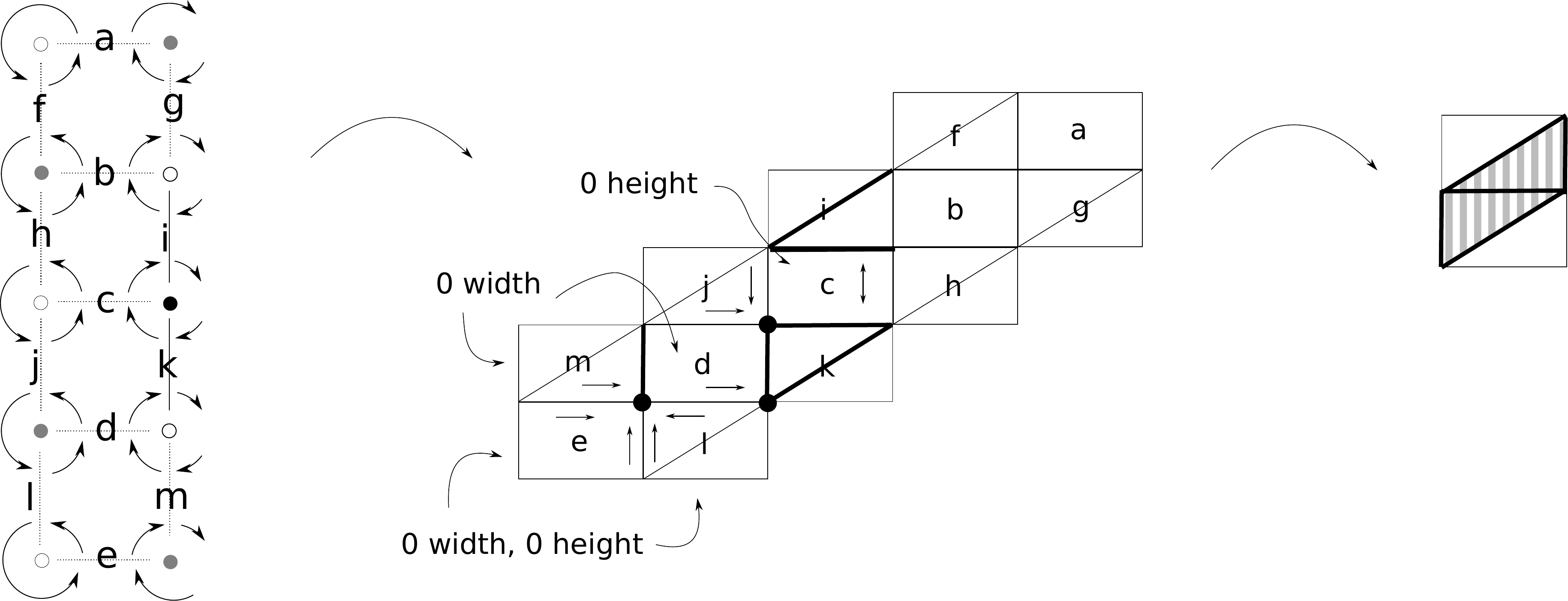}
\begin{quote}\caption{The first piece of the vertical decomposition of $\M{3}{4}$ and its orthogonal presentation. \label{piece1}} \end{quote}
\end{figure}

The remaining basic rectangles $k$ and $i$ are the only non-degenerate ones, each corresponding to half of a basic rectangle.
It is evident that the gluing between $k$ and $i$ internal to the piece is the one going upwards, passing through the horizontal edge represented by $c$. 
The result is a parallelogram as in the right picture of Figure \ref{piece1}. 
The diagonals in $k$ and $i$ will be glued to the other triangles, missing from the basic rectangle and that will appear in the following polygon. 
The other two sides will be glued to the next polygon and this is because the gluing correspond to the ``hanging arrows" shown in the left part of Figure \ref{octort}: a gluing on the left (arrow pointing to $m$ around a white vertex) for $m$ and a gluing on the right (arrow starting from $g$ around a white vertex). 

Doing the same thing for the other two pieces of the Hooper diagram, we get two parallelograms and a octagon glued together. 
We can see them in Figure \ref{octort} in what we will call the \emph{orthogonal presentation}. 
To return to the original polygonal presentation as described in section \ref{bmdefsection}, we need to shear back the cylinders to put them back in the original slope. 
The grid in the orthogonal presentation is in fact the vertical and horizontal cylinder decomposition. (Recall that the angle  we call \emph{vertical} is not ${\pi}/{2}$, but $\pi/n$.) 

\begin{figure}[!h]
\centering
\includegraphics[width=0.5\textwidth]{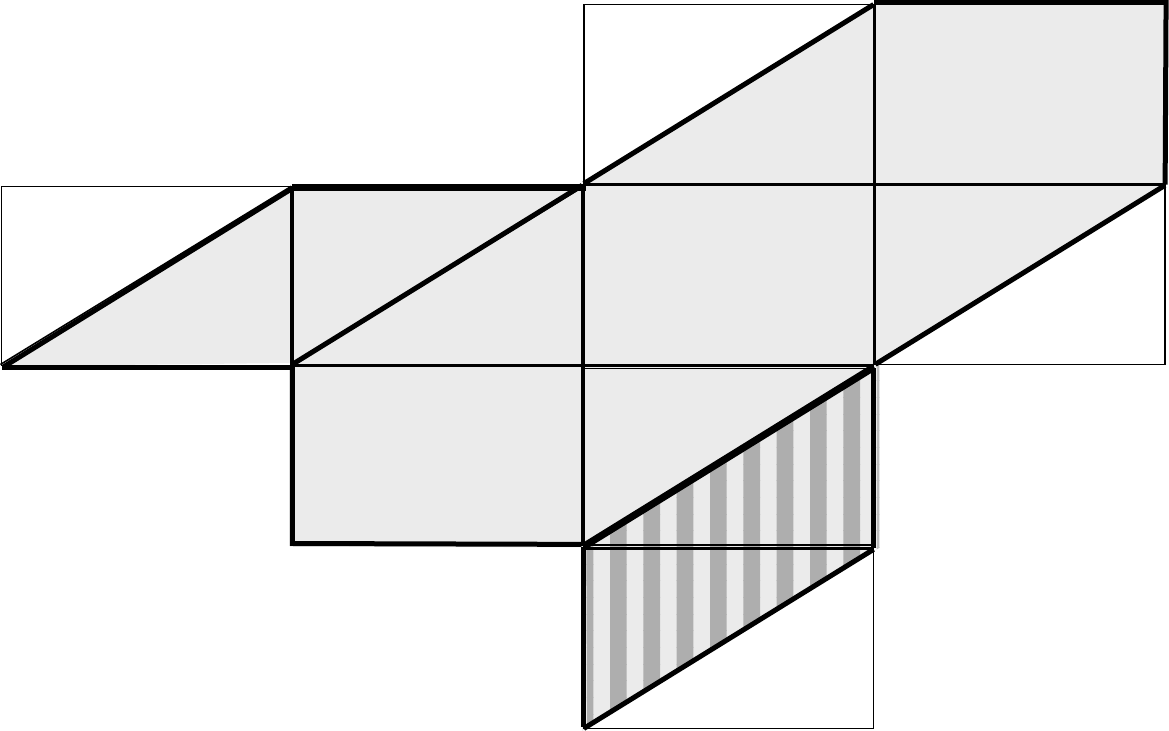}
\begin{quote}\caption{The orthogonal polygonal presentation of $\M{3}{4}$. \label{octort}} \end{quote}
\end{figure}

\subsection{From Hooper diagrams to \bm surfaces: widths of cylinders} \label{moduli}

Now that we reconstructed the combinatorial structure of the surface from a Hooper diagram,  we will explain how to recover the widths of the cylinders, which is the last piece of information to completely  determine the geometry of the surface. 
Indeed, widths automatically determine the heights of the cylinders as well: for how the two cylinder decompositions intersect, we can recover the heights from the formula:
\begin{align}\label{heigth}
\text{height}(\beta_i)= \sum_{j \in \Lambda} \#(\beta_i \cap \alpha_j) \text{width} (\alpha_j). 
\end{align}

This is because each part of the cylinder $\beta_i$ is in the surface, hence also in a cylinder $\alpha_j$. 
Measuring along the height of such a cylinder means counting each $\alpha_j$ we are intersecting and having a segment as long as its width. 

To recover the width we need to explain the concept of \emph{critical eigenfunctions}.
\begin{definition}
Let's assume in a general setting that $\graphg$ is a graph, connected, with no multiple edges or loops and $\mathcal V$ is the vertex set. 
Let $\mathcal E(x) \in \mathcal V$ be the set of vertices adjacent to $x \in \mathcal V$, which we assume is finite. 
Let $\mathbb C^\mathcal V$ be the set of functions $f \colon \mathcal V \to \mathbb C$. 
The adjacency operator is $H \colon \mathbb C ^\mathcal V \to \mathbb C ^\mathcal V$ defined by
\[
(Hf)(x)=\sum_{y \in \mathcal E(x)} f(x).
\]

An eigenfunction for $H$ corresponding to the eigenvalue $\lambda \in \mathbb C$ is a function $f \in \mathbb C^\mathcal V$, such that $Hf=\lambda f$.
\end{definition}

Now let $\mathcal L_\mathbb Z$ be the graph with integer vertices whose edges consist of pairs of integers whose difference is $\pm 1$.
$\graphg$ will now be a connected subgraph of $\mathcal L_\mathbb Z$, and again $\mathcal V$ is its vertex set.
If we assume $\mathcal V=\{1, \dots, n-1\}$, which will be our case, then:

\begin{definition}
The \emph{critical eigenfunction of $H$} is defined by 
\[
f(x)=\sin \frac{x \pi}{n}, \text{ corresponding to the eigenvalue } \lambda=\cos \frac{\pi}{n}.
\]
\end{definition}

We now consider $\mathcal I$ and $\mathcal J$, two connected subgraphs of $\mathcal L_\mathbb Z$, with vertex sets $\mathcal V_\mathcal I$ and $\mathcal V_\mathcal J$ respectively. 
Let $\graphg$ be the Cartesian product of the two graphs, as described in \cite{Hooper}.

Clearly our cylinder intersection graph is a graph of this type. 
For the graph $\G{m}{n}$, we then choose the widths of the cylinders to be defined by 
\[
w(\alpha_{i,j})=w(\beta_{i,j})=f_\mathcal I (i) f_\mathcal J(j),
\]
where $f_\mathcal I \colon \mathcal V_\mathcal I \to \mathbb R$ and $f_\mathcal J \colon \mathcal V_ \mathcal J \to \mathbb R$ are the critical eigenfunctions.

As we said, the graph fully determines the combinatorial structure for the surface. 
The flat structure is fully determined by choosing the widths of the cylinders, corresponding to vertices of the graph.
We take the critical eigenfunctions of the graph to be the widths of the cylinders:

\begin{corollary}
The \bm surface $\M m n$ has cylinder widths 
\[
w_{i,j}= \sin \left(\frac{ i \pi}{m} \right) \sin \left( \frac{j\pi}{n} \right),
\]
where $w_{i,j}$ is the width of the cylinder corresponging to the vertex $(i,j)$ of the Hooper diagram $\G m n $.
The height can be calculated using (\ref{heigth}).
\end{corollary}

\section{Affine equivalence of $\M mn$ and $\M nm$}\label{sec:affine}
In this section we explicitly describe the affine equivalence between the dual \bm surfaces $\M mn$ and $\M nm$ 
(see Theorem \ref{thm:Hooper}) that we will use to characterize cutting sequences.  As usual, we first describe it through the concrete example of   $\M 34$  and  $\M 43$, then comment on the general case.   
The semi-regular polygon presentation of $\M 34$ was shown in Figure \ref{octcyl}, with its horizontal and vertical cylinder decompositions, and the analogous picture for $\M 43$ is shown in Figure \ref{hexcyl}. In this section, we will show that the two surfaces are affinely equivalent.

\begin{figure}[!h]
\centering
\includegraphics[width=1\textwidth]{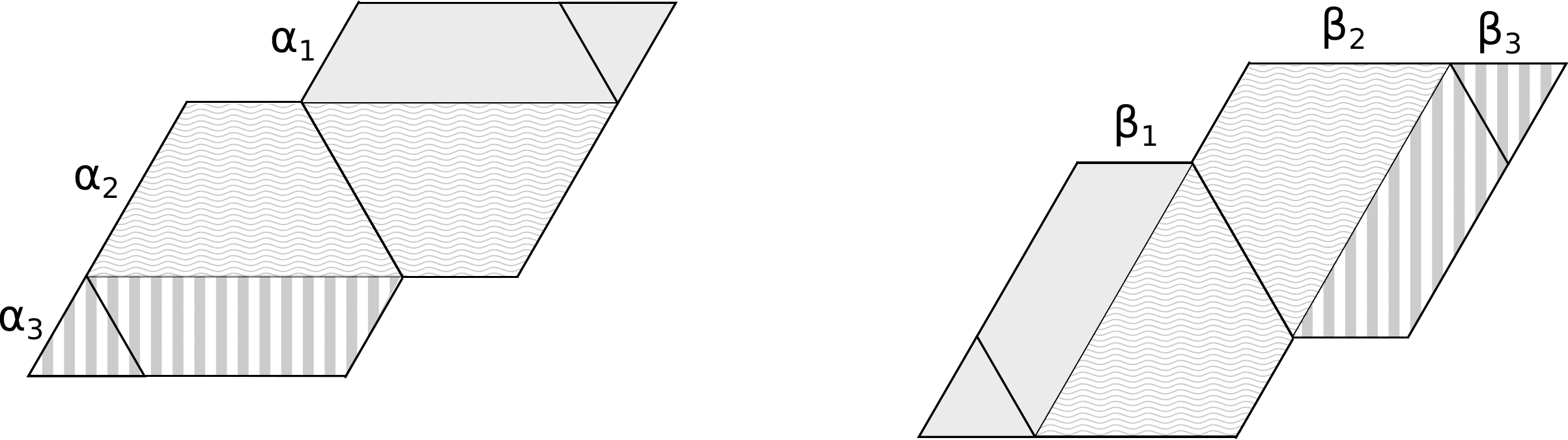}
\begin{quote}\caption{Horizontal and vertical cylinder decompositions for $\M{4}{3}$. \label{hexcyl}} \end{quote}
\end{figure}

We exploit  the orthogonal presentation we built in \S \ref{hooperdiagrams} (shown  for $\M{3}{4}$ in Figure \ref{octort}), which provides a convenient way  to visualize the central step of this equivalence. We first discuss how to go from one orthogonal decomposition to the dual one. We then combine this step with flips and shears, see \S\ref{flipandshears}.

\subsection{The dual orthogonal decomposition}

In \S \ref{hoopertobm}, we constructed an orthogonal presentation for $\M 34$, by \emph{cutting} the Hooper diagram vertically and associating to each piece a semi-regular polygon. This orthogonal presentation is in Figure \ref{octort}, and is in the top left of Figure \ref{hexort1} below.
We can consider the same graph $\G{3}{4}$ and decompose it into \emph{horizontal pieces}, instead of using the vertical pieces as we did in \S \ref{hoopertobm}, and then apply the same procedure.  
This produces the orthogonal presentation of the \emph{dual} \bm surface $\M 43$, shown  on the top right in Figure \ref{hexort1}.

\begin{figure}[!h]
\centering
\includegraphics[width=1\textwidth]{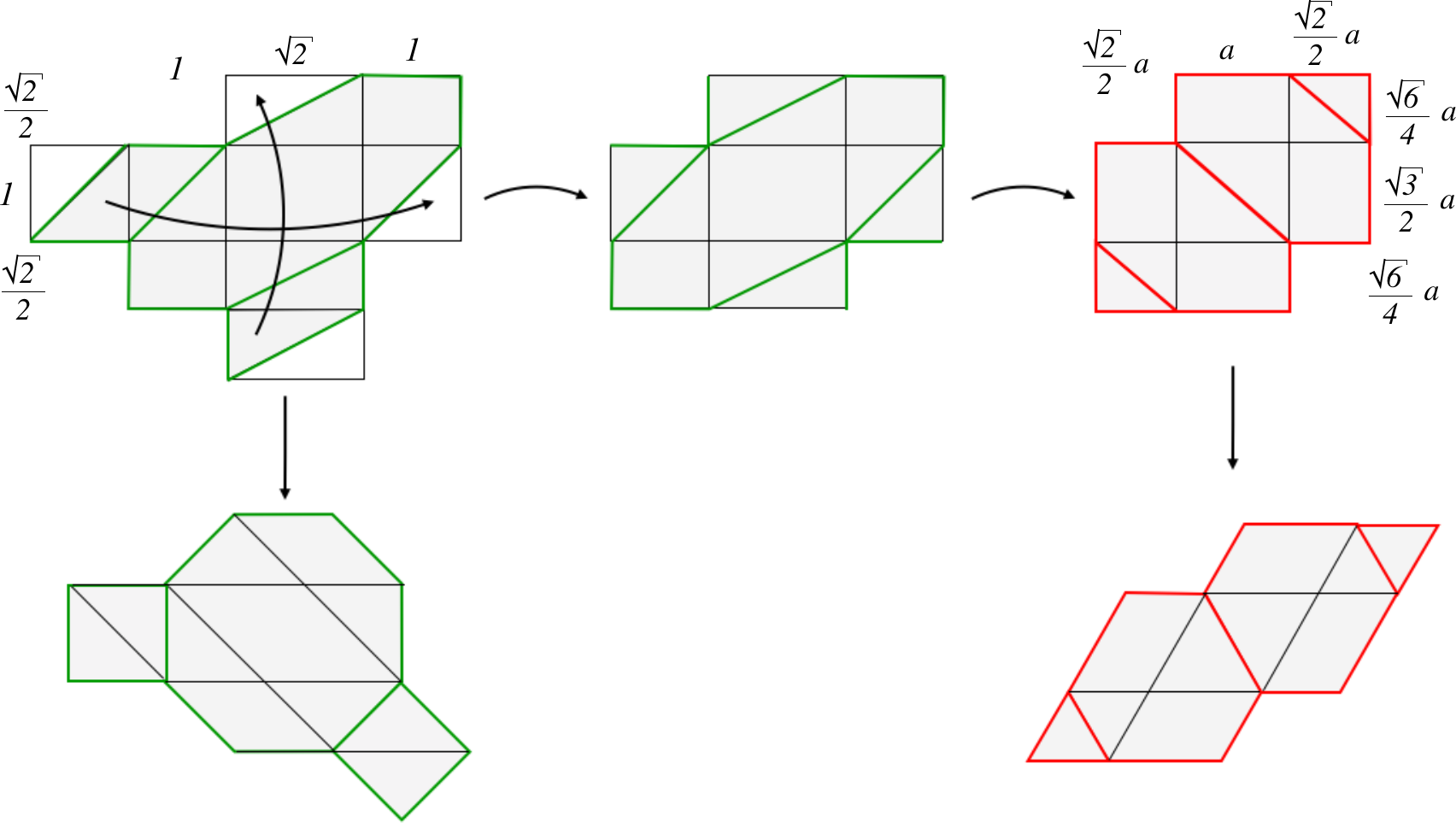}
\begin{quote}\caption{
The top row shows the orthogonal polygonal presentation of $\M 43$ from Figure \ref{octort}, and then a rearrangement of the same pieces, and then a different dissection of the same pieces into the dual orthogonal presentation of $\M 34$, after rescaling the lengths by a diagonal matrix. The second row shows the semi-regular presentations of $\M 43$ and $\M 34$, respectively, which are sheared images of the figures above. It also shows the length comparison for the two surfaces, where the number $a$ is calculated in (\ref{calculatea}). \label{hexort1}} \end{quote}
\end{figure}


\begin{remark}\label{rk:rotating} The same figure (i.e. the dual orthogonal presentation shown on the top right in Figure \ref{hexort1}) could  also be obtained by vertically decomposing  the graph $\G{4}{3}$ and repeating a procedure similar to the one described in \S \ref{hoopertobm}. 
The fact that the surface $\M mn$  and the surface $\M nm$ can each be obtained  from either of the respective graphs, $\G{m}{n}$ and $\G{n}{m}$ (decomposing each vertically), or both from the same graph $\G{m}{n}$ (one decomposing vertically, the other horizontally) is coherent with the construction of the graphs $\G{m}{n}$ and $\G{n}{m}$, because it is easy to check that we can obtain one from the other by rotating the graph by $\frac{\pi}{2}$ and changing the directions of the permutation cycles around the black dots.  

More precisely, from a point of view of the combinatorics of the surface, the change of direction of the permutations does not change it, because since the black dots represent a lateral gluing, the rectangles gluing on the right will be glued on the left instead and vice versa, which corresponds on the surface to a vertical flip. 
The rotation corresponds to the equivalence between the vertical decomposition of the first graph and the horizontal decomposition of the second one and vice versa. 
This can be seen on the polygonal presentation from the fact that we consider diagonals with different slopes if we decompose the graph vertically or horizontally. 
\end{remark}

\subsection{Cut and paste and rescaling between orthogonal presentations}\label{sec:cut_paste_rescale}
The procedure described above  can be done starting from any graph $\G{m}{n}$: by decomposing it vertically, we obtain an orthogonal presentation of $\M mn$; by decomposing it horizontally, a dual orthogonal presentation of $\M nm$. The two presentations, if we consider only the combinatorics of the surface, 
differ only by a cut and paste map. 
We can in fact cut along the horizontal and vertical diagonals in the two parallelograms that come from shearing a square and paste them along the side that was a diagonal of one of our basic rectangles, as shown in Figure \ref{hexort1} for the example of  $\M 43$ and $\M 34$.


\begin{remark} \label{alternate}
The rectangles containing a diagonal in \mnbms  are exactly the complementary ones to the rectangles containing a diagonal in $\M nm$.
This comes from the fact that by construction, only the vertical edges in one case and the horizontal ones in the other case are repeated and that the edges repeated in two different pieces are the ones which have a diagonal. 
One can see this in the top line of Figure \ref{hexort1}, and also later we will see them superimposed on the same picture in Figure \ref{coincide}.
\end{remark}

While from the point of view of the combinatorics of the surface the two presentations can be cut and pasted to each other, if we computed the associated widths of cylinders as described in \S \ref{moduli} we we would see that the lengths of the sides of the basic rectangles are not the same. Since we want both surfaces to have the \emph{same area} (in particular, we want them to be in the \emph{same Teichm\"uller disc}, 
we want to define a \emph{similarity} that allows us to rescale the lengths of the sides of the basic rectangles suitably. 

To determine the similarity, let us impose for the two surfaces in the orthogonal presentations  to have same areas, by keeping constant the ratios between the side lengths in the semi-regular polygon presentation.  Let us recall that the lengths, obtained from the polygonal description, or equivalently from the critical eigenfunctions for the graph, give us the lengths of the sides of the basic rectangles up to similarity. 

Let us work out  this explicitly in the $\M{3}{4}$ and $\M{4}{3}$ example. 
We can assume that the sides of the original octagon  all have length $1$. 
The area of the polygonal presentation of $\M{3}{4}$ will then be clearly $A_1=2(2+ \sqrt 2)$. 
Denoting by $a$ and $a'= \frac{\sqrt 2}{2} a$ the two side lengths of the sides of the polygons in $\M{4}{3}$, the area is $A_2= \sqrt 3 (1+\sqrt 2) a$. 
Requiring them to have the same area, $A_1=A_2$ gives us 
\be\label{calculatea}
a=\sqrt{\frac{2\sqrt 6}{3}} \quad \text{ and } \quad a'=\frac{\sqrt 2}{2} \sqrt{\frac{2\sqrt 6}{3}}.
\ee
From now on we will assume that $\M{4}{3}$ has these side lengths. 
Shearing the surface to make the two cylinder decomposition directions orthogonal gives us basic rectangles with the side lengths marked in Figure \ref{hexort1}.


The transformation that rescales the basic rectangles can be easily deduced from the figure, as the sides on the left and the corresponding sides on the right have the same ratio if we consider  the vertical ones and the horizontal ones separately. 
The transformation will hence be achieved by a diagonal matrix, with the two ratios as its entries. 
We remark that since we imposed for the area to be preserved, the matrix will be unitary.

For our example taking the orthogonal presentation of $\M 43$ to the orthogonal presentation of $\M 34$, this diagonal matrix is
\[
\diag{4}{3} = \matp {\sqrt[4]{\frac 32}} 0 0 {\sqrt[4]{\frac 23}}.
\]

We can extend all this reasoning to any generic \bm surface, and compute in a similar way a similarity that rescales the dual orthogonal presentation of  $\M{n}{m}$ so that it has the same area as the orthogonal presentation of  $\M{m}{n}$; see (\ref{def:generalmatrices}) for the general form.

\subsection{Flip and shears}\label{flipandshears}
So far we built a sheared copy of $\M{m}{n}$ (its orthogonal presentation) which can be cut and pasted and rescaled (as described in \S \ref{sec:cut_paste_rescale}) to obtained a  sheared copy of $\M{n}{m}$ (its dual orthogonal presentation). Thus, one can obtain an affine diffeomorphism between $\M{m}{n}$ and $\M{n}{m}$ through a shear, a cut and paste, a rescaling and another shear. In order to renormalize cutting sequences, we  also add a \emph{flip} (the reason will be clear later, see \S \ref{derivation}), to obtain the affine diffeomorphism $\AD m n: \M mn \to \M nm$ defined in formulas below. Let us first describe it in a concrete example.

\begin{example}
The affine diffeomorphism $\AD 43$ which we  use to map $\M 43$ to $\M 34$,  is realized by a sequence of flips, shears and geodesic flow shown in Figure \ref{hexTOoct}: starting from $\M 43$ we first apply the vertical flip $\flip$, then the shear $\shear {4}{3}$ to bring it to the orthogonal presentation. By cutting and pasting as explained in Figure \ref{hexort1} and then applying the diagonal matrix $\diag{4}{3}$ computed in the previous section, we obtain the  dual orthogonal presentation of $\M 34$. Finally, we shear the dual orthogonal presentation of  $\M 34$ to the semi-regular presentation of $\M 34$ by the shear $\shear {3}{4}$.

\begin{figure}[!h]
\centering
\includegraphics[width=1\textwidth]{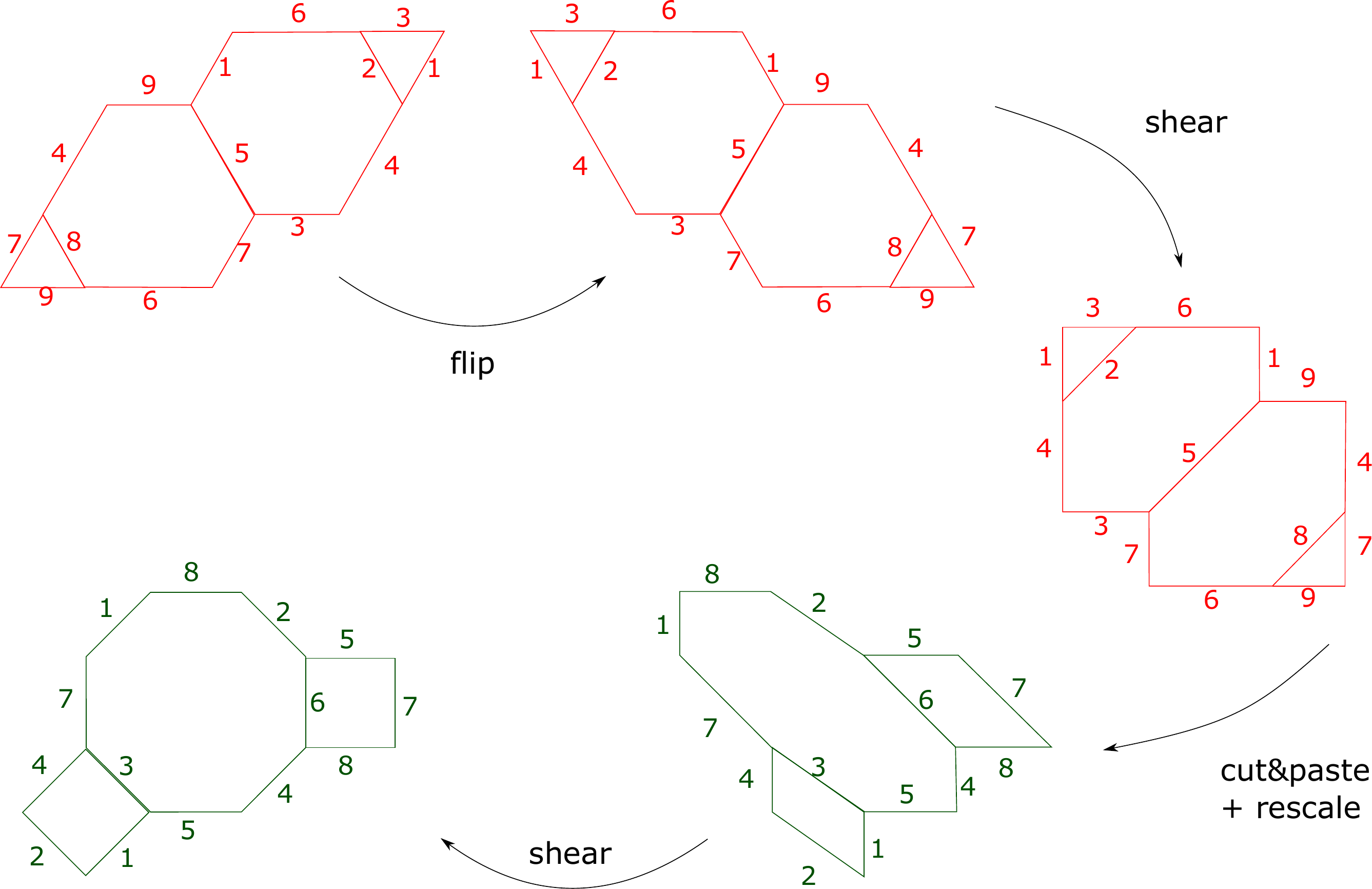}
\begin{quote}\caption{The affine diffeomorphism $\AD 43$ from $\M 43$ to $\M 34$, composition of $\flip$, $\shear {4}{3}$, cut and paste and $\diag{4}{3}$ and finally $\shear{3}{4}$. \label{hexTOoct}} \end{quote}
\end{figure} 
\end{example}

To define $\AD m n$ in  the \emph{general case},  consider a vertical flip $\flip$, the shear $\shear{m}{n}$ and the diagonal matrix $\diag{m}{n}$ given by:
\be \label{def:generalmatrices}
\flip  = \begin{pmatrix} -1 &  0 \\ 0 &  1  \end{pmatrix}, \qquad 
\shear m n  = \begin{pmatrix} 1 &  \cot \left( \frac {\pi}{n} \right) \\ 0 &  1  \end{pmatrix}, \qquad 
\diag{m}{n}  = \begin{pmatrix} \sqrt {\frac {\sin \frac {\pi}{n}}{\sin \frac {\pi}{m}}} & 0 \\ 0 & \sqrt {\frac {\sin \frac {\pi}{m}}{\sin \frac {\pi}{n}}}  \end{pmatrix}. 
\ee
The affine diffeomorphism $\AD m n$ is obtained by first applying the flip $f$ to $\M mn$ and shearing it by 
$\shear{m}{n}$, which produces the orthogonal presentation of $\M mn$. We then compose with the cut and paste map and the similarity given by $\diag{m}{n}$, which maps the orthogonal presentation of $\M mn$ to the dual orthogonal presentation of $\M nm$. Finally, we compose with the other shear $s_{nm}$ which produces the semi-regular presentation of $\M nm$. 

Thus, the \emph{linear part} of $\AD{m}{n}$, which we will denote by $\derAD m n $, is given by the following product:

\begin{equation} \label{def:derivativeAD}
\derAD m n   = \shear nm \diag{m}{n} \shear{m}{n} f = \begin{pmatrix} -\sqrt {\frac {\sin \frac {\pi}{n}}{\sin \frac {\pi}{m}}} & \frac{\cos\frac{\pi}{m}+\cos \frac{\pi}{n}}{\sqrt{\sin \frac{\pi}{m} \sin \frac{\pi}{n}}} \\ 0 &  \sqrt {\frac {\sin \frac {\pi}{m}}{\sin \frac {\pi}{n}}} \end{pmatrix}.
\end{equation}
The action of $\AD m n$ on directions will be described in \S \ref{sec:2farey}, and the action on cutting sequences in \S \ref{derivation}.

\section{Stairs and hats}\label{stairsandhats}

In this section we will explain in detail one particular configuration of basic rectangles in the  orthogonal presentation, and the corresponding configuration in the Hooper diagram. 
We put particular emphasis on it because we will be using throughout the next sections. 

Let us consider a piece of an orthogonal presentation given by six basic rectangles, glued together as in Figure \ref{stair}. 

\begin{figure}[!h]
\centering
\includegraphics[width=0.3\textwidth]{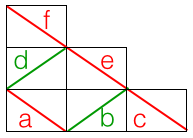}
\begin{quote}\caption{Configuration of a stair. \label{stair}} \end{quote}
\end{figure}

\begin{definition}
A \emph{stair} is a piece of an orthogonal presentation made of six basic rectangles. 
They are glued together so that we have three columns, made of three, two and one rectangle respectively, as shown in Figure \ref{stair}.
\end{definition}

As we did all through \S \ref{hooperdiagrams}, we will need to pass from the Hooper diagram to the orthogonal presentation. 
First, we will explain what a stair corresponds to in a Hooper diagram. 
Clearly, it will be a piece of diagram made of six edges, with some vertices between them. 
The exact configuration will depend on the parity of the vertices, i.e. on the position of the piece in the diagram. 

The piece corresponding to a stair will be one of the configurations in Figure \ref{hat}, which we call a \emph{hat}.

\begin{figure}[!h]
\centering
\includegraphics[width=0.8\textwidth]{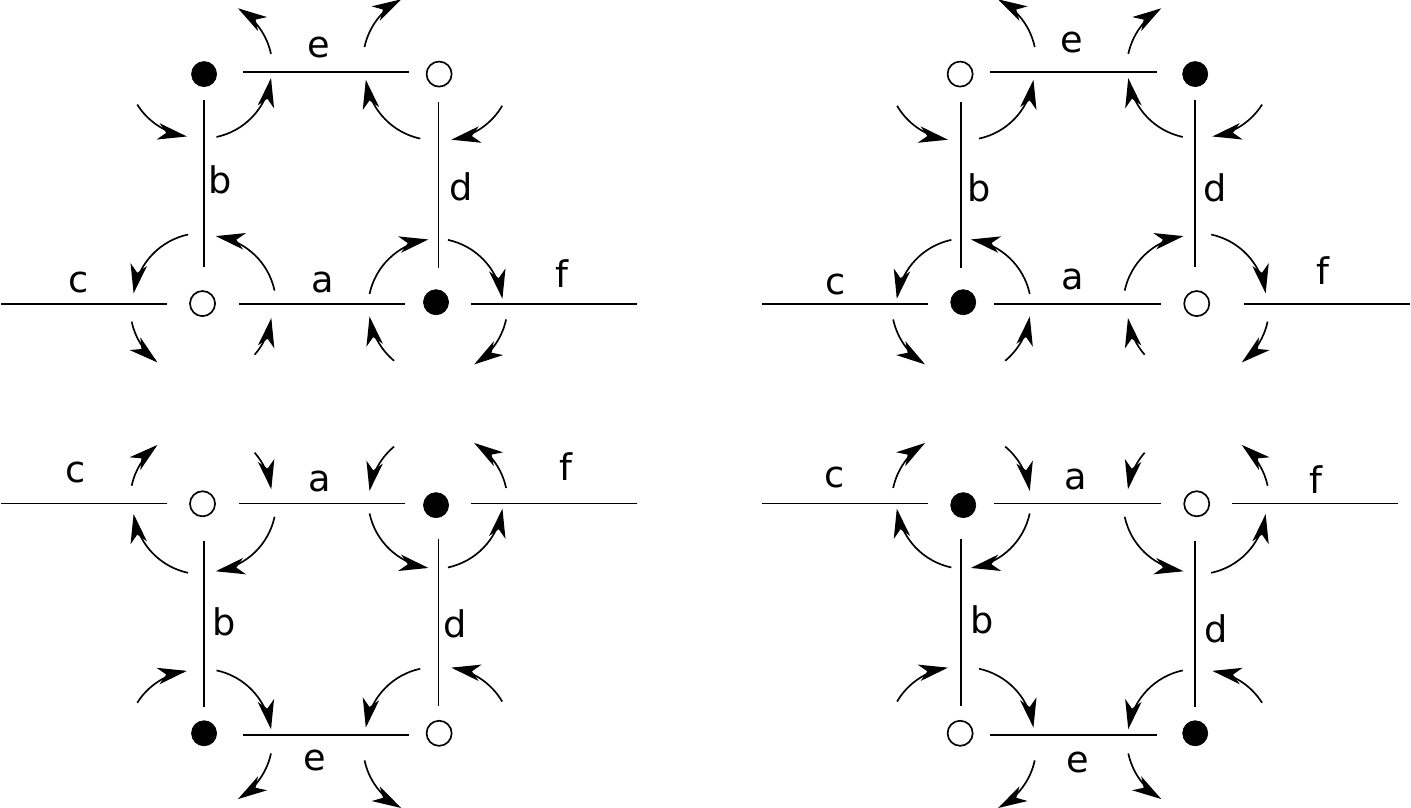}
\begin{quote}\caption{Possible configurations of hats. \label{hat}} \end{quote}
\end{figure}

More precisely:

\begin{definition}\label{def:hat}
A \emph{hat} is a piece of a Hooper diagram made of six edges. 
Two of them are vertical and the others are horizontal, in the configuration shown in Figure \ref{hat}. 
Moreover, if the two vertical ones go upwards from the vertices of the three-piece base, the first column has counter-clockwise permutation arrows; it has clockwise permutation arrows otherwise. 
\end{definition}

According to the parity, these vertices described can be black or white and have permutation arrows turning around clockwise or counter-clockwise. 
This gives us four possible configurations, as in Figure \ref{hat}.

The direction of the permutation arrows depends on the number of the column. 
As we saw, in fact, in odd columns we have arrows turning clockwise, while in even columns we have arrows turning counter-clockwise. 
As we explained in the definition, this determines also the position of the two vertical edges. 

Given the parity of the column, the two different possibilities of the vertex colorings are determined from the parity of the row. 
In an odd column we will have white vertices on odd rows and black vertices on even rows, 
and the opposite in an even column.

Notice that the vertex in the lower left corner determines everything: 
Its color together with the direction of its arrows determines the parity of the row and column of its position, and determines in which of the four possible hats we are in. 

The first case, with a white vertex and counter-clockwise arrows, corresponds to a corner position in an even row and an even column. 
The second one, with a black vertex but still counter-clockwise arrows, corresponds to an odd row and an even column. 
The third one, with a black vertex but clockwise arrows, corresponds to an even row and an odd column. 
The last one, with a white vertex and clockwise arrows again, corresponds to an odd row and an odd column. 

\subsection{Stairs and hats correspondence} 
We will now show that the stair and hat configurations correspond to each other. 
To do that we will use our method of passing from the Hooper diagram to the orthogonal decomposition and vice-versa. 

\begin{lemma}[Hat Lemma]\label{hatlemma}
The stair configurations correspond exactly to the four possible hat configurations.
\end{lemma}

\begin{proof}
First, we show that if we have one of the hat configurations, it actually gives a stair configuration. 
We will show it in detail for the first case and the others will work in exactly the same way. 

Let us consider a labeling on the hat in the upper-left of Figure \ref{hat}. 
As before, each edge corresponds to a basic rectangle. 
The three edges around the white vertex in the left bottom corner and the arrows around it, tell us that we will have three basic rectangles, glued one to each other on the right, in the order $a$ glued to $b$, glued to $c$. 
On the other hand, the three edges around the black vertex at the other extremity of the edge $a$, and its arrows, tell us that a basic rectangle labeled $f$ is glued on top of one labeled $d$ which is glued on top of the one labeled $a$. 
Finally, the basic rectangle $e$ is glued on top of $b$, and on the right of $d$, and we obtain the configuration in Figure \ref{stair}. 

The other three cases work the same way.

Secondly, we show that if we have a stair configuration, it will necessarily give a hat configuration on the Hooper diagram. 
Let us consider a stair configuration, with the same labels as in Figure \ref{stair}. 
The basic rectangle $a$ will correspond to an edge, and we do not know if it will be horizontal or vertical. 
We assume for the moment that it is a horizontal edge (we will explain later why the same figure, but rotated so that $a$ is vertical, is not acceptable). 
At this point we have the choice of on which extremity of $a$ we want to record the left-right adjacency and the upwards-downwards one. In other words, we have the choice of where to put a black vertex and where to put a white one. 
This gives us two possible cases. 

If we have the white vertex (resp. the black vertex) on the left of the edge $a$, we will record the gluing with $b$ and then $c$ (resp. $d$ and $f$) on that side. 
Again, we have the choice of recording it putting $b$ (resp. $d$) going upwards from the vertex or downwards. 
This leads to split each of the two cases in two more.
If $b$ (resp. $d$) is above the line of $a$, the permutation arrows around the white vertex will go counter-clockwise (resp. clockwise).
Now, on the other extremity of the edge $a$, we record the other adjacency and add the edges $d$ and $f$ (resp. $b$ and $c$) in order. 
It looks like we have again a choice of whether to draw $d$ (resp. $b$) going upwards or downwards, but it is not difficult to see that the previous choice determines also this one. 
In fact, if edge $b$ was going upwards, the edge $d$ will have to go upwards as well, because the edge $e$ is obtained both from the upwards gluing from $b$ and from the right gluing from $d$. 
The diagram does not intersect itself and we cannot repeat an edge, hence the two vertical ones have to be in the same direction. 

This also shows that $a$ needs to be horizontal, because having the two vertical edges in the same direction makes the permutation arrow go in different orientation around the two vertices, and we saw that if they are in the same column, then they must have the same orientation. 

It is clear that we cannot have any other possibility and the four possibilities just described correspond to the four hat configurations in Figure \ref{hat}. 
\end{proof}

\subsection{Degenerate hat configurations}
Let us recall that to unify and simplify the description of \bm surfaces via Hooper diagrams we introduced a \emph{augmented diagram}, which allows us to treat the boundary of the Hooper diagram as a degenerate case of a larger diagram (see \S \ref{hooperdiagram}). We now describe degenerate hat configurations that correspond to boundary configurations in the Hooper diagram. We will use them later, in \S \ref{sec:structure}, to prove our main structure theorem.

We will shade the six edges to pick out a hat configuration, as shown in Figure \ref{hat-cases}. The \emph{middle edge} is the one that is numbered in Figure \ref{hat-cases} (or edge $a$ from Figure \ref{hat}).


\begin{lemma}[Degenerate Hat Lemma] \label{degeneratehat}
All edges of an augmented Hooper diagram that form a subset of a hat, such that the middle edge of the hat is an edge of the augmented diagram, fall into one of the four cases in Figure \ref{hat-cases}.

%
%
%
%
\end{lemma}

\begin{figure}[!h] 
\centering
\includegraphics[width=.6\textwidth]{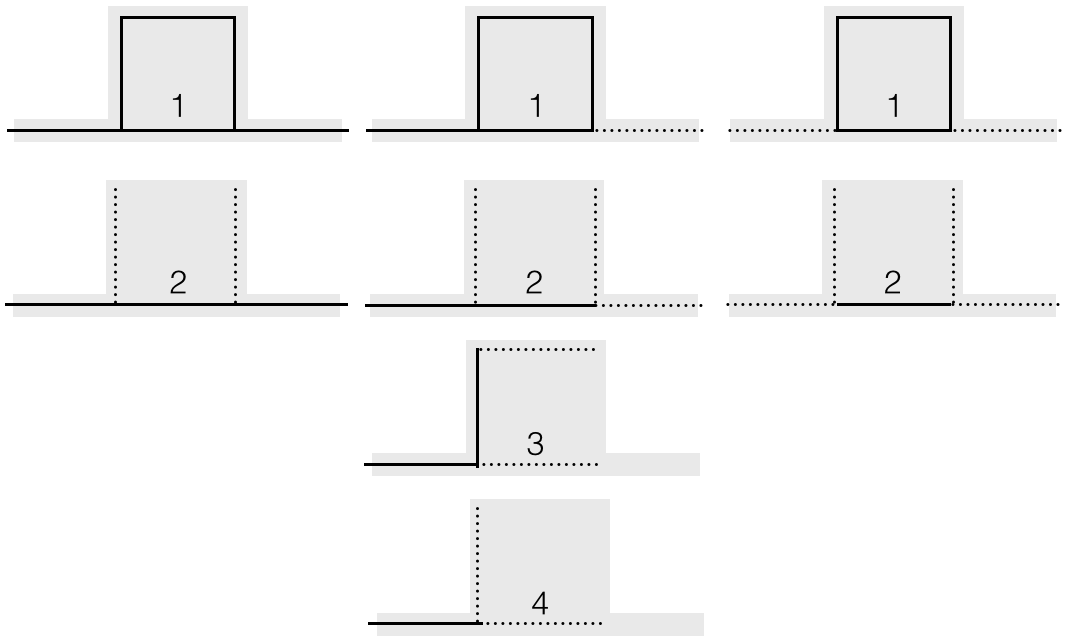} 
\begin{quote}\caption{The cases $1-4$ for hats \label{hat-cases}} \end{quote}
\end{figure}

\begin{proof}
The reader can easily verify, using a diagram such as Figure \ref{all-hats}, that any orientation and placement of a hat whose middle edge is an edge of an augmented Hooper diagram falls into one of the cases $1-4$.
\end{proof}

\begin{figure}[!h] 
\centering
\includegraphics[width=400pt]{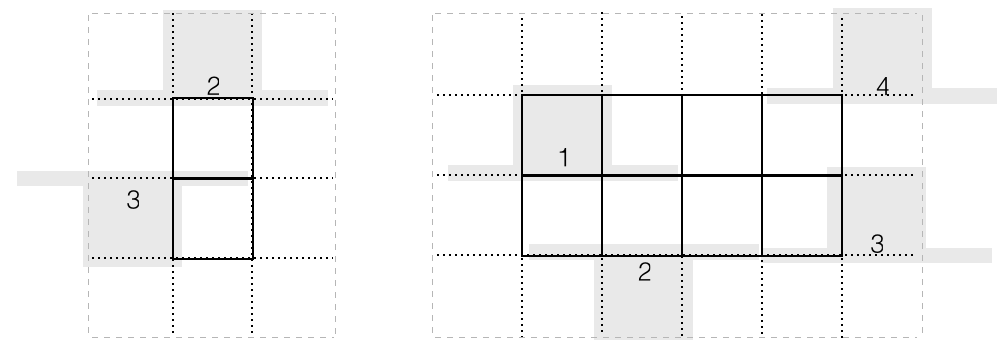} 
\begin{quote}\caption{Examples of the four possible cases for hats. We do not include the outer degenerate edges (dashed gray) in our hats because they are not adjacent to the middle edge. \label{all-hats}} \end{quote}
\end{figure}

These cases are illustrated in Figure \ref{all-hats}. We use hats, and degenerate hats, in Lemma \ref{degeneratearrow}, which is a main step in our Structure Theorem \ref{tdtheorem} for derivation diagrams.

%

\subsection{Dual surfaces}

In this section, we will prove a Lemma that uses the stairs and degenerate stair configurations, and  will be used later to define the derivation diagrams that give  derivation rules. 

Consider the superposition of the orthogonal presentation of $\M mn$ and the dual orthogonal presentation $\M nm$, as shown in Figure \ref{coincide}. 
Recall that sides of (the sheared images of) $\M mn$ and of $\M nm$ appear as diagonals of alternating basic rectangles. Sides of either presentation that are horizontal or vertical  can be thought of as degenerate diagonals, i.e. degenerate basic rectangles of zero width or zero height, described by the  \emph{augmented} Hooper diagram (see \S \ref{hoopertobm}).

As before, let us call \emph{positive diagonals} the sheared images of sides of $\M nm$ (which, if not vertical or horizotal, have slope $1$) and let us now call \emph{negative diagonals } the sheared images of sides of $\M mn$ (which, if not vertical or horizontal, have slope $-1$).  As observed in Remark \ref{alternate}, positive and negative diagonals \emph{alternate}, in the sense that the neighboring basic rectangles with a positive diagonal are adjacent (right/left or up/down) to basic rectangles with a negative diagonal. This remark holds true for all sides, including vertical and horizontal ones,  if we think of them as degenerated diaganals and draw them according to the following convention:

\begin{convention}\label{convention:ordered_sides}
When degenerate sides of the orthogonal presentation of $\M nm$ and of the dual orthogonal presentation of $\M nm$ \emph{coincide}, we think of them as degenerated diagonals and hence we draw them adjacent to each other and ordered so that degenerate positive (red) and negative (green) diagonals  \emph{alternate} in horizontal and vertical directions, as shown in Figure \ref{coincide} for $\M  43 $ and $\M 34$. 
\end{convention}

\begin{figure}[!h] 
\centering
\includegraphics[width=400pt]{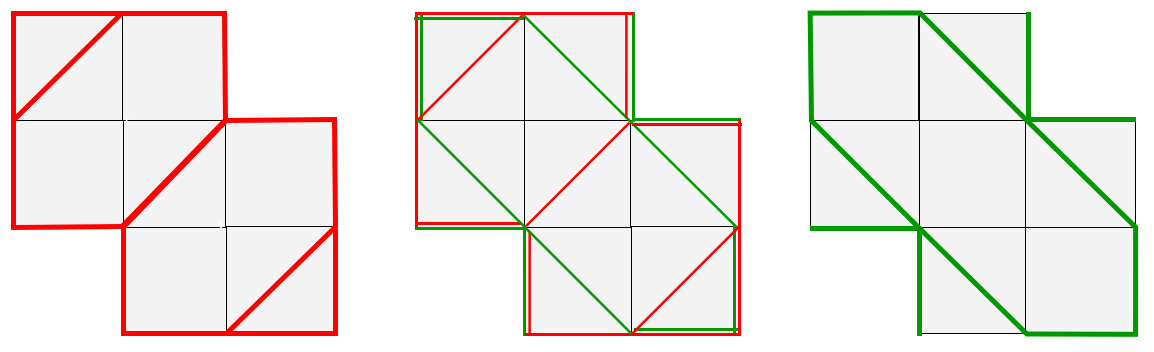} 
\begin{quote}\caption{The orthogonal presentations of $\M 43$ (green, left) and $\M34$ (red, right), superimposed on the same figure (center). Coinciding horizontal and vertical edges alternate red and green in the horizontal and vertical directions. Here the space between coinciding edges is exaggerated for clarity. \label{coincide}} \end{quote}
\end{figure}

Consider trajectories whose direction belongs to the first quadrant, i.e. such that $\theta \in [0, \pi/2]$. Let us say that a pair of negative diagonals is \emph{consecutive} if there exists such a trajectory which hits these two diagonals one after the other. 
 
 
\begin{lemma}\label{lemma:intertwined}
Consider a pair of consecutive negative diagonals. Then, the following dichotomy holds: for \emph{any}    trajectory whose direction belong to the first quadrant, i.e. such that $\theta \in [0, \pi/2]$, either
\begin{itemize}
\item   between any consecutive crossings of these pairs of negative diagonals, no positive diagonal is crossed, or
\item between any consecutive crossings of these pairs of negative diagonals, exactly one and the same positive diagonal is crossed.
\end{itemize} 
\end{lemma}

\begin{proof} 
Assume first that the two negative adjacent negative diagonals are non-degenerate. The fact that they are adjacent means that one can find a stair configuration as in Figure \ref{stair}, in which the two diagonals are the ones labeled by either $a$ and $c$, or $a$ and $e$, or $a$ and $f$.  It is then clear from the stairs picture in Figure \ref{stair} that (referring to the labeling in that figure)  if the pair is given by $a$ and $e$, then a trajectory whose direction belongs to the first quadrant that crosses these two negative diagonals never crosses any positive diagonal in between, while if the pair is $a$ and $c$  or $a$ and $f$, such a trajectory will always cross a negative diagonal between, i.e. $b$ (for the pair $a$ and $c$) or $d$ (for the pair $a$ and $f$).

Convention \ref{convention:ordered_sides}, which treats vertical and horizontal sides as degenerate diagonals, allows us to use exactly the same proof for degenerate stairs (when some of the $6$ edges in the stair are degenerate). 
\end{proof}

In \S \ref{sec:labeled_def} this Lemma will be used to define derivation diagrams.

\section{Transition and derivation diagrams} \label{transitiondiagrams}

We now return to the polygon decomposition of the \bm surfaces and to our goal of characterizing all cutting sequences. First, in \S \ref{howtolabel} we will describe how to label the edges of the semi-regular presentation of the \bm surface. We then show in \S \ref{sec:labelingHooper} how this labeling induces a labeling on the corresponding Hooper diagram.  In \S\ref{sec:transition_diagrams} we define \emph{transition diagrams}, which are essential for understanding cutting sequences. In \S\ref{sec:admissibility} we define \emph{admissible} cutting sequences, generalizing the work of Series and Smillie-Ulcigrai discussed in \S\ref{sec:Sturmian}-\ref{sec:polygons}. In \S\ref{sec:labeled_def} we define \emph{derivation diagrams}, which are the key tool we will use to characterize cutting sequences on \bm surfaces. In \S\ref{sec:structure}, we prove our structure theorem for derivation diagrams for trajectories in $\Sec 0mn$, which is the main result of this section. In \S\ref{sec:normalization}, we describe how to \emph{normalize} trajectories in other sectors to $\Sec 0mn$. In \S\ref{sec:other_sectors}, we describe transition diagrams for trajectories in other sectors.


\subsection{Edge labeling}\label{howtolabel}

To label the edges of the \bm surfaces, we use a ``\emph{zig-zag}'' pattern as follows. First, we label the lower-right diagonal edge of $P(0)$ with a $1$, and then go horizontally left to the lower-left diagonal edge of $P(0)$ and label it with a $2$ (see Figure \ref{labeling}). Then we go up diagonally to the right (at angle $\pi/n$) and label the next edge $3$, and then go horizontally left and label that edge $4$, and so on until label $n$. The $n$ edges of $P(1)$ that are identified with these edges have the same labels. 

Now we label the remaining $n$ edges of $P(1)$. If the bottom horizontal edge is already labeled (as in Figure \ref{labeling}a below), we start with the lowest-right diagonal edge and label it $n+1$, and then go horizontally to the left and label that edge $n+2$, and then zig-zag as before. If the bottom horizontal edge is not yet labeled (as in Figure \ref{labeling}b below), we label it $n+1$, and then go diagonally up to the right and label that edge $n+2$, and so on in the zig-zag. We do the same for $P(2)$ and the remaining polygons until all the edges are labeled.

\begin{figure}[!h] 
\centering
\includegraphics[width=0.8\textwidth]{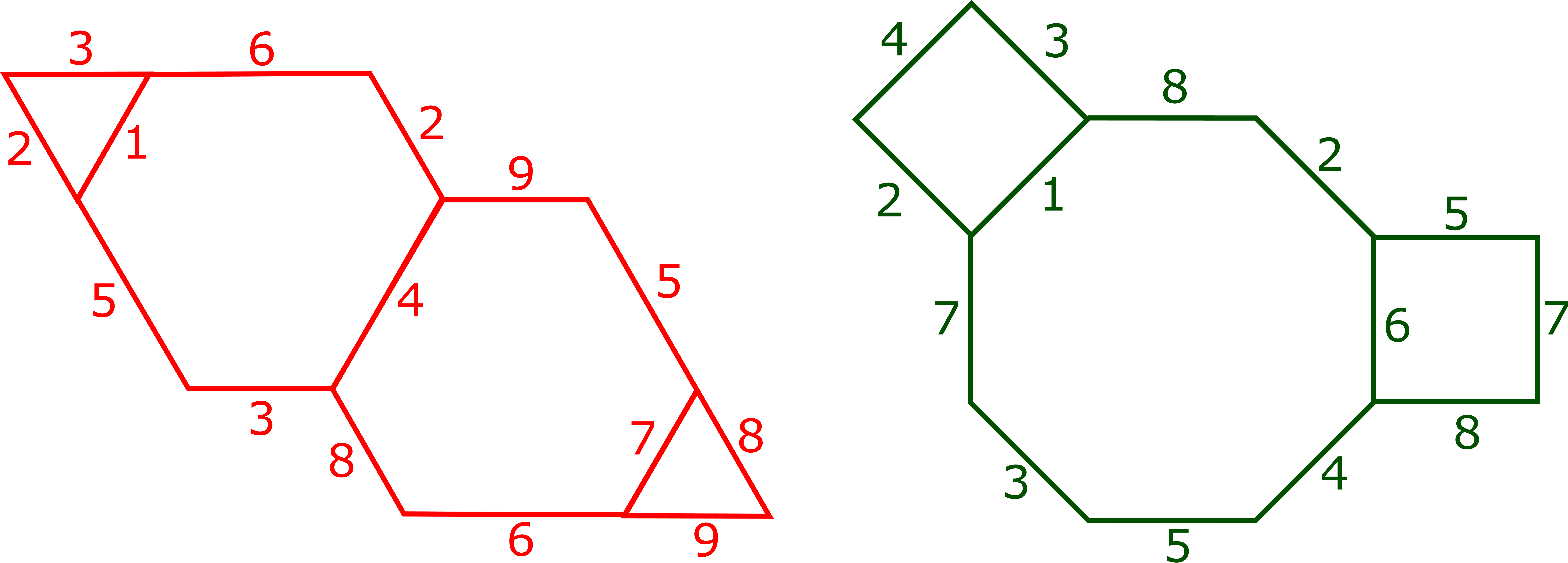}
\begin{quote}\caption{The edge labelings for $\M 34$ and $\M 43$ \label{labeling}} \end{quote}
\end{figure}

We choose to label the edges in this way because it makes the \emph{transition diagrams} easy to describe, as we will see. We can first reap the benefits of this labeling system by labeling the edges of the Hooper diagram.


\subsection{Labeling the Hooper diagram}\label{sec:labelingHooper}

Each edge of the Hooper diagram $\G mn$ corresponds to the intersection of a horizontal cylinder and a vertical cylinder, which is a basic rectangle in the orthogonal decomposition. Each non-degenerate basic rectangle is crossed by an edge of either $\M mn$ or $\M nm$: a negative diagonal for the (red) edges of $\M mn$ or a positive diagonal for the (green) edges of $\M nm$. We can label the edges of the Hooper diagram with the label of the edge that crosses the corresponding basic rectangle.

\begin{proposition}\label{hd-snake}
In $\G mn$, the labels are as follows:

The upper-left horizontal auxiliary edge is edge ${\rd 1}$ of $\M mn$, and thereafter the horizontal edges are labeled ${\rd 2,3,4}$, etc., ``snaking'' horizontally back and forth from top to bottom, as shown in Figure \ref{hd-labeled}b.

The upper-left vertical auxiliary edge is edge ${\gr 1}$ of $\M nm$, and thereafter the vertical edges are labeled ${\gr 2,3,4}$, etc., ``snaking'' vertically up and down from left to right, as shown in Figure \ref{hd-labeled}b.
\end{proposition}

In Figure \ref{hd-labeled}a, ``up'' and ``down'' are reversed because of the conventions in the Hooper diagram, but we choose to orient the $1$s in the upper left; see Remark \ref{hd-reflect}.

\begin{proof}
We begin with an Hooper diagram, including the edges that are either horizontally degenerate, or vertically degenerate. (We omit edges that are completely degenerate, because they are points and thus do not have polygon edges associated with them.) This is the black part of the diagram in Figure \ref{hd-labeled}. We will determine where the (colored) edge labels go on the diagram in several steps.

Recall that the white vertices represent horizontal cylinders, with the arrows indicating movement to the right, and the black vertices represent vertical cylinders, with the arrows indicating movement up.

\emph{Step 1}: The (red) edges of $\M mn$ and the (green) edges of $\M nm$ comprise the horizontal and vertical sets of edges of the Hooper diagram. We can determine which is which by counting: $\M mn$ has $n(m-1)$ edges and $\M nm$ has $n(m-1)$ edges. If $m=n$, the diagram is symmetric so it doesn't matter which is which.

For our example, $\M 43$ has $9$ edges, so they are the horizontal edges in Figure \ref{hd-labeled}a, and $\M 34$ has $8$ edges, so they are the vertical edges in Figure \ref{hd-labeled}a. This means that the horizontal edges will have red edge labels, and the vertical edges will have green edge labels.

\emph{Step 2}: We determine where to put the edge label ${\rd 1}$. ${\rd 1}$ is a degenerate edge, so it must be one of the outer (dotted) diagram edges. ${\rd 1}$ is in $\M mn$, so it must be a horizontal diagram edge. ${\rd 1}$ is parallel to the vertical cylinder decomposition, so it lies in a horizontal cylinder, so it emanates from a white vertex. When we go against the arrow direction from ${\rd 1}$, we get to ${\gr 1}$, which is also a degenerate edge, so it must be on a corner (see Figure \ref{hd34aux}). 

All of these narrow our choices to just one, or sometimes two when the diagram has extra symmetry; in that case, the two choices are equivalent. In our example, there is only one choice, the edge labeled ${\rd 1}$ in Figure \ref{hd-labeled}. 

\begin{figure}[!h] 
\centering
\includegraphics[width=350pt]{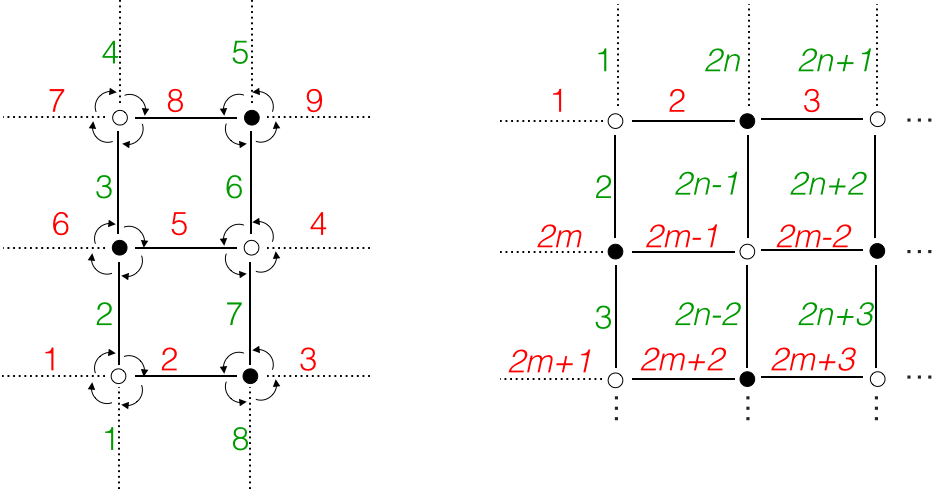}
\begin{quote}\caption{The labeled Hooper diagram for $\M 43$, and the general form (see Remark \ref{reflect}). We do not include the bottom edge, the right edge, or the bottom-right corner of the general form in Figure \ref{hd-labeled}, because the edge labels and the vertex colors depend on the parity of $m$ and $n$, so it is clearer to look at the example. \label{hd-labeled}} \end{quote}
\end{figure}

\emph{Step 3}: We determine where to place edges ${\gr 1},{\rd 2,\ldots,n,n+1}$. 

From edge ${\rd 1}$ in $\M mn$, we go horizontally to the left to get to ${\rd 2}$, and in between we pass through ${\gr 1}$ (see Figure \ref{hd34aux}). On the Hooper diagram, from edge ${\rd 1}$ we go against the arrows around the white vertex, and label the vertical edge ${\gr 1}$ and the next horizontal edge ${\rd 2}$. 

From edge ${\rd 2}$ in $\M mn$, we go in the direction of the vertical cylinder decomposition to get to ${\rd 3}$, so we go with the arrows around the black vertex and label the next horizontal edge ${\rd 3}$. In our example $\M 43$, this is the end of the row; for $m>3$, we continue until we get to ${\rd n}$, going left and up in the polygons and correspondingly going around the white and black vertices in the Hooper diagram.

To get from edge ${\rd n}$ to ${\rd n+1}$, in the polygons we go up and right for $n$ odd, and left and down for $n$ even, and we follow the arrows in the Hooper diagram to do the same. For our example $\M 43$, from ${\rd 3}$ to ${\rd 4}$, we go up and right, so in the Hooper diagram we follow the arrow around the black vertex to the vertical edge, and then at the other end of the vertical edge we follow the arrow around the white vertex, and label the horizontal edge ${\rd 4}$. The same is true for any odd $n$. When $n$ is even, we follow the same pattern on the Hooper diagram to go left and down and label edge ${n+1}$ in the same location.

\emph{Step 4}: We complete the labels of $\M mn$ and also label with $\M nm$.

The construction in Step $3$ shows why moving horizontally across a line in the Hooper diagram corresponds to the zig-zag labeling in each polygon of the \bm surface: going around white and black vertices corresponds to alternately going horizontally and vertically in the polygons. To get from one horizontal line to the next in the Hooper diagram, we follow the direction in the polygons. Thus, the ``snaking'' labeling in the Hooper diagram corresponds to the labeling described in Section \ref{howtolabel}. 

We already placed edge ${\gr 1}$ of $\M nm$, and we follow exactly the same method for the rest of the edges as we just described for $\M mn$. This leads to the overlaid ``snaking'' patterns shown in Figure \ref{hd-labeled}.
\end{proof}

\begin{remark}\label{hd-reflect}
When we defined the Hooper diagrams in Section \ref{hooperdiagrams}, we followed Hooper's convention of the arrangement of white and black vertices and arrow directions. In fact, this choice is somewhat arbitrary; the diagrams lead to the same polygon construction if we rotate them by a half-turn, or reflect them horizontally or vertically. Using Hooper's convention, along with our left-to-right numbering system in the polygons where we first label $P(0)$ with $1,\ldots,n$ and so on, leads to the edges ${\rd1}$ and ${\gr 1}$ being in the lower-left corner of the labeled Hooper diagram, with the numbering going up. We prefer to have the $1$s in the upper-left corner with the numbers going down, so after we finish labeling it, we will reflect the diagram horizontally, as in Figure \ref{hd-labeled}b for the general form. This choice is merely stylistic.
\end{remark}

\subsection{Transition diagrams: definitions and examples}\label{sec:transition_diagrams}
In this section we define \emph{transition diagrams}, which describe all possible transitions between edge labels for trajectories that belong to a given sector of directions (see Definition \ref{def:transition} below). 
We will first describe in this section transition diagrams for cutting sequences of trajectories whose direction belongs to the sector  $[0,\pi/n]$. Then, exploiting the symmetries of the polygonal presentation of \bm surfaces,  we will describe transition diagrams for the other sectors of width $\pi/n$, see \S\ref{sec:other_sectors}.


\begin{definition} \label{sectordef}
For $i=0,\ldots,2n-1$, let $\Sec i m n = [i\pi/n,(i+1)\pi/n]$. We  call \mbox{$\Sec 0 m n = [0,\pi/n]$} the \emph{standard sector}. For a trajectory $\tau$, we say $\tau \in \Sec i m n$ if the angle of the trajectory is in $\Sec i m n$.
\end{definition}

Let us first describe the \emph{transitions} that are allowed in each sector:

\begin{definition}\label{def:transition}
The \emph{transition} $n_1 n_2$ is \emph{allowed} in sector $\Sec i m n$ if some trajectory in $\Sec i m n$ cuts through edge $n_1$ and then through edge $n_2$.
\end{definition}

The main result of this section (Theorem \ref{tdtheorem}) is the description of the structure of  diagrams which describe of all possible transitions in $\Sec 0 mn$ for $\M mn$.


\begin{definition}
The \emph{transition diagram} $\T i m n$ for trajectories in $\Sec i  m n $ on $\M m n$ is a directed graph whose vertices are edge labels of the polygon decomposition of the surface, with an arrow from edge label $n_1$ to edge label $n_2$ if and only if the transition $n_1 n_2$ is allowed in $\Sec i  m n $.
\end{definition}


\begin{example}
We  construct $\T 0 43$ which is for sector \mbox{$\Sec 0  4 3 =[0,\pi/3]$} (Figure \ref{34td}). A trajectory passing through edge ${\rd 1}$ can then go horizontally across through edge ${\rd 2}$ or diagonally up through edge ${\rd 6}$, so we draw arrows from \mbox{${\rd 1}\to {\rd 2}$} and \mbox{${\rd 1}\to {\rd 6}$}. A trajectory passing through edge ${\rd 2}$ can go across through edge ${\rd 1}$, or up through edge ${\rd 3}$, so we draw arrows \mbox{${\rd 2}\to {\rd 1}$} and \mbox{${\rd 2}\to {\rd 3}$}. From edge ${\rd 3}$, we can only go up to edge ${\rd 4}$, so we draw \mbox{${\rd 3}\to {\rd 4}$}. The rest of the diagram is constructed in the same manner. We do \emph{not} draw (for example) an arrow from ${\rd 3}$ to ${\rd 6}$, because such a trajectory is not in $\Sec 0  43$ (it is in $\Sec 143$).
\end{example}

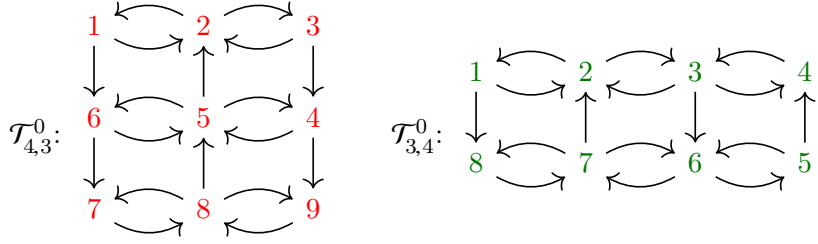
\begin{figure}[!h] 
\centering
$\T 0 43$:  \begin{tikzcd}
{\color{red}1}\arrow[bend right]{r} \arrow{d} 
&{\color{red}2} \arrow[bend right]{l} \arrow[bend left]{r} 
&{\color{red}3} \arrow[bend left]{l}   \arrow{d}  \\
{\color{red}6}\arrow[bend right]{r} \arrow{d}
&{\color{red}5} \arrow[bend right]{l} \arrow[bend left]{r}  \arrow{u}
&{\color{red}4} \arrow[bend left]{l} \arrow{d} \\ 
{\color{red}7}\arrow[bend right]{r}
&{\color{red}8} \arrow[bend right]{l} \arrow[bend left]{r}  \arrow{u}
&{\color{red}9} \arrow[bend left]{l}  \\ 
\end{tikzcd} \ \ \ \ \ 
$\T 0 34$:  \begin{tikzcd}
{\gr1}\arrow[bend right]{r} \arrow{d} 
&{\gr2} \arrow[bend right]{l} \arrow[bend left]{r} 
&{\gr3} \arrow[bend left]{l} \arrow[bend right]{r}  \arrow{d} 
&{\gr4}   \arrow[bend right]{l} \\
{\gr8}\arrow[bend right]{r}
&{\gr7} \arrow[bend right]{l} \arrow[bend left]{r}  \arrow{u}
&{\gr6} \arrow[bend left]{l} \arrow[bend right]{r} 
&{\gr5}  \arrow[bend right]{l} \arrow{u} \\ 
\end{tikzcd}
\begin{quote}\caption{Transition diagrams for the standard sector \label{34td}} \end{quote}
\end{figure}

\begin{example}
In Figure \ref{34td}, we also show $\T 0 34$, which is constructed in the same way for trajectories in sector \mbox{$\Sec 0 34 = [0,\pi/4]$} on $\M 34$.
\end{example}

We chose to label the edges as we did so that the numbers in the transition diagrams ``\emph{snake}'' back and forth across the table in this convenient way, just as in the Hooper diagram. The arrows are always as in Figure \ref{34td}: The arrows in the upper-left corner of every diagram are exactly as in the figure, and if $m$ and $n$ are larger, the same alternating pattern is extended down and to the right. We prove this general structure in the main result of this section, Theorem \ref{tdtheorem}.

\subsection{Admissibility of sequences}\label{sec:admissibility}
Consider the space ${\LL mn}^{\mathbb{Z}}$ of bi-infinite words $w$ in the symbols (edge label numbers) of the alphabet ${\LL mn}$ used to label the edges of the polygon presentation of $\M mn$.

\begin{definition}\label{admissibledef} Let us say that the word  $w$ in ${\LL mn}^{\mathbb{Z}}$ is \emph{admissible} if there exists a diagram $\T i m n$ for $i\in \{0, \dots, n-1\}$ such that all transitions in $w$ correspond to labels of edges of $\T i m n$. In this case, we will say that $w$ is \emph{admissible in (diagram) $\T i m n$}. Equivalently, the sequence $w$ is admissible in $\T i m n$  if it describes an infinite path on $\T i m n$. Similarly, a finite word $u$ is admissible (admissible in $\T i m n$) if it describes a finite path on a diagram (on $\T i m n$). 
\end{definition}

Admissibility is clearly a necessary condition for a sequence to be a cutting sequence:
\begin{lemma}\label{admissiblelemma}
Cutting sequences are admissible.
\end{lemma}
\begin{proof}
Let $w$ be a  cutting sequence of a linear a trajectory $\tau$ on $\M mn$. Up to orienting it suitably (and reversing the indexing by $\mathbb{Z}$ if necessary) we can assume without loss of generality that its direction $\theta$ belongs to $[0, \pi]$.  Then there exists some $0\leq i \leq n-1$ such that  $\theta \in \Sec i m n$. Since the diagram $\T i m n$ contains by definition all transitions which can occurr for cutting sequences of linear trajectories with direction in $\Sec i m n$, it follows that $w$ is admissible in $\T i m n$. 
\end{proof}

We remark that some words are admissible in more than one diagram. For example, since we are using closed sectors, a trajectory in direction $k\pi/n$ is admissible in sector $k$ and in sector $k+1$.
On the other hand, if $w$ is a non-periodic sequence, then it is admissible in a \emph{unique} diagram:

\begin{lemma}\label{lemma:uniqueness_sector}
If $w$ in ${\LL mn}^{\mathbb{Z}}$  is a \emph{non-periodic} cutting sequence of a linear trajectory on $\M mn$, then there exists a \emph{unique} $i\in \{0, \dots, n-1\}$ such that  $w$ is admissible in diagram $\T i m n$. 
\end{lemma}
\begin{proof}
 We know that $w$ is the cutting sequence of some $\tau$ in an unknown  direction $\theta$. 
 Let $0\leq i \leq n-1$ be so that $w$ is admissible in $\T i mn$. A priori $w$ could be admissible in some other diagram too  and we want to rule out this possibility. 
 We are going to show that all transitions which are allowed in $\T i mn$ actually occur. 

Since $w$ is non-periodic, the trajectory $\tau$ cannot be periodic. The Veech dichotomy 
(see \S\ref{sec:Veech}) 
implies that $\tau$ is dense in $\M mn$.
Let $n_1 n_2$ be a transition allowed in $\T i mn$. This means that we can choose inside the polygons forming $\M mn $ a segment in direction $\theta$ that connects an interior point on a side labeled by $n_1$ with an interior point on a side labeled $n_2$. Since $\tau$ is dense, it comes arbitrarily close to the segment. Since by construction $\tau$ and the segment are parallel, this shows that $w$ contains the transition $n_1 n_2$. 

Repeating the argument for all transitions in $\T i mn$, we get that $w$ gives a path on $\T i mn$ which goes through all arrows.  
This implies that the the diagram in which $w$ is admissible is uniquely determined, since one can verify by inspection that there is a unique diagram which contains certain transitions. 
\end{proof}

\subsection{Derivation diagrams }\label{sec:labeled_def}
We now define \emph{derivation diagrams} and explain how to construct them. These diagrams, as explained in the introduction, will provide a concise way to encode the rule to derive cutting sequences. As usual, we start with a concrete example for $\M 43 $, then give the general definition and results.

As explained in Section \ref{hooperdiagrams}, the \bm surfaces $\M mn $ and $\M nm $ are cut-and-paste affinely  equivalent via a diffeomorphism $\Psi_{m,n}$. Hence, we can draw a flip-sheared version of $\M nm $ surface on the $\M nm$ polygon decomposition. This is shown for the special case of $m=4, n=3$ in Figure \ref{hd34aux}. When two edges coincide, we arrange them so that red and green edges alternate going horizontally, and also vertically (as shown in Figure \ref{hd34aux} for the example). 

\begin{figure}[!h] 
\centering
\includegraphics[height=140pt]{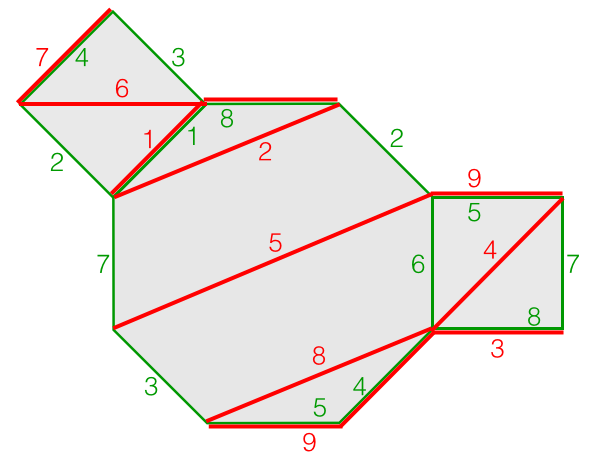}
\includegraphics[height=140pt]{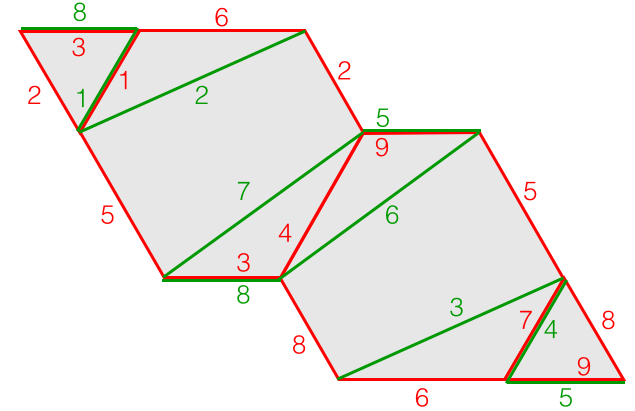}
\begin{quote}\caption{$\M 34$ with flip-sheared edges of $\M 43$, and $\M 43$ with flip-sheared edges of $\M 34$.  \label{hd34aux}} \end{quote}
\end{figure}


We add the following labeling to the transition diagram, thus making it into a derivation diagram. Recall that each arrow \mbox{$n_1\to n_2$} in the diagram represents a possible transition from edge $n_1$ to $n_2$  for a trajectory in $\Sec i m n$  in $\M m n $. We label the arrow \mbox{$n_1\to n_2$} with the edge label $n_3$ if  trajectories which hit the edge $n_1$ and then the edge  $n_2$ passes through some edge labeled  $n_3$ of the flip-sheared $\M n m $. 
 It turns out that, with a suitable convention to treat degenerate cases, this definition is well posed: either \emph{every} trajectory from $n_1$ to $n_2$ passes through $n_3$, or \emph{no} trajectory from $n_1$ to $n_2$ passes through $n_3$. This will be shown below in Lemma \ref{lemma:wellposed}.

\begin{example}
Figure \ref{hd34aux} shows $\M 34$ in  red with the flip-sheared edges of $\M 43$ in  green, and shows $\M 43$ in  green with the flip-sheared edges of $\M 34$ in  red. { The picture on the left and the picture on the right are two pictures of the same surface, under an affine automorphism: if we start with the figure on the left, flip it horizontally (via $x\to -x$), shear it to the right, and then cut and paste the pieces, we get the figure on the right. Similarly, if we do the same thing to the figure on the right, we get the figure on the left. We have shown both the red and green edges on both pictures; some edges are edges in both decompositions, so they have a red label and a green label.} We will construct the derivation diagram for each picture. 

The transition diagram for $\M 34$ is as before, but now we will add arrow labels (Figure \ref{34auxtd}). A trajectory passing from  {\rd 1} to {\rd 2} crosses edge {\gr 2}, so we label \mbox{ {\rd 1} $\to$ {\rd 2}} with {\gr 2}. A trajectory passing from {\rd 6} to {\rd 5} also passes through {\gr 2}, so we label \mbox{ {\rd 6} $\to$ {\rd 5}} with {\gr 2} as well. Since these arrows are next to each other, we just write one {\gr 2} and the arrows share the label. The rest of the diagram, and the diagram for $\M 43$, is constructed in the same way.

The only exceptions to this are the ``\emph{degenerate cases}'', where edges coincide. The edges that coincide here are {\rd 1} with {\gr 1}, {\rd 3} with {\gr 8}, {\rd 7} with {\gr 4}, and {\rd 9} with {\gr 5}. 
Four pairs of edges coincide in this way in the four corners of every transition diagram.

\begin{figure}[!h] 
\centering
\includegraphics[width=350pt]{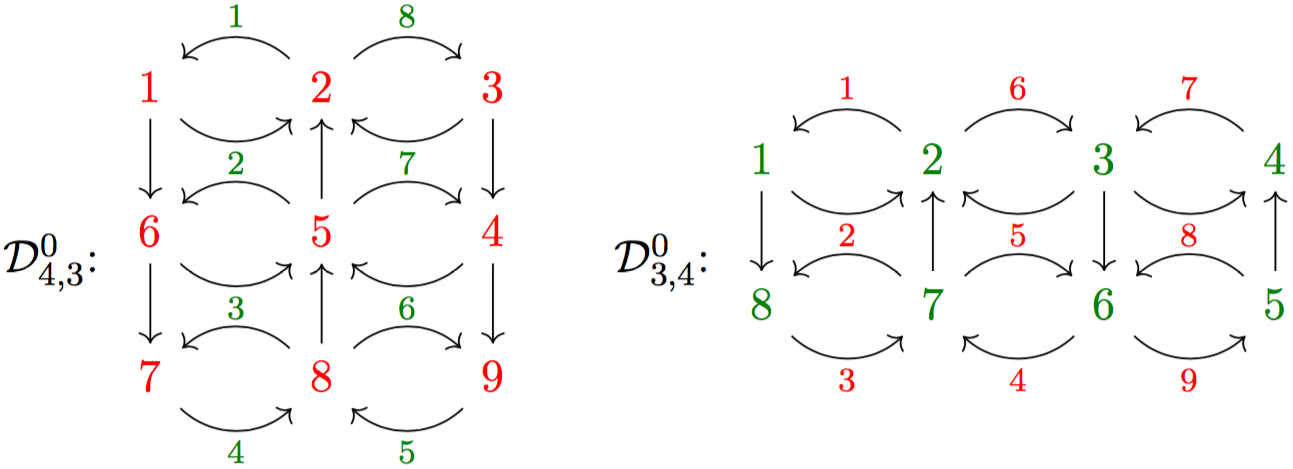}
\begin{quote}\caption{Derivation diagrams for $\M 43$ and $\M 34$ \label{34auxtd}} \end{quote}
\end{figure}
\end{example}


In general, we adopt the following convention, which corresponds (after a shear) to Convention \ref{convention:ordered_sides} for the orthogonal presentations. 

\begin{convention}\label{convention:ordered_sides_affine}
When sides of  $\M nm$ and of the flip and sheared pre-image of $\M nm$ by $\AD mn$ \emph{coincide}, we draw them adjacent to each other and ordered so that sides of $\M nm$ (red) and sides of $\M nm$ (green) diagonals  \emph{alternate}, as shown in Figures \ref{coincide} and \ref{hd34aux} for $\M  43 $ and $\M 34$. 
\end{convention}

With this convention, the following Lemma holds, which is essentially a restating of Lemma \ref{lemma:intertwined} from the orthogonal presentations: 

\begin{lemma}\label{lemma:wellposed}
Consider any segment of a trajectory on $\M mn $ with direction $\theta$ in the standard sector $\Sec 0 m n$ which crosses from the side of $\M mn $ labeled $n_1$ to the side of $\M mn $ labeled $n_2$. Consider the interwoven sides of the flip-sheared copy of $\M nm$ obtained asa  preimage of $\AD mn$. Then only one of the following is possible:
\begin{enumerate}
\item either no such segment crosses a side of the flip-sheared edges of $\M 34$, or
\item every such segment crosses the same side  of the flip-sheared edges of $\M 34$.
\end{enumerate}
\end{lemma}
\begin{proof} Remark that the affine diffeomorphism that maps the orthogonal presentation of $\M mn$ to $\M mn$, by mapping negative diagonals to sides of $\M mn$, and the dual orthogonal presentation of $\M nm$ to the flip and sheard preimage of $\M nm$, by mapping positive diagonals to flip and sheared preimages of sides of $\M nm$ by $\AD mn$. Thus, Convention \ref{convention:ordered_sides_affine} for the sides of $\M mn$ and the sides of the preimage of $\M nm$ by $\AD mn$  correspond to Convention \ref{convention:ordered_sides_affine} for diagonals in the orthogonal presentations. Thus, the lemma follows immediately from  Lemma \ref{lemma:intertwined} for the orthogonal presentations. 
\end{proof}

With the above convention (Convention \ref{convention:ordered_sides_affine}), in virtue of  Lemma \ref{lemma:wellposed} the following definition is well posed.

\begin{definition} The \emph{derivation diagram} $\D 0 mn$ is the transition diagram $\T 0 mn$ for the standard sector with arrows labeled as follows. 
We label the arrow \mbox{$n_1\to n_2$} with the edge label $n_3$ if all the segments of trajectories with direction in the standard sector which hit the edge $n_1$ and then the edge  $n_2$ passes through some edge labeled  $n_3$ of the flip-sheared $\M nm$. Otherwise, we leave the arrow \mbox{$n_1\to n_2$} without a label.
\end{definition}

In the example of derivation diagram for the surface $\M 34$ in Figure \ref{34auxtd}, 
one  can see that the arrow labels in the example are also arranged elegantly: they snake up and down, interlaced with the edge labels in two alternating grids. The relation between the diagrams for $\M 34$ and $\M 43$ is simple as well: flip the edge labels across the diagonal, and then overlay the arrows in the standard pattern. 

This  structure holds for every \bm surface, as we prove in the following main  theorem of this section:

\begin{theorem}[Structure theorem for derivation diagrams]  \label{tdtheorem}
The structure of the derivation diagram for $\M mn$ in sector $[0,\pi/n]$ is as follows:
\begin{itemize}
\item The diagram consists of $n$ columns and $m-1$ rows of edge labels of $\M mn$.
\item The edge labels start with {\rd 1} in the upper-left corner and go left to right across the top row, then right to left across the second row, and so on, ``snaking'' back and forth down the diagram until the last edge label {\rd n(m-1)} is in the left or right bottom corner, depending on parity of $m$.
\item Vertical arrows between edge labels go down in odd-numbered columns and up in even-numbered columns.
\item Vertical arrows have no arrow labels.
\item A pair of left and right horizontal arrows connects every pair of horizontally-adjacentedge labels. 
\item Horizontal arrows have arrow labels, which are edge labels of $\M nm$. 
\end{itemize}
For convenience, we choose to arrange these arrow pairs so that the top arrow goes left and the bottom arrow goes right for odd-numbered columns of arrows, and vice-versa in even-numbered columns of arrows. With this arrangement, the arrow labels are as follows:
\begin{itemize}
\item The top-left arrow label is {\gr 1}, and then going down, the next two arrows are both labeled {\gr 2}, and the rest of the pairs are numbered consecutively, until the last remaining arrow is labeled {\gr n}. Then the arrow to the right is labeled {\gr n+1}, and going up the next two arrows are both labeled {\gr n+2}, and so on, ``snaking'' up and down across the diagram until the last arrow is labeled {\gr m(n-1)}.
\end{itemize}
\end{theorem}

There are two examples of derivation diagrams in Figure \ref{34auxtd}, and the general form is shown in Figure \ref{gentransdiag}. Essentially, the two transition diagrams in Figure \ref{34td} are laid over each other as overlapping grids.

%
%
%

\begin{figure}[!h] 
\centering
\includegraphics[width=250pt]{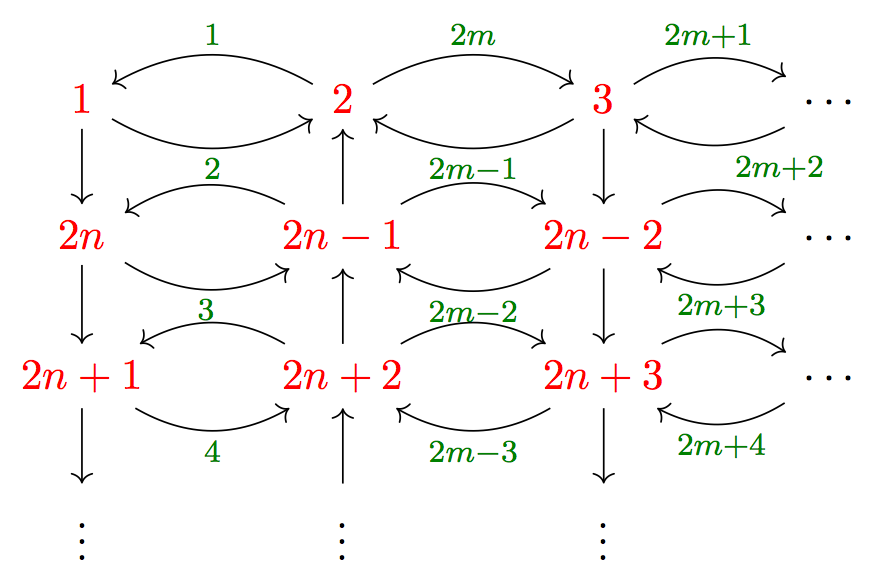} \\
\begin{quote}\caption{The form of a derivation diagram for $\M mn$ \label{gentransdiag}} \end{quote}
\end{figure}

Again, we omit the right and bottom edges of the diagram because their labels depend on the parity of $m$ and $n$; to understand the full diagram, it is clearer to look at an example such as Figure \ref{34auxtd}.

\subsection{The structure theorem for derivation diagrams }\label{sec:structure}
In this section we prove Theorem \ref{tdtheorem} describing the structure of derivation diagrams. 
For the proof 
we will use the \emph{stairs} and \emph{hats} that we defined in Section \ref{stairsandhats}. 

Let us recall that each edge  in the Hooper diagram corresponds to a basic rectangle, which is the intersection of two cylinders, as explained in Section \ref{hoogen}.  Each stair configuration of basic rectangles corresponds exactly to the four possible hat configurations, see Lemma \ref{hatlemma} and also Figure \ref{hat-cases}.  Recall that the \emph{middle edge} is the one that is numbered in Figure \ref{hat-cases}, and called $a$ in Figure \ref{hatarrow} below.

 We will now describe the labeling on these hats that corresponds to a given labeling by $a,b,c,d,e,f$ of the  basic rectangles in the stairs. Each basic rectangle either contains an edge of $\M mn$ (red, a negative diagonal) or an edge of $\M nm$ (green, a positive diagonal). 
Thus, giving a labeling of diagonals is equivalent to giving a labeling of basic rectangles. Furthermore, if we work with augmented diagrams and degenerate basic rectangles,  each edge of the \bm surface and of its dual \bm surface is in correspondence with a diagonal (positive or negative) of a basic rectangle (possibly degenerate).

Let us first establish: 

\begin{lemma}\label{hatdirection} 
Hats are right-side-up when the middle edge is in an even-numbered column, and upside-down when the middle edge is in an odd-numbered column.
\end{lemma}

\begin{proof}
Recall from Definition \ref{def:hat} that we have defined a \emph{hat} in such a way that the arrows from the Hooper diagram always go from the middle edge of the hat to each of the adjacent vertical edges $-$ from edge $a$ to edges $b$ and $d$ as shown in in Figure \ref{hat}. Since the arrows go down in even-numbered columns and up in odd-numbered columns of the Hooper diagram, as discussed in \S \ref{hoopertobm} and shown in Figure \ref{augmented}, the directions of the hats also alternate accordingly. When we perform the reflection discussed in Remark \ref{hd-reflect}, the directions are reversed, as desired.
\end{proof}


The following Lemma is key to proving the structure theorem, since it describes the \emph{local} structure of a transition diagram that corresponds to a (non-degenerate) hat/stair configuration. (Recall the cases $1-4$ for hats from the Degenerate Hat Lemma \ref{degeneratehat}.)

\begin{lemma} \label{degeneratearrow} Consider an edge $a$ of the \bm surface $\M mn$.
If the corresponding edge $a$ of $\G mn$ is the middle edge of a hat in case $1$, with adjacent edges $b,c,d,e,f$ as positioned in Figure \ref{hatarrow}a, then the allowed transitions starting with $a$ are as shown in Figure \ref{hatarrow}c.

\begin{figure}[!h] 
\centering
\includegraphics[width=400pt]{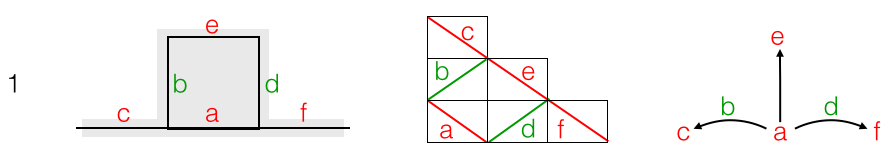} 
\begin{quote}\caption{(a) a hat in case $1$ (b) the corresponding stair diagram (c) the transitions from edge $a$ \label{hatarrow}} \end{quote}
\end{figure}

Furthermore, if $a$ is the middle edge of a hat in any of the degenerate cases $2-4$, the corresponding arrow picture is a subset of that picture, with exactly the edges that appear in the degenerate hat, as shown in Figure \ref{deghatarr}.

\begin{figure}[!h] 
\centering
\includegraphics[width=300pt]{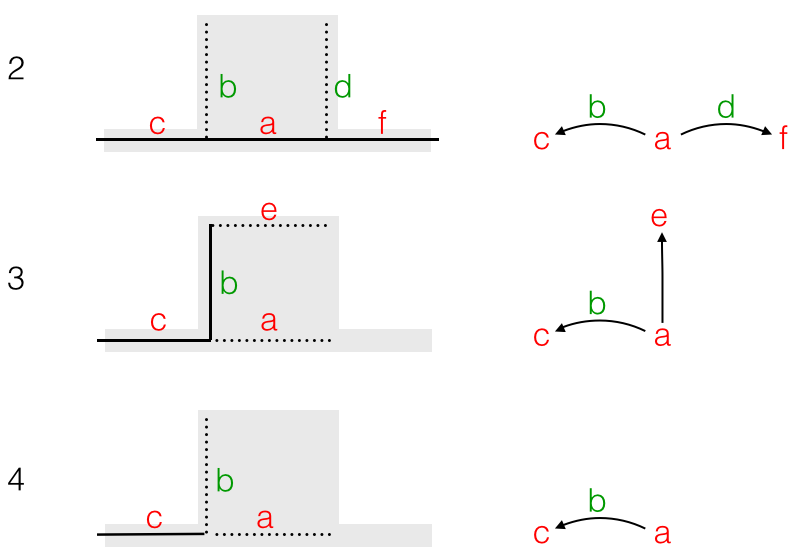} 
\begin{quote}\caption{The degenerate hats of cases $2-4$, and their corresponding transitions \label{deghatarr}} \end{quote}
\end{figure}
\end{lemma}

\begin{proof}
First, we consider the case where $a$ is the middle edge of a hat in case $1$ (Figure \ref{hatarrow}a).
Assume that edges $a,b,c$ in the Hooper diagram are adjacent in a vertical cylinder, so then $a,d,f$ are adjacent in a horizontal cylinder. Then the stair corresponding to this hat is as in Figure \ref{hatarrow}b.

Now we can determine the possible transitions from edge $a$ to other edges of $\M mn$ -- in this case, edges $c,e$ and $f$. Going vertically, $a$ can go to $c$ through $b$; going horizontally, $a$ can go to $f$ through $d$, and going diagonally, $a$ can go to $e$ without passing through any edge of $\M nm$. We record this data with the arrows in Figure \ref{hatarrow}c.

If instead the edges $a,b,c$ are adjacent in a horizontal cylinder, and $a,d,f$ are adjacent in a vertical cylinder, the roles of $b$ and $d$ are exchanged, and the roles of $c$ and $f$ are exchanged, but the allowed transitions and arrows remain the same.

Now we consider the case where $a$ is the middle edge of a hat in cases $2-4$. The analysis about basic rectangles and diagonals is the same as in case $1$; the only difference is that the basic rectangles corresponding to auxiliary (dotted) edges are degenerate, and the basic rectangles corresponding to missing edges are missing. 

The degeneracy of the rectangles does not affect the adjacency, so the degenerate edges act the same as normal edges, and remain in the arrow diagram. The missing edges clearly cannot be included in transitions, so these are removed from the arrow diagram (Figure \ref{deghatarr}).
\end{proof}

We can now use these Lemmas to give the proof of Theorem \ref{tdtheorem}.

\begin{proof} [Proof of Theorem \ref{tdtheorem}]
We begin with a Hooper diagram as in Figure \ref{construct1}a. The edges are labeled corresponding to the case of the hat that has that edge as its middle edge. The label is above the edge if that hat is right-side-up, and below the edge if the label is upside-down, from Lemma \ref{hatdirection}.

Lemma \ref{degeneratearrow} tells us the allowed transitions in each case, and we copy the arrows onto the corresponding locations in the Hooper diagram, in Figure \ref{construct1}. Here the node at the tail of each arrow is the hat case number, and we have spaced out the arrows so that it is clear which arrows come from which hat.

\begin{figure}[!h] 
\centering
\includegraphics[height=200pt]{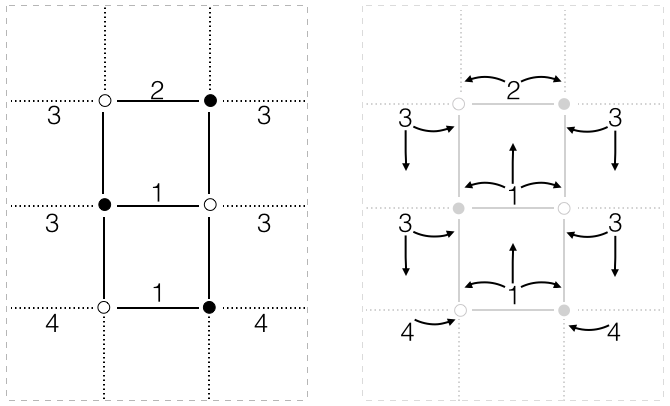}  
\begin{quote}\caption{The first steps of constructing the derivation diagram for $\M 43$ \label{construct1}} \end{quote}
\end{figure}

Now we determine the arrow labels. Proposition \ref{hd-snake} tells us that the edge labels from $\M mn$ and $\M nm$ snake back and forth and up and down, respectively, so we copy the labels in onto the Hooper diagram in Figure \ref{construct2}a. Then we use Lemma \ref{degeneratearrow} to copy these labels onto the arrow picture. For $\M 43$, this yields the derivation diagram in Figure \ref{construct2}b, and for $\M mn$ in general it yields the derivation diagram in Figure \ref{gentransdiag}.

\begin{figure}[!h] 
\centering
\includegraphics[height=200pt]{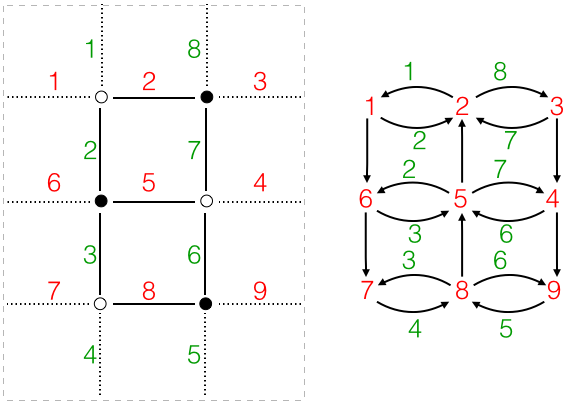} 
\begin{quote}\caption{Finishing the construction of the derivation diagram for $\M 43$ \label{construct2}} \end{quote}
\end{figure}

Where two identical arrow labels are adjacent (as for ${\gr 2, 3, 6, 7}$ here), we only write one label, and then get the diagram in Figure \ref{gentransdiag}, as desired.
\end{proof}


\subsection{Normalization}\label{sec:normalization}

Theorem \ref{tdtheorem} describes the transition diagram for \mbox{$\Sec 0 m n = [0,\pi/n]$}. Now we will describe how to transform any trajectory into a trajectory in $\Sec 0 m n$. 
To \emph{normalize} trajectories whose direction does not belong to the standard sector, we reflect each other sector $\Sec i m n$ for $1\leq i \leq 2n-1$ onto $\Sec 0 m n$. Remark that geodesic are line in a given direction and we can choose how to orient it. We can decide that all trajectories are ``\emph{going up},'' i.e. have their angle $\theta \in[0,\pi]$. Hence, often we will consider only sectors  $\Sec i m n$ for $1\leq i \leq n-1$. 



Recall that for $\M mn$, we defined $\Sec i m n  = [i\pi/n, (i+1)\pi/n]$.  

\begin{definition}\label{def:reflections}
For $0\leq i < 2n$ the transformation $\refl i m n$ is a reflection across the line $\theta = (i+1)\pi/(2n)$. Thus, $\refl i m n$ maps $\Sec i m n$ bijectively to $\Sec 0 m n$.
\end{definition}
\noindent See Example \ref{ex:refl_matrices} for the explicit form of the reflection matrices for $n=3$.

The reflection $\refl i m n$ also gives an
\emph{affine diffeomorphism} of $\M mn$, which is obtained by reflecting each polygon of $\M mn$ (see Example \ref{ex:reflections} below). 

\begin{convention}\label{convention:refl}
We use the same symbols $\refl i m n$  to denote matrices  in $SL(2 , \mathbb{R})$ and the corresponding affine diffeomorphisms of the \bm surface $\M mn$ .
\end{convention}

Each of the affine diffeomorphisms $\refl i m n$  also induces a \emph{permutation} on the edge labels   of $\M mn$, i.e. on the alphabet in $\LL m n $ (see Example \ref{ex:permutations} below).
We will denote the permutation corresponding to $\refl i mn$ by $\perm i mn$.  We now want to describe these permutations explicitly. 
To do this, we first notice that each flip can be seen as a composition of two flips that are easier to study (see Lemma \ref{comptransf} below). 
%
%
%
The following Definition \ref{diagram-actions-def} and  Lemma \ref{diagram-actions-lem}  then explain the actions of these fundamental transformations on the labels of the polygons. 
%
%

\begin{lemma}\label{comptransf}
Each of the reflections $\refl i m n$ can be written as a composition of the following:
\begin{itemize}
\item a flip along the axis at angle $\pi /n$, denoted by $\fl {n}$.
\item a flip along the axis at angle $\pi /(2n)$, denoted by $\fl {2n}$.
\end{itemize}
\end{lemma}

\begin{proof}
Recall that we numbered the sectors with $\Sec i m n  = [i\pi/n, (i+1)\pi/n]$, and that $\refl i m n$ reflects sector $\Sec imn$ into sector $\Sec 0mn$.  
Applying $\fl {2n}$ to $\Sec imn$ yields $\Sec {2n-1}mn$, with the opposite orientation. The composition $\fl n \circ \fl {2n}$ is a counter-clockwise rotation by $\pi/n$, preserving orientation. Thus, $$\refl imn = (\fl n \circ \fl {2n})^{2n-i} \circ \fl {2n}.$$ 
Notice that this is a composition of an odd number of flips, so it reverses orientation, as required.
\end{proof}

\begin{definition}\label{diagram-actions-def}
We define two actions on transition diagrams, which leave the arrows in place but move the numbers (edge labels) around.

The action $\nu$ is a flip that exchanges the top row with the bottom row, the second row with the next-to-bottom row, etc., as shown in the lower left of Figure \ref{diagram-actions-fig}. The action $\beta$ is a switching of certain adjacent pairs: the $1$ in the upper-left corner is preserved, and the $2$ and $3$ exchange places, $4$ and $5$ exchange places, and so on across the first row. { Then in the second row, the first and second are exchanged, the third and fourth are exchanged, etc., as shown in the lower right of Figure \ref{diagram-actions-fig} with bold arrows connecting pairs that are switched. The pairs that are switched in the second row are offset from those that are switched in the first, just as the second row of bricks is offset from the first when a mason builds a brick wall. In the same way as bricks, all the odd rows are the same as each other, and all the even rows are the same as each other. }
%
\end{definition}

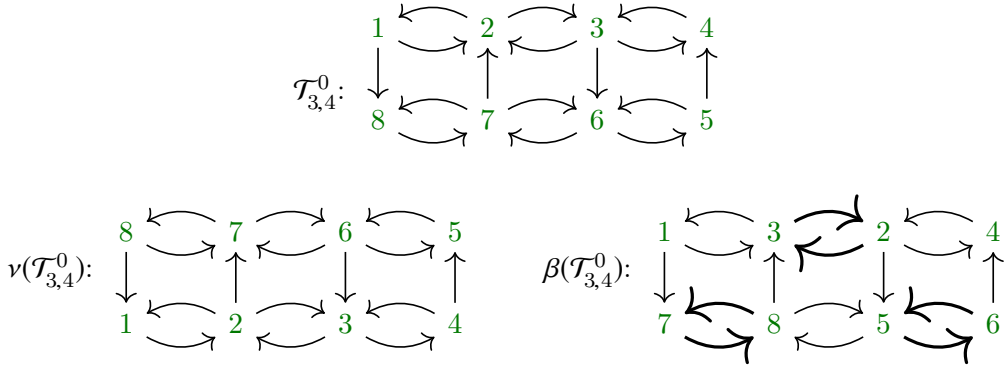
\begin{figure}[!h]
$\T 0 3 4$:  \begin{tikzcd}
{\gr1}\arrow[bend right]{r} \arrow{d} 
&{\gr2} \arrow[bend right]{l} \arrow[bend left]{r} 
&{\gr3} \arrow[bend left]{l} \arrow[bend right]{r}  \arrow{d} 
&{\gr4}   \arrow[bend right]{l} \\
{\gr8}\arrow[bend right]{r}
&{\gr7} \arrow[bend right]{l} \arrow[bend left]{r}  \arrow{u}
&{\gr6} \arrow[bend left]{l} \arrow[bend right]{r} 
&{\gr5}  \arrow[bend right]{l} \arrow{u} \\ 
\end{tikzcd} \\
$\nu(\T 0 3 4)$:   \begin{tikzcd}
{\gr8}\arrow[bend right]{r} \arrow{d} 
&{\gr7} \arrow[bend right]{l} \arrow[bend left]{r} 
&{\gr6} \arrow[bend left]{l} \arrow[bend right]{r}  \arrow{d} 
&{\gr5}   \arrow[bend right]{l} \\
{\gr1}\arrow[bend right]{r}
&{\gr2} \arrow[bend right]{l} \arrow[bend left]{r}  \arrow{u}
&{\gr3} \arrow[bend left]{l} \arrow[bend right]{r} 
&{\gr4}  \arrow[bend right]{l} \arrow{u} 
\end{tikzcd} \ \ \ \ \ \ 
$\beta(\T 0 3 4)$:   \begin{tikzcd}
{\gr1} \arrow[bend right]{r} \arrow{d} 
 & {\gr3} \arrow[bend right]{l} \arrow[very thick,bend left]{r} 
 & {\gr2} \arrow[very thick,bend left]{l} \arrow[bend right]{r}  \arrow{d} 
 & {\gr4} \arrow[bend right]{l} \\
{\gr7} \arrow[very thick,bend right]{r}
 & {\gr8} \arrow[very thick,bend right]{l} \arrow[bend left]{r}  \arrow{u}
 & {\gr5} \arrow[bend left]{l} \arrow[very thick,bend right]{r} 
 & {\gr6} \arrow[very thick,bend right]{l} \arrow{u} 
\end{tikzcd}
\begin{quote}\caption{The actions $\nu$ and $ \beta$ on a transition diagram. \label{diagram-actions-fig}}\end{quote}
\end{figure}

\begin{lemma} \label{diagram-actions-lem}
\begin{enumerate}
\item The flip $\fl {2n}$ has the effect of $\nu$ on the transition diagram. \label{pirotate}
\item The flip $\fl n$ has the effect of of $\beta$ on the transition diagram.\label{diagbuddy}
\end{enumerate}
\end{lemma}

\begin{proof}
\begin{enumerate}
\item Recall Definition \ref{pk}, where we named the polygons $P(0), P(1), \ldots, P(m-1)$ from left to right. By the Structure Theorem for derivation diagrams \ref{tdtheorem}, the first row of a transition diagram has the edge labels of $P(0)$, the second row has the edge labels of $P(1)$, and so on until the last row has the edge labels of $P(m-1)$.  A flip along the line at angle $\pi/(2n)$ exchanges the locations of the ``short'' and ``long'' sides, so it takes $P(0)$ to $P(m-1)$, and takes $P(1)$ to $P(m-1)$, etc. Thus it exchanges the rows by the action of $\nu$.
\item A flip along $\pi/n$ exchanges pairs of edge labels that are opposite each other in direction $\pi/n$ in the polygons, which because of the zig-zag labeling are exactly the ones exchanged by $\beta$.
\end{enumerate}
\end{proof}

\begin{corollary}\label{preserves-rows-cor}
The actions $\nu$ and $ \beta$ on the transition diagram corresponding to the actions of $\fl {2n}$ and $\fl {n}$, respectively, preserve the rows of the transition diagram $\T i mn$.
Consequently, the permutations $\perm i mn$ preserve the rows of $\T i mn$.
\end{corollary}

\begin{proof}
It follows immediately from Lemma \ref{diagram-actions-lem} that the action described on the diagrams preserve rows. 
Now, each permutation $\perm i mn$ corresponds to a reflection $\refl i mn$, which by Lemma \ref{comptransf} is obtained as a composition of the transformations $\fl {2n}$ and $\fl {n}$.
Thus the permutations are obtained by composing the permutations corresponding to $\fl {2n}$ and $\fl {n}$. 
Each permutation preserves the rows, hence their composition does too.
\end{proof}

%

\begin{example}[Matrices for $n=3$]\label{ex:refl_matrices}
 For $n=3$, the reflections $\refl i 4 3$ for $0\leq i \leq 2$ 
 that act on $\M 4 3$ { are  given by the following matrices}: 
 
$$\refl 043 = \mat {\bf  {-1}}001 \qquad \refl 1 43 = \mat {-1/2}{\sqrt{3}/2}{\sqrt{3}/2}{1/2} \qquad \refl 2 43 = \mat {-1}001.$$

\end{example}

\begin{example} [Reflections for $n=3$]\label{ex:reflections}
In Figure \ref{reflect} we show how the reflections $\refl i 4 3$ and $\refl i 3 4$ act as affine diffeomorphism on $\M 43$  and $\M 34$ respectively. The solid line reflects $\Sigma_1$ to $\Sigma_0$; the dashed line reflects $\Sigma_2$ to $\Sigma_0$, and for $\M 34$ the dotted line reflects $\Sigma_3$ to $\Sigma_0$.
\begin{figure}[!h] 
\centering
\includegraphics[width=.45\textwidth]{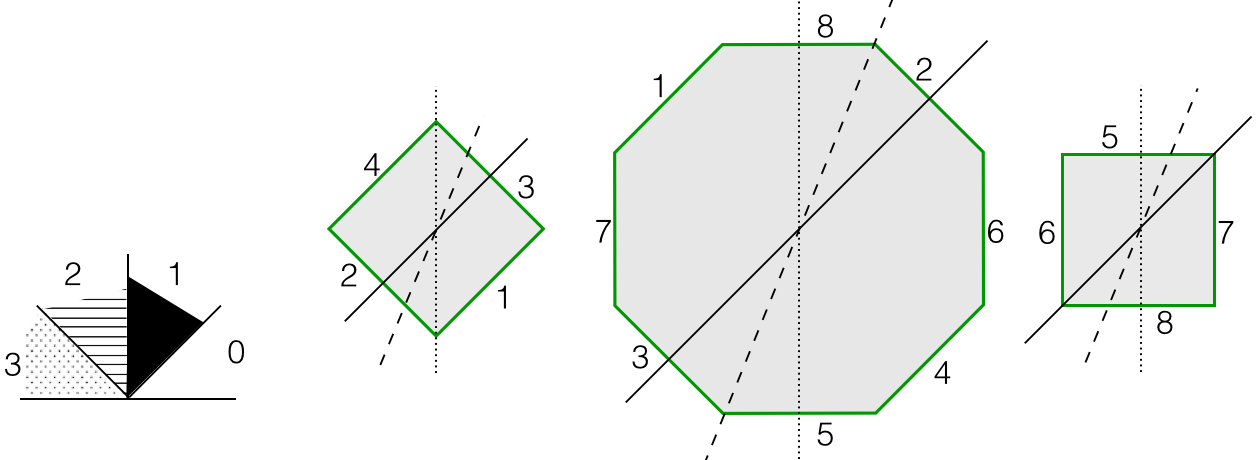} \hspace{0.08\textwidth}
\includegraphics[width=.45\textwidth]{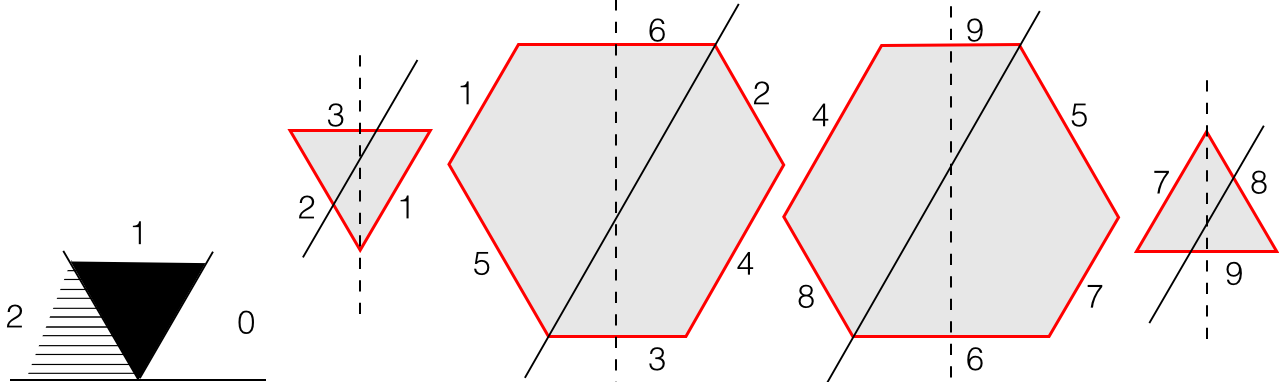}
\begin{quote}\caption{The action of reflections $\refl i m n$ on $\M 34$ and $\M 43$. \label{reflect}}\end{quote}
\end{figure}

\end{example}


\begin{example}[Permutations for $n=3$]\label{ex:permutations} Looking at Figure \ref{reflect}, we can see the permutation on edge labels induced by each $\refl i m n$.
\begin{align*}
&\M 43: &\perm 1 43  = &(17)(29)(38)(56) &\M 34: &&\perm 1 34 =& (14)(57)(68) \\
&&\perm 2 43 = &(12)(45)(78) &&&\perm 2 34 =& (16)(28)(35)(47) \\
& &&&&&\perm 3 34 =&  (12)(34)(67).
\end{align*}
\end{example}

\subsection{Transition diagrams for other sectors}\label{sec:other_sectors}
We can now explain how to draw a \emph{transition diagram} for trajectories in each sector.
Let us start with some examples, and then give a general rule to produce any such diagram. 

For our example surfaces $\M 43$ and $\M 34$, the transition diagrams for each sector are in Figures \ref{all-td43} and \ref{all-td34}, respectively.

\begin{figure}[!h] 
\centering

$\T 0 43$:  \begin{tikzcd}
{\color{red}1}\arrow[bend right]{r} \arrow{d} 
&{\color{red}2} \arrow[bend right]{l} \arrow[bend left]{r} 
&{\color{red}3} \arrow[bend left]{l}   \arrow{d}  \\
{\color{red}6}\arrow[bend right]{r} \arrow{d}
&{\color{red}5} \arrow[bend right]{l} \arrow[bend left]{r}  \arrow{u}
&{\color{red}4} \arrow[bend left]{l} \arrow{d} \\ 
{\color{red}7}\arrow[bend right]{r}
&{\color{red}8} \arrow[bend right]{l} \arrow[bend left]{r}  \arrow{u}
&{\color{red}9} \arrow[bend left]{l}  \\ 
\end{tikzcd}  \ \ \ \ \ \ 
$\T 1 43$:  \begin{tikzcd}
{\color{red}7}\arrow[bend right]{r} \arrow{d} 
&{\color{red}9} \arrow[bend right]{l} \arrow[bend left]{r} 
&{\color{red}8} \arrow[bend left]{l}   \arrow{d}  \\
{\color{red}5}\arrow[bend right]{r} \arrow{d}
&{\color{red}6} \arrow[bend right]{l} \arrow[bend left]{r}  \arrow{u}
&{\color{red}4} \arrow[bend left]{l} \arrow{d} \\ 
{\color{red}1}\arrow[bend right]{r}
&{\color{red}3} \arrow[bend right]{l} \arrow[bend left]{r}  \arrow{u}
&{\color{red}2} \arrow[bend left]{l}  \\ 
\end{tikzcd}  \ \ \ \ \ \
$\T 2 43$:  \begin{tikzcd}
{\color{red}2}\arrow[bend right]{r} \arrow{d} 
&{\color{red}1} \arrow[bend right]{l} \arrow[bend left]{r} 
&{\color{red}3} \arrow[bend left]{l}   \arrow{d}  \\
{\color{red}6}\arrow[bend right]{r} \arrow{d}
&{\color{red}4} \arrow[bend right]{l} \arrow[bend left]{r}  \arrow{u}
&{\color{red}5} \arrow[bend left]{l} \arrow{d} \\ 
{\color{red}8}\arrow[bend right]{r}
&{\color{red}7} \arrow[bend right]{l} \arrow[bend left]{r}  \arrow{u}
&{\color{red}9} \arrow[bend left]{l}  \\ 
\end{tikzcd} 

\begin{quote}\caption{Transition diagrams in each sector for $\M 43$ \label{all-td43}} \end{quote}
\end{figure}

\begin{figure}[!h] 
\centering
$\T 0 34$:  \begin{tikzcd}
{\gr1}\arrow[bend right]{r} \arrow{d} 
&{\gr2} \arrow[bend right]{l} \arrow[bend left]{r} 
&{\gr3} \arrow[bend left]{l} \arrow[bend right]{r}  \arrow{d} 
&{\gr4}   \arrow[bend right]{l} \\
{\gr8}\arrow[bend right]{r}
&{\gr7} \arrow[bend right]{l} \arrow[bend left]{r}  \arrow{u}
&{\gr6} \arrow[bend left]{l} \arrow[bend right]{r} 
&{\gr5}  \arrow[bend right]{l} \arrow{u} \\ 
\end{tikzcd} \ \ \ \ \ 
$\T 1 34$:  \begin{tikzcd}
{\gr4}\arrow[bend right]{r} \arrow{d} 
&{\gr2} \arrow[bend right]{l} \arrow[bend left]{r} 
&{\gr3} \arrow[bend left]{l} \arrow[bend right]{r}  \arrow{d} 
&{\gr1}   \arrow[bend right]{l} \\
{\gr6}\arrow[bend right]{r}
&{\gr5} \arrow[bend right]{l} \arrow[bend left]{r}  \arrow{u}
&{\gr8} \arrow[bend left]{l} \arrow[bend right]{r} 
&{\gr7}  \arrow[bend right]{l} \arrow{u} \\ 
\end{tikzcd} \\
$\T 2 34$:  \begin{tikzcd}
{\gr6}\arrow[bend right]{r} \arrow{d} 
&{\gr8} \arrow[bend right]{l} \arrow[bend left]{r} 
&{\gr5} \arrow[bend left]{l} \arrow[bend right]{r}  \arrow{d} 
&{\gr7}   \arrow[bend right]{l} \\
{\gr2}\arrow[bend right]{r}
&{\gr4} \arrow[bend right]{l} \arrow[bend left]{r}  \arrow{u}
&{\gr1} \arrow[bend left]{l} \arrow[bend right]{r} 
&{\gr3}  \arrow[bend right]{l} \arrow{u} \\ 
\end{tikzcd} \ \ \ \ \ 
$\T 3 34$:  \begin{tikzcd}
{\gr2}\arrow[bend right]{r} \arrow{d} 
&{\gr1} \arrow[bend right]{l} \arrow[bend left]{r} 
&{\gr4} \arrow[bend left]{l} \arrow[bend right]{r}  \arrow{d} 
&{\gr3}   \arrow[bend right]{l} \\
{\gr8}\arrow[bend right]{r}
&{\gr6} \arrow[bend right]{l} \arrow[bend left]{r}  \arrow{u}
&{\gr7} \arrow[bend left]{l} \arrow[bend right]{r} 
&{\gr5}  \arrow[bend right]{l} \arrow{u} \\ 
\end{tikzcd}

\begin{quote}\caption{Transition diagrams in each sector for $\M 34$ \label{all-td34}} \end{quote}
\end{figure}

\begin{corollary}\label{cor:shape}
Up to permuting the labels, the shape of the transition diagram is always the same. For  $\T i mn$, the labels in $\T 0 mn$ are permuted by $\perm i m n$. 
\end{corollary}

\begin{definition}\label{def:universal}
We will call \emph{universal diagram} and denote by $\UD mn$ the unlabeled version of the diagrams $\T i mn$. 
\end{definition}
The universal diagrams for $\M 43$ and $\M 34$ are shown in Figure \ref{fig:universalex}. All transition diagrams for $\M mn$ have the same arrow structure, $\UD mn$, with different labels at the nodes.

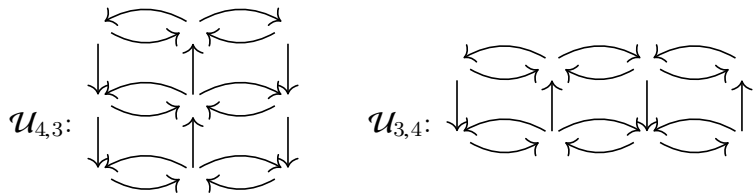
\begin{figure}[!h] 
\centering
$\UD  43$: \begin{tikzcd}
\arrow[bend right]{r} \arrow{d} 
& \arrow[bend right]{l} \arrow[bend left]{r} 
& \arrow[bend left]{l}   \arrow{d}  \\
\arrow[bend right]{r} \arrow{d}
& \arrow[bend right]{l} \arrow[bend left]{r}  \arrow{u}
& \arrow[bend left]{l} \arrow{d} \\ 
\arrow[bend right]{r}
& \arrow[bend right]{l} \arrow[bend left]{r}  \arrow{u}
& \arrow[bend left]{l}  \\ 
\end{tikzcd}  \ \ \ \ \ \ 
$\UD  34$:  \begin{tikzcd}
\arrow[bend right]{r} \arrow{d} 
& \arrow[bend right]{l} \arrow[bend left]{r} 
& \arrow[bend left]{l} \arrow[bend right]{r}  \arrow{d} 
&   \arrow[bend right]{l} \\
\arrow[bend right]{r}
& \arrow[bend right]{l} \arrow[bend left]{r}  \arrow{u}
& \arrow[bend left]{l} \arrow[bend right]{r} 
&  \arrow[bend right]{l} \arrow{u} \\ 
\end{tikzcd} 
\begin{quote}\caption{The universal diagrams $\UD 43$ and $\UD 34$ \label{fig:universalex}} \end{quote}
\end{figure}

\smallskip
\section{Derivation} \label{derivation}
In this section we will describe the \emph{renormalization} procedure which will be our key tool to help us characterize (the closure of) all possible trajectories on a given surface. The idea is that we will describe geometric renormalization operators (given by compositions of affine maps and reflections) which transform a linear trajectory into another linear trajectory, and at the same time we will describe the corresponding combinatorial operations on cutting sequences. The renormalization will happen in two steps, by first transforming a trajectory on $\M mn$ into one on $\M nm$ (and describing how to transform the corresponding cutting sequence), then mapping a trajectory on $\M nm$ into a  new trajectory on $\M mn$ (and a new cutting sequence).  
The combinatorial operators which shadow this geometric renormalization at the level of cutting sequences will be our derivation operators, followed by a suitable normalization (given by permutations). Since linear trajectories will be by construction infinitely renormalizable under this procedure, their cutting sequences will have to be infinitely derivable (in the sense made precise in \S \ref{sec:infinitely_derivable} below).

More precisely, we begin this section by describing in \S \ref{sec:der_ex} an example which, for $\M 4 3$, shows geometrically how to renormalize trajectories and their cutting sequences.   In \S \ref{sec:der_general} we then define  the combinatorial derivation operator $\Der mn$ and prove that it has the   geometric meaning described in the example, which implies in particular that the derived sequence of a cutting sequence on $\M mn$ is a cutting sequence of a linear trajectory on the dual surface $\M nm$. In \S\ref{sec:nor} we define this operator for sequences admissible in the standard sector, and then define  a \emph{normalization} operator that maps admissible sequences to sequences admissible in the standard sector. Then, by combining $\Norm nm \circ \Der mn$ with the operator $\Norm mn \circ \Der nm$  one gets a derivation operator on cutting sequences for $\M mn$ back to itself. In \S \ref{sec:infinitely_derivable} we use this composition to give the definition of infinitely derivable and prove that cutting sequences are infinitely derivable (Theorem \ref{thm:infinitely_derivable}). 
In \S\ref{sec:sectors_sequences}, we use this result to associate to any given cutting sequence an infinite sequence (the sequence of admissible diagrams) which records combinatorial information on the sequence of derivatives. We will explain in the next section \S \ref{farey} how this sequence can be used to provide an analogue of the continued fraction expansion for directions. Finally, in  \S\ref{sec:fixedpoints} we characterize the (periodic) sequences which are fixed under the renormalization process, since this description will be needed for the characterization in \S\ref{sec:characterization}.





\subsection{A motivational example: derivation for $\M 43$  geometrically}\label{sec:der_ex}

Let us start with an example to  geometrically motivate the definition of  derivation.   
In \S\ref{sec:affine} we described an explicit affine diffeomorphim $\AD {4}{3} $  from $\M 43$ to $\M 34$, which is obtained by combining a flip, a shear, a cut and paste, a similarity and another shear. The effect of these steps on $\M 43$ are shown in Figure \ref{hexoct}. 
\begin{figure}[!h]
\centering
\includegraphics[width=1\textwidth]{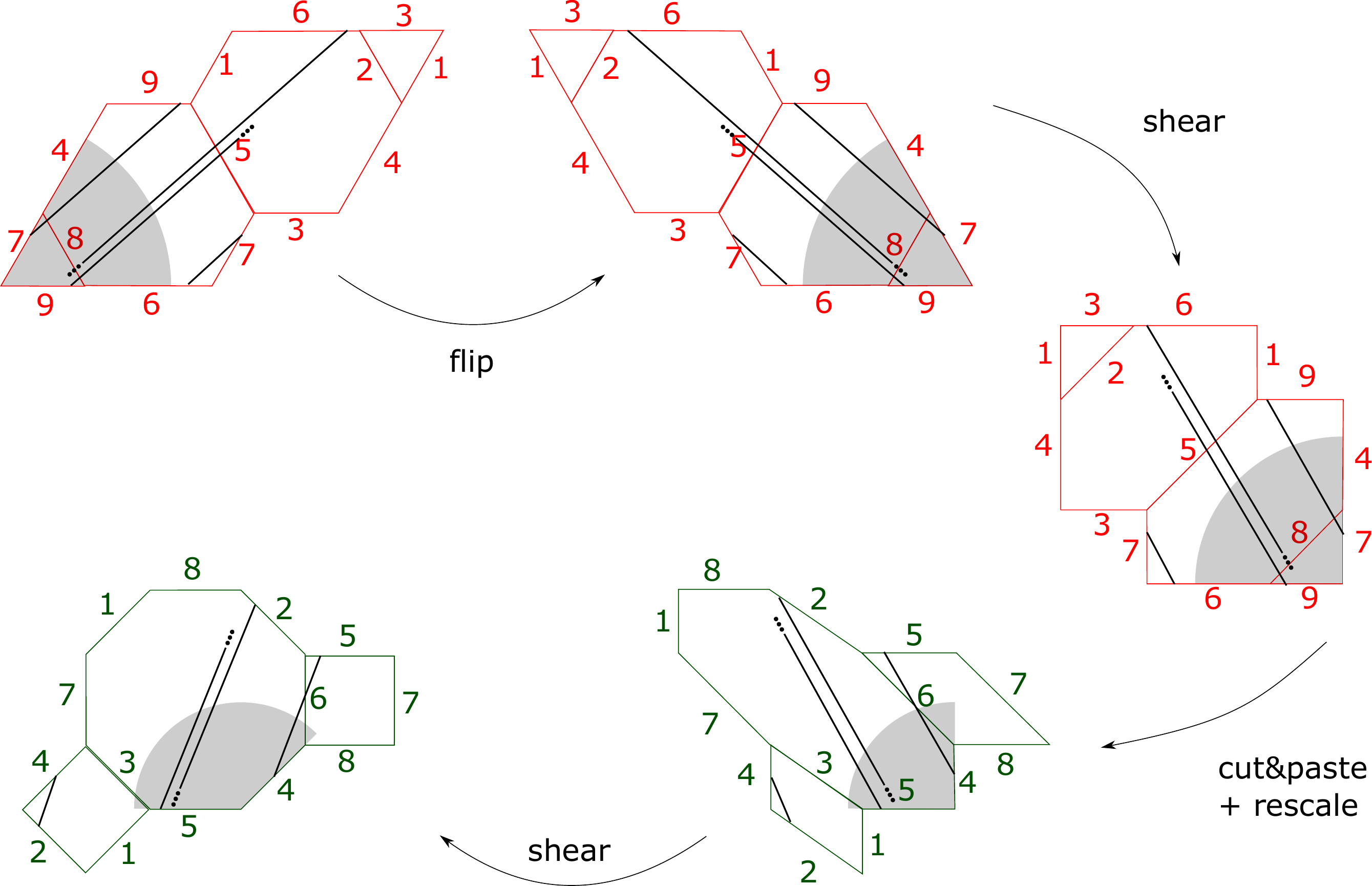}
\begin{quote}\caption{The effect of $\AD 43$ from $\M 43$ to $\M 34$. \label{hexoct}} \end{quote}
\end{figure}
Remark that the transformation $\AD {4}{3} $ acts on directions by mapping the standard sector $\Sec {0}43 $ for $\M 43$ to the sector $[\pi/4,\pi]$, which is  the complement in $[0,\pi] $ of the standard sector $\Sec {0} 34 $ for $\M 34$. This is shown in Figure \ref{hexoct}, where the image of the standard sector is followed step by step. 

In Figure \ref{fig:intertwined}, the preimages of the edges of $\M 34$ (with their edge labels) by  $\AD  43 $ are shown inside $\M 43$. Given a trajectory $\tau$ on $\M 43$ in a direction $\theta \in \Sec 043=[0,\pi/3]$ with cutting sequence $w$ in $\LL 43^{\mathbb{Z}}$ with respect to the edge labels of $\M 43$, one can also write the cutting sequence $w'$ of the same trajectory $\tau$ with respect to the edge labels of the pullback of the sides of $\M 34$  by $\AD {4}{3}$. This gives a symbolic sequence $w'$ in $\LL 34^{\mathbb{Z}}$. We want to define a combinatorial derivation operator so that the sequence $w'$ is obtained from the sequence $w$ by derivation. 

\begin{example}\label{reprise}
The periodic trajectory $\tau$ on $\M 43$ in Figure \ref{fig:intertwined}a has corresponding cutting sequence $w=\overline{\rd 1678785452}$, and the same trajectory on $\M 34$ has corresponding cutting sequence $w'=\overline{\gr 143476}$.\footnote{Here the overbar indicates a bi-infinite periodic sequence.} We can read off $w'$ on the left side of the figure as the pullback of the sides of $\M 34$, or on the right side of the figure in $\M 34$ itself. The reader can confirm that the path ${\rd 1678785452}$ in $\D 043$ in Figure \ref{34auxtd} collects exactly the arrow labels ${\gr 434761}$.
\end{example}

\begin{figure}[!h] 
\centering
\includegraphics[width=200pt]{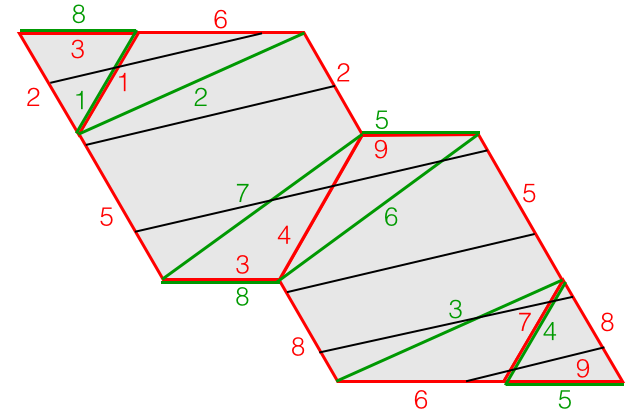}
\includegraphics[width=200pt]{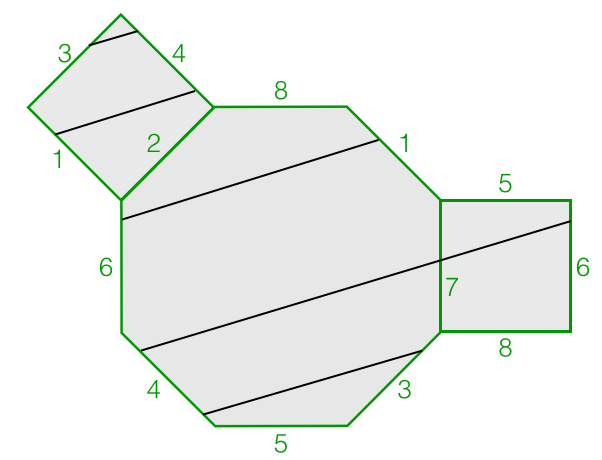}
\begin{quote}\caption{A trajectory in $\M 43$ and $\M 34$: The green edges inside $\M 43$ on the left are the preimages under $\AD  43 $ of $\M 34$ on the right. Figure \ref{hd34aux} showed this construction, and here we now show the same trajectory (the black line) on both surfaces. Note that the trajectories in the two surfaces are not parallel; on the left it is at an angle of about $13^{\circ}$, and on the right about $17^{\circ}$. \label{fig:intertwined}} \end{quote}
\end{figure}

In this explicit example, one can check by hand that  for each possible  transition  in $\T 043$ from an edge label of $\M 43$ to another one, either there are no pullbacks of edges of $\M 34$ crossed, or there is exactly one edge crossed (in general this follows from Lemma \ref{lemma:intertwined}).  By writing  the edge labels of these edges on top of the arrows in   $\T 043$ representing the transition, one obtains exactly the derivation diagram in Figure \ref{34auxtd}a.  Thus, consider the \emph{derivation operator} $\Der {4}{3}$ already mentioned in the introduction. It maps sequences admissible in $\T  043$ to sequences in the $\M 34$ edge labels, and is given by reading off the $\M 34$ edge labels on the arrows of the sequence described by the original sequence on $ \D  043$. It is clear from this geometric interpretation that the derivation operator $\Der {4}{3}$ is exactly such that the cutting sequence $w'$ of $\tau$ with respect to the pullback of the $\M 34$ edges satisfies $w'= \Der {4}{3} w$.

Let us now apply the affine diffeomorphism $\AD {4}{3}$. Then $\M 43$ is mapped to $\M 34$ and the trajectory $\tau$ is mapped to a linear trajectory $\tau'$ on $\M 34$ in a new direction $\theta'$, as shown in Figures \ref{hexoct}-\ref{fig:intertwined}. By construction, the sequence $w'= \Der {4}{3} w$ is the cutting sequence of a linear trajectory on $\M 34$. Since cutting sequences are admissible, this shows in particular that the derived sequence of a cutting sequence is admissible. 

The direction $\theta'$ image of $\theta \in [0,\pi/3]$ belongs   to $[\pi/4,\pi]$ by the initial remark on the action of $\AD {4}{3} $ on sectors.  Thus, $\Der  {4}{3}$ maps cutting sequences on $\M 43$ in a direction $\theta \in [0,\pi/3]$  to cutting sequences on  $\M 34$ in a direction $\theta \in [\pi/4,\pi]$. By applying a symmetry on $\M 34$ that maps the direction $\theta'$ to the standard sector $[0,\pi/4]$ for $\M 34$ and the correspoding permutation on edge labels, one obtains a new cutting sequence $\Norm 34 w'$ on $\M 34$. The map that sends the direction $\theta$ of $\tau$ to the direction $\theta'$ of $\tau'$ is the \emph{Farey map} $\F 4 3$, which will be described in \S \ref{farey}.

One can then perform a similar process again, starting this time from $\M 34$ and thus showing that $\Der {3}{4} \Norm 34 w'$ is a cutting sequence of the trajectory $\tau'$ with respect to the edge labels of the pullback of the sides of a sheared $\M 43$ by $\AD {3}{4}$, or equivalently, the cutting sequence of a new trajectory $\tau''$ on $\M 43$. For symmetry, we also apply a final flip $f'$ to reduce once more to trajectories with direction in the standard sector.  This second step is shown in Figure \ref{octhex}.

\begin{figure}[!h]
\centering
\includegraphics[width=1\textwidth]{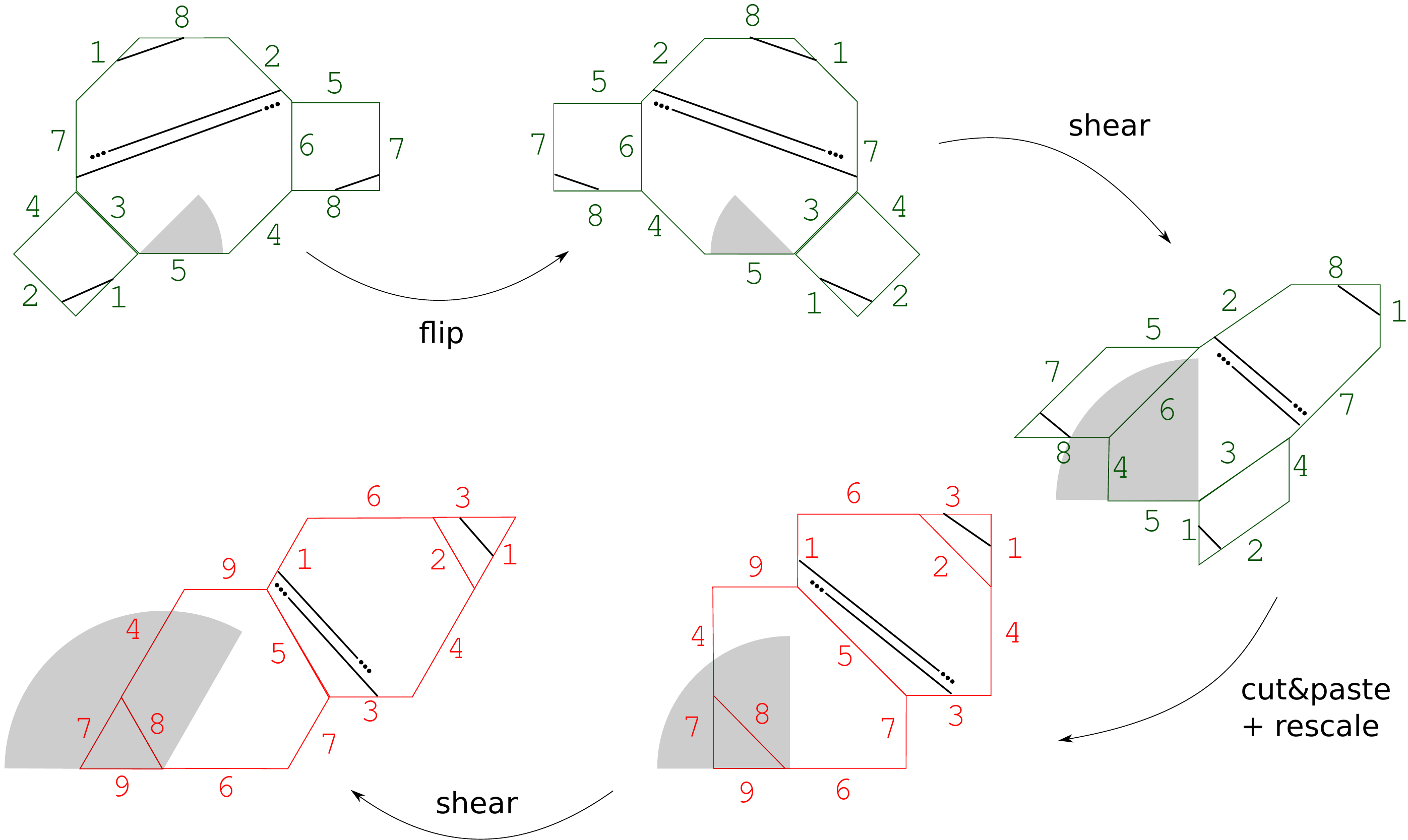}
\begin{quote}\caption{The effect of $\AD 34$ from $\M 34$ to $\M 43$. \label{octhex}} \end{quote}
\end{figure}

In  Figure \ref{fullmovie} we show the combined effect of applying $\AD 43$, then a flip, then $\AD 34$, then another flip.
By applying $\AD 43 \circ f \circ \AD 34$, we then obtain a new linear trajectory in $\M 43$  whose cutting sequence is $\Der 34 \Norm 34 \Der 43 w$. We can then apply another flip $\flip'$ to reduce again to a trajectory with direction in $\Sec 0 {4}{3}$ and repeat the same process. 

The effect of the composition $f' \circ \AD 43 \circ f \circ \AD 34$, which we call \emph{renormalization}, (see Remark \ref{rk:renormalization} below) corresponds to  applying derivation twice, with normalization to reduce to the standard sector in between and at the end. One gets in this way an operator from cutting sequences on $\M 43$ back to cutting sequences on $\M 43$. The cutting sequence of the \emph{initial} trajectory with respect to the sides of the image of $\M43$ by this element is the sequence $w$ derived, normalized, then derived and normalized again. If we apply $f' \circ \AD 43 \circ f \circ \AD 34$, it maps the original trajectory to a new trajectory whose cutting sequence with respect to $\M43$ is the sequence $w$ derived and normalized twice. Thus, deriving and normalizing twice produces cutting sequences. Repeating this renormalization process allows us to show, with this observation, that cutting sequences are infinitely derivable (see \S \ref{sec:infinitely_derivable}). 

\begin{remark}\label{rk:renormalization}
We will call \emph{renormalization} the  process just described (in the specific case of $m=4$, $n=3$), obtained by applying to $\M mn$ first $\AD mn$, then a flip to reduce to $\Sec 0 nm$ for $\M nm$, then $\AD nm$ and finally another flip. The name \emph{renormalization} must not be confused with the name \emph{normalization}, used to describe just the reduction (by flips) to standard sectors.  Renormalization maps trajectories on $\M mn$ back to trajectories on $\M mn$ but with the effect of \emph{shortening} long pieces of trajectories with directions in the standard sector. This follows since the standard sector is \emph{opened up} to the complementary sector, as shown in Figure \ref{fullmovie}. 

At the combinatorial level of cutting sequences, the effect of renormalization  
 corresponds to  applying derivation twice, once acting on cutting sequences on $\M mn$, once on $\M nm$, with \emph{normalization} in between and at the end, which acts by applying a permutation on cutting sequences to reduce to  sequences admissible in the  standard sector. One gets in this way an operator from cutting sequences on $\M mn$ back to cutting sequences on $\M mn$. 
The geometric fact that finite pieces of trajectories are shortened by renormalization has, as its combinatorial counterpart, that finite pieces of cutting sequences, when derived, become \emph{shorter}; see Remark \ref{rk:derivation_short}. 
\end{remark}





\begin{figure}[!h]
\centering
\includegraphics[width=1\textwidth]{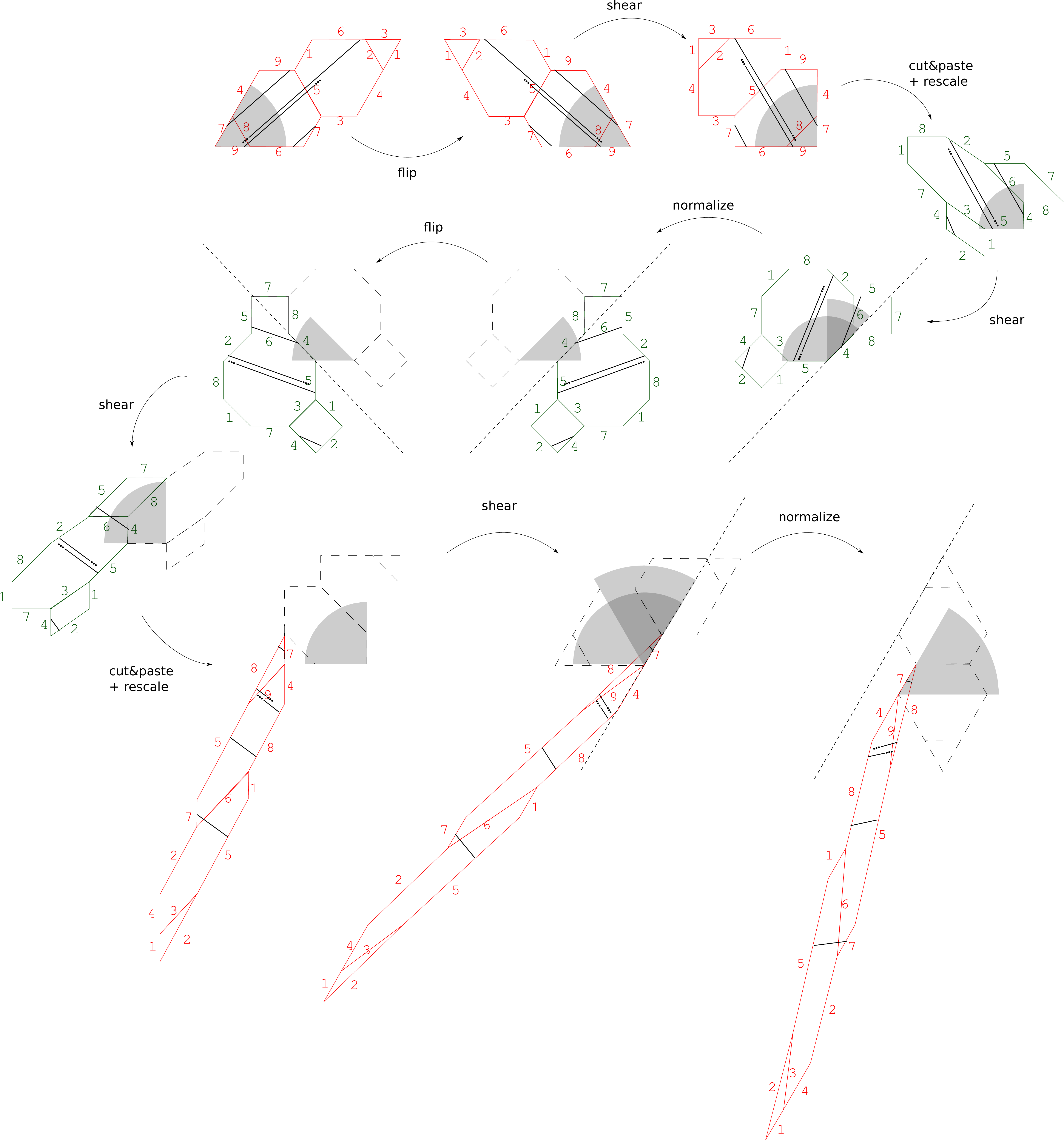}
\begin{quote}\caption{The full renormalization from $\M 43$ to itself. \label{fullmovie}} \end{quote}
\end{figure}





\subsection{Derivation operator for general $m,n$}\label{sec:der_general}
In general, we will now define an operator $\Der {m}{n}$ combinatorially, and then prove that it admits a 
geometric interpretation as in the example in the previous section. The \emph{derivation operator} $\Der {m}{n}$ is defined using the \emph{derivation diagram} $\D  0mn$ defined in \S \ref{sec:labeled_def} (see Theorem \ref{tdtheorem}) as follows. Recall that a sequence   \emph{admissible} in $\T  0mn$  describes a bi-infinite path on $\D  0mn$ (see the Definition \ref{def:transition} of admissible in \S\ref{sec:transition_diagrams}).

\begin{definition}\label{def:derivation}
 Given a sequence $w =(w_i) \in \LL mn ^{\mathbb{Z}}$ admissible in $\T  0mn$, the sequence   $\Der mn w$ is the sequence of  labels of the \emph{arrows} of $\D  0mn$  that are crossed by   an infinite path on $\D  0mn$ that goes through the vertices labeled by $(w_i)_{i \in \mathbb{Z}}$.
\end{definition}
An example was already given, in the introduction (Example \ref{ex:der34}) and in the previous section (Example \ref{reprise}). Derivation is well defined as an operator that maps  \emph{admissible} sequences in $\LL mn ^{\mathbb{Z}}$    to sequences in $\LL nm ^{\mathbb{Z}}$,  by virtue of the following Lemma:

\begin{lemma}\label{lem:derivation}
If $w =(w_i)$ is a bi-infinite sequence in  $\LL mn$ admissible in $\T  0mn$,  $\Der mn w$ is a bi-infinite sequence   in $\LL nm $. Thus,  
the operator $\Der {m}{n}$ maps sequences  sequences in $\LL mn ^{\mathbb{Z}}$ which are \emph{admissible}   in $\T  0mn$ to $\LL nm ^{\mathbb{Z}}$.
\end{lemma}
\begin{proof}
The proof of the Lemma is a consequence of the definition of derivation diagrams.
Let us recall from the structure theorem for derivation diagrams (Theorem \ref{tdtheorem}) that in $\D 0 mn$ the only arrows without edge labels are vertical. 
Since there are only $m-2$ vertical arrows in a row, the bi-infinite path described by $w$  is going to have at least one horizontal arrow out of every $m-1$ arrows. Thus, for every block of $m-1$ edge labels of $\LL mn$ in $w$ one get at least one edge label  of $\LL nm$ in the derived sequence $\Der mn w$. It follows that also $\Der mn w$ is an infinite sequence.
\end{proof}

Derivation is defined so that the following geometric interpretation holds. 

\begin{lemma}[Geometric interpretation for  derivation.]\label{lem:derivation_interpretation}
If $w$ is the cutting sequence of a linear trajectory of $\M mn$ in a direction $\theta \in \Sec 0 mn$, then the sequence $\Der mn w$  is the sequence of edge  labels of the crossed sides of the flip-sheared copy of $\M nm$ which is the preimage of $\M nm$ by the affine diffeomorphism $\AD {m}{n}$. 
\end{lemma}
The proof of the Lemma is a consequence of the definition of derivation diagrams and of their Structure Theorem \ref{tdtheorem}.
\begin{proof}
Consider the sequence $w_{i} w_{i+1}$, $i \in \mathbb{Z}$ of transitions in $w$, which by assumption of $\theta \in \Sec 0 mn$ are all transitions  which appear in $\T 0 m n$.  Since $w$ is a cutting sequence, each transition $w_{i} w_{i+1}$  corresponds to a segment of the trajectory $\tau$ that crosses first the edge of $\M mn$ labeled $w_i$, then the edge labeled $ w_{i+1}$. By the definitions of derivation diagrams and derivation, if this segment crosses an  edge of the flip-sheared copy of $\M nm$ obtained as preimage by $\AD mn$, the derived sequence $w'$ contains the label of this edge. Thus, the derived sequence describes exactly the cutting sequence of $\tau$ with respect to the flip-sheared copy of $\M nm$ in the statement.
\end{proof}

\begin{remark}\label{rk:derivation_short}
As we can intuitively see in Figure \ref{fig:intertwined}, when from the cutting sequence of a trajectory in $\M 43$ we pass to the cutting sequence with respect to the edge labels that are the preimages by $\AD 43$ of $\M 34$, the number of sides crossed reduces. 
Combinatorially, this can be seen on the derivation diagram $\D 0 mn$ in general, by remarking that horizontal arrows have exactly one label, while vertical arrows have none. 
This means that when we consider a subsequence that comes from a finite path which travels along horizontal arrows, it will have the same number of labels after derivation, while if the subsequence contains also vertical arrows, the length of the subsequence after derivation will be shorter. 
Thus, \emph{the effect of derivation on finite subsequences of a word is not to increase their length}. 
\end{remark}

From Lemma \ref{lem:derivation_interpretation} we will now deduce that cutting sequences are derivable in the following sense. Recall that the permutations $\perm i mn$ introduced in \S \ref{sec:normalization} map sequences admissible in $\T i mn $ to  sequences admissible  in $\T 0mn $. 

\begin{definition}\label{def:derivable}
A  sequence $w$ admissible  in $\T 0mn $ is \emph{derivable} if   $\Der mn (w)$ is admissible in some diagram $\T i nm$. A  sequence $w$ admissible  in $\T i mn $ is derivable if $\perm i mn w$ is derivable. 
\end{definition}

\begin{proposition}\label{thm:derivable}
 Cutting sequences of linear trajectories in $\M mn$ are derivable. Furthermore:
\begin{enumerate}
\item The derived sequence $\Der mn (w)$ of a cutting sequence $w$ on $\M mn$ is the cutting sequence of a trajectory in $\M nm$.
\item  If $w$ is admissible in $\T 0mn $, then $\Der mn (w)$  is admissible in  some $\T i nm $ with $1\leq i \leq m$.
\end{enumerate}
\end{proposition}
\begin{proof}
We will first prove the first claim. 
Normalizing by $\perm i mn$ we can assume without loss of generality  by that $w$ is admissible in $\T 0mn $ (recall from Definition \ref{def:derivable} the notion of derivable in other sectors). The proof follows from Lemma \ref{lem:derivation_interpretation} by applying the affine diffeomorphism $\AD {m}{n}$:  $\AD {m}{n}$ maps the flip-sheared copy of $\M nm$ to the semi-regular presentation of $\M nm$, and the trajectory $\tau$ with cutting sequence $w$ to a new trajectory $\tau'$ on $\M nm$. This trajectory has cutting sequence $\Der mn w$  by Lemma \ref{lem:derivation_interpretation}. 

Since we just showed that the derived sequence is a cutting sequence and cutting sequences are admissible (see \ref{admissiblelemma}),
it follows that cutting sequences are derivable (according to Definition \ref{def:derivable}). 

Finally, the second claim in the Theorem follows by showing that the sector $\Sec 0 m n$ is mapped by $\AD {m}{n}$ to $[0,\pi]\backslash \Sec 0 n m $. This can be verified by describing the image of $\Sec 0 m n$ by each of the elementary maps (see for example Figure \ref{octhex}) comprising $\derAD mn$.
\end{proof}

\subsection{Normalization}\label{sec:nor}
After deriving a derivable sequence using $\Der mn$, we now want to apply  derivation again, this time using the operator $\Der nm$ for $\M nm$. 
Since we defined derivation only for sequences admissible in $\T 0 mn$ and  derived sequences are admissible in a sector $\T i mn$ with $i \neq 0$ (by the second claim in the Theorem \ref{thm:derivable}), so in order to apply derivation one more time we first want to \emph{normalize} the derived sequence as follows. 

\begin{definition}
Given a sequence $w = (w_j)_{j \in \mathbb{Z}}$ which is admissible in diagram $\T i mn$, the \emph{normalized sequence} which will be denoted by $\Norm mn w$ is the sequence $\Norm mn w = (\perm i mn w_j)_{j \in \mathbb{Z}}$. Thus, the operator $\Norm mn$ maps sequences admissible in $\T i mn$ to sequences admissible in $\T 0 mn$.
\end{definition}

Let us remark that a sequence $w$ can in principle be admissible in more than one diagram $\T i mn$. In this case, we can use any of the corresponding $\perm i mn$ to normalize $w$. On the other hand, we will apply $\Norm mn$ to cutting sequences and one can show that if a \bm cutting sequence is not \emph{periodic}, then it is admissible in a \emph{unique} diagram $\T i mn$ (Lemma \ref{lemma:uniqueness_sector}), hence $\Norm mn w$ is uniquely defined. 

\subsection{Cutting sequences are infinitely derivable}\label{sec:infinitely_derivable}
Let $w$ be a derivable sequence in $\T 0 mn$. Then by definition of derivability, $\Der  mn w$ is admissible in some  $\T i nm$. By normalizing, $\Norm nm \Der m n  w$ is admissible in $\T 0 nm$ and we can now apply $\Der nm $. 

\begin{definition}\label{def:infinitely_derivable}
A sequence $w$ in $\LL mn ^{\mathbb{Z}}$ is \emph{infinitely derivable} if it is admissible and, by alternatively applying derivation and normalization operators, one always gets admissible sequences, i.e. for any even integer  $k=2l$,
$$ \underbrace{(\Der nm  \Norm nm  \Der mn \Norm mn ) \cdots  (\Der nm  \Norm nm  \Der mn \Norm mn )}_{\text{$l$ times }} w$$
 is admissible on some $\T i mn $ for some $0\leq i \leq n-1$ 
 and for any odd integer $k=2l+1$,
$$ \underbrace{(\Der mn  \Norm mn  \Der nm \Norm nm )\cdots  (\Der mn \Norm mn \Der nm  \Norm nm ) }_{\text{$l$ times }} \Der mn \Norm mn w$$
for any $l \in \mathbb{N}$ is admissible on some $\T i  nm $ for some $0\leq i \leq m-1$. 
\end{definition}

We can  now show that cutting sequences are infinitely derivable. 

\begin{theorem}\label{thm:infinitely_derivable}
Let $w$ be a cutting sequence of a bi-infinite linear trajectory on $\M mn$. Then $w$ is \emph{infinitely derivable} in the sense of Definition \ref{def:infinitely_derivable}. 

\end{theorem}
\begin{proof}
We will prove this by induction on the number $k$ of times one has derived and normalized the sequence $w$ that the resulting sequence $w^{k}$  is admissible (on some $\T i  mn $ for $k$ even, on some $\T i  nm $ for $k$ odd).  Assume that we have proved it for $k$. Say that $k$ is odd, the other case being analogous.  In this case $w^{k}$ is a cutting sequence of a trajectory $\tau^k$ in some sector $\Sec i nm$. Since $\Norm nm $ acts on $\tau^k$ by the permutation $\perm i nm$ induced by an isometry of $\M nm$, also $\Norm nm w^k$ is a cutting sequence, of the trajectory $\perm i nm \tau^k$ which belongs to the standard sector $\Sec 0 n m$.   By applying $\Der mnm$, by the first part of Proposition \ref{thm:derivable}, the sequence  $w^{k+1}:=\Der nm \Norm nm w^{(n)}$ is again a cutting sequence.  Since cutting sequences are admissible (Lemma \ref{admissiblelemma}) this concludes the proof. 
\end{proof}

\subsection{Sequences of admissible sectors}\label{sec:sectors_sequences}
Let  $w$ be a cutting sequence of a bi-infinite linear trajectory $\tau$ on $\M mn$. 
Since by Theorem \ref{thm:infinitely_derivable} $w$ is infinitely derivable, one can alternatively derive it and normalize it to obtain a sequence $(w^k)_k$ of cutting sequences. More precisely:

\begin{definition}\label{def:derivatives}
The \emph{sequence of derivatives} $(w^k)_k$ starting from a cutting sequence of a bi-infinite linear trajectory $w$ on $\M mn$, is the sequence recursively defined by:  
\bes
w^0:=w, \qquad w^{k+1} := \begin{cases}  \Der nm (\Norm nm w^k)  , & k \text{ odd} ; \\ \Der mn  (\Norm mn w^k), & k \text{ even} .  \end{cases}
\ees
\end{definition}
This sequence is well-defined by Theorem \ref{thm:infinitely_derivable}, and by the same Theorem, 
 $(w^k)_k$ are all admissible sequences 
in at least one sector. Furthermore, if $w$ is non-periodic, each $w^k$ is admissible in a unique sector by Lemma \ref{lemma:uniqueness_sector}. 
   
 We now want to record after each derivation the sectors in which the cutting sequences $w^k$ are admissible. By Proposition \ref{thm:derivable},  for $k$ even $w^k$ is admissible in (at least one) of the sectors $\Sec i m n $ where $1\leq i \leq n-1$, while for $k$ odd $w^k$ is admissible in (at least one) of the sectors $\Sec i  nm  $ where $1\leq i \leq m-1$. Let us hence define two sequences of indices in $n-1$ and $m-1$ symbols respectively as follows.


\begin{definition}[Sequences of admissible sectors]\label{def:seq_sectors}
Let $w$ be a 
cutting sequence of a linear trajectory on $\M mn$ in the standard sector $\Sec 0 mn$.  Let us say that the two sequences $(a_k)_k \in \{1, \dots , m-1\}^\mathbb{N}$ and $(b_k)_k \in \{1, \dots ,n-1\}^\mathbb{N}$ are a \emph{pair of sequences of admissible sectors associated to $w$} if 
\begin{itemize}
\item
$w^{2k-1}$ is admissible in  $\Sec {a_k} nm$ for any $k\geq 1$;
\item
$w^{2k}$ is admissible in  $\Sec {b_k} mn $ for any $k\geq 1$.
\end{itemize}
\end{definition}
Thus, the sequence of admissible sectors for $w$, i.e. the sequence of sectors in which the derivatives $w^1, w^2, \dots$ are admissible, is given by 
$$  \Sec {a_1} {n}{m}, \Sec {b_1} {m}{n},  \Sec {a_2} {n}{m}, \Sec {b_2} {m}{n},\dots ,  \Sec {a_k} {n}{m},\Sec {b_k} {m}{n},  \dots .$$
We remark that the sequences of admissible sectors associated to a cutting sequence $w$ are \emph{unique} as long as $w$ is non-periodic, by virtue of Lemma \ref{lemma:uniqueness_sector}. 

\begin{convention}
If $w$ is a cutting sequence of a linear trajectory $\tau$ on $\M mn$ in a sector different from the standard one, we will denote the sector index by $b_0$, so that $\tau$ is admissible in $\Sec {b_0} mn$, where $0\leq b_0 \leq 2n-1$. 
\end{convention}

In \S \ref{sec:itineraries_vs_sectors}, after introducing the \bm Farey map $\F mn$, we will show that this pair of sequences is related to a symbolic coding of the map $\F mn$ and can be used to reconstruct from the sequence $w$, via a continued-fraction like algorithm, the direction of the trajectory of which $w$ is a cutting sequence (see \S \ref{sec:direction_recognition}, in particular Proposition \ref{prop:itineraries_vs_sectors}).

\subsection{Sequences fixed by renormalization}\label{sec:fixedpoints}
For the characterization of the closure of cutting sequences, we will also need the following characterization of periodic sequences, which are fixed points of our renormalization procedure. 
Let us first remark that  it makes sense to consider the restriction of the operators $ \Norm mn \Der nm $ and $ \Norm nm  \Der mn  $ to subwords of cutting sequences,   by following up in the process how a subword $u$ of the bi-infinite cutting sequence $w$ changes under derivation (some edge labels in $u$ will be dropped, while others will persist) and  under normalization (a permutation will act on the remaining edge labels). If $w' $ is obtained from a cutting sequence $w$ by applying a sequence of operators of the form $ \Norm mn \Der nm $ and $ \Norm nm  \Der mn  $ and $u'$ is the subword (possibly empty) obtained by following a subword $u$ of $w$ in the process, we will say that \emph{$u'$ is the image of $u$ in $w'$}.

\begin{lemma}\label{periodicABAB}
Let $w$ be the cutting sequence of a linear trajectory on $\M mn$ admissible in the standard sector $\Sec 0 mn$. If $w$  is fixed by our renormalization procedure, i.e. 
$$ \Norm mn \Der nm  \Norm nm  \Der mn w = w , $$ 
then $w$ is an infinite periodic word of the form $\dots n_1 n_2 n_1 n_2 \dots$ for some edge labels $n_1, n_2 \in {\LL mn} $.

Furthermore, if  $u$ is a subword of a cutting sequence $w$ on  $\M mn$ such that the image $u'$ of $u$ in $w' =  \Norm mn \Der nm  \Norm nm  \Der mn w$ has the same length (i.e. the same number of edge labels)  of $u$, then $u$ is a finite subword of the infinite periodic word $\dots n_1 n_2 n_1 n_2 \dots$ for some edge labels $n_1, n_2 \in {\LL mn}$. 
\end{lemma}

\begin{proof}
Let us first remark that it is enough to prove the second statement, since if $ w=w'$ where $w'$ is by definition $\Norm mn \Der nm  \Norm nm  \Der mn w $, in particular for any finite subword $u$ of $w$, the image $u'$ of $u$ in $w$ has the same number of edge labels as $u$. Thus, considering arbitrarily long finite subwords of $w$, the second statement implies that $w$ is the infinite word $\dots n_1 n_2 n_1 n_2 \dots$. 

Let $w$ be a cutting sequence of a linear trajectory $\tau$ on   $\M mn$. Without loss of generality we can assume that $\tau$ is in direction $\Sec 0 mn$, as the second part of the statement does not change up to applying the relabeling induced by permutations. Thus, we can assume that  the  cutting sequence $w$ (by definition of the transition diagrams) describes  an infinite path in $\T 0 mn$.   
Consider now the same path in the derivation diagram $\D 0 mn$. Any given finite subword $u$ of $w$ corresponds to a finite path on $\D 0 mn$. If we assume that the image $u'$ of $u$ in $\Der mn w $ has the same number of edge labels as $u$, the path cannot cross any vertical single arrow, as otherwise $u'$ would be shorter (see  Remark \ref{rk:derivation_short}). Thus, the path described by $u$ should consist of arrows all belonging to  the same row of $\D 0 mn$. 

Without loss of generality, we can assume that the finite path described by $u$ contains the arrow $r_1$ in the piece of diagram drawn here below.
\begin{figure}[!h] 
\centering 
\includegraphics[width=180pt]{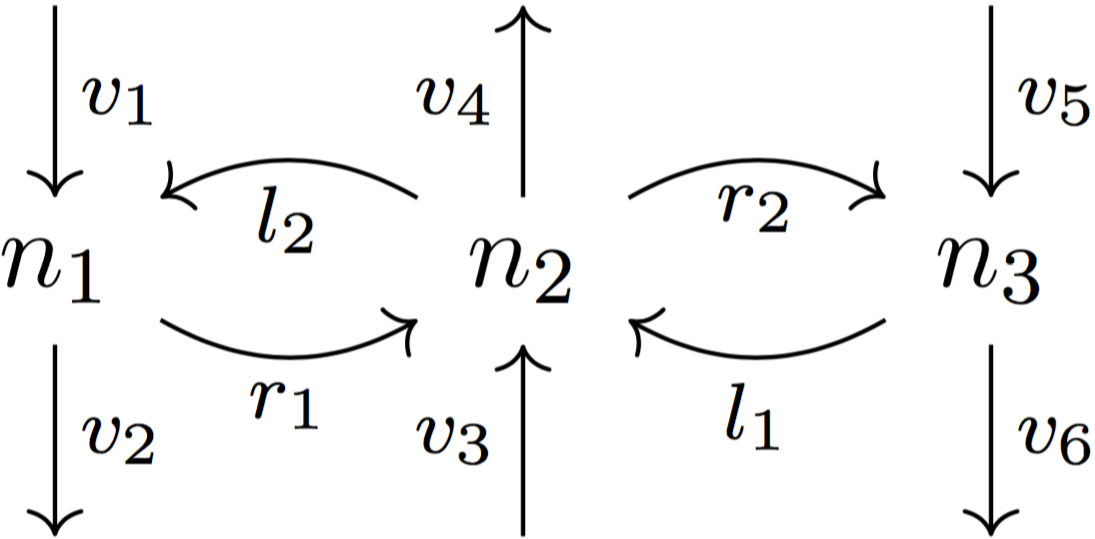}
\end{figure}
Since it cannot contain vertical arrows, the only arrows that can follow $r_1$ in the path can either be $l_2$ or $r_2$. Correspondingly, the sequence of edge labels that can appear in the word are either  $ n_1n_2n_1 $ or $n_1n_2n_3 $.  If we prove that the latter case leads to a contradiction, then by repeating the same type of argument, we get that the path must be going back and forth between edge labels $n_1$ and $n_2$ and hence the word $u$ if a finite subword of the infinite periodic word $\dots n_1 n_2 n_1 n_2 \dots$.

Let us assume that the arrow $r_1$ is followed by $r_2$ and show that it leads to a contradiction. Let us denote by $n_4$ and $n_5$  the edge labels of $r_1$ and $r_2$ respectively in the derivation diagram $\D 0 mn$. Then, by definition of the derivation operator $\Der nm$,  the image $u'$ of $u$ in $\Der nm w$  will contain the string $ \dots n_4 n_5 \dots$.   
 We claim that the transition diagram $\T 0 nm$ (for the standard sector of the dual surface $\M nm$) will contain a piece that looks like the following figure,  up to change the direction of the arrows:
\begin{figure}[!h] 
\centering 
\begin{tikzcd}
{n_4}\arrow[bend right]{r} \arrow{d} 
&{} \arrow[bend right]{l}  \\
{}\arrow[bend right]{r}
&{n_5} \arrow[bend right]{l} \arrow{u} \\ 
\end{tikzcd} \ \ \ \ \
\end{figure}
In particular we claim that the location of the edge labels $A',B'$ in $\T 0 nm$  is at opposite vertices as shown in the figure above. This can be deduced by the Structure Theorem \ref{tdtheorem} looking at how the labels of the arrows of the derivation diagram $\D 0 mn$ snake in  Figure \ref{gentransdiag} and comparing them with the labels of  the 
transition diagram $\T 0 mn$. 
For a concrete example, refer to  Figure \ref{34auxtd}: pairs of labels of arrows like $r_1$ and $r_2$ are for example {\color{green}$2$} and {\color{green}$8$} or {\color{green}$6$} and {\color{green}$2$} which indeed lie at opposite vertices in the derivation diagram for $\M 34$ and one can verify from Figure \ref{gentransdiag} that this is indeed always the case. In particular, $n_4,n_5$ \emph{do not} belong to the same row of $\T 0 nm$.

We know that the derived sequence $w'= \Der nm w$ is the cutting sequence of a linear trajectory on $\M nm$ and that it is admissible in (at least one) transition diagram $\T i nm$ for some $1\leq i \leq m$ (see Proposition \ref{thm:derivable}).  This means that there will be an arrow between the two vertices labeled $n_4$ and $n_5$ in the transition diagram $\T i nm$ \footnote{We remark that this does not yet imply though that there has to be an arrow connecting $n_4, n_5$ in $\T 0 nm$ and indeed this does not have to be the case in general.}.  

 Let us now apply the normalization operator $\Norm nm$, which corresponds to acting by the permutation $\perm i nm $ for some $1\leq i \leq m$ on the labels in $w$. Denote by $n_4',n_5'$ the images of the labels $n_4, n_5$.  
Since $\Norm nm  \Der mn w$ is admissible on $\T 0 nm$ and contains the transition $n_4' n_5'$, by construction $n_4'$ and $n_5'$ must be connected by an arrow of $\T 0 nm$. Furthermore,  the assumption on $u $ (that the image of $u$ in $w' =  \Norm mn \Der nm  \Norm nm  \Der mn w$ has the same length as $u$) implies that $n_4'$ and $n_5'$ also have to be on the same row of $\T 0 nm$ (otherwise  the image of $u$ in $w'$ would be shorter than $u$ because of the effect of $\Der nm$, see Remark \ref{rk:derivation_short}).  This means that also $n_4$ and $n_5$ were on the same row of $\T i nm$ (since by definition of $\perm i nm $, $n_4'= \pi_i(n_4)$ and $n_5'= \pi_i(n_5)$ are the labels on $\T 0 nm$ of the vertices which were labeled $n_4$ and $n_5$, respectively, on $\T i nm$). 
 
By Corollary \ref{preserves-rows-cor}, the action of $(\perm i nm) ^{-1}= \perm i nm$ ($\perm i nm$ are involutions since the reflections $\refl i nm$ are)  maps the transition diagram  $\T i nm$ to $\T 0 nm$ by mapping labels on the same rows to labels on the same row. Thus, we get that $n_4$ and $n_5$, which we just said are on the same row of $\T i nm$, are also on the same row of $\T 0 nm$, in contradiction with what we proved earlier (see the above Figure, that shows that  $n_4$ and $n_5$ are \emph{not} on the same row on $\T 0 nm$).   This concludes the proof that $u$ has the desired form and hence, by the initial remark, the proof of the Lemma.
\end{proof}
Finally, for the characterization we will also need to use that all sequences which have the form of fixed sequences under derivation, i.e. of the form $\dots n_1n_2n_1n_2 \dots$, are actually cutting sequences:

\begin{lemma}\label{realizecs}
Given a transition $n_1n_2$, such that the word $n_1n_2n_1n_2$ is admissible in some diagram $\T i mn$, 
the periodic sequence of form $\dots n_1n_2n_1n_2 \dots$ is the cutting sequence of a periodic trajectory in $\M mn$.
\end{lemma}

\begin{proof}
If $n_1n_2n_1n_2$ is admissible in a diagram $\T i mn$, then both $n_1n_2$ and $n_2n_1$ are admissible in that diagram, so  the labels $n_1$ and $n_2$ are  connected by arrows in both directions, so $n_1$ and $n_2$ must be on the same row of the diagram $\T i mn$ and must be adjacent.

We recall from Section \ref{hooperdiagrams} that white vertices in a Hooper diagram correspond to horizontal cylinders and black ones correspond to "vertical" cylinders.
Edges around a vertex correspond to (possibly degenerate) basic rectangles composing the cylinder.
Moreover, sides in the polygonal representation are diagonals of the basic rectangles which correspond to horizontal sides in the Hooper diagram. 
Using this last fact, we can label the horizontal edges of a Hooper diagram with the edge labels of the polygonal representation (see Figure \ref{hd-labeled}) and we proved in Section \ref{transitiondiagrams} that the labels of the Hooper diagram have the same structure as the transition diagram in the standard sector (see Figure \ref{construct2}).
The latter observation means that since $n_1$ and $n_2$ were adjacent and in the same row of $\T i mn$, they will label two adjacent horizontal edges of the Hooper diagram.

Let us consider first the case when $n_1$ and $n_2$ label horizontal edges of the Hooper diagram which share a white vertex as a common endpoint. 
They will hence correspond to the two basic rectangles of the horizontal cylinder represented by the white vertex in the Hooper diagram, which have the sides labeled by $n_1$ and $n_2$ as diagonals. 
Consider now a horizontal trajectory in the polygon contained in this horizontal cylinder. By looking at the way the arrows in the Hooper diagram follow each other around a white edge, we can see that the horizontal trajectory will cross in cyclical order the four (possibly degenerate) basic rectangles forming the cylinder, crossing first a basic rectangle which does not contain sides of the polygonal representation (corresponding to a vertical edge in the Hooper diagram), then the one with a diagaonl labeled by $n_1$ (corresponding to a horizontal edge in the Hooper diagram), then another basic rectangle which does not contain any side (corresponding to the other vertical edge) and finally the one with a diagaonl labeled by $n_2$.
This means that the cutting sequence of such trajectory corresponds to the periodic trajectory $\ldots n_1n_2n_1n_2\ldots$.

Recalling that "vertical" means vertical in the orthogonal decomposition, and at an angle of $\pi/n$ in the polygon decomposition (see Figure \ref{hexort1} for the correspondence), the same argument can be used for the case when $n_1$ and $n_2$ label horizontal edges of the Hooper diagram connected by a black vertex. 
This proves that a path at angle $\pi/n$ across the cylinder represented by the black dot corresponds to the periodic trajectory $\ldots n_1n_2n_1n_2\ldots$.
\end{proof}

\section{The \bm Farey maps} \label{farey}
We will describe in this section a one-dimensional map that describes the effect of renormalization on the direction of a trajectory. We call this map the \emph{\bm Farey map}, since it plays an analogous role to the Farey map for Sturmian sequences.  We define the \bm Farey map $\FF mn$ in two steps, i.e. as composition of the two maps $\F mn$ and $\F nm $, which correspond respectively to the action of $\Der mn $ and $\Der nm $, each composed with normalization.  

\subsection{Projective transformations and projective coordinates}\label{sec:projective}
Let us introduce some preliminary notation on projective maps. 
Let $\RP$ be the space of lines in $\R^2$. A line in $\R^2$ is determined by a non-zero column vector with coordinates $x$ and $y$. There are two \emph{coordinate systems} on $\RP$ which will prove to be useful in what follows and that we will use interchangably.  The first is the \emph{inverse slope coordinate}, $u$.  We set $u((x,y))=x/y$.   The second useful coordinate is the \emph{angle coordinate} $\theta \in [0, \pi]$, where $\theta$ corresponds to the line generated by the vector with coordinates $x=\cos(\theta)$ and $y=\sin(\theta)$. Note that since we are parametrizing lines rather than vectors, $\theta$ runs from $0$ to $\pi$ rather than from $0$ to $2\pi$.

An interval in $\RP$ corresponds to a collection of lines in $\R^2$. We will think of such an interval as corresponding to a sector in the upper half plane (the same convention is adopted in \cite{SU, SU2}).

\begin{convention}
We will still denote by  $ \Sec i mn$ for $0\leq i \leq n-1$ (see  Definition \ref{sectordef})
the  sector of $\RP$ 
corresponding to the angle coordinate sectors $ \left[ i \pi /n ,  (i+1)\pi/n \right]$ for $i=0,\dots, 2n-1$, each  of length $\pi/n$ in $[0,\pi]$. 
 We will abuse notation by writing $u \in \Sec i mn$ or $\theta \in \Sec i mn$, meaning that the coordinates belong to the corresponding interval of coordinates.
\end{convention}

A linear transformation of $\R^2$  induces a \emph{projective transformation} of $\RP$ as follows.  If $L= \left(\begin{smallmatrix}a& b \\ c& d  \end{smallmatrix} \right)$ is a matrix in $GL(2, \mathbb{R})$, the induced projective transformation is given by the associated \emph{linear fractional transformation} $ L[x]= \frac{ax+b}{cx+d}$. This linear fractional transformation records the action of $L$ on the space of directions in inverse slope coordinates. Let $PGL(2,\R)$ be the quotient of $GL(2, \mathbb{R})$ by   all homotheties $\{ \lambda I,\ \lambda \in \mathbb{R}\}$, where $I$ denotes the identity matrix. Remark that the linear fractional transformation associated to a homothety is the identity.
The \emph{group of projective transformations} of $\RP$ is   $PGL(2,\R)$. 

\subsection{The projective action  $\F mn$ of $\AD mn$} \label{sec:2farey}
Let us recall that in \S \ref{sec:affine} we defined an affine diffeomorphism $\AD m n$ from $\M mn$ to $\M nm$ which acts as a \emph{flip and shear}. The linear part of $\AD mn$ is the $SL(2, \R)$ matrix 
$\derAD m n $ in \eqref{def:derivativeAD}, obtained as the product $\shear nm \diag{m}{n} \shear{m}{n} f$ of the matrices in  \eqref{def:generalmatrices}.  
 The diffeomorphism $\AD m n$ acts projectively by sending the standard sector $\Sec 0 mn$ to the complement $(\pi/m, \pi)$ of the standard sector $\Sec 0 nm$. 
 This is shown for $m=4,n=3$ in Figures \ref{hexoct} and \ref{octhex}, for $\AD4 3$ and $\AD 3 4 $ respectively, where the effect of each elementary matrix in the product giving $\derAD m n $ is illustrated. 

Let $\refl  i nm$ for $i=0, \dots m-1$ be the isometry of $\M nm$ described in \S \ref{transitiondiagrams}, which maps $\Sec i  nm$ to $\Sec 0  nm$. We will abuse the notation and also denote by $\refl  i nm$  the matrix in  $PGL(2,\R)$ which represents them (see Example \ref{ex:refl_matrices} for the matrices $\refl  i 4 3$ for $\M 43$). We stress that when we consider products of the matrices $\refl  i nm$ we are always thinking of the product as representing the corresponding coset in $PGL(2,\R)$.


Let us define the map $\F m n$ so that it  records the projective action of $ \AD m n$ on the standard sector  $\Sec 0 mn$, composed with \emph{normalization}. Let us recall  from \S\ref{sec:normalization} that we normalize trajectories in  $\Sec j n m$ by applying the reflection $\refl  j nm $ that maps them to $\Sec 0 n m$ (see Definition \ref{def:reflections}). Thus, we have to compose $\derAD m n $ with $\refl  j nm $ exactly when the image under  $\derAD m n $ is contained in $\Sec j nm$. Let us hence define the subsectors $\Subsec j mn \subset \Sec 0 mn$ for $1\leq j \leq m-1$ which are given in inverse slope coordinates by
\be \label{def:subsectors}
\Subsec j mn : = \{ 
({\derAD m n})^{-1} [u], u \in   \Sec j nm \} \quad \text{for} \quad 1 \leq j \leq m-1.
\ee
\begin{remark}\label{rk:subsec}
Thus, $u \in \Subsec j mn$ iff $\derAD m n [u] \in   \Sec j nm$. 
\end{remark}
We can then define   the map $\F m n : \Sec 0 mn \to\Sec 0 nm $ to be the piecewise-projective map, whose action on the subsector of directions corresponding to $\Subsec j mn$ is given by 
the  projective action given by $\refl  j nm \derAD m n $, that is, in inverse slope coordinates, by the following piecewise linear fractional transformation: 
\be \label{def:Fmn}
 \F mn (u)  = \refl  j nm \derAD m n [u ] =  \frac{a_i u + b_i}{c_i u + d_i}, \qquad \mathrm{where} \, \, \begin{pmatrix} a_i & b_i \\ c_i & d_i \end{pmatrix}:= \refl  j  n m \derAD m n, \quad \mathrm{for\ } u  \in  \Subsec j mn , \leq j \leq m-1. 
\ee
 
The action in angle coordinates is obtained by conjugating  by conjugating by $\cot : [0,\pi] \rightarrow \mathbb{R}$, so that if $\theta \in \overline{\Sigma_i}$ we have $F (\theta) =   \cot^{-1}\left( \frac{a_1 \cot(\theta) + b_i}{c_i \cot(\theta) + d_i}\right)$. Let us remark that the change from the coordinates $u$ to $\theta$ through cotangent reverses orientation.  

In Figure \ref{fareymap} we show the graphs of the map $\F 43$ and $\F 34$ in angle coordinates.   



\begin{figure}[!h]
\centering
\includegraphics[width=.44\textwidth]{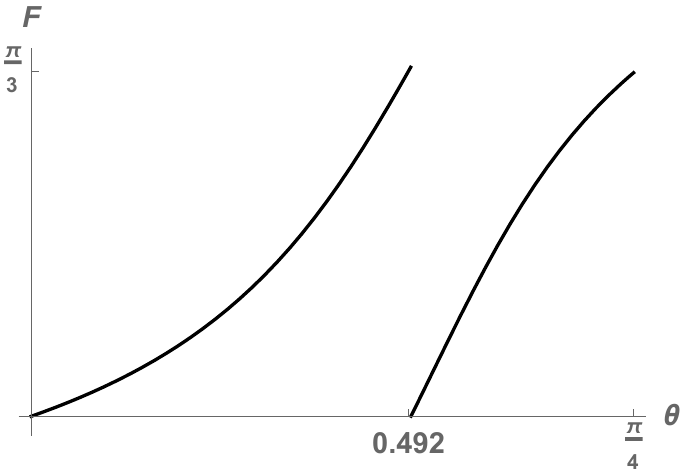} \ \ \ \ \ 
\includegraphics[width=.44\textwidth]{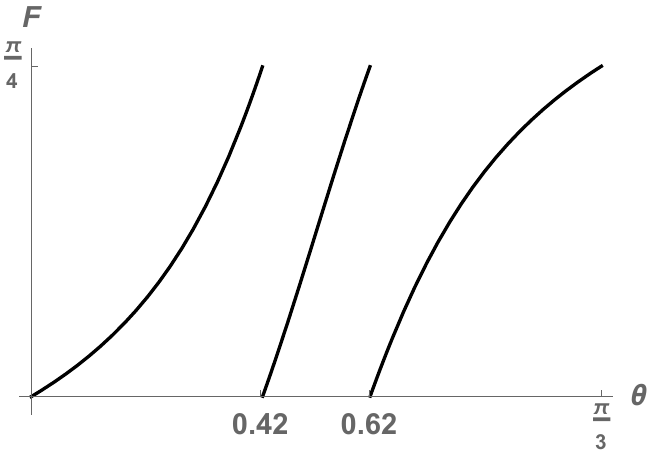}
\begin{quote}\caption{The Farey maps $\F 43$ and $\F 34$ .  \label{fareymap}} \end{quote}
\end{figure}

Remark that since the image of the standard sector $\Sec 0 m n$ by $\F m n$  is contained in the standard sector $\Sec 0 n m$, we can compose $\F m n$ with $\F n m$.

\begin{definition}\label{def:FareyFF}
The \emph{\bm Farey map} $\FF m n: \Sec 0 mn \to \Sec 0 mn$ for \M mn is the composition $\FF m n :=  \F n m \circ \F m n$ of the maps $\F m n$ and $\F n m$ given by \eqref{def:Fmn}.
\end{definition}
In Figure \ref{ffareymap} we show the graphs of the maps $\FF 4 3$ and $\FF 34$ in angle coordinates.  
The  map $\FF m n $  is  pointwise expanding, but not uniformly expanding since the expansion constant tends to $1$ at the endpoints of the sectors. Since each branch   of $\FF m n $   is monotonic,  the inverse maps 
of each branch 
are well defined. 

\begin{figure}[!h]
\centering
\includegraphics[width=.44\textwidth]{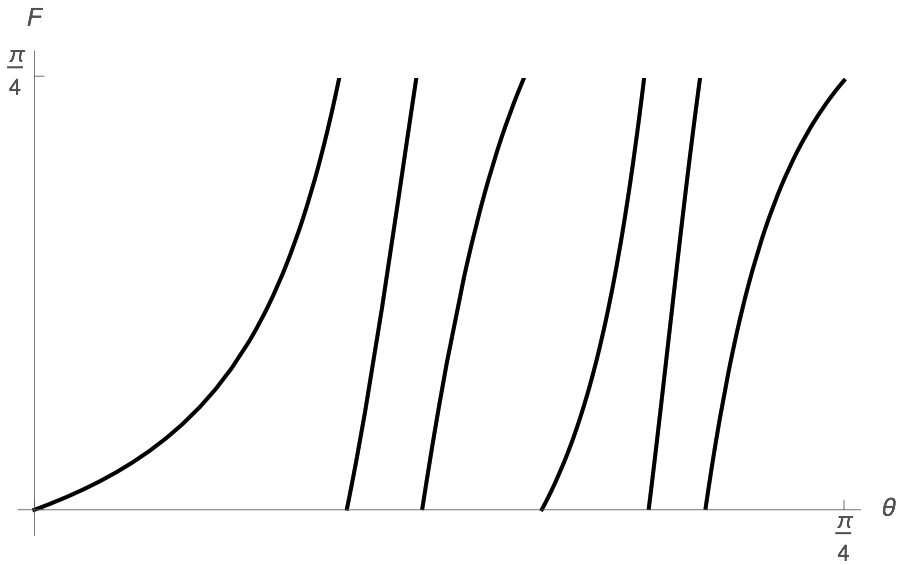} \ \ \ \ \ 
\includegraphics[width=.44\textwidth]{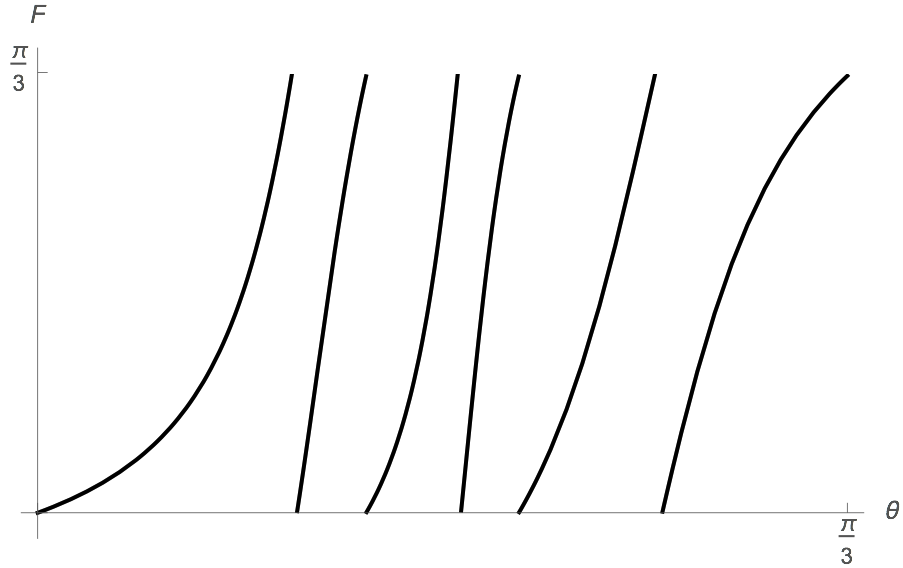}
\begin{quote}\caption{The Farey maps $\FF 34$ and $\FF 43$.  \label{ffareymap}} \end{quote}
\end{figure}

\subsection{Itineraries and sectors}\label{sec:itineraries_vs_sectors}
The \bm Farey map $\FF m n$ has $(m-1)( n-1) $ branches. The intervals of definitions of these branches are the following subsectors of $\Sec 0 m n $ that we will denote by $\Subsubsec i j m n$ for $1\leq i \leq m-1, 1\leq j \leq n-1$:
\be\label{def:subsubsec}
\Subsubsec i j m n =  \Subsec i   mn \cap (\F m n)^{-1} (\Subsec j  nm), \qquad i=1, \dots,  m-1, \quad  j =1 \dots,  n-1.
\ee
Thus, explicitely,
\be \label{explicitdefFareyFF}
\FF m n(u) = \refl {b_1} m n \derAD nm \refl {a_1} n m \derAD mn [u], \qquad \text{iff} \ u \in \Subsubsec i j m n, 
\ee
Remark that if $\theta \in \Subsubsec i j m n$, then the  
affine diffeomorphism $\AD m n$ sends the direction $\theta$ to a direction $\theta'$ in the sector $\Sec i  n m$ and then,  after normalizing the direction $\theta'$ to a direction $\refl i nm [\theta]$ in the standard sector $\Sec 0  n m$, the  
affine diffeomorphism $\AD n m$ sends it to a direction $\theta''$ in the sector $\Sec j  m n$. Thus the indices $i, j$ record the visited sectors.

\smallskip
Let us \emph{code} the orbit of a direction under $\FF m n$  as follows. 
\begin{definition}[Itinerary]\label{itinerarydef}
To any $\theta\in \Sec 0 mn$ we can assign two  sequences $(a_k)_k \in \{1, \dots , m-1\}^\mathbb{N}$ and $(b_k)_k \in \{1, \dots , n-1\}^\mathbb{N}$  defined by  
\bes
{\left(\F m n\right)}^{k-1}(\theta)\in \Subsubsec {a_k} {b_k} m n   \text{ for each } k \in \mathbb{N}. 
\ees
We call the sequence $\left((a_k, b_k) \right)_k  $ the \emph{itinerary} of $\theta$ under $\FF m n$. %
\end{definition}


Let us recall that in \S \ref{sec:sectors_sequences} given a cutting sequence $w$, by performing derivation and normalization, we  assigned to it a pair of sequences recording the sectors in which derivatives of $w$ are admissible (see Definition \ref{def:seq_sectors}), uniquely when $w$ is non-periodic (by Lemma \ref{lemma:uniqueness_sector}). These sequences are related to  itineraries of $\F m n$ as follows:


\begin{proposition}\label{prop:itineraries_vs_sectors}
Let $w$ be a non-periodic cutting sequence of a bi-infinite linear trajectory on $\M mn$ in a direction $\theta$ in $\Sec 0 mn$. Let 
 $(a_k)_k \in \{1, \dots , m-1\}^\mathbb{N}$ and $(b_k)_k \in \{1, \dots , n-1\}^\mathbb{N}$ be the pair of sequences of admissible sectors associated to $w$ (see Definiton \ref{def:seq_sectors}).   
Then  the sequence $\left((a_k, b_k) \right)_k  $ is the \emph{itinerary} of $\theta$ under $\FF m n$. 
\end{proposition}
The proof is based on the fact  that the \bm Farey map shadows at the projective level the action of the geometric renormalization procedure that is behind the combinatorial derivation and normalization procedure on cutting sequences. 
\begin{proof}
Let $(w^k)_k$ be the sequence of derivatives of $w$ given by Definition \ref{def:derivatives}. Remark that since $\tau$ is not periodic, none of its derivatives $w^k$ is periodic. Thus, since $w^k$ is non-periodic, it is admissible in a unique diagram that (by definition of $(a_k)_k, (b_k)_k$ as  sequences of admissible sectors) is $\T {a_j} {n}{m}$  for $k=2j-1$ odd and  $\T {b_j} {m}{n}$  for $k=2j$ even. Thus, the sequence  $(u^k)_k$ of \emph{normalized derivatives}, given by $u^k := \Norm nm w^k$ for $k$ odd and $u^k := \Norm mn w^k$ for $k$ even, is well defined and is explicitly given by $u^k= \perm {a_j} {n}{m} w^k$  for $k=2j-1$ odd and  $u^k= \perm {b_j} {m}{n} w^k$  for $k=2j$ even.   

 From Proposition \ref{thm:derivable} we know that, for any $k$, $w^k$ is the cutting sequence of a trajectory $\tau^k$ and thus $u^k$ is the cutting sequence of a trajectory $\overline{\tau}^k$ in the standard sector obtained by normalizing $\tau^k$. These trajectories can be constructed recursively as follows.  Set $\tau^0:= \tau= \overline{\tau}^0$ (since we are assuming that $\tau$ is in the standard sector). Assume that for some $k\geq 1$ we have already defined ${\tau}^{k-1}$ and $\overline{\tau}^{k-1}$ so that their cutting sequences are respectively $w^{k-1}$ and $u^{k-1} $. Deriving $u^{k-1}$, we get $w^{k}$, which, by definition of sequence of sectors and the initial observation, is admissible \emph{only} in $\T {a_j} {n}{m}$  for $k=2j-1$ odd and \emph{only}  in $\D {b_j} {m}{n}$  for $k=2j$ even. It follows that the trajectory  $\tau_{k}$ of which $w^{k} $ is a cutting sequence belongs to $\Sec {a_j} {n}{m}$  for $k=2j-1$ odd and  $\Sec {b_j} {m}{n}$  for $k=2j$ even.  Thus, to normalize it we should apply $\refl {a_j} n m $ or $\refl {b_j} m n $ according to the parity. Hence, set for any $k\geq 1$: 
 \begin {eqnarray} \label{trajectories_recursively}
  \tau_{k} &:= &\begin{cases}     \AD m n \overline{\tau}^{k-1}, & k-1\ \text{even} ,\\ \AD n m \overline{\tau}^{k-1} , & k-1\ \text{odd} ; \end{cases}\\
 \overline{\tau}_{k} &:= & \begin{cases}  \refl {a_j}{n}{m} \tau^{k} , & k=2j-1\ \text{odd} ; \\   \refl {b_j}{m}{n} \tau^{k}, & k=2j\ \text{even} .  \end{cases}
\end{eqnarray}
Let $(\overline{\theta}^k)_k$ be the directions of the normalized trajectories $(\overline{\tau}^k)_k$. 
Recalling now the definition of the  \bm Farey map $\FF mn = \F nm \circ \F m n$, and of each of the maps $\F nm $ and $\F m n$ defined in (\ref{def:Fmn}), we see then that the directions $(\overline{\theta}^k)_k$ of $(\overline{\tau}^k)_k$ satisfy for any $k\geq 1$:
\bes
\overline{\theta}^{k} = \begin{cases}  \F mn (\overline{\theta}^{k-1}) =  \refl {a_j}{n}{m} \derAD mn [\overline{\theta}^{k-1}] , & k=2j-1 \ \text{odd} ; \\  \F nm (\overline{\theta}^{k-1}) =  \refl {b_j}{m}{n} \derAD nm [\theta^{k-1}], & k=2j \ \text{even} .  \end{cases}
\ees
It follows from Remark \ref{rk:subsec} that  $\overline{\theta}^{k-1}$ belongs to $\Subsec {a_j}{n}{m}$ for $k=2j-1$ odd (since  $\theta^k= \derAD mn [\overline{\theta}^{k-1}] $, which belongs to $\Sec {a_j}{n}{m}$) and  to $\Subsec {b_j}{m}{n}$ for $k=2j$ even (since  $\theta^k= \derAD nm [\overline{\theta}^{k-1}] $, which belongs to $\Sec {b_j}{m}{n}$). 
 
From the definition of $\FF mn^l$ as composition, we also have that $\overline{\theta}^{2l} = \FF mn ^l (\theta)$ for every $l \geq 0$.
Thus, by definition \eqref{def:subsubsec} of the subsectors $\Subsubsec i j mn$ and using the previous formulas for $k-1:=2l$ (so that   $k$ is odd and we can write it as $k=2j-1$ for $j=l+1$), we have that   $\FF mn ^l (\theta)$ belongs to $\Subsubsec {a_{l+1}} {b_{l+1}} mn$ for every $l\geq 0$. This shows that $\left((a_k,b_k)\right)_k$ is the itinerary of $\theta$ under the \bm Farey map and concludes the proof. 
\end{proof}

In the next section we show that, thanks to Proposition \ref{prop:itineraries_vs_sectors} and the expanding nature of the map $\F mn$, one can recover the direction of a trajectory from its cutting sequence (see Proposition \ref{directionsthm}).



\subsection{Direction recognition}\label{sec:direction_recognition}
Given a cutting sequence $w$, we might want to recover the direction of the corresponding trajectory. This can be done by exploiting a continued fraction-like algorithm associated to the \bm Farey map, as follows. 

 Let $\mathcal{I}$  be the set of all possible itineraries of $\FF mn$, i.e. the set
\bes
\mathcal{I}_{m,n}:= \{ \left((a_k, b_k) \right)_k , \quad (a_k)_k \in \{1, \dots , m-1\}^\mathbb{N}, \quad (b_k)_k \in \{1, \dots , n-1\}^\mathbb{N}\}.
\ees
Let us recall that $\FF mn$ is monotonic and surjective when restricted to each subiterval 
 $\Subsubsec i j m n$ for $1\leq i \leq m-1, 1\leq j \leq n-1$. Let us denote by $\FFF m n  i j$ the restriction of $\FF mn$ to $\Subsubsec i j m n$. Each of these branches $\FFF m n  i j$ is invertible. 

 Given  $\left((a_k, b_k) \right)_k \in \mathcal{I}_{m,n}$, one can check that intersection 
\be \label{sectorCFdef}
\bigcap_{k \in \mathbb{N}} (\FFF m n  {a_1} {b_1} )^{-1} (\FFF m n  {a_2} {b_2})^{-1} \cdots  (\FFF m n  {a_k} {b_k})^{-1} [0, \pi] 
\ee is non empty and consists of a single point $\theta$ {(the inverse branches \FF m n are indeed strictly contracting apart from finitely many parabolic points, and one can argue this as in the proof of  Lemma 2.2.15 in \S~4.3 of \cite{SU})}. In this case we write
$\theta = [a_1, b_1, a_2, b_2 , \dots ]_{m,n}$
 and say that $[a_1, b_1, a_2, b_2 , \dots ]_{m,n} $ is a \emph{\bm additive continued fraction expansion} of $\theta$. To extend this continued fraction beyond directions in the standard sector, we set the following convention. For an integer $0\leq b_0 \leq 2n-1 $,  and sequences $(a_k)_k, (b_k)_k$ as above, we set
\be\label{bmCFdef}
[b_0; a_1, b_1, a_2, b_2 , \dots ]_{m,n}:= 
\left(\refl {b_0} {m}{n}\right)^{-1} [\theta], \qquad \text{where \ } \theta = [a_1, b_1, a_2, b_2 , \dots ]_{m,n}.
\ee
The index $b_0$ is here such that the above angle lies in $\Sec {b_0} mn$. 
The notation recalls the standard continued fraction notation, and $b_0$ plays the role analogous to the integer part. 

\smallskip
We have the following result, which allows us to reconstruct the direction of a trajectory from the combinatorial knowledge of the sequence of admissible sectors of its cutting sequence:

\begin{proposition}[Direction recognition]\label{directionsthm}
If $w$ is a \emph{non-periodic} cutting sequence of a linear trajectory in direction $\theta \in [0, 2\pi]$, then
\bes
\theta=[b_0; a_1, b_1, a_2, b_2 , \dots ]_{m,n}, 
\ees 
where $b_0$ is such that $\theta \in \Sec {b_0} {m}{n}$ and the two sequences \mbox{$(a_k)_k \in \{1, \dots , m-1\}^\mathbb{N}$} and \mbox{$(b_k)_k \in \{1, \dots , n-1\}^\mathbb{N}$} are a \emph{pair of sequences of admissible sectors associated to $w$}.
\end{proposition}
Let us recall that $b_0$ and the sequences $(a_k)_k \in \{1, \dots , n-1\}^\mathbb{N}$ and $(b_k)_k \in \{1, \dots , m-1\}^\mathbb{N}$  are uniquely determined when $w$ is non-periodic (see \S \ref{sec:sectors_sequences}). Let us also remark that the above Proposition implies in particular that the direction $\theta$ is \emph{uniquely determined} by the combinatorial information given by deriving $w$. 

\begin{proof}
Without loss of generality, by applying $\perm {b_0}{m}{n}$ to $w$ and $\refl {b_0}{m}{n}$ to $\tau$, we can assume that the direction $\theta $ of $\tau$ is in the standard sector and reduce to proving that $\theta $ is the unique point of intersection of \eqref{sectorCFdef}. 
By Proposition \ref{prop:itineraries_vs_sectors}, the itinerary of $\theta$ under $\FF mn$ is $\left((a_k, b_k)\right)_k$. 
 By definition of itinerary, for every $k \in \mathbb{N}$ we have that  $\theta^k:= {\left(\F m n\right)}^k(\theta) \in \Subsubsec {a_k}{b_k}{m}{n} $. Thus, since $\FF mn$ restricted to  $\Subsubsec {a_k}{b_k}{m}{n} $ is by definition the branch $\FFF m n  {a_k} {b_k}$, we can write $$\theta^{k}=(\FFF m n  {a_k} {b_k})^{-1} (\theta^{k+1}), \qquad  \forall \ k \in \mathbb{N}.$$ 
  This shows that $\theta$ belongs to the intersection \eqref{sectorCFdef} and,  
 since the intersection consists of an unique point, it shows that $\theta = [ a_1,b_1,a_2,b_2,\dots]_{m,n}$. 
\end{proof} 
\section{Characterization of cutting sequences} \label{sec:characterization}
In this section we will give a complete characterization of (the  closure of) \bm cutting sequences in the set of all bi-infinite sequences in the alphabet $\LL mn $. 
As in \cite{SU} we cannot give a straightforward characterization of the cutting sequences as the closure of infinitely derivable sequences. 
In fact, such a characterization holds only for the case of the Sturmian sequences on the square, presented in \cite{Series}. As in the regular $2n$-gons case for $n \geq 3$,  we can still give a full characterization and we will present it in two different ways. 

The first way, as in \cite{SU}, consists in introducing the so-called \emph{generation rules}, a combinatorial operation on sequences inverting the derivation previously defined. These are defined in \S \ref{sec:generation}, where it is shown that they allow us to \emph{invert} derivation. In \S \ref{sec:generation_characterization} we then state and prove a characterization using generation (see Theorem \ref{thm:generation_characterization}). 
The second way, presented in \S \ref{sec:substitutionscharacterization}, will be obtained from the previous one by replacing the generation rules with the better known substitutions, in order to obtain an $\mathcal{S}$-adic presentation, i.e. Theorem \ref{thm:substitutionscharacterization}. 

\subsection{Generation as an inverse to derivation}\label{sec:generation}
In this section we  define \emph{generation operators},\footnote{The name was introduced in \cite{SU}, where this type of operator was also used to invert derivation.} which will allow us to \emph{invert derivation}. Generations are combinatorial operations on sequences which, like derivation, act by interpolating a sequence with new edge labels and dropping the previous ones. They will be used to produce sequences which, derived, give back the original sequence.  
Let us recall that in our renormalization procedure we always compose the derivation operators (alternatively $\Der m n$ and $\Der n m$) with a normalization operator ($\Norm nm $ or $\Norm mn$ respectively) which maps sequences admissible in other sectors back to sequences admissible in the standard sectors (on which the derivation operators are defined). Thus, we want more precisely to define operators that invert the action of the composition $\Norm nm \Der m n$ of derivation and normalization on sequences in $\M mn$. 
 
 It turns out that the operator $\Norm nm \Der m n$ cannot be inverted uniquely. This is because, as we saw in \S \ref{derivation}, under the action of $\AD m n$, one of our sectors in $\M mn$ opens up to a whole range of sectors in $\M nm$ (more precisely $m-1$ sectors, as many as the sectors in the complement of the standard one for $\M nm$.) 
Then by normalizing, we bring each of these sectors back to the standard one. 
As a consequence, when we have the cutting sequence of a trajectory in $\Sec 0mn$ for $\M nm $, there exist $m-1$ cutting sequences of trajectories in the standard sector for $\M nm$ which, derived and normalized, produce the same cutting sequence. To uniquely determine an inverse, we have to specify the sector in which the derived sequence is admissible before normalizing.

We will hence define $m-1$ generations $\gen i mn $ for $1\leq i \leq m-1$, each of which inverts $\Norm nm \Der m n$: each $\gen i mn $  will send an admissible sequence $w$ in $\T 0 nm$
 to an admissible sequence  $\gen i mn w $  in $\T 0 mn$ which, when derived and normalized, gives back the sequence $w$. For how we defined derivation and normalization, the derived and normalized sequence of the cutting sequence of a trajectory in $\M{m}{n}$, is the cutting sequence of a trajectory in $\M{n}{m}$.  Generations $\gen i mn $  will act in the same way: applying them to a cutting sequence of a trajectory in the standard sector on $\M{n}{m}$ will give a cutting sequence in $\M{m}{n}$.
 
We will first define operators $\mathfrak g_i^0$, for $1\leq i \leq m-1$ (which invert $\Der m n $), and then to use them to define  $\gen i n m$ (which inverts $\Norm nm \Der m n $). The operator $\mathfrak g_i^0$ applied to the sequence $w$ of a trajectory in $\Sec j mn$ in $\M{m}{n}$, will produce a sequence $W=\mathfrak g_i^j w$ admissible in transition diagram $\T i n m$, and such that $\Der m n W=w$. 

As usual let us first start with defining generations for the $\M{3}{4}$ case.
First we will define $\mathfrak g_i^0$.
Such an operator, applied to the cutting sequence of a trajectory in $\M{3}{4}$ admissible in the diagram $\T i 3 4$, gives a sequence admissible in diagram $\T 0 4 3$, for trajectories in $\M{4}{3}$.
In the proof of Proposition \ref{inverse}, we will explain how to construct the diagram from which we deduce such an operator. 
For the general case, the following definition will remain the same, but the diagrams in Figures \ref{gendiagram-43} and \ref{gendiagram-34} will be obviously different, find in the way described in the proof of the Proposition.

\begin{definition}
Let $w$ be a sequence admissible in diagram $\T k n m$. 
Then $\mathfrak g_k^0 w$ is the sequence obtained by following the path defined by $w$ in $\T k n m$, interpolating the elements of $w$ with the labels on the arrows of a diagram analogous to the ones in Figure \ref{gendiagram-43} (which we will call \emph{generation diagrams} and denote by $\GD k n m$), and dropping the previous ones.

For example, if our sequence contains $w= \dots n_1 n_2 n_3 \dots$, and the arrow from  $n_1$ to $n_2$ in diagram $\GD k n m$ has the label $w_{n_1 n_2}^k$, while the arrow  from  $n_2$ to $n_3$ in  $\GD k n m$ has the label $w_{n_2 n_3}^k$, then $\mathfrak g_k^0 w= \dots w^k_{n_1 n_2} w^k_{n_2 n_3}  \dots$.
\end{definition}

\begin{figure}[!h] 
\centering 
\includegraphics[width=280pt]{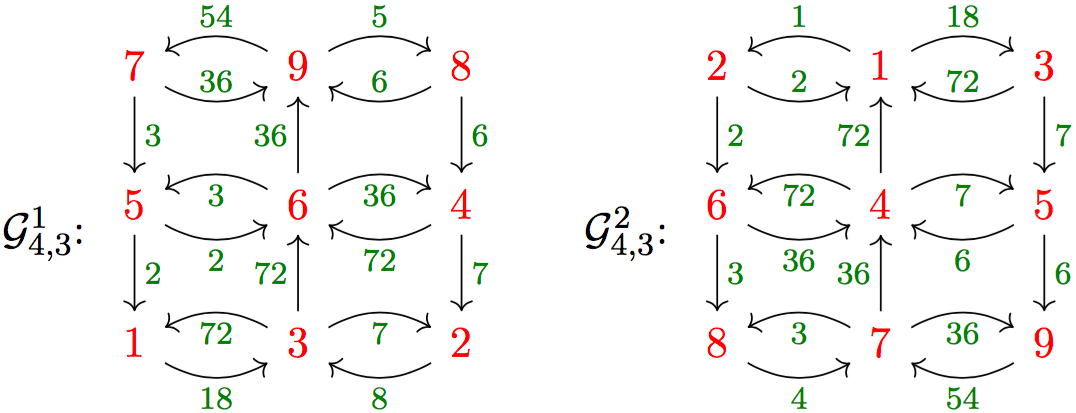}
\begin{quote}
\caption{Generation diagrams describing the operator $\mathfrak g_i^0$ for $\M 43$  \label{gendiagram-43}} \end{quote}
\end{figure}

\begin{figure}[!h] 
\centering 
\includegraphics[width=0.8\textwidth]{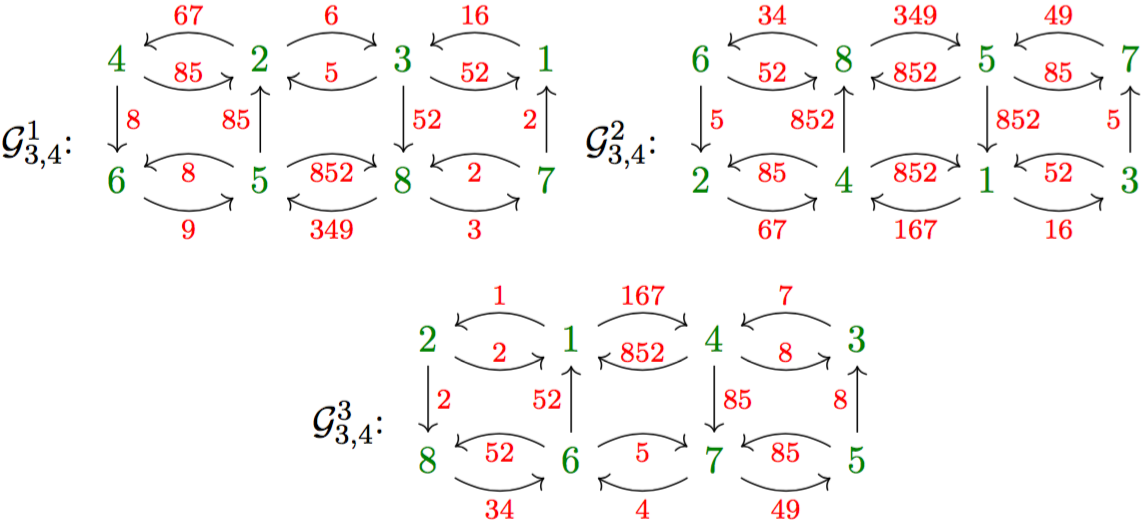}
\begin{quote}
\caption{Generation diagrams describing the generation operator $\mathfrak g_i^0$ for $\M 34$} \label{gendiagram-34} \end{quote}
\end{figure}

We then define the generation operator as follows:

\begin{definition} \label{genoperator}
The generation operator $\gen i n m$ is defined by
\[
\gen i n m w= \mathfrak g_i^0 (\perm i n m)^{-1} w,
\]
where $w$ is a sequence admissible in $\T 0 n m$ and $\perm i n m$ is the $i^{\text{th}}$ isometry permutation in $\M n m$.
\end{definition}

As we said, this operator inverts the derivation and normalization operation on sequences. More specifically, we have the following:

\begin{proposition}[Generation as an inverse to derivation.]\label{inverse}
Let $w$ be a sequence admissible in diagram $\T 0 n m$. 
Then for every $1\leq i \leq n-1$, the sequence $W=\gen i n m w$ is admissible in diagram $\T 0 m n$ and satisfies the equation $\Norm nm \Der m n W=w$. Moreover, the derivative $\Der mn W$ (before normalization) is a sequence admissible in diagram $\T i m n$.
\end{proposition}

In order to prove the Proposition, the following Lemma will be crucial. 
The proof of the Proposition relies in fact on the idea that the diagrams in Figures \ref{gendiagram-43} and \ref{gendiagram-34} are constructed exactly in such a way to invert the derivation operation. 
This Lemma is useful exactly in this sense and is true for generic \bm surfaces.

\begin{lemma} \label{uniquepath}
Let us consider the derivation diagram for $\M{m}{n}$, 
 as in Figure \ref{hextooct}. 
{ Given two green edge labels on arrows of the derivation diagram, if there is a path from one to the other that follows the arrows on the diagram without crossing another green edge label in between, then the path is unique. }
In other words, we cannot always go from a green edge label to an other green edge label following a path satisfying such conditions, but if it is possible, then there exists only one such path.
\end{lemma}

\begin{proof}
The condition of not passing through another green edge label implies that we can move either in the same column, upwards or downwards, or on the next one, on the left or on the right, because we have a green edge label on each horizontal arrow. 
Starting from a green edge label, unless we are on a boundary edge, we have the first choice to do. 
We will have the choice of which of the two arrows carrying that edge label to follow.
The choice will be related to the two different cases of moving upwards or downwards if we are going to the same column, or if we are moving left or right if we are changing the column.

Let us now assume that we want to reach an edge label on the same column. 
For the structure of the derivation diagram, we know that if we follow one of the arrows we will get to a red vertex whose column has arrows going upwards, while if we choose the other one we will get to arrows going downwards.
According to whether the green edge label we want to reach is higher or lower with respect to the starting one, we will choose which way to go. 
Clearly, in the opposite case, we will be restricted to go on the wrong side and we will never reach the targeted green edge label. 
At that point, we follow arrows upwards or downwards until reaching the level of the green edge label where we want to arrive.
In fact, trying to move again to try to change column would make us cross a new green edge label, which we can afford to do only once we reach the level of the edge label we want.
After stopping on the right red vertex we have half more arrow to move back to the column of the green edge label, reaching the one we were targeting. 
This is obviously possible, because the horizontal arrows are always double. 

On the contrary, we consider now that we want to reach a green edge label on a column adjacent to the previous one.
In this case, the first choice of the edge to follow for the first half arrow depends on whether the adjacent column is the right or the left one. 
Clearly, going in the other direction would make it impossible to reach the green edge label we want. 
As before, at that point, we can only follow the arrows going upwards or downwards, according to the parity of the column.
As we said, we are assuming that there is a path connecting the two edge labels. 
In fact,  if for example the arrows are going downwards and the edge label is on a higher row, then such a path  does not exist, but this is a case we are not considering.

From what we said, it is clear that at each step the choice made is the only possible one to reach the targeted edge label.
\end{proof}
The path that was found in the proof of the Lemma \ref{uniquepath} will be used again later in the proof of Proposition \ref{inverse}.

We also prove the following Lemma, which will be used later in the proof of Theorem \ref{thm:substitutionscharacterization}. 

\begin{lemma}\label{lemma:uniqueprecedent}
For any vertex $n_1$ of any generation diagram $\GD i mn$, the labels of all the arrows of $\GD i mn$ which end in vertex $ n_1$ end with the same edge label of the alphabet $\LL nm$. 
\end{lemma}
The proof of the Lemma is given below. For example, in Figure \ref{gendiagram-43}, one can see that the three arrows which end in {\rd 9} carry the labels {\gr 6} and {\gr 36}, which all end with {\gr 6}. In this case one can verify by inspection of $\GD i 43$ that the same is true for any other vertex.
\begin{definition}[Unique precedent]
For any ${ n_1} \in \LL mn$, the unique edge label ${ n_1} \in \LL nm$ given by Lemma \ref{lemma:uniqueprecedent} (i.e. the edge label of $\LL nm$ with which all labels of arrows ending at vertex ${ n_1}$ ends) will be called the \emph{unique precedent} of ${n_1}$. 
\end{definition}
\begin{proof}[Proof of Lemma \ref{lemma:uniqueprecedent}]
The proof uses the stairs configuration introduced in \S \ref{stairsandhats}. 
Let us first recall that (by definition of generation diagrams and  Proposition \ref{inverse}) given a path in $\GD i mn$, the generated sequences obtained by reading off the labeles of arrows  of $\GD i mn$ along the path are by construction admissible sequences in $\T 0 nm$ that, derived, give the sequence of labels of vertices crossed by the  path. Furthermore, each label of an arrow on $\GD i mn$ is a cutting sequence of a piece of a trajectory in the standard sector $\Sec 0 nm$ that crosses the sequence of sides of $\M nm$ described by the label string. This is because, when following on $\GD i mn$ a path coming from a cutting sequence, we produce a cutting sequence, with the labels of the arrows crossed as subsequences.
Such a label string will hence be part of a cutting sequence in sector $\Sec 0 nm$. 
The label of the incoming vertex is an edge label of the flip and sheared copy of $\M mn$ that is hit next by the same trajectory. If we apply a shear to pass to the orthogonal presentation, we are considering trajectories with slope in the first quadrant, and the labels of an arrow describe the sequence of {\rd negative} diagonals of basic rectangles hit (see for example Figure \ref{coincide}), while the vertex label is the label of a {\gr positive} diagonal. 

Without loss of generality,  we can  assume that the edge label of $\LL mn$ that we are considering is the label of the {\gr positive} diagonal $\gr b$ in the stair configuration in Figure \ref{stair}, since   
recalling Convention \ref{convention:ordered_sides}, vertical or horizontal sides can be considered as degenerated diagonals in a degenerated stair (corresponding to a degenerated hat in the augmented Hooper diagram). 
One can then see from Figure \ref{stair} that any  trajectories with slope in the first quadrant which hit the {\gr positive} diagonal labeled by  {\gr b} in Figure \ref{stair}, last hit the  {\rd negative} diagonal labeled by  {\rd a}. This hence shows that all labels of arrows in  $\GD i mn$ which end in the vertex corresponding to the side {\gr b} end with the edge label $\rd a$ of $\LL nm$.
\end{proof}

We are now ready to prove Proposition \ref{inverse} and at the same time explain how to construct in general the generation diagrams for the operator $\mathfrak g_i^0$.

\begin{proof}[Proof of Proposition \ref{inverse}] 
As we explained in \S \ref{derivation}, the operation of derivation consists of taking a cutting sequence in $\M{m}{n}$ and interpolating pairs of edge labels with new ones, then dropping the previous ones.
In this way we get a cutting sequence in $\M{n}{m}$. 
To invert it, given a cutting sequence in $\M{n}{m}$, we want to recover the previous edge labels to appear in the new ones. 
As we saw, derivation might insert or not an edge label between two original ones, and if it does, it is only one. 
This implies that generation will add edge labels (one or a string) between each and every pair of edge labels of the new sequence. 

For clarity, we first explain how to recover the edge labels to interpolate through the example of $\M{3}{4}$. The  proof for general $(m,n)$ follows verbatim the proof in this special case.
In \S \ref{derivation}, we started from a sequence in the standard sector of $\M{4}{3}$, colored in red in the figures, and got a sequence in $\M 3 4$, colored in green in the figures. 
The method consisted in interpolating the red edge labels with the green ones, following the diagram in figure \ref{hextooct} (see also Figure \ref{34auxtd} in \S \ref{transitiondiagrams}).
 
\begin{figure}[!h]
\centering
\includegraphics[width=100pt]{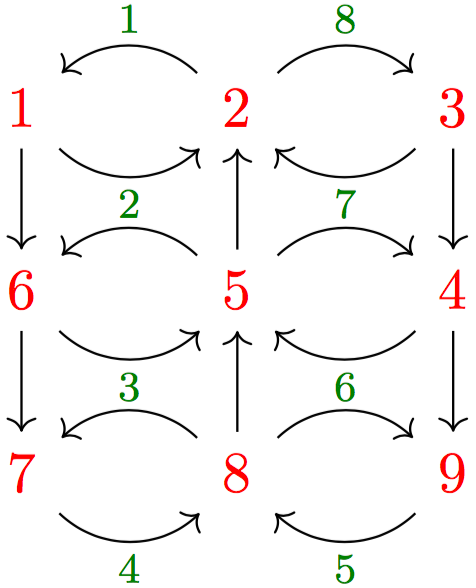}
\begin{quote}\caption{The derivation diagram for $\M{4}{3}$. \label{hextooct}} \end{quote}
\end{figure}

Let us now assume that we have a sequence $w$ in green edge labels.
It will be a path in one of the transition diagrams $\T i 3 4$.
Since we saw that a sequence in the standard sector gives a sequence admissible in one of the other ones for the other surface, $i$ will be between 1 and $n-1$, so here $i=1,2,3$. 

Let us define $k$ such that $w$ is admissible in $\T k 3 4$.
Each pair of green edge labels is hence a transition in $\T k 3 4$.
For each of these pairs, we want to recover which path in the diagram in Figure \ref{hextooct} (i. e. from which cutting sequence in red edge labels) it can come from. 
This means that we have two green edge labels and we want to find a path leading from one to the other through edges and vertices of our derivation diagram for $\M{4}{3}$.
Since we are considering a transition in $w$, we want a path which does not intersect other green edge labels in the middle, or we would have the corresponding transition instead. 
A path connecting two green edge labels admitted in $\T k 3 4$ will always exist, because the derivation opens the standard sectors surjectively on all the others.
These are exactly the hypotheses of  Lemma \ref{uniquepath}, so we can find a unique such path. 
We then record the red edge labels crossed by such a path on the arrow in $\T k 3 4$ corresponding to the transition that we are considering. 

Such diagrams with labels are called $\GD k nm$, and in the case $\M{3}{4}$ this procedure gives the diagrams in Figure \ref{gendiagram-34}.
By construction, each of these strings that we add on the arrows represents the unique string the transition in $w$ can come from. 
Hence, it creates an operator that inverts derivation.

The same procedure can be applied to a generic \bm surface, as we saw that in all cases the two transition diagrams for the $(m,n)$ and $(n,m)$ surfaces are combined together forming the derivation diagram we described in  \S \ref{transitiondiagrams}.
\end{proof}

%
%
%
%
%
%
%

\subsection{Characterization via generation operators\label{sec:generation_characterization}}
The following theorem gives a characterization of the closure of the set of cutting sequences. 

\begin{theorem}[Characterization of \bm cutting sequences via generation]\label{thm:generation_characterization}
A word $w$ is in the closure of the set of cutting sequences of bi-infinite linear trajectories on $\M mn$  if and only if there exists $0\leq b_0 \leq 2n-1$ 
 and two sequences $(a_k)_k \in \{1, \dots , m-1\}^\mathbb{N}$ and $(b_k)_k \in \{1, \dots , n-1\}^\mathbb{N}$ 
such that
\be \label{infiniteintersection}
w \in \mathscr{G}\left(b_0, (a_1,b_1)\dots,(a_k,b_k)  \right) := \bigcap_{k\in\mathbb{N}} (\perm {b_0}{m} n)^{-1}(\gen {a_1} n m  \gen {b_1} mn)( \gen {a_2} n m  \gen {b_2}  mn )  \ldots ( \gen {a_k} n m  \gen {b_k} m n) \textrm{Ad}_{m,n},
\ee
where $\textrm{Ad}_{m,n}$ denotes  the set of words in $\LL mn ^\mathbb{Z} $ which are admissible in  $\T 0m n $.
 
Thus, a word $w$ belongs to the  closure of the set of cutting sequences if and only if 
\be \label{unionfiniteintersection}
w \in \bigcup_{0 \leq b_0 \leq 2n-1} \, \, \bigcap_{k\in\mathbb{N}} \, \, \bigcup_{\substack{ 1\leq a_k \leq m-1 \\ 1\leq b_k  \leq n-1} }   \mathscr{G}\left(b_0,(a_1,b_1)\dots,(a_k,b_k)  \right). 
\ee
\end{theorem}
\begin{remark}
As we will show in the proof, the sequences $(a_k)_k \in \{1, \dots , n-1\}^\mathbb{N}$ and $(b_k)_k \in \{1, \dots , m-1\}^\mathbb{N}$  in Theorem \ref{thm:generation_characterization} will be given by the itinerary of the direction $\theta$ of the trajectory of which $w$ is cutting sequence under the \bm Farey map $\FF mn$.
\end{remark}

\begin{proof}[Proof of \ref{thm:generation_characterization}]
Let  us denote by $I \subset {\LL mn}^{\mathbb{Z}}$  the union of intersections in \eqref{unionfiniteintersection}, by  $CS$   the set of cutting sequences of bi-infinite linear trajectories on $\M mn$ and by  $\overline{CS}$ be its closure in ${\LL mn}^{\mathbb{Z}}$.     
In order to show that $\overline{CS} = \I$, one has to show that $CS \subset \I$, that $\I$ is closed and that $CS$ is dense in  $\I$.

\smallskip
{\it Step 1 ($CS \subset \I$)} 
Let $w$ be the cutting sequence of a trajectory $\tau$ in direction $\theta$. Let $b_0$ be such that $\theta \in \Sec {b_0}{m}{n} $ and let $(a_k)_k $ and $(b_k)_k $ be such that  $\left( (a_k, b_k) \right)_k$ is the itinerary of $\theta_0:=\refl {b_0} m n [\theta] \in \Sec 0 mn$ under $\FF mn$ (where $\refl {b_0} m n [\theta] $ denotes the action of $\refl {b_0} m n $ on directions, see the notation introduced in \S \ref{sec:projective}). 

 Let  $(w^k)_k$ be the sequence of derivatives, see Definition \ref{def:derivatives}. From Propostion \ref{thm:derivable}, it follows that $w^k$ is the cutting sequence of a trajectory $\tau^k$. Furtheremore,   the sequence $(\tau^k)_k$ is obtained by the following recursive definition (which gives the geometric counterpart of the renormalization process on cutting sequeneces obtained by alternatively deriving and normalizing):
 \be \label{trajectories_rec}
\tau^0:=\tau, \qquad \tau^{k+1} := \begin{cases}  \AD m n (\refl {b_j} mn)  \tau^k , & k=2j \text{ even} ; \\  \AD nm (\refl {a_j} nm) \tau^k, & k=2j-1 \text{ odd} .  \end{cases}
\ee
 The direction of the trajectory $\tau^k$ belongs to $\Sec {a_j}{n}{m}$ for $k=2j-1$ odd and to $\Sec {b_j}{n}{m}$ for $k=2j$ even, as shown in the proof of Proposition \ref{prop:itineraries_vs_sectors}. 

Let  $(u_k)_k$ be the sequence of \emph{normalized derivatives}, given by
\bes
u_k := \begin{cases}  \perm {a_j} mn w_k , & k=2j-1 \, \text{ odd} ; \\  \perm {b_j} mn w_k, & k=2j \, \text{ even} . 
\end{cases}
\ees 
 Remark that when $w$ is non-periodic, this could be simply written as  $u_k := \Norm mn w_k$ or $u_k := \Norm nm w_k$ according to the parity of $k$, but for periodic sequences the operators $\Norm mn$ and $\Norm nm$ are a priori not defined (since a derivative could possibly be admissible in more than one sector), so we are using the knowledge of the direction of the associated trajectory to define normalizations. 

 We will  then show that for any $k\geq 0$:  
 \be \label{inductionass}
w= (\perm {b_0} m n)^{-1} (\gen {a_1} n m  \gen {b_1} mn)( \gen {a_2} n m  \gen {b_2}  mn )  \ldots ( \gen {a_k} n m  \gen {b_k} m n) u_{2k}.
\ee
This will show  that $w$ belongs to the intersections \eqref{infiniteintersection} and hence that $CS \subset \I$.

First let us remark that by replacing $w$ with $\Norm mn w= \perm {b_0} m n w$ we can assume without loss of generality that $b_0=0$, so $\perm {b_0} m n$ is the identity. Notice also that by Proposition \ref{inverse} $w^k$ is the cutting sequence of a trajectory $\tau^k$ whose direction, by definition of the Farey map and its itinerary,   is in $\Sec {a_j}{n}{ m}$ for $k=2j-1$ odd and in $\Sec {b_j}{ m}{ n}$ for $k=2j$ even. Thus, if $k=2j-1$ is odd (respectively $k=2j$ is even), $u^{k}= \Norm nm  \Der m n u^{k-1}$ (respectively $ u^{k}=\Norm mn  \Der nm  u^k$)  and $w^{k}$ is the cutting sequence of a trajectory in sector $\Sec {a_j} nm$ (respectively $\Sec {b_j} mn$. By Proposition \ref{inverse}, $u^{k-1}$ is hence equal to $ \gen {a_j} n m u^{k}$ (respectively $ \gen {b_j} m n u^{k-1}$). Thus, if by the inductive assumption  we have \eqref{inductionass} for $k-1$, we can write $u_{2(k-1)} = \gen {a_k} n m  \gen {b_k} m n) u^{2k}$ and get \eqref{inductionass} for $k$.   This  concludes the proof of this step.

\smallskip
{\it Step 2 ($\I$ is closed)}
 $\I$ is given by \eqref{unionfiniteintersection} as a union of  countable intersections of finite unions.   
 Since the set $\mathrm{Ad}_{m,n}$ of admissible words in $\T 0 m n$  is a subshift of finite type, $\mathrm{Ad}_{m,n}$  is closed (see for example Chapter $6$  of \cite{LM:sym}). Moreover, one can check that the composition $\gen {i}{n}{m} \gen {j}{m}{n}$ is an operator from $\LL m n^\mathbb{Z}$ back to itself which is Lipschiz, since if $u, v \in \mathrm{Ad}_{m,n}$  have a common subword, the interpolated words $\gen {i}{n}{m} \gen {j}{m}{n}{u}$ and  $\gen {i}{n}{m} \gen {j}{m}{n}{v}$ have an even longer common subword. Thus, the sets $\mathscr{G}\left((a_1,b_1)\dots,(a_k,b_k)  \right)$ in (\ref{unionfiniteintersection}) are closed, since they are the image of a closed set under a continuous map from the compact space ${\LL mn}^{\mathbb{Z}}$.  
  Since  in (\ref{unionfiniteintersection}), for each  $k$, one considers a finite union of closed sets, 
  $\I$  is a finite union of countable intersection of closed sets  and thus it is closed.  

\smallskip
{\it Step 3 ($CS$ is dense in  $\I$)}   
  By  the definition of topology on ${\LL mn}^\mathbb{Z}$ (see for example \cite{LM:sym}), to show that cutting sequences are dense in \eqref{unionfiniteintersection}, it is enough to show that each arbitrarily long finite subword $u$ of a word $w$  in the intersection  \eqref{infiniteintersection} is contained in a bi-infinite cutting sequence  of a trajectory on $\M mn$.

Let $v$ be such a finite subword  and let $b_0$ and $\left((a_k, b_k)\right)_k$ be the integer and sequences, respectively, that appear in the expression \eqref{infiniteintersection}. 
Let $(w^k)_k$ be the sequence of derivatives given by Definition \ref{def:derivatives} and let $(v^k)_k$ be the subwords (possibly empty) which are images of $v$ in $w^k$ (using the terminology introduced at the very beginning of \S~\ref{sec:fixedpoints}). 
Recall that  the operator $\Der nm  \Norm nm \Der mn  \Norm mn$ either strictly decreases or does not increase the length of finite subwords (see Remark \ref{rk:derivation_short}).  
Thus, either there exists a minimal $\overline{k}$ such that $v^{{\overline{k}+1}}$ is empty (let us call this situation Case $(i)$), or there exists a minimal $\overline{k}$ such that  $v^{\overline{k}}$ has the same length as $v^{{k}}$ for all $k\geq \overline{k}$  (Case $(ii)$).

Let us show that in both cases  $v^{\overline{k}}$ is a subword of the cutting sequence of some periodic trajectory $\tau^{\overline{k}}$. 
In Case $(i)$, 
let $n_1$ (respectively $n_2$) be the  last (respectively the first) edge label of $w^{\overline{k}}$ which survives in $w^{\overline{k}+1} $ before (respectively after) the occurrence of the subword $ { u^{\overline{k}}}$. Thus, since $u^{\overline{k}+1}$ is the empty word by definition of $\overline{k}$,  $n_1n_2$ is a transition in $w^{\overline{k}+1}$. 
By definition of a transition, we can hence find a trajectory ${\tau}^{\overline{k}+1} $ which contains the transition $n_1n_2$ in its cutting sequence.  
If we set $\tau^{\overline{k}}$ to be equal to $(\refl {a_j} nm)^{-1} (\AD nm )^{-1}  {\tau}^{\overline{k}+1} $, 
 if $\overline{k}=2j-1$ is odd  (respectively  $(\refl {b_j} mn)^{-1} (\AD mn )^{-1}   {\tau}^{\overline{k}+1} $ if $k=2j$ is even), the cutting sequence of $\tau^{\overline{k}}$  contains the block $ { v^{\overline{k}}}$ in its cutting sequence. 

In Case $(ii)$, note that since  $v^{\overline{k}}$ has the same length as $v^{\overline{k}+2}$, by Lemma \ref{periodicABAB} it must be a finite subword of the infinite periodic word $\dots n_1n_2n_1n_2 \dots$ for some edge labels $n_1,n_2$. 
Now, by Lemma \ref{realizecs}, all  words of this type are cutting sequences of periodic trajectories, so there exists a periodic trajectory  ${\tau}^{\overline{k}} $  which contains $v^{\overline{k}}$ in its cutting sequence. 



Finally, once we have found a trajectory  $\tau^{\overline{k}}$  which contains $v^{\overline{k}}$ in its cutting sequence, we will reconstruct a trajectory $\tau$  which contains $v$ in its cutting sequence
 by applying  in reverse order the steps which invert derivation at the combinatorial level (i.e. the generations given by the knowledge of the sequences of admissible sectors) on cutting sequences, and at the same time applying the corresponding affine diffeomorphisms on trajectories. 
More precisely, we can define by recursion trajectories $\tau^{{k}}$ which contain $v^{{k}}$ in their cutting sequence for $k= \overline{k}-1, \overline{k}-2, \dots, 1 ,0$ as follows.  
Let us make the  inductive assumption that  $v^{k}$ is contained in the cutting sequence of $\tau^k$. 
Let us denote by $\overline{w}^k$  the cutting sequence of the normalized trajectory $\overline{\tau}^k$ and by $\overline{v}^k$ the block in $\overline{w}^k$  which corresponds to ${v}^k$ in $w^k$. By definition of itinerary and by Proposition \ref{inverse}, we then have that  
$u^{k-1} = \gen {a_j}{n}{m} u^{k}$ for $k=2j-1$ odd or $u^{k-1} = \gen {b_j} mn w^{k}$ for $k=2j$ even. 
 Thus, setting $\overline{\tau}^{k-1}$ to be equal to ${\refl{a_j}{n}{m}}^{-1} {\AD mn}^{-1}  \overline{\tau}^{k}$ or $ {\AD nm}^{-1} {\refl{b_j}{m}{n}}^{-1} \overline{\tau}^{k}$ respectively, we have that by Proposition \ref{inverse} the derived sequence $\overline{w}^{k}$ contains $\overline{v}^k$. Thus, if we  set $\tau^{k-1} $ to be respectively 
${\refl{b_{j-1}}{m}{n}} \overline{\tau}^{k-1}$ or ${\refl{a_{j-1}}{n}{m}} \overline{\tau}^{k-1}$, $\tau^{k-1}$ has a cutting sequence which  contains $v^{k-1}$.  

Continuing this recursion for $\overline{k}$ steps, we finally obtain a trajectory $\tau^0$ which contains the finite subword $v$. This concludes the proof that cutting sequences are dense in $\I$.
\end{proof}

\subsection{An $\mathcal{S}$-adic characterization via substitutions}\label{sec:substitutionscharacterization}
In this section we present an alternative characterization using the more familiar language of substitutions. This will be obtained by starting from the characterization via generations (Theorem \ref{thm:generation_characterization}) in the previous section \S\ref{sec:generation_characterization},  and showing that generations can be converted to substitutions on a different alphabet corresponding to \emph{arrows} (or transitions) in transition diagrams. Let us first recall the formal definition of a substitution.
 
\begin{definition}[Substitution]\label{def:substitution}
A \emph{substitution} $\sigma $ on the alphabet $\mathcal{A}$ is a map that sends each symbol in the alphabet to a finite word  in the same alphabet, then extended to act on $\mathcal{A} ^{\mathbb{Z}}$  by juxtaposition, so that if for $a\in\mathcal{A}$ we have $\sigma (a) = w_a $ where $w_a$ are finite words in $\mathcal{A}$, then for $w= (a_i)^\mathbb{Z} \in \{0,1\}^ \mathbb{Z}$ we have that $\sigma(\cdots a_{-1}  a_0  a_1  \cdots) = \cdots w_{a_{-1}}  w_{a_0}  w_{a_1}  \cdots$. 
\end{definition}

Let us now define a new alphabet $\Ar mn$, which we will use to label \emph{arrows} of a transition diagram of $\M mn$. The cardinality of the alphabet $\Ar mn$ is $N_{m,n}:=\NA mn $ since this is the number of arrows in the diagrams $\T i mn$. Recall that from each vertex in  $\T i mn$ there is at most one outgoing vertical arrow, for a total of $n(m-2)$ vertical arrows. On the other hand,  there can be two outgoing horizontal arrows, going one right and one left, for a total of $2(m-1)(n-2)$ horizontal arrows.  Hence, we will use as edge labels $v_i$, 
$l_i, r_i$ where $v, l, r$ will stays respectively for  \emph{vertical}, \emph{left} and  \emph{right} and the index $i$ runs from $1$ to the number of arrows in each group, i.e. 
\bes
\Ar mn =  \{ v_i, 1\leq i \leq n(m-2)\} \cup \{ r_i, 1\leq i \leq (m-1)(n-2)\} \cup  \{ l_i, 1\leq i \leq (m-1)(n-2)\}.
\ees


We label the \emph{arrows} of the universal diagram $\UD  m n$ in a \emph{snaking pattern} starting from the upper left corner, as shown in Figure \ref{fig:arrows_names_ex}  for $\M 43 $ and $\M 34$,
 where the labels of the alphabet $\Ar 43$ are all in red (since they represent transitions between the red vertices), while the labels of $\Ar 43$ are all in green. In particular for vertical arrows $v_i$, $v_1$ is the vertical arrow from the top left vertex, then $i$ increases by going down on odd columns and up on even ones; right arrows $r_i$ are numbered so that $r_1$  is also exiting the top left vertex and $i$ always increases going from left to right in each row; finally left arrows $l_i$ are numbered so that $l_1$  exits the top right vertex and $i$ always increases going from right to left in each row. 

This labeling of $\UD  m n$  induces a labeling of arrows on each $\T i m n$ for $0\leq i \leq n-1$, where all arrows are labeled in the same way in each diagram.



\begin{figure}[!h] 
\centering
\includegraphics[width=350pt]{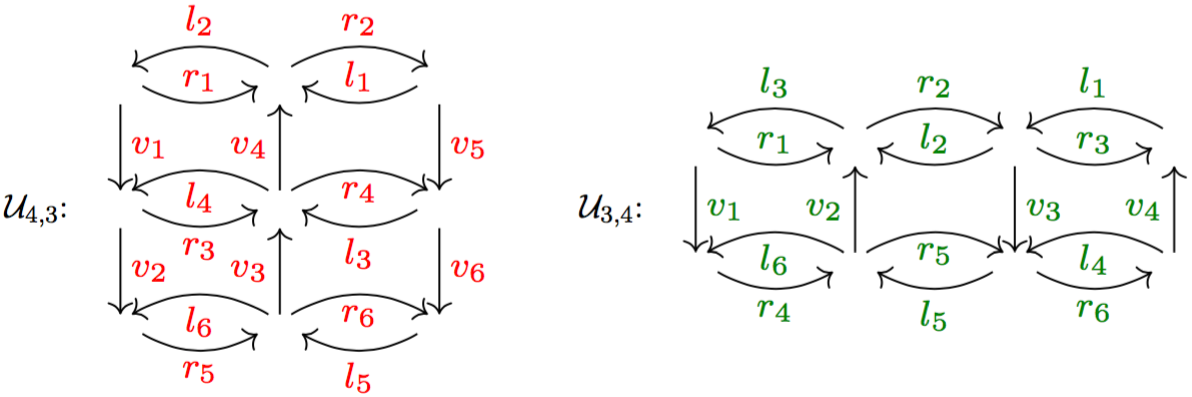}
\begin{quote}
\caption{The labeling of $\UD  m n$ with the labels of $\Ar mn$ for \mbox{$m=4,n=3$} and for \mbox{$m=3,n=4$}   \label{fig:arrows_names_ex}} \end{quote}
\end{figure}


Let us call \emph{admissible} the words in the alphabet $\Ar mn$  that correspond to  paths of arrows  in a transition diagram (in a similar way to Definition \ref{admissibledef}). 
\begin{definition}\label{def:admissible} Let us say that the word  $a$ in ${\Ar mn}^{\mathbb{Z}}$ is \emph{admissible} 
if it describes an infinite path on $\UD  m n$, i.e. all pairs of consecutive labels $a_i a_{i+1}$ are such that $a_i$ labels an arrow that ends in a vertex in which the arrow labeled by $a_{i+1}$ starts.
\end{definition}

Let us also define an operator
$\Tr i m n$ which,  for $0 \leq i \leq 2n-1$,  allows us to convert admissible words in $\Ar mn^{\mathbb{Z}}$ to words in $\LL mn^{\mathbb{Z}}$ that are admissible in diagram $\T i m n$. 

\begin{definition}\label{def:Tr}
The operator $\Tr 0 m n $ sends an \emph{admissible} sequence $(a_k)_k$ in $ \Ar mn^{\mathbb{Z}}$ to $(w_k)_k$ the sequence  in $\LL mn^{\mathbb{Z}}$ admissible in $\T 0 mn $ obtained by reading off the names of the vertices of a path in $\T 0 mn $ which goes through all the arrows $\dots a_{-1}, a_0, a_1, \dots$.

The operators $\Tr  i m n $ for $0\leq i \leq 2n-1$ are obtained by composing $\Tr 0 m n $ with the action on $\LL mn$ of $\perm i mn$, so that $\Tr  i m n := \perm i mn \circ \Tr 0 m n$ maps admissible sequences in $ \Ar mn^{\mathbb{Z}}$ to  the sequences in $\LL mn^{\mathbb{Z}}$ admissible in $\T i mn $.
\end{definition}

\begin{example}\label{ex:Tr}
Let $m=4$ and $n=3$. 
Consider an admissible sequence in $\Ar 43 ^ \mathbb Z$ containing ${\re r_1l_2v_1}$. 
This is possible because the string represents a path in $\UD 43$ (see Figure \ref{fig:arrows_names_ex}). 
Now, to calculate $\Tr 0 mn (\dots {\re r_1l_2v_1} \dots)$, we need to look at $\T 0 43$ (Figure \ref{all-td43}). 
We record the vertices of the path represented by these arrows and we get a word which will contain the subword ${\re 1216}$.
So $\Tr 0 mn (\dots {\rd r_1l_2v_1}\dots )=\dots {\re 1216} \dots $.
\end{example}

\begin{remark}\label{invertibilityTr}
The operator $\Tr  i m n $  is invertible and for $0\leq i \leq n-1$ the inverse $(\Tr  i m n )^{-1}$ maps a sequence $(w_k)_k$ in $\LL mn^{\mathbb{Z}}$ admissible in $\T i mn $ to the admissible sequence $(a_k)_k$ in $\Ar mn^{\mathbb{Z}}$  obtained by reading off the names  $\dots a_{-1}, a_0, a_1, \dots$ of the arrows of a path in $\T i mn $ which goes through all the vertices $\dots w_{-1}, w_0, w_1, \dots$.
\end{remark}

\smallskip
The main result of this section is the following characterization.

\begin{theorem}[An $\mathcal{S}$-adic characterization of \bm cutting sequences.]\label{thm:substitutionscharacterization}
There exist \mbox{$(n-1)(m-1)$} substitutions $\sigma_{i,j}$ for $1\leq i \leq n-1 $ and  $1\leq j \leq m-1 $ on the alphabet $ \Ar mn$ such that the following holds:

The sequence $w$ is the closure of the set of cutting sequences of a bi-infinite linear trajectory on $\M mn$ 
if and only if there exist two sequences $(a_k)_k \in \{1, \dots , n-1\}^\mathbb{N}$ and $ (b_k)_k \in \{1, \dots , m-1\}^\mathbb{N}$ and $0\leq b_0 \leq 2n-1$ such that

\be\label{intersection_substitutions}
w \in \bigcap_{k\in \N} \Tr {b_0}{m}{n} {\Sub{a_1}{b_1}{m}{n}} {\Sub{a_2}{b_2}{m}{n}} \dots {\Sub{a_k}{b_k}{m}{n}} \Ar mn^{\mathbb{Z}}.
\ee

Furthermore,  the sequence $\left((a_k, b_k) \right)_k  $ the {itinerary} of $\theta$ under $\FF m n$. 
\end{theorem}
This gives the desired $\mathcal{S}$-adic characterization, where 
\bes
\mathcal{S}= \mathcal{S}_{m,n}= \{ \sigma_{i,j}, \qquad 1\leq i \leq n-1 , 1\leq i \leq m-1 \}.
\ees
Equivalently, \eqref{intersection_substitutions} can be rephrased by saying that  any sequence in the closure of the set of cutting sequences is obtained as an \emph{inverse limit} of products of the substitutions in $\mathcal{S}_{m,n}$, i.e. there exists a sequence of labels $a_k$ in $\Ar mn$ such that
\begin{equation}\label{inverse_limit}
w=  \lim_{k \to \infty} \Tr {b_0}{m}{n} {\Sub{a_1}{b_1}{m}{n}} {\Sub{a_2}{b_2}{m}{n}} \dots {\Sub{a_k}{b_k}{m}{n}} a_k.
\end{equation}
The above expression is known as \emph{$\mathcal{S}$-adic expansion} of $w$. We refer to \cite{BD} for details. 

The proof of Theorem \ref{intersection_substitutions}, which is presented in the next section \S\ref{proof:substitutions}, essentially consists of rephrasing Theorem \ref{thm:generation_characterization} in the language of substitutions.

As an example of the substitutions which occur, we list one of  the substitutions for $m=4, n=3$ below  (Example \ref{ex:substitution1for43}) and give the other substitutions for $m=4, n=3$ in Example \ref{ex:substitutionsfor43} as composition of the pseudosubstitutions (see Definition \ref{def:pseudosubstitutions} below) in Example \ref{ex:pseudosubstitutionsfor43}.   We explain in the next section how these substitutions were computed (see in particular Example \ref{ex:substitution_howto}).

\begin{example}[Substitutions for $\M 43$]\label{ex:substitution1for43}
The substitution $\Sub 11  43$ for cutting sequences on $\M 43$ is the following:
\begin{align*}
&\Sub 11 43:&
 \Sub 11  43( r_1)&= l_2 v_1  r_3v_4  
&  \Sub 11  43( l_1)&= l_1  
&  \Sub 11  43( v_1)&= l_2 v_1  \\&&
 \Sub 11  43(r_2)&= r_2  
&  \Sub 11  43(l_2)&= r_2 l_1 
&  \Sub 11  43(v_2)&= r_3  \\&&
 \Sub 11  43(r_3)&= r_3  
&  \Sub 11  43(l_3)&= r_2 v_5 v_6l_5 v_3  
&  \Sub 11  43(v_3)&= r_6 l_5 v_3  \\&&
 \Sub 11  43( r_4)&= l_4 r_3 v_4  
&  \Sub 11  43( l_4)&= l_4  
&  \Sub 11  43( v_4)&= l_4 r_3 v_4  \\&&
 \Sub 11  43( r_5)&= l_4 v_2 r_5  
&  \Sub 11  43( l_5)&= l_5  
&  \Sub 11  43( v_5)&= l_1  \\&&
 \Sub 11  43( r_6)&= r_6  
&  \Sub 11  43( l_6)&= r_6 l_5 v_3 
&  \Sub 11  43( v_6)&= r_2 v_5 v_6 
\end{align*}
In Example \ref{ex:substitution1for43} below we explain how the above substitution can be  obtained from the generation rules in the previous section.  The other substitutions for $\M 43$ are given in Example \ref{ex:substitutionsfor43}, see also Example \ref{ex:pseudosubstitutionsfor43}.
%

\subsection{From generations to substitutions}\label{proof:substitutions}
We will now provide the \emph{recipe} of how to translate generation operators (in the alphabet $\LL mn$) into a substitution (in the alphabet $\AA mn$), and in particular to obtain the substitutions in the previous example. This is done in Definition \ref{def:substitutions} and Lemma \ref{lemma:conjugationgensub}.   They constitute the heart of the proof of Theorem \ref{thm:substitutionscharacterization} from Theorem \ref{thm:generation_characterization}, which is presented at the end of this section.
 We begin first with a concrete example, which the definitions below will then formalize.
 


\begin{example}\label{ex:substitution_howto}
Let $m=4$ and $n=3$. Let us explain how to associate to the composition of the two generation operators $\gen 1 34 \circ \gen 1 43 $ a substitution on the arrows alphabet $\Ar 43$. For clarity, we will denote in red the symbols (edge labels) of the alphabet  $\rd \Ar 43$ and in green the ones of $\gr \Ar 34$. 
 Let us first consider the generation diagram $\T 1 43$ used to define $\gen 1 43 $. Start from the arrow labeled by $\rd r_1$ on the universal diagram $\UD 43$, which in this diagram is the arrow from the vertex labeled by $\rd 7$ to the vertex $\rd 9$. The generating word on this arrow in $\GD 1 43$ is the green word $\gr 3 6$. Remark also that all the arrows incoming to the red vertex $\rd 7$ (in this case only one) carry a green word which ends with $\gr 4$ (in this case {\gr 54}), while all the arrows outgoing from the red vertex $\rd 9$ (two) carry a green word which starts with $\gr 5$ ({\gr 5} and {\gr 54}). Thus, the derived sequence of a sequence which contains the transition $\rd 7 9$ contain the word  $\gr 4  3 6 5$. We look now at the transition diagram $\T 0 34$ on page \pageref{all-td34} (the first in Figure \ref{all-td34}) and see that a path which goes through $\gr 4  3 6 5 $ crosses the arrows which are labeled by ${\gr l_1 v_3 r_6}$ in $\UD 34$ (see Figure \ref{fig:arrows_names_ex}). We choose to include in the green path associated to the transition $\rd r_1$ the first arrow but not the last one, which will be included in the green path associated to the following red transition. Thus, we say that the label $\rd r_1$ of  $\rd \Ar 43$ is mapped to  the word ${\gr l_1 v_3 }$ in the alphabet $\gr \Ar 34$. We repeat the same process for every arrow on $\GD 0 43$. This gives a map from edge labels in $\rd \Ar 43$  to words in $\gr \Ar 34$, which can be extended to words in $\rd \Ar 43$ by juxtaposition. We call this type of operator a \emph{pseudo-substitution} (see Definition \ref{def:pseudosubstitutions}), since it acts as a substitution but between two different alphabets. { Note that a pseudo-substitution is sometimes called a \emph{free semi-group morphism.}}


Similarly, we repeat the same process for arrows for the dual \bm surface $\M 34$. For example, for the generation diagram   $\GD 1 34$ used to define $\gen 1 34 $ we see that the arrow labeled by $\gr l_1$ is the arrow from $\gr 1$ to $\gr 3$ and carries the word $\rd 16$. Furthermore, the unique incoming arrow to $\gr 1$ carries the label $\rd 2$. Since the word $\rd 216$ describes  in the diagram $\T 0 43$, describes a path through the arrows labeled by $\rd l_2 v_1$ in $\UD 34$ in Figure \ref{fig:arrows_names_ex}, we  associate to $\gr l_1$ the word $\rd l_2 v_1$.
Reasoning in a similar way, we associate to $\gr v_3$  the word  $\rd r_3v_4$ (given by the path $\rd 652$). Thus, by juxtaposition, the word $\gr l_3 v_4 r_6$ in  $\gr \Ar 43$ maps to the word $\rd l_2 v_1  r_3v_4$ in $\rd \Ar 43$. 

Thus, the composition $\gen 2 34 \circ \gen 1 43 $ sends $\rd r_1$ to $\rd l_2 v_1  r_3v_4$. Thus we can define a standard substitution $\sigma_{1,2}^{43}$ in the alphabet $\rd \Ar 43$, by setting $\sigma_{1,2}^{43}({\rd r_1})= {\rd l_2 v_1  r_3v_4 r_2}$ and similarly for the other labels of $\rd \Ar 43$. This produces the substitutions in the Example \ref{ex:substitution1for43} above.
\end{example}

We will now state formally how to obtain substitutions from generations, thus formalizing the process explained in the Example \ref{ex:substitution_howto} above.  
As we already saw, since each generation operator maps cutting sequences on $\M nm$ to cutting sequences on $\M nm$, in order to get substitutions (in the standard sense of Definition \ref{def:substitution}) we will need to compose \emph{two} generation operators. It is easier though to first describe the substitutions in two steps, each of which correspond to one of the generation operators. 
Since the alphabet $\AA mn$ on which the substitution acts  corresponds to \emph{transitions} in the original alphabet $\LL mn$ and the transitions for $\M nm$ and for $\M mn$, the  intermediate steps will be described by \emph{pseudo-substitutions}, which are like substitutions but act on two different alphabets in departure and arrival:

\begin{definition}\label{def:pseudosubstitution}[Pseudo-substitution]
A \emph{pseudo-substitution} $\sigma $ from alphabet $\mathcal{A} $ to an alphabet $\mathcal{A}'$ is a map that sends each letter $a \in \mathcal{A}$  to a finite word  in $\mathcal{A}'$, then extended to act on $\mathcal{A} ^{\mathbb{Z}}$  by justapposition, so that if $\sigma (a) = w_a $ for some finite words $w_a$  in the letters of $\mathcal{A}'$ as $a \in \mathcal{A}$, then for $w= (a_i)^\mathbb{Z} \in \mathcal{A}^ \mathbb{Z}$ we have that $\sigma(\cdots a_{-1}  a_0  a_1  \cdots) = \cdots w_{a_{-1}}  w_{a_0}  w_{a_1}  \cdots$. 
\end{definition}

\begin{definition}[Pseudo-substitution associated to a generation]\label{def:pseudosubstitutions}
Let  $\PSub i  m n  $ for $\leq i \leq n-1$  
be the pseudo-substitution between the alphabets $\AA mn$ and $\AA nm$  
 defined as follows. 
Assume that the arrow from vertex $j$ to vertex $k$   of $\T i mn $ is labeled by $a$ in $\UD mn$. Let  $w_1 w_2 \dots w_N$ be the finite word  associated to this arrow in the generation diagram ${ \GD i m n}$. Then set
\bes
\PSub i  m n (a) = a_0 a_1 a_2 \dots a_{N-1}, 
\ees
where $a_k$ for $1\leq k \leq N-1$ are the labels in $\AA mn$ of the arrow from $w_i$ to $w_{i+1}$, while $a_0$ is the label of the arrow from the unique edge label in $\LL nm$ which always preceeds $j$ in paths on $\GD i mn$ 
 to $w_1$. 
\end{definition}

\begin{example}\label{ex:pseudosubstitutionsfor43}
Let $m=4$, $n=3$. For $i=1$, as we already saw in the beginning of Example \ref{ex:substitution_howto}, the arrow labeled by $\rd r_1$ on the universal diagram $\UD 43$, which in  the arrow from the vertex labeled by $\rd 7$ to the vertex $\rd 9$ in $\GD 1 43$, is labeled by  the green word $\gr 3 6$ and the labels of arrows incoming to the red vertex $\rd 7$ end with $\gr 4$. Furthermore, the path $\gr 436$ on $\T 0 43$ correspond to the arrows labeled by ${\gr l_1 v_3  }$. Thus we  set $\PSub i 43({\rd r_1})={\gr l_1 v_3 } $. Similarly, the arrow ${\rd r_2}$ in $\GD 1 43$ goes from {\rd 9} to {\rd 8}, is labeled by {\gr 5} and all three arrows which and in {\rd 9} have labels which end with {\gr 6}. Thus, since the path {\gr 65} correspond to the arrow $r_6$ in $\T 0 34$, we set  $\PSub i 43({\rd r_1})={r_6 } $.
Reasoning in the same way and generalizing it to $i=2$, we get the full pseudosubstitutions for $\M 43$ (Figure \ref{pseudo43}).

\begin{figure}[!h] 
\begin{align*}
&\PSub 1 43: 
 &\PSub 1 43 (r_1) &=l_1v_3
 &\PSub 1 43 (l_1) &=l_4
 &\PSub 1 43 (v_1) &=l_1 \\&
 &\PSub 1 43 (r_2) &=r_6
 &\PSub 1 43 (l_2) &=r_6v_4
 &\PSub 1 43 (v_2) &=l_2 \\&
 &\PSub 1 43 (r_3) &=l_2
 &\PSub 1 43 (l_3) &=l_5v_2
 &\PSub 1 43 (v_3) &=r_4v_2 \\&
 &\PSub 1 43 (r_4) &=r_2v_3
 &\PSub 1 43 (l_4) &=r_2
 &\PSub 1 43 (v_4) &=r_2v_3 \\&
 &\PSub 1 43 (r_5) &=l_3v_1
 &\PSub 1 43 (l_5) &=l_6
 &\PSub 1 43 (v_5) &=l_4 \\&
 &\PSub 1 43 (r_6) &=r_4
 &\PSub 1 43 (l_6) &=r_4v_2
 &\PSub 1 43 (v_6) &=l_5 \\
 &&&&&&& \\
&\PSub 2 43: 
 &\PSub 2 43 (r_1) &=r_1
 &\PSub 2 43 (l_1) &=r_4v_2
 &\PSub 2 43 (v_1) &=r_1 \\&
 &\PSub 2 43 (r_2) &=l_3v_1
 &\PSub 2 43 (l_2) &=l_3
 &\PSub 2 43 (v_2) &=r_2 \\&
 &\PSub 2 43 (r_3) &=r_2v_3l_5
 &\PSub 2 43 (l_3) &=r_5
 &\PSub 2 43 (v_3) &=l_1v_3 \\&
 &\PSub 2 43 (r_4) &=l_5
 &\PSub 2 43 (l_4) &=l_5v_2
 &\PSub 2 43 (v_4) &=l_5v_2 \\&
 &\PSub 2 43 (r_5) &=r_3
 &\PSub 2 43 (l_5) &=r_6v_4
 &\PSub 2 43 (v_5) &=r_4 \\&
 &\PSub 2 43 (r_6) &=r_4
 &\PSub 2 43 (l_6) &=l_1
 &\PSub 2 43 (v_6) &=r_5
\end{align*}
\begin{quote}\caption{The pseudosubstitutions for $\M 43$\label{pseudo43}} \end{quote}
\end{figure}

Let now $m=3$ and $n=4$. In the same way, for $i=1,2,3$, we can calculate the pseudosubstitutions for $\M 34$ (Figure \ref{pseudo34}).

\begin{figure}
\begin{align*}
&\PSub 1 34: 
&\PSub 1 34 (r_1)&=r_5v_3
&\PSub 1 34 (l_1)&=l_2v_1
&\PSub 1 34 (v_1)&=r_5 \\&
&\PSub 1 34 (r_2)&=l_4
&\PSub 1 34 (l_2)&=r_3
&\PSub 1 34 (v_2)&=l_5v_3 \\&
&\PSub 1 34 (r_3)&=r_3v_4
&\PSub 1 34 (l_3)&=l_4v_2
&\PSub 1 34 (v_3)&=r_3v_4 \\&
&\PSub 1 34 (r_4)&=r_6
&\PSub 1 34 (l_4)&=l_1
&\PSub 1 34 (v_4)&=l_1 \\&
&\PSub 1 34 (r_5)&=l_5v_3v_4
&\PSub 1 34 (l_5)&=r_2v_5v_6 \\&
&\PSub 1 34 (r_6)&=r_2
&\PSub 1 34 (l_6)&=l_5 \\&
 &&&&&&& \\
&\PSub 2 34: 
&\PSub 2 34 (r_1)&=l_3v_4
&\PSub 2 34 (l_1)&=r_4v_6
&\PSub 2 34 (v_1)&=l_3 \\&
&\PSub 2 34 (r_2)&=r_2v_5v_6
&\PSub 2 34 (l_2)&=l_5v_3v_4
&\PSub 2 34 (v_2)&=r_5v_3v_4 \\&
&\PSub 2 34 (r_3)&=l_5v_3
&\PSub 2 34 (l_3)&=r_2v_5
&\PSub 2 34 (v_3)&=l_5v_3v_4 \\&
&\PSub 2 34 (r_4)&=l_4v_2
&\PSub 2 34 (l_4)&=r_3v_4
&\PSub 2 34 (v_4)&=r_3 \\&
&\PSub 2 34 (r_5)&=r_5v_3v_4
&\PSub 2 34 (l_5)&=l_2v_1v_2 \\&
&\PSub 2 34 (r_6)&=l_2v_1
&\PSub 2 34 (l_6)&=r_5v_3 \\&
 &&&&&&& \\
&\PSub 3 34: 
&\PSub 3 34 (r_1)&=r_1
&\PSub 3 34 (l_1)&=l_6
&\PSub 3 34 (v_1)&=r_1 \\&
&\PSub 3 34 (r_2)&=l_2v_1v_2
&\PSub 3 34 (l_2)&=r_5v_3v_4
&\PSub 3 34 (v_2)&=l_3v_4 \\&
&\PSub 3 34 (r_3)&=r_5
&\PSub 3 34 (l_3)&=l_2
&\PSub 3 34 (v_3)&=r_5v_3 \\&
&\PSub 3 34 (r_4)&=r_2v_5
&\PSub 3 34 (l_4)&=l_5v_3
&\PSub 3 34 (v_4)&=l_5 \\&
&\PSub 3 34 (r_5)&=l_3
&\PSub 3 34 (l_5)&=r_4 \\&
&\PSub 3 34 (r_6)&=r_4v_6
&\PSub 3 34 (l_6)&=l_3v_4 
\end{align*}
\begin{quote}\caption{The pseudosubstitutions for $\M 34$\label{pseudo34}} \end{quote}
\end{figure}
\end{example}
\end{example}


It is easy to check from the definition  that given a pseudo-substitution $\sigma$  between the alphabets $\mathcal{A}$ and $\mathcal{A}'$ and a pseudo-substitution $\tau$  between the alphabets $\mathcal{A}'$ and $\mathcal{A}$, their composition $\tau \circ \sigma$ is a substitution on the alphabet $\mathcal{A}$. Thus the following definition is well posed. 

\begin{definition}[Substitution associated to pair of generation]\label{def:substitutions}
For $\leq i \leq n-1$, $1\leq j \leq m-1$, let  $\Sub i j  m n  $ be the  substitution on the alphabets $\AA mn$ defined by
$$
\Sub i j  m n := \PSub j  n m  \circ \PSub i  m n . 
$$
\end{definition}
\begin{example}\label{ex:substitutionsfor43}
In Example \ref{ex:substitution1for43} we wrote the substitution $\Sub 11 43$ explicitly. The full list of substitutions for $\M 43$ can be produced by composing the pseudosubstitutions in Example \ref{ex:pseudosubstitutionsfor43}, as by Definition \ref{def:substitutions}, i.e. 
\bes
\mathcal{S}_{4,3} =\{ \Sub ij  43 :=  \PSub j  n m  \circ \PSub i  m n,  \qquad \text{for} \ 1\leq i \leq 2, 1\leq j\leq 3\}.
\ees
\end{example}

The following Lemma shows that, up to changing  alphabet from vertices labels to arrows labels as given by the operator $\Tr 0 m n $ and its inverse (see Remark \ref{invertibilityTr}),  the  substitutions $\Sub i k m n$  act as the composition of two generation operators. 
\begin{lemma}[From generations to substitutions]
\label{lemma:conjugationgensub}
 The substitutions $\Sub i j m n$ defined in Definition \ref{def:substitutions}  for any $1\leq i \leq n-1, 1\leq j \leq m-1$ are such that
\be\label{covariance}
\Tr 0 m n \circ \Sub i j m n \circ (\Tr 0 m n )^{-1} = \gen j nm \circ \gen i mn.
\ee
\end{lemma}

Before giving the proof, we show by  an example the action of the two sides of the above formula. 
\begin{example}
Let us verify for example that the formula in Lemma \ref{lemma:conjugationgensub} holds for $m=4,n=3$ and $i=1,j=1$ when applied to a word  $w$ admissible in $\T 0 43$  which contains the transition  $\rd 123$. 

Let us first compute the action of the right hand side of \eqref{covariance}. 
Recall (see Definition \ref{genoperator}) that $\gen 1 43$ is given by first applying  $\perm 1 43$, then $\mathfrak g_1^0$. Since $\perm 1 43$ maps  $\rd 123$ to 
 $\rd 7 9 8$, by looking at the generation diagram $\GD 1 43$ (Figure \ref{gendiagram-43}), we see that $\gen 1 43 w$ will contain the string   $\gr 4365$. Then, to apply  $\gen 1 43$ we first apply $\perm 1 34$, which sends $\gr 4365$ to $\gr 1387$, then look at the generation diagram $\GD 1 34$ (Figure \ref{gendiagram-34}), to see that a path which contains $\gr 1387$ will also contain $\rd 216523$.  
Hence $\gen 1 34  \gen 1 43 w$ will contain $\rd 216523$. 

Let us now compute the action of the left hand side of (\ref{covariance}).
Since the arrow from the vertex $\rd 1$ to $\rd 2$ in $\T 1 43$ is labeled by $\rd r_1$ (Figure \ref{fig:arrows_names_ex}), 
and the one from $\rd 2$ to $\rd 3$ is labeled by $\rd r_2$, the operator $(\Tr 0 m n )^{-1}$ sends $\rd 123$ to $\rd r_1 r_2$. Then, from Example \ref{ex:pseudosubstitutionsfor43} we have that $\Sub 11 43 ({\rd r_1 r_2}) = {l_2v_1r_3v_4r_2 }$. Finally, 
$\Tr 0 m n$ maps this word in $\AA 43$ to $\rd 216523$ (see Example \ref{ex:Tr}). 
Thus, we have verified again that $\Tr 0 43 \circ \Sub 11 43 \circ (\Tr 0 43 )^{-1} (w)$ contains $\rd 216523$. 
\end{example}

\begin{proof}[Proof of Lemma \ref{lemma:conjugationgensub}]
Since $$\Sub i j  m n := \PSub j  n m  \circ \PSub i  m n = \PSub j  n m  \circ  (\Tr 0 m n )^{-1} \circ \Tr 0 m n \circ \PSub i  m n,$$ 
it is enough to show that 
\bes
\Tr 0 n m \circ \PSub i  m n \circ (\Tr 0 m n )^{-1} =  \gen i mn, \qquad \text{for all} \ 1\leq i \leq m-1, \quad \text{for all} \ m,n.
\ees
Consider any sequence $w \in {\LL mn}^\mathbb{Z}$. Let $ n_1 n_2$ and $ n_2 n_3$ be any two pairs of transitions in $\T 0 mn$.  Let $u_1 u_2 \dots u_{N}$ be the label of the arrow from $n_1$ to $n_2$ in $\GD i mn$ and $v_1 v_2 \dots v_{M}$ the one of the arrow 
between $n_2$ and $n_3$.  Then, for every occurrence of the transitions $n_1 n_2 n_3$ transitions  in  $w$ (i.e. every time $w_{p-1}=n_1 , w_p=n_2, w_{p+1} =n_3 $ for some $p \in \mathbb{N}$) gives rise to a block of the form  $ u_1 u_2 \dots u_N v_1 v_2 \dots v_M$ in $\LL nm$ in $\gen i mn(w)$.


Let $a$ be the label in $\AA mn$ of the arrow from $n_1$ to $n_2$ and $b$ be the label of the arrow from $n_2$ to $n_3$. Thus, when we apply $(\Tr 0 m n )^{-1}$ to $w$, each block  $n_1 n_2 n_3$ in $w$ is mapped to the word $ab$. 

Let  $a_i$ for $1\leq i \leq N-1$ be the labels of the arrows from $u_i$ to $u_{i+1}$ and $a_0$ be the label of the arrow from the unique label in $\AA mn$ which preceeds $n_1$ 
 to $u_1$. Thus, by Definition \ref{def:pseudosubstitutions}, we have that $\PSub i  m n (a) = a_0 a_1 a_2 \dots a_{N}$. Now, let $b_i$ for $1\leq i \leq M-1$ be the labels of the arrow from $v_i$ to $v_{i+1}$. Let $b_0$ be the arrow from  $u_{N}$ to $v_1$ and remark that $u_N$ is (by uniqueness) the unique label in $\AA mn$ which preceeds $n_2$. Thus, again by  Definition \ref{def:pseudosubstitutions},  $\PSub i  m n (b) = b_0 b_1 b_2 \dots b_{M-1}$.


 
Thus, we have that $\PSub i  m n (ab) = a_0 a_1 a_2 \dots a_{N} b_0 b_1 \dots b_M$.  Finally, by definition of the operator $\Tr 0 mn$ (recall Definition \ref{def:Tr}) and of the arrows $a_i$ and $b_i$ given above, $\Tr 0 m n \circ \PSub i  m n \circ (\Tr 0 m n )^{-1} (w)$ contains the word $u_1 u_2 \dots u_N v_1 \dots v_M$.  This shows the equality between the two sides of \eqref{covariance}.
\end{proof}

\begin{proof}[Proof of Theorem \ref{thm:substitutionscharacterization}]
By Theorem \ref{thm:generation_characterization}, $w \in {\LL mn}^\mathbb{Z}$ belongs to the closure of cutting sequences on $\M mn$ if and only if  there exists $(a_k)_k, (b_k)_k$ such that it belongs to the intersection \eqref{intersection_substitutions}, i.e.  for every $k$ there exists a word $u^k$  in $\LL mn ^\mathbb{Z} $ which is admissible in  $\T m n 0$ such that
\begin{align*}
w&= (\perm {b_0} m n)^{-1}(\gen {a_1} n m  \gen {b_1} mn) ( \gen {a_2} n m  \gen {b_2}  mn )  \ldots ( \gen {a_k} n m  \gen {b_k} m n) u^k \\
&=(\perm {b_0} m n)^{-1} \Tr 0 m n (\Tr 0 m n)^{-1}  (\gen {a_1} n m  \gen {b_1} mn) \Tr 0 m n (\Tr 0 m n)^{-1} ( \gen {a_2} n m  \gen {b_2}  mn ) \Tr 0 m n \ldots (\Tr 0 m n)^{-1}( \gen {a_k} n m  \gen {b_k} m n) \Tr 0 m n u^k \\
&=  \Tr  {b_0} m n \Sub {a_1} {b_1}  m n  \Sub {a_2} {b_2}  m n \ldots \Sub {a_k} {b_k}  m n \Tr 0 m n u^k ,  
\end{align*}
where in the last line we applied Lemma \ref{lemma:conjugationgensub} and recalled the definition  $\Tr  i m n := (\perm i mn)^{-1} \circ \Tr 0 m n$  of the operators $\Tr i mn$ (see Definition \ref{def:Tr}).  Remarking that $ \Tr 0 m n u^k $ is a sequence in the alphabet $\AA mn$ which is  admissible by definition of $\Tr 0 m n$, this shows that $w$ is in the closure of cutting sequences if and only if it belongs to the intersection \eqref{intersection_substitutions}.    
\end{proof}

\appendix
\section{Renormalization on the Teichm\"uller disk}  \label{teich}
In this section we describe how the renormalization algorithm for cutting sequences and linear trajectories defined in this paper for \bm surfaces can be visualized on the Teichm\"uller disk of $\M mn$. This is analogous to what was described in \cite{SU2} by Smillie and the third author for the analogous renormalization algorithm for the regular octagon and other regular $2n$-gons introduced in \cite{SU}, so we will only give a brief overview and refer to \cite{SU2} for details.

\subsection{The Teichm\"uller disk of a translation surface}\label{Teichdisksec}
The Teichm\"uller disk of a translation surface $S$ can be identified with  a space of marked translation surfaces as follows. 
Let $S$ be a translation surface.   
Using the convention that a map determines its range and domain we can identify a triple with a map and denote it by $[f]$. 
We say two triples  $f:S\to S'$ and $g:S\to S''$ are equivalent if there is a translation equivalence $h:S'\to S''$ such that $g=fh$. 
Let $\tilde{\mathcal{ M}}_A(S)$ be the set of equivalence classes of triples. We call this the set of \emph{marked translation surfaces affinely equivalent to $S$}. There is a canonical basepoint corresponding to the identity map $id:S\to S$. 
We can also consider marked translation surfaces up to \emph{isometry}. We say that two triples $f:S\to S'$ and $g:S\to S''$ are equivalent up to isometry if there is an isometry $h:S'\to S''$ such that $g=fh$.  Let ${\mathcal{\tilde M}}_{I}(S)$ be the collection of isometry classes of triples. 

Let us denote by  $\mathbb{H}$ the upper half plane,
and by $\mathbb{D}$ the unit disk. 
In what follows, we will identify  them by the conformal map $\phi: \mathbb{H} \rightarrow \mathbb{D}$ given by $\phi(z) = \frac{z-i}{z+i}$. One can show that the set ${\mathcal{\tilde M}}_A(S)$ can be canonically identified with $SL_\pm(2,\mathbb{R})$. More precisely, one can map the matrix  $\nu \in SL_\pm(2,\R)$ to the marked triple $\Psi_\nu: S \to \nu S$, where $\Psi_\nu$ is the standard affine deformation of $S$ given by  $\nu $ and show that this map is injective and surjective (see the the proof of Proposition 2.2 in \cite{SU2}).   
The space ${\mathcal{\tilde M}}_{I}(S)$ of marked translation surfaces up to isometry is hence isomorphic to $\mathbb{H}$ (and hence to $\mathbb{D}$) (see Proposition 2.3 in \cite{SU2}).   The hyperbolic plane has a natural \emph{boundary}, which 
can be naturally identified with  the projective space $\RP$.  
The point  $\begin{pmatrix} x_1&x_2\end{pmatrix}$  in $\RP$ is sent 
to the point $e^{i \theta_x} \in \partial \mathbb{D}$ where $\sin  \theta_x = -2x_1x_2 /(x_1^2+x_2^2)$ and $\cos \theta_x= (x_1^2- x_2^2) /(x_1^2+x_2^2)$. 



\subsubsection*{Actions on the Teichm\"uller disk and Teichm\"uller orbifolds.}
The subgroup $SL_\pm(2,\R)\subset GL(2,\R) $ acts naturally on ${\mathcal{\tilde M}}_A(S)$ by the following \emph{left action}. { Given a triple $f:S\to S'$, an element $\eta \in SL_\pm(2,\R)$ maps $[f]$ to $[\eta f]$ where $\eta f :S\to S''$ is obtained by post-composing $f$ with the map $\eta : S' \to S'':= \eta S$ given by the linear action of $\eta$ on the translation surface defined by post-composing the charts of the translation surface with $\eta$}.  Using the identification of ${\mathcal{\tilde M}}_A(S)$ with $SL_\pm(2,\R)$, this action corresponds to left multiplication by $\eta$. One can see that 
this action is simply transitively on ${\mathcal{\tilde M}}_A(S)$. 
There is also a natural \emph{right action} of $\Aff(S)$ on the set of triples. Given an affine automorphism $\Psi:S\to S$ we send $f:S\to S'$ to $f \Psi:S\to S'$. This action induces a right action of $V(S)$ on ${\mathcal{\tilde M}}_A(S)$. Using the identification of ${\mathcal{\tilde M}}_A(S)$ with $SL_\pm(2,\R)$, this action corresponds to right multiplication by $D\Psi$. It follows from the associativity of composition of functions that \emph{these two actions commute}.  

\smallskip
The Veech group acts via isometries with respect to the hyperbolic metric of constant curvature on $\mathbb{H}$. 
This action induces an action of the Veech group on $\RP$ seen as boundary of $\mathbb{H}$ (or $\mathbb{D}$),  
which correspond to the  projective action of $GL(2,\R)$ on row vectors coming from multiplication on the right, namely 
$\begin{pmatrix} z_1 & z_2\end{pmatrix}\mapsto \begin{pmatrix} z_1 & z_2 \end{pmatrix} \left(\begin{smallmatrix} a & b \\ c & d\end{smallmatrix}\right)$. 
When the matrix $\nu = \left( \begin{smallmatrix} a & b \\ c & d\end{smallmatrix}\right)$ has positive determinant it takes the upper and lower half-planes to themselves and the formula is 
$z \mapsto \frac{az+c}{bz+ d}$. 
When the matrix $\nu$ has negative determinant the formula is $z \mapsto \frac{a \overline{z}+c}{b \overline{z}+ d}$.
 



\smallskip
The \emph{Teichm\"uller flow} is given by the action  of the $1$-parameter subgroup $g_t$ of $SL(2,\mathbb{R})$ given by the diagonal matrices $$g_t : = \begin{pmatrix} e^{t/2} & 0 \\  0 & e^{-t/2}  \end{pmatrix}, \qquad t \in \mathbb{R},$$ on ${\mathcal{\tilde M}}_{A}(S)$. 
If we project ${\mathcal{\tilde M}}_{A}(S)$ to ${\mathcal{\tilde M}}_{I}(S)$ by sending a triple to its isometry class and  using the identification  ${\mathcal{\tilde M}}_{I}(S)$ with $ \mathbb{H}$ described in \S~\ref{Teichdisksec}, then the Teichm\"uller flow corresponds to the hyperbolic geodesic flow on $T_1 \mathbb{H}$, i.e.~orbits of the $g_t$-action on ${\mathcal{\tilde M}}_{A}(S)$ project to geodesics in $\mathbb{H}$ parametrized at unit speed. We call a $g_t$-orbit in ${\mathcal{\tilde M}}_{A}(S)$ (or, under the identifications, in $T_1 \mathbb{D}$) a \emph{Teichm\"uller geodesic}. 

\smallskip
The quotient of ${\mathcal{\tilde M}}_{I}(S)$ by the natural right action of the Veech group $V(S)$ is the moduli space of unmarked translation surfaces, which we call ${\mathcal{M}}_{I}(S) = {\mathcal{\tilde M}}_{I}(S) / V(S)$. This space is usually called the \emph{Teichm\"uller curve associated to $S$}, but since since we allow orientation-reversing automorphisms, this quotient might be a surface with boundary, so the term  \emph{Teichm\"uller orbifold associated to $S$} is  more appropriate. 
We denote by ${\mathcal{ M}}_{A}(S)$ the quotient ${\mathcal{\tilde M}}_{A}(S)/ V(S)$ of ${\mathcal{\tilde M}}_{A}(S)$ by the right action of the Veech group (this space is a four-fold cover of the tangent bundle to ${\mathcal{M}}_{I}(S)$ in the sense of orbifolds, see Lemma~2.5 in \cite{SU2}).  
 The Teichm\"uller flow on ${\mathcal{ M}}_{A}(S)$ can be identified with the geodesic flow on the Teichm\"uller orbifold, which, in the particular case where the space ${\mathcal{M}}_{I}(S)$ is a geodesic polygon in the hyperbolic plane, is just the \emph{hyperbolic billiard flow} on the polygon. 

\subsection{Veech group action on a tessellation of the Teichm\"uller disk of a \bm surface} \label{sec:tessellation}
Let $\M mn$ be the $(m,n)$ \bm translation surface. We recall from \S~\ref{veechofbm} that the Veech group of $\M mn$ (as well as the Veech group of the dual surface $\M nm$) is isomorphic to the $(m,n,\infty)$ triangle group or it has index $2$ in it (when $n,m $ are both even, see \cite{Hooper}). Thus, the fundamental domain for the action described above of the Veech group on the Teichm\"uller disk is a hyperbolic triangle whose angles are $\pi/n,\pi/m$, and $0$. The action of the Veech group can be easily visualized by considering a tessellation of the hyperbolic plane by $(m,n,\infty)$ triangles, as shown in Figure \ref{hypdisk1} for the $(3,4)$ \bm surface. In this example, the tessellation  consists of triangles whose angles are $\pi/3,\pi/4$, and $0$. The rotational symmetries of order $3$ and $4$ appear clearly at alternating interior vertices. Triangles in the tessellation can be grouped to get a tessellation into hyperbolic polygons which are either $2m$-gons or $2n$-gons. The $2m$-gons (respectively  the $2n$-gons) have as a center an elliptic point of order $m$ (respectively $n$) and have exactly $m$ (respectively $n$) ideal vertices. For example, the tessellation  in Figure \ref{hypdisk1} contains a supertessellation by octagons with four ideal vertices and hexagons with three ideal vertices.

\begin{figure}[!h] 
\centering
\includegraphics[width=.4\textwidth]{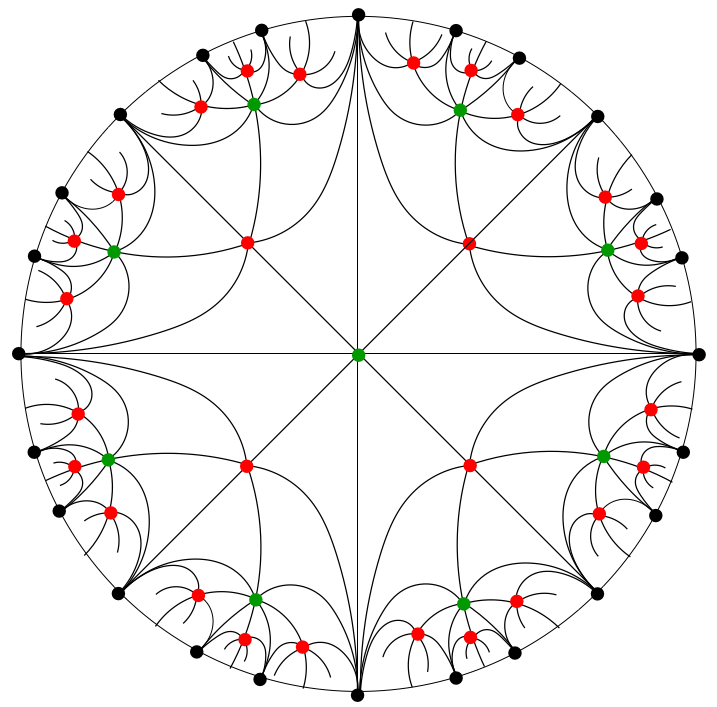}
\hspace{1cm}
\includegraphics[width=.4\textwidth]{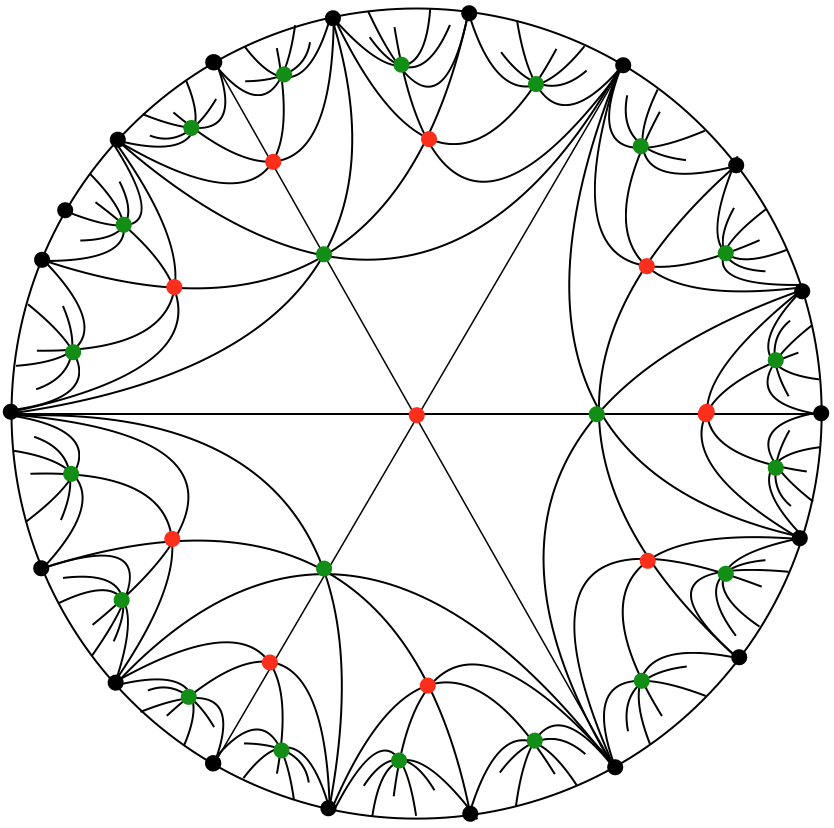}
\begin{quote}\caption{The first four steps of the tessellation of the hyperbolic disk by $(3,4,\infty)$ triangles, with $\M 34$ or $\M 43$ in the center respectively. Angles of $\pi/3, \pi/4$ and $0$ are indicated by red, green and black dots, respectively. \label{hypdisk1}} \end{quote}
\end{figure}

If we consider the Teichm\"uller disk of $\M mn$ pointed at $\M mn$, i.e. we choose the center of the disk $\Disk$ to represent the base triple $id: \M mn \to \M mn$, 
the rotation of order $\pi/n$ on the plane acts as a rotation by angle $2\pi/n$ of the Teichm\"uller disk. On the other hand, if we center the Teichm\"uller disk at  $\M nm$, i.e. we choose the center of the disk $\Disk$ to represent the base triple $id: \M nm \to \M nm$ and mark triples by $\M mn$,  the rotation of order $\pi/m$ acts a a rotation by an angle $2\pi/m$ of the Teichm\"uller disk.  The derivative $\derAD m n$  of the affine diffeomorphism $\AD m n$ (described in \S~\ref{sec:affine}) acts on the right on $\Disk$ by mapping the center of the disk, which in this case is a center of an ideal $2n$-gon, into the center of an ideal $2m$-gon.  Thus, the elliptic element of order $2m$ in the Veech group of $\M mn$ can be obtained by conjugating the rotation $\rho_m$ by an angle $\pi/m$ acting on $\M nm$ by  the derivative $\derAD m n$ of the affine diffeomorphism  $\AD m n$ sending $\M mn$ to the dual surface $\M nm $ (described in \S~\ref{sec:affine}), i.e. it has the form ${\derAD mn} ^{-1}\rho_m \derAD mn$. Finally, the parabolic element which generates parabolic points in the tessellation is the shear automorphism from $\M mn$ to itself given by the composition $\shear nm  \shear mn $ of the shearing matrices defined in \S~\ref{flipandshears}, see \eqref{def:generalmatrices}. 
We remark also that all reflections  $\refl i m n$ for $0\leq i \leq n$ defined in \S~\ref{sec:normalization} (see Definition \ref{def:reflections}) belong to the Veech group of $\M mn$. Each of them acts on the Teichm\"uller disk as a reflection at one of the hyperbolic diameters which are diagonals of the central $2n$-gon. 

\subsubsection*{The tree of renormalization moves.}
We now define a bipartite tree associated to the  tessellations of the disk described above. Paths in this tree  will prove helpful in visualizing and describing the possible sequences of renormalization moves. 
Consider the graph in the hyperbolic plane which has a bipartite set of vertices $V=V_m \cup V_n$ where vertices in $V_m$, which we will call $m$-vertices, are in one-to-one correspondence with centers of ideal $2m$-gons of the tessellation, while  vertices in $V_n$, called $n$-vertices, are in one-to-one correspondence with centers of ideal $2n$-gons. Edges connect vertices in $V_m$ with vertices in $V_n$, and there is a vertex connecting an $m$-vertex to an $n$-vertex if and only if the corresponding $2m$-gon and $2n$-gon share a side.  The graph can be naturally
embedded in $\Disk$, so that vertices in $V_m$ (respectively $V_n$) are  centers of $2m$-gons (respectively $2n$ gons) in the tessellation and each edge is realized by a hyperbolic geodesic segment, i.e. by the side of a triangle in the tessellation which connects the center of an $2m$-gon with the center of an adjacent $2n$-gon. We will call $\Tree{m}{n}$ the embedding of the graph in the tessellation associated to $\M mn$, i.e. the embedding such that the center of the disk is a vertex of order $n$ (which will be the root of the tree). Examples of the embedded graph $\Tree{m}{n}$ are given in Figure~\ref{hypdisk} for $m=3$, $n=4$ and for $m=4$, $n=3$.

\begin{figure}[!h] 
\centering
\includegraphics[width=.4\textwidth]{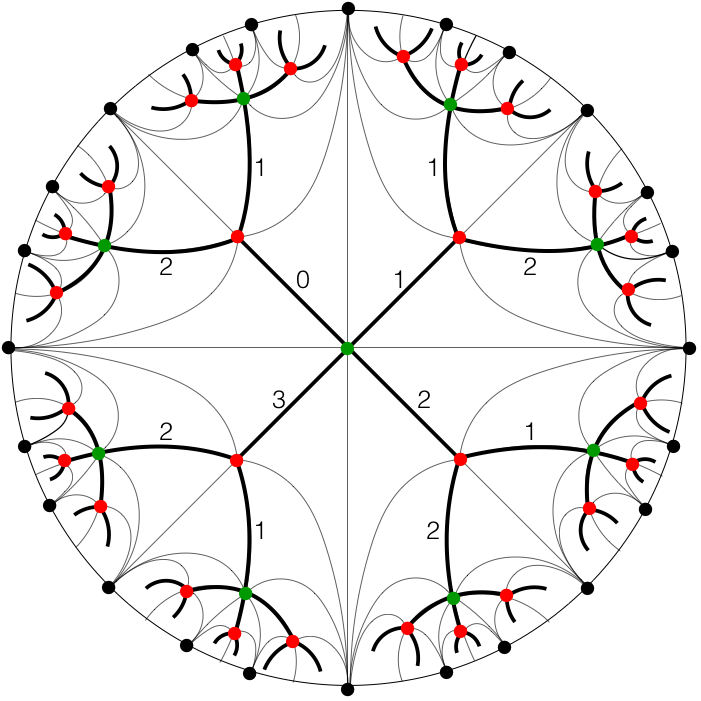}
\hspace{1cm}
\includegraphics[width=.4\textwidth]{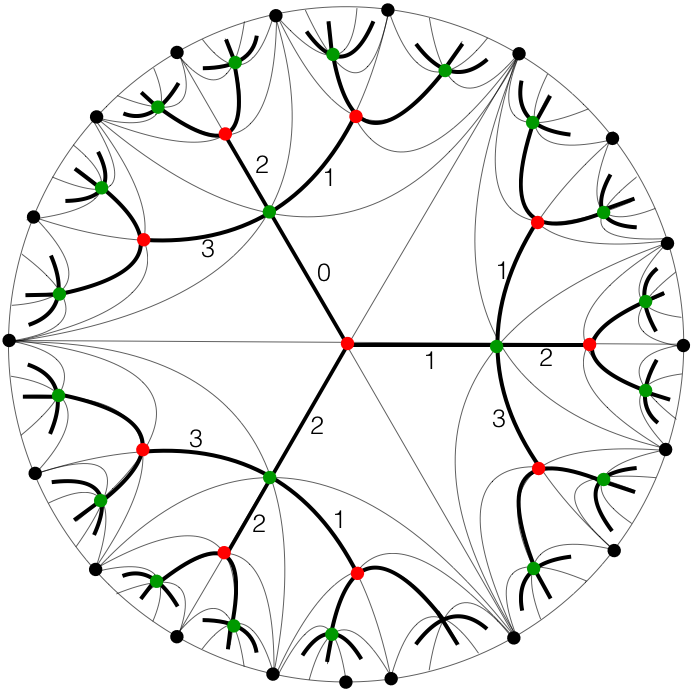}
\begin{quote}\caption{The tree $\Tree{m}{n}$ for the tessellation associated to $\M34$ and to $\M 43$ with the labels of the first two generations of edges. \label{hypdisk}} \end{quote}
\end{figure}

One can  see that the graph $\Tree{m}{n}$ is a bipartite \emph{tree}, with a root in the center of the disk.   We define the \emph{level $k$} of the tree to be composed of all vertices which have distance $k$ from the root, where the distance here is the natural  distance on a graph which gives distance $1$ to vertices which are connected by an edge. For $k\ge 1$ we call \emph{edges of level $k$} the edges which connect a vertex of level $k-1$ with a vertex of level $k$. 

\subsubsection*{Labeling of the tree} 
Let us now describe how to {label  the edges of the tree} $\Tree{m}{n}$ so that the labels will code renormalization moves. 
{ We first remark  that edges of level $1$ are in one-to-one correspondence with the $n$ \emph{sectors} $\Sec i m n$ for $0 \leq i \leq n-1$ defined in \S~\ref{sec:transition_diagrams} (see Definition \ref{sectordef}) as follows. Consider all the points on $\partial \mathbb{D}$ that are endpoints of paths on the tree that start with a given edge $e$ of level $1$. These give an arc on $\partial \mathbb{D}$, which,  via the identification of  $\partial \Disk$ with $\RP$ described in \S~\ref{Teichdisksec}, maps to one of the sectors $\Sec i m n$ for $0 \leq i \leq n-1$. Thus, we label by $i$ the edge $e$ of level 1 that corresponds to the sector $\Sec i m n$.  }

We remark now that the right action of the derivative $\derAD mn $ of the affine diffeomorphism $\AD mn $ (which, we recall, was described in \S~\ref{sec:affine}) maps the level $1$ edge labeled by $0$ in $\Tree{m}{n}$ to the level $1$ edge labeled by $0$ in $\Tree{n}{m}$ flipping its orientation, in particular by mapping 
 the center of the disk (i.e. the root of $\Tree{m}{n}$) to the endpoint $v_0$ of the edge of level $1$ labeled by $0$ in $\Tree{n}{m}$. Thus, the inverse $({\derAD nm })^{-1}$ sends the endpoint $v_0$ of the edge of level $1$ labeled by $0$ in $\Tree{m}{n}$ to the root of $\Tree{n}{m}$ in the center of the disk (and maps  the  $2m$-gon which has $v_0$ as a center in the tessellation for $\M mn $ to the  central $2m$-gon in the tessellation for the dual surface $\M nm$). For example,  $(\derAD 34)^{-1}$ maps the hexagon which has as center the red endpoint of the $0$-edge of level $1$ in the left disk tessellation in Figure \ref{hypdisk}  to the central hexagon in the right disk tessellation in the same Figure \ref{hypdisk}.   Since the edges of level $1$ of $\Tree{n}{m}$ are labeled by $0 \leq i \leq m-1$  and $\derAD mn $ maps the $0$-edges of level $1$ of $\Tree{m}{n}$ and $\Tree{n}{m}$ to each other, it follows that $(\derAD mn)^{-1} $ induces a labeling of $m-1$ edges of level $2$ which start from the endpoint of the $0$-edge of level $1$ as follows. One of such edges $e$ is labeled by $1\leq i\leq m$ if $({\derAD mn})^{-1} $ maps $e$ to the edge of level $1$ of $\Tree{n}{m}$ labeled by $i$.  


To label the edges of level $2$ which branch out of the other level $1$ edges, just recall that the reflection $\refl i m n$ (see Definition \ref{def:reflections}) maps the $i$-edge to the $0$ edge,  
 and hence can be used in the same way to induce a labeling of all the edges of level $1$ branching out of the $i$ edge (by labeling an edge by $1\leq j\leq m$ if it is mapped to an edge of level $1$ already labeled by $j$).
The same definitions for the tree embedded in the tessellation for the dual surface $\M nm$ also produce a labeling of the edges of level $1$ and $2$ of  $\Tree{n}{m}$.    We refer to  Figure \ref{hypdisk} for an example of these labelings of edges of level $1$ and $2$ for $\Tree{3}{4}$ and $\Tree{4}{3}$.

\begin{figure}[!h] 
\centering
\includegraphics[width=.4\textwidth]{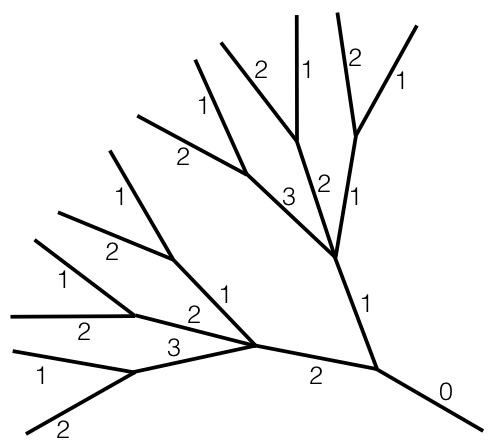}
\begin{quote}\caption{The labeling on a schematic representation of a portion of tree $\Tree{m}{n}$ for  $\M34$, consisting of paths starting with the $0$-edge. \label{treelabeling}} \end{quote}
\end{figure}

We will  now describe how to label all paths which start with the $0$-edge in $\Tree{m}{n}$, since these are the ones needed to describe renormalization on $\M mn$. 
Since we already labeled edges of levels $1$ and $2$ both in $\Tree{m}{n}$ and $\Tree{n}{m}$, to label the edges of level $3$ which belong to paths in $\Tree{m}{n}$ which start with the $0$-edge, we can use that $({\derAD nm })^{-1}$  maps them to edges of level $2$ in $\Tree{n}{m}$ and hence this induces a labeling for them.   For example, see Figure \ref{treelabeling} to see this labeling for $\Tree{3}{4}$.  

We can then transport this labeling of paths made by $3$ edges starting with the $0$-edge to all paths starting with the $0$ edge in $\Tree{m}{n}$ via the action of elements of the Veech group as follows. 
Consider the elements $  (\refl i n m \derAD mn)^{-1 } (\refl j m n \derAD nm)^{-1 } $, $i= 1,\dots, m$, $j= 1,\dots, n$ and consider their right action on the embedded copy of $\Tree{m}{n}$. One can check the following. 

\smallskip
For $1 \leq i\leq m$ and $1\leq j\leq n$, let us denote by $v_{i,j}$ the vertex of level $3$ which is the endpoint of the path starting with the $0$ edge at level $1$, the $i$ edge at level $2$ and the $j$ edge at level $3$.  

\begin{lemma}\label{actionontree}
For every $k\geq 1$, $1\leq i \leq n$ and $1\leq j \leq m$, the right action of the element $   (\refl i n m \derAD mn)^{-1 } (\refl j m n \derAD nm)^{-1 } $ gives a tree automorphism of $\Tree{m}{n}$, 
 which maps 
 $v_{i,j}$ (defined just above)   to the endpoint of the $0$-edge of level $1$ in $\Tree{m}{n}$. 
  The edge ending in $v_{i,j}$ is mapped to the $0$-edge of level $1$. Furthermore, the   edges of level $2k$, $k\geq 2$, which branch out of $v_{i,j}$ are mapped to edges of level $2k-2$ and the edges of level $2k+1$ branching from those to edges of level $2k-1$. 
\end{lemma} 
This Lemma, whose proof we leave to the reader, can be used to define  a labeling of the  edges of paths starting with the $0$-edge by induction on the \emph{level} of the edges. 
One can check by induction by repeatedly applying Lemma \ref{actionontree} that the labeling is defined so that the following holds.

\begin{lemma}\label{actionontreepath}
Consider a  finite path on the tree $\Tree{m}{n}$ starting from the root and ending in an $n$-vertex, whose edge labels are in order $b_0, a_1, b_1, \dots, a_k, b_k$ where $0\leq b_0 \leq n$ and, for any $k\geq 1$, $1\leq a_k \leq m $ and $1\leq b_k \leq n $. Then the  element 
$$ \refl {b_0} m n (\refl {a_1} n m \derAD mn)^{-1 } (\refl {b_1} m n \derAD nm)^{-1 }  \cdots (\refl {a_k} n m \derAD mn)^{-1 } (\refl {b_k} m n \derAD nm)^{-1 } $$ 
 acts on the right by giving a tree automorphism of  $\Tree{m}{n}$ which maps the last edge, i.e. the one labeled by $b_k$, to the $0$-edge of level $1$ and the final vertex of the path to the ending vertex of the $0$-edge of level $1$. 
\end{lemma}

{ Notice that the the labeling satisfies the following description. For $k=2i$ even, finite paths of length $k$ starting from the root are labeled by sequences of the form   $(b_0, a_1, b_1, \dots, b_{i-1}, a_i)$ where $0\leq b_0 \leq n-1$, and for $1\leq j \leq i $, $1\leq a_j \leq m-1$ and  $1\leq b_j \leq n-1$. Sharing the endpoint of such a path as an initial vertex,  there are $n-1$ edges of level ${k+1}$ which  are  labeled by an index $j$ that increases from $1$ to $n-1$  as one moves counterclockwise (see Figure~\ref{hypdisk} and Figure~\ref{treelabeling}). Similarly, for $k=2i+1$ odd, finite paths of length $k$ starting from the root are labeled by sequences of the form $(b_0, a_1, b_1, \dots, b_{i}, a_i)$  where $0\leq b_0 \leq n-1$, and for $1\leq j \leq i $, $1\leq a_j \leq m-1$ and  $1\leq b_j \leq n-1$.  The final arc of any such path is the initial endpoint of $m-1$ vertices of level $k+1$, labeled by an index $i$ that increases from $1$ to $m-1$, also  counterclockwise  (see Figure~\ref{hypdisk} and Figure~\ref{treelabeling}).}



\subsection{Renormalization on the Teichm{\"u}ller disk.}\label{renormcutseq}
In this section we link the sequences of labels of paths on the tree to itineraries of the \bm Farey map and to sequences of admissible diagrams for derivatives of cutting sequences. 

Let $\theta$ be a fixed direction, that we think of as the direction of a trajectory $\tau $ on  $\M mn$. Denote by $\rho_\theta$ the matrix corresponding to counterclockwise rotation by $\theta$ and by $g_t ^{\theta}:=  \rho_{\frac{\pi}{2}-\theta}^{-1} \, g_{t}\ \rho_{\frac{\pi}{2}-\theta} $ a $1$-parameter subgroup  conjugate to the geodesic flow. 
Let us hence consider the \emph{Teichm{\"u}ller geodesic ray}
\be\label{raydef}
\tilde{r}_\theta : = \{ g_t ^{\theta} \cdot M_{m,n} \}_{t\geq 0}, 
\ee\noindent
which, using the identification of ${\mathcal{\tilde M}}_A(S)$ with $T_1 \mathbb{D}$ explained in \S~\ref{Teichdisksec}, corresponds to a geodesic ray in $T_1 \mathbb{D}$.  The projection $r_\theta$ of the Teichm{\"u}ller ray $\tilde{r}_\theta$ to   $\Disk$ is a half ray, starting at the center $0 \in \mathbb{D}$ and converging to the point $e^{(\pi+ 2 \theta)i} \in \partial \mathbb{D}$.  
  In particular, $r_0$ is the ray in $\mathbb{D}$ obtained by intersecting   the negative real axes in $\mathbb{C}$ with $\mathbb{D}$ and  $r_\theta$ is the ray that makes  an angle $2\theta$ (measured clockwise) with the ray $r_0$. 
 
\subsubsection*{Combinatorial geodesics.}
Let us explain how to associate to the geodesic path $r_\theta$  a path $p_\theta$ in the tree $\Tree{m}{n}$,  which we call the \emph{combinatorial geodesic} approximating $r_\theta$.  We say that $\theta$ is a \emph{cuspidal direction} if the ray $r_{\theta}$ converges to a vertex of an ideal polygon of the tessellation. One can show that this is equivalent to saying that the corresponding flow on $\M mn$ consists of periodic trajectories. 
{ Assume first that $\theta$ is \emph{not} a cuspidal direction. In this case, there exists a  
unique continuous semi-infinite path on $\Tree{m}{n}$, which starts at $\underline{0}$ and converges to the endpoint of $r_\theta$ on $\partial \Disk$.   We will call this infinite path on $\Tree{m}{n}$  
 the \emph{combinatorial geodesic} associated to $r_\theta$ and denote it by $p_\theta$. We can think of this path $p_\theta$ as the image of $r_\theta$  under the  retraction that sends the whole disk $\Disk$ onto the deformation retract $\Tree{m}{n}$.   
If $\theta$ is a  cuspidal direction, there exist exactly  two such paths, 
  sharing the cuspidal point as a common limit point, 
  and thus two combinatorial geodesics that approximate $r_\theta$. }

\subsubsection*{Interpretations of the labeling sequences}
Given a direction $\theta$, let $p_\theta$ be a combinatorial geodesic associated to $r_{\theta}$. Let us denote by 
$$ l(p_\theta) = (b_0, a_1, b_1, \dots, a_i, b_i, \dots), \quad \text{where} \ 0\leq b_0 \leq n, \ 1\leq a_k \leq m ,\ 1\leq b_k \leq n, $$ 
the sequence of labels of the edges of $p_\theta$ in increasing order 
. 
This sequence coincides both with the itinerary of $\theta$ under the \bm Farey map (as defined in \S~\ref{sec:itineraries_vs_sectors}), see Proposition~\ref{CFandcuttseq} below, and with the pair of sequences of admissible sectors of any (bi-infinite, non periodic) cutting sequence of a linear trajectory on $\M mn $ in direction $\theta$ (see  Definition~\ref{def:seq_sectors} in \S~\ref{sec:sectors_sequences}), see Corollary \ref{l_vs_diagrams} below.

Let us recall that in \S~\ref{sec:direction_recognition} we have defined a \bm continued fraction expansion, see Definition~\ref{bmCFdef}. Definitions of the labeling of the tree are given so that the following holds:
\begin{prop}\label{CFandcuttseq}
If  a non-cuspidal direction $\theta$ has  \bm continued fraction expansion
\be\label{CFexp}
\theta = [b_0; a_1, b_1, a_2, b_2 , \dots ]_{m,n} ,
\ee
then the labeling sequence $l(p_\theta)$ of the unique combinatorial geodesics associated to the the Teich\-m\"uller geodesics  ray $r_{\theta}$ is given by the entries, i.e. 
\bes\label{labelexp} l(p_\theta) = (b_0, a_1, b_1, a_2, b_2 ,  \dots ).
\ees
If $\theta$ is a cuspidal direction,  $\theta$ admits two \bm continued fraction expansions of the form \eqref{CFexp}, which give the labellings of the two combinatorial geodesics approximating $r_\theta$.
\end{prop}

To prove the proposition, 
one defines a renormalization scheme on paths on the tree $\Tree{m}{n}$ (or combinatorial geodesics)  acting by the elements $(\refl i n m \derAD mn)^{-1 } (\refl j m n \derAD nm)^{-1 }$, $1\leq i \leq m$, $1\leq j \leq n$, and show that this renormalization extends to an action on $\partial \Disk$ that can be identified with the action of the \bm Farey map.

As a consequence of Proposition \ref{CFandcuttseq} and the correspondence between itineraries and sequences of admissible sectors given by Proposition \ref{prop:itineraries_vs_sectors}, we hence also have the following.

\begin{corollary}\label{l_vs_diagrams}
Let $w$ be a non-periodic cutting sequence of a bi-infinite linear trajectory on $\M mn$ in a direction $\theta$ in $\Sec 0 mn$. Let 
 $(a_k)_k \in \{1, \dots , m-1\}^\mathbb{N}$ and $(b_k)_k \in \{1, \dots , n-1\}^\mathbb{N}$ be the pair of sequences of admissible sectors associated to $w$ (see Definiton \ref{def:seq_sectors}).   
Then  the labeling $l(p_\theta)$ of the combinatorial geodesic $p_\theta $ approximating $r_\theta$ is  $l(p_\theta)= (0, a_0, b_0, a_1, b_1, \dots )$.
\end{corollary}

\subsubsection*{Derived cutting sequences  and vertices on the combinatorial geodesic.}
The sequence of vertices of the combinatorial geodesic $p_\theta$ has a geometric interpretation which helps to understand derivation on cutting sequences. More precisely, 
if $w$ is a cutting sequence of a trajectory $\tau$ in direction $\theta$, let $r_{\theta}$ be the geodesic ray which contracts the direction $\theta$ given in (\ref{raydef}) and let $p_\theta$ be the associated combinatorial geodesic, i.e. the path on $\Tree{m}{n}$ that we defined above. Recall that given a cutting sequence $w$ on $\M mn $ of a trajectory in direction $\theta$, in \S~\ref{sec:sectors_sequences} we  recursively defined its {sequence of derivatives} $(w^k)_k$ obtained by alternatively deriving it and normalizing it, see Definition~\ref{def:derivatives}. 
These derived sequences can be seen as cutting sequences of the same trajectory with respect to a sequence of polygonal decompositions of $\M mn$ dertermined by the vertices of the combinatorial path $p_\theta$ as explained below.
 
If the label sequence $l(p_{\theta})$ starts with $b_0, a_1, b_1, \dots ,a_l, b_l, \dots$, then for each $k\geq 1$, define the affine diffeomorphisms 
$$
\Psi^k:= \begin{cases} \refl{b_0} mn (\AD n m)^{-1 } \refl {a_1} n m  ( \AD mn)^{-1 } \refl {b_1} m n (\AD n m)^{-1 } \dots  \refl {a_l} n m ( \AD mn)^{-1 }& \text{if }\ k=2l, \\  \refl{b_0}mn ( \AD n m)^{-1 }\refl {a_1} n m ( \AD mn)^{-1 }\refl {b_1} m n ( \AD n m)^{-1 } \dots  \refl {a_l} n m ( \AD mn)^{-1 } \refl {b_l} m n ( \AD n m)^{-1 } & \text{if }\ k=2l+1. \end{cases}
$$
 
We will denote by $\gamma^k$ the derivative of $\Psi^k$.  We claim that $\gamma^k$ 
acts on the right on  $\Disk$ by mapping the $k^{th}$ vertex of $p_{\theta}$ back to the origin. This can be deduced from Lemma \ref{actionontreepath} for  even indices, by remarking that
$\gamma^{2k} \refl {b_k} m n = \left( \refl {b_0} m n (\refl {a_1} n m \derAD mn)^{-1 } (\refl {b_1} m n \derAD nm)^{-1 }  \cdots (\refl {a_k} n m \derAD mn)^{-1 } (\refl {b_k} m n \derAD nm)^{-1 } \right), $
which are the elements considered in Lemma \ref{actionontreepath} and by noticing that the additional reflection $\refl {b_k} m n$ does not change the isometry class of the final vertex. For odd indices, this can be obtained by combining Lemma \ref{actionontreepath} with the description of the action of $\derAD n m$ on the disk. We omit the details. 


 We now remark that  $\Psi^k(\M mn)=\M mn$ when $k$ is even while    $\Psi^k(\M nm)=\M mn$ when $k$ is odd. Let us consider the marked triple $(\Psi^k)^{-1}: \M mn \to \M mn$ for $k$ even or  $(\Psi^k)^{-1}: \M mn \to \M nm$ for $k$ odd. As explained at the beginning of this Appendix (see \S~\ref{Teichdisksec}), this is an affine deformation of $\M mn$ and considering its isometry equivalence class in  ${\mathcal{\tilde M}}_{I}(S)$  we can identify it with a point in the Teichm\"uller disk $\Disk$ centered at $id: \M mn \to \M mn $. The corresponding point  is a vertex level $k$ of $\Tree{m}{n}$, or more precisely, it is the $k^{th}$ vertex in the combinatorial geodesic $p_\theta$. Thus, under the identification of  $\Disk$ with ${\mathcal{\tilde M}}_{I}(S)$, the vertices of the path $p_\theta$ are, in order, the isometry classes of the marked triples $[\Psi^k]$, for $k=1, 2, \dots$. 

One can visualize these affine deformations by a corresponding sequence of polygonal presentations as follows. 
Recall that both $\M mn$ and $\M nm$ are equipped for us with a semi-regular polygonal presentation, whose sides are labeled by the alphabets $\LL mn$ and $\LL nm$ respectively as explained in \S~\ref{howtolabel}. For every $k\geq 1$, let $\mathcal{P}^k$ be the image in $\M mn$ under the affine diffeomorphism $\Psi^k$   of the polygonal presentation of $\M mn$ if $k$ is even or of $\M nm $ if $k$ is odd. This polygonal decomposition $\mathcal{P}^k$ carries furthermore a labeling of its sides by $\LL mn$ or $\LL nm$ (according to the parity of $k$) induced by  $\Psi^k$: if for $k$ even (respectively $k$ odd) a side of $\M mn$ (respectively $\M nm$) is labeled by $i \in \LL mn$ (respectively by $i \in \LL nm$), let us also label by $i$  its image under $\Psi^k$. This gives  a labeling of the sides of $\mathcal{P}^k$ by $\LL mn$ for $k$ even or by $\LL nm$ for $k$ odd,  which we call the \emph{labeling induced by} $\Psi^k$.


Thus the sequence of vertices in $p_\theta$ determines a sequence of affine deformations of $\M mn$ and a sequence $(\mathcal{P}^k)_k$ of labeled polygonal decompositions. The connection between $(\mathcal{P}^k)_k$  and the sequence of derived cutting sequences (see Definition~\ref{def:derivatives}) is the following.	




\begin{prop}\label{wknormalized}
Let $w, \theta$ and $\mathcal{P}^k$ be as above. The $k^{th}$ derived sequence $w^k$ of the cutting sequence $w$ of a trajectory on $\M mn$ is the cutting sequence of the same trajectory  with respect to the labels of the sides of the polygonal decompositions $\mathcal{P}^k$ with the labeling induced by $\Psi^k$.
\end{prop}
\noindent 
{ The proof of Proposition~\ref{wknormalized} (as well as some of the other results stated in this Appendix) can be found in the arXiv preprint version of this paper.}  
Let us remark that  if we think of $\mathcal{P}^k$  as a collection of polygons in $\mathbb{R}^2$  obtained by  linearly deforming the semi-regular polygonal presentation of $\M mn$ if $k$ is even or of $\M nm $ if $k$ is odd by the linear action of 
$\gamma^k$,   as $k$ increases the polygons in these decompositions  become more and more  stretched in the direction $\theta$,  meaning that the directions of the sides of polygons  tend to $\theta$. 
This can be checked by first reflecting by $\refl{b_0}mn$ to reduce to the $b_0=0$ case and then by  verifying that the sector of directions which is the image of $ \Sec 0 mn$ under the projective action of $\gamma^k $  is shrinking to the point corresponding to the line in direction $\theta$. This distortion of the polygons corresponds to the fact that as $k$ increases a fixed trajectory hits the sides of  $\mathcal{P}^k$ less often which is reflected by the fact that in deriving a sequence labels are erased.  

Finally, let us remark that, as was done in \cite{SU2} for the octagon Teichm\"uller disk and octagon Farey map, it is possible to use the hyperbolic picture introduced in this Appendix to define a cross section of the geodesic flow on the Teichm\"uller orbifold of a \bm surface. More precisely, one can consider a section corresponding to geodesics which have as forward endpoint { a point on the arc of $\partial \mathbb{D}$ given by endpoints of paths on the tree starting with the edge labeled by $0$, and a backward endpoint in the complementary arc of $\partial \mathbb{D}$. } The Poincar{\'e} map of the geodesic flow on this section provides a geometric realization of the  natural extension of the \bm Farey map $\FF m n$. More precisely, one can define a \emph{backward} \bm Farey map which can be used to define the natural extension and describes the behavior of the backward endpoint under the Poincar{\'e} map. The natural extension can be then used to explicitly compute an invariant measure  for $\FF mn$ which is absolutely continuous with respect to the Lebesgue measure but infinite. In order to have a finite absolutely continuous invariant measure, one can accelerate branches of $\FF mn$ which correspond to the parabolic fixed points of $\FF mn$ at $0$ and $\theta= \pi/n$. We leave the computations to the interested reader, following the model given by \cite{SU2}.

\subsubsection*{Towards a characterization of cutting sequences on Veech surfaces}
{ We conclude by explaining why we believe that the techniques introduced in this paper might be helpful in trying to characterize cutting sequences on \emph{any} Veech translation surface. 
We recall that in the characterization of Sturmian sequences or of cutting sequences on regular $2n$-gons, derivation and its inverse, namely generation, are combinatorial operations that correspond to the action on cutting sequences induced by (good and carefully chosen) generators of the group of affine diffeomorphisms of the corresponding Veech surface (i.e. the surface obtained from the torus, or the regular $2n$-gon respectively). In our treatment of \bm surfaces, though, the basic operations that we use in order to characterize cutting sequences (e.g. the derivation operators or the substitutions $\sigma_{i,j}^{m,n}$) correspond to \emph{intermediate} affine diffeomorphisms, which are not {auto}morphisms of the Veech surface $\mathscr{M}_{m,n}$ to itself, but map $\mathscr{M}_{m,n}$ to $\mathscr{M}_{n,m}$ and vice-versa. On the tree embedded in the Teichmueller disk of $\M mn $ that we described in this Appendix, these operations are associated  to \emph{edges} of the tree, which connect a vertex with index $m$ (or resp. $n$) to a vertex with valency $n$ (or resp. $m$). Given any Veech surface $M$, one can similarly associate a \emph{tree} of elementary moves, but this tree will not necessarily be \emph{bipartite} as in the case of \bm surfaces. The basic steps one wants to describe combinatorially will hence be operations that send cutting sequences on the (marked) surface corresponding to any vertex of the tree, to cutting sequences of the (marked) surface which correspond to a vertex of the next tree level. We believe that it should be possible to describe these moves using techniques similar to the ones in this paper, and in particular by exploiting cylinder intersection diagrams similar to the Hooper diagrams for \bm surfaces. We hope  to pursue this approach in future work.    
}

\end{document}